\documentclass[11pt]{article}
\usepackage[OT1,T1]{fontenc}
\usepackage{amsmath,amssymb}
\usepackage{mathrsfs}
\usepackage[all]{xy}
\SelectTips{cm}{10}
\usepackage[colorlinks,citecolor=black,linkcolor=black,urlcolor=black,
  pdfpagemode=None,pdfstartview=FitH]{hyperref}
\usepackage{tikz}
\usetikzlibrary{calc}

\renewenvironment{description}{\list{}
  {\setlength\leftmargin{2.5em}\setlength\labelwidth{0pt}
  }}
  {\endlist}

\newtheorem{theorem}{Theorem}[section]
\newtheorem{lemma}[theorem]{Lemma}
\newtheorem{proposition}[theorem]{Proposition}
\newtheorem{corollary}[theorem]{Corollary}
\newenvironment{proof}{\trivlist
  \item[\hskip\labelsep{\itshape Proof.}]\upshape}{\nobreak\noindent
  $\square$\endtrivlist}
\newenvironment{other}[1]{\refstepcounter{theorem}\trivlist
  \item[\hskip\labelsep{\itshape #1~\thesection.\arabic{theorem}.}]
  \upshape}{\endtrivlist\bigbreak}
\numberwithin{equation}{section}

\renewcommand\appendix[1]{\def\thesection{A}
  \setcounter{theorem}0\refstepcounter{section}
  \subsection*{Appendix: #1}}

\DeclareMathOperator\Hom{Hom}
\DeclareMathOperator\End{End}
\DeclareMathOperator\Ext{Ext}
\DeclareMathOperator\ext{ext}
\DeclareMathOperator\Tor{Tor}
\DeclareMathOperator\re{re}
\DeclareMathOperator\im{im}
\DeclareMathOperator\coker{coker}
\DeclareMathOperator\id{id}
\DeclareMathOperator\hd{hd}
\DeclareMathOperator\soc{soc}
\DeclareMathOperator\wt{wt}
\DeclareMathOperator\height{ht}
\DeclareMathOperator\rk{rk}
\DeclareMathOperator\Pol{Pol}
\DeclareMathOperator\Irr{Irr}
\DeclareMathOperator\Rep{Rep}
\DeclareMathOperator\add{add}
\DeclareMathOperator\Sub{Sub}
\DeclareMathOperator\Fac{Fac}
\DeclareMathOperator\Tr{Tr}
\DeclareMathOperator\GL{GL}
\DeclareMathOperator\Card{Card}
\DeclareMathOperator\linspan{span}
\newcommand\dimvec{{\,\underline\dim\,}}
\newcommand\bsm{\begin{smallmatrix}}
\newcommand\esm{\end{smallmatrix}}
\newcommand\out{\mathrm{out}}
\newcommand\iin{\mathrm{in}}
\newcommand\mmod{\text{\upshape-}\mathrm{mod}}

\begin{document}
\title{Affine Mirkovi\'c-Vilonen polytopes}
\author{Pierre Baumann, Joel Kamnitzer and Peter Tingley}
\date{}
\maketitle

\begin{abstract}
\noindent
Each integrable lowest weight representation of a symmetrizable Kac-Moody
Lie algebra $\mathfrak g$ has a crystal in the sense of Kashiwara, which
describes its combinatorial properties. For a given $\mathfrak g$, there
is a limit crystal, usually denoted by $B(-\infty)$, which contains all
the other crystals. When $\mathfrak g$ is finite dimensional, a convex
polytope, called the Mirkovi\'c-Vilonen polytope, can be
associated to each element in $B(-\infty)$. This polytope sits in
the dual space of a Cartan subalgebra of $\mathfrak g$,
and its edges are parallel to the roots of $\mathfrak g$.
In this paper, we generalize this construction to the case where
$\mathfrak g$ is a symmetric affine Kac-Moody algebra. The datum
of the polytope must however be complemented by partitions attached
to the edges parallel to the imaginary root $\delta$. We prove that
these decorated polytopes are characterized by conditions on their
normal fans and on their $2$-faces. In addition, we discuss how our
polytopes provide an analog of the notion of Lusztig datum for affine
Kac-Moody algebras. Our main tool is an algebro-geometric model for
$B(-\infty)$ constructed by Lusztig and by Kashiwara and Saito, based
on representations of the completed preprojective algebra $\Lambda$
of the same type as~$\mathfrak g$. The underlying polytopes in our
construction are described with the help of Buan, Iyama, Reiten and
Scott's tilting theory for the category $\Lambda\mmod$. The partitions
we need come from studying the category of semistable $\Lambda$-modules
of dimension-vector a multiple of~$\delta$.
\end{abstract}

\section{Introduction}
\label{se:Intro}
Let $A$ be a symmetrizable generalized Cartan matrix, with rows and
columns indexed by a set $I$. We denote by $\mathfrak g$ the Kac-Moody
algebra defined by $A$. It comes with a triangular decomposition
$\mathfrak g=\mathfrak n_-\oplus\mathfrak h\oplus\mathfrak n_+$,
with a root system $\Phi$, and with a Weyl group $W$. The simple
roots $\alpha_i$ are indexed by $I$ and the group $W$ is a Coxeter
system, generated by the simple reflections $s_i$. We denote the
length function of $W$ by $\ell:W\to\mathbb N$ and the
set of positive (respectively, negative) roots by $\Phi_+$
(respectively, $\Phi_-$). The root lattice is denoted by
$\mathbb ZI=\bigoplus_{i\in I}\mathbb Z\alpha_i$ and we set
$\mathbb RI=\mathbb ZI\otimes_{\mathbb Z}\mathbb R$. The canonical
pairing between $\mathbb RI$ and its dual $(\mathbb RI)^*$ will
be denoted by angle brackets. Lastly, we denote by
$\mathbb R_{\geq0}I$ the set of linear combinations of the simple
roots with nonnegative coefficients and we set
$\mathbb NI=\mathbb ZI\cap\mathbb R_{\geq0}I$.

\subsection{Crystals}
\label{ss:Crystals}
The combinatorics of the representation theory of $\mathfrak g$ is
captured by Kashiwara's theory of crystals. Let us summarize quickly
this theory; we refer the reader to the nice survey~\cite{Kashiwara95}
for detailed explanations.

A $\mathfrak g$-crystal is a set $B$ endowed with maps $\wt$,
$\varepsilon_i$, $\varphi_i$, $\tilde e_i$ and $\tilde f_i$,
for each $i\in I$, that satisfy certain axioms. This definition is
of combinatorial nature and the axioms stipulate the local behavior
of the structure maps around an element~$b\in B$. This definition
is however quite permissive, so one wants to restrict to crystals
that actually come from representations.

In this respect, an important object is the crystal $B(-\infty)$,
which contains the crystals of all the irreducible lowest weight
integrable representations of $\mathfrak g$ (see Theorem~8.1
in~\cite{Kashiwara95}). This crystal contains a lowest weight element
$u_{-\infty}\in B(-\infty)$ annihilated by all the lowering operators
$\tilde f_i$, and any element of $B(-\infty)$ can be obtained by
applying a sequence of raising operators $\tilde e_i$ to~$u_{-\infty}$.

The crystal $B(-\infty)$ itself is defined as a basis of the
quantum group $U_q(\mathfrak n_+)$ in the limit $q\to0$. Working
with this algebraic construction is cumbersome, and there exist other,
more handy, algebro-geometric or combinatorial models for $B(-\infty)$.

One of these combinatorial models is Mirkovi\'c-Vilonen (MV) polytopes.
In this model, proposed by Anderson~\cite{Anderson03}, one associates
a convex polytope $\Pol(b)\subseteq\mathbb RI$ to each element
$b\in B(-\infty)$. The construction of $\Pol(b)$ is based on the
geometric Satake correspondence. More precisely, the affine
Grassmannian of the Langlands dual of $\mathfrak g$ contains
remarkable subvarieties, called MV cycles after Mirkovi\'c and
Vilonen~\cite{MirkovicVilonen04}. There is a natural bijection
$b\mapsto Z_b$ from $B(-\infty)$ onto the set of all MV cycles
\cite{BravermanFinkelbergGaitsgory06,BravermanGaitsgory01,
GaussentLittelmann05}, and $\Pol(b)$ is simply the image of $Z_b$
by the moment map.

Using Berenstein and Zelevinsky's work~\cite{BerensteinZelevinsky01},
the second author showed in~\cite{Kamnitzer10} that these MV
polytopes can be described in a completely combinatorial fashion:
these are the convex lattice polytopes whose normal fan is a coarsening
of the Weyl fan in the dual of $\mathbb RI$, and whose $2$-faces have
a shape constrained by the tropical Pl\"ucker relations. In addition,
the length of the edges of $\Pol(b)$ is given by the Lusztig data of
$b$, which indicate how $b$, viewed as a basis element of
$U_q(\mathfrak g)$ at the limit $q\to0$, compares with the PBW bases.

\subsection{Generalization to the affine case}
This paper aims at generalizing this model of MV
polytopes to the case where $\mathfrak g$ is an affine Kac-Moody
algebra.

Obstacles pop up when one tries to generalize the above constructions
of $\Pol(b)$ to the affine case. Despite difficulties in defining the
double-affine Grassmannian, the algebro-geometric model of $B(-\infty)$
using MV cycles still exists in the affine case, thanks to Braverman,
Finkelberg and Gaitsgory's work~\cite{BravermanFinkelbergGaitsgory06};
however, there is no obvious way to go from MV cycles to MV polytopes.

On the algebraic side, several PBW bases for $U_q(\mathfrak n_+)$
have been defined in the affine case by Beck~\cite{Beck94}, Beck
and Nakajima~\cite{BeckNakajima04}, and Ito~\cite{Ito10}, but
the relationship between the different Lusztig data they provide
has not been studied\footnote{Recently, Muthiah and Tingley
\cite{MuthiahTingley13} have considered this problem in the case
$\mathfrak g=\widehat{\mathfrak{sl}_2}$. They have shown that the
resulting combinatorics matches that produced in the present paper,
in the sense that the MV polytopes coming from the Lusztig data
provided by the PBW bases match those defined here. It should be
easy to extend this result to the case of an arbitrary symmetric
affine Kac-Moody algebra.}.

As recalled above, in finite type, the normal fan of an MV polytope
is a coarsening of the Weyl fan, so the facets of an MV polytope are
orthogonal to the rays in the Weyl fan. Therefore an MV polytope is
determined just by the position of theses facets, which form a set of
numerical values dubbed ``Berenstein-Zelevinsky (BZ) data''. In the
case $\mathfrak g=\widehat{\mathfrak{sl}_n}$, a combinatorial model
for an analog of these BZ data was introduced by Naito, Sagaki, and
Saito~\cite{NaitoSagakiSaito12,NaitoSagakiSaito13}. Later, Muthiah
\cite{Muthiah11} related this combinatorial model to the geometry of
the MV cycles. However, the complete relationship between this
combinatorial model and our affine MV polytopes is not yet clear.

\subsection{The preprojective model}
\label{ss:PrepMod}
Due to the difficulties in the MV cycle and PBW bases models,
we are led to use a third construction of $\Pol(b)$,
recently obtained by the first two authors for the case of a
finite dimensional $\mathfrak g$ \cite{BaumannKamnitzer12}.
This construction uses a geometric model for $B(-\infty)$
based on quiver varieties, which we now recall.

This model exists for any Kac-Moody algebra $\mathfrak g$ (not
necessarily of finite or affine type) but only when the
generalized Cartan matrix $A$ is symmetric. Then $2\id-A$ is the
incidence matrix of the Dynkin graph $(I,E)$; here our index set
$I$ serves as the set of vertices and $E$ is the set of edges.
Choosing an orientation of this graph yields a quiver $Q$, and one
can then define the completed preprojective algebra $\Lambda$ of $Q$.

A $\Lambda$-module is an $I$-graded vector space equipped with
linear maps. If the dimension-vector is given, we can work with a
fixed vector space; the datum of a $\Lambda$-module then amounts
to the family of linear maps, which can be regarded as a point of
an algebraic variety. This variety is called Lusztig's nilpotent
variety; we denote it by $\Lambda(\nu)$, where $\nu\in\mathbb NI$
is the dimension-vector. Abusing slightly the language, we often
view a point $T\in\Lambda(\nu)$ as a $\Lambda$-module.

For $\nu\in\mathbb NI$, let $\mathfrak B(\nu)$ be the set
of irreducible components of $\Lambda(\nu)$. We set
$\mathfrak B=\bigsqcup_{\nu\in\mathbb NI}\mathfrak B(\nu)$.
In~\cite{Lusztig90b}, Lusztig endows $\mathfrak B$ with a
crystal structure, and in~\cite{KashiwaraSaito97},
Kashiwara and Saito show the existence of an isomorphism of
crystals $b\mapsto\Lambda_b$ from $B(-\infty)$ onto $\mathfrak B$.
This isomorphism is unique since $B(-\infty)$ has no non-trivial
automorphisms.

Given a finite-dimensional $\Lambda$-module $T$, we can consider
the dimension-vectors of the $\Lambda$-submodules of $T$; they
are finitely many, since they belong to a bounded subset of the
lattice $\mathbb ZI$. The convex hull in $\mathbb RI$ of these
dimension-vectors will be called the Harder-Narasimhan (HN)
polytope of $T$ and will be denoted by $\Pol(T)$.

The main result of \cite{BaumannKamnitzer12} is equivalent to
the following statement: if $\mathfrak g$ is finite dimensional,
then for each $b\in B(-\infty)$, the set
$\{T\in\Lambda_b\mid\Pol(T)=\Pol(b)\}$ contains a dense open
subset of $\Lambda_b$. In other words, $\Pol(b)$ is the general
value of the map $T\mapsto\Pol(T)$ on $\Lambda_b$.

This result obviously suggests a general definition for MV
polytopes. We will however see that for $\mathfrak g$ of affine
type, another piece of information is needed to have an complete
model for $B(-\infty)$; namely, we need to equip each polytope
with a family of partitions. Our task now is to explain what our
polytopes look like, and where these partitions come from.

\subsection{Faces of HN polytopes}
\label{ss:FacesHNPol}
Choose a linear form $\theta:\mathbb RI\to\mathbb R$ and let
$\psi_{\Pol(T)}(\theta)$ denote the maximum value of $\theta$
on $\Pol(T)$. Then $P_\theta=\{x\in\Pol(T)\mid
\langle\theta,x\rangle=\psi_{\Pol(T)}(\theta)\}$ is a face of
$\Pol(T)$. Moreover, the set of submodules $X\subseteq T$ whose
dimension-vectors belong to $P_\theta$ has a smallest element
$T_\theta^{\min}$ and a largest element $T_\theta^{\max}$.

The existence of $T_\theta^{\min}$ and $T_\theta^{\max}$ follows
from general considerations: if we define the slope of a finite
dimensional $\Lambda$-module $X$ as $\langle\theta,\dimvec
X\rangle/\dim X$, then $T_\theta^{\max}/T_\theta^{\min}$ is the
semistable subquotient of slope zero in the Harder-Narasimhan
filtration of~$T$. Introducing the abelian subcategory
$\mathscr R_\theta$ of semistable $\Lambda$-modules of slope
zero, it follows that, for each submodule $X\subseteq T$,
$$\dimvec(X)\in P_\theta\ \Longleftrightarrow\
\Bigl(T_\theta^{\min}\subseteq X\subseteq T_\theta^{\max}\quad
\text{and}\quad X/T_\theta^{\min}\in\mathscr R_\theta\Bigr).$$
In other words, the face $P_\theta$ coincides with the HN
polytope of $T_\theta^{\max}\strut/T_\theta^{\min}$, computed
relative to the category $\mathscr R_\theta$, and shifted by
$\dimvec T_\theta^{\min}$.

Our aim now is to describe the normal fan to $\Pol(T)$, that
is, to understand how $T_\theta^{\min}$, $T_\theta^{\max}$
and $\mathscr R_\theta$ depend on $\theta$. For that, we need
tools that are specific to preprojective algebras.

\subsection{Tits cone and tilting theory}
\label{ss:TitsTilt}
One of these tools is Buan, Iyama, Reiten and Scott's tilting ideals
for $\Lambda$~\cite{BuanIyamaReitenScott09}. Let $S_i$ be the simple
$\Lambda$-module of dimension-vector $\alpha_i$ and let $I_i$ be its
annihilator, a one-codimensional two-sided ideal of $\Lambda$. The
products of these ideals $I_i$ are known to satisfy the braid
relations, so to each $w$ in the Weyl group of $\mathfrak g$, we can
attach a two-sided ideal $I_w$ of $\Lambda$ by the rule
$I_w=I_{i_1}\cdots I_{i_\ell}$, where $s_{i_1}\cdots s_{i_\ell}$
is any reduced decomposition of $w$. Given a finite-dimensional
$\Lambda$-module $T$, we denote the image of the evaluation map
$I_w\otimes_\Lambda\Hom_\Lambda(I_w,T)\to T$ by $T^w$.

Recall that the dominant Weyl chamber $C_0$ and the Tits cone
$C_T$ are the convex cones in the dual of $\mathbb RI$ defined as
$$C_0=\{\theta\in(\mathbb RI)^*\mid\forall
i\in I,\;\langle\theta,\alpha_i\rangle>0\}\quad\text{and}\quad
C_T=\bigcup_{w\in W}w\,\overline{C_0}.$$

We will show the equality $T_\theta^{\min}=T_\theta^{\max}=T^w$
for any finite dimensional $\Lambda$-module $T$, any $w\in W$ and
any linear form $\theta\in wC_0$. This implies that $\dimvec T^w$
is a vertex of $\Pol(T)$ and that the normal cone to $\Pol(T)$ at
this vertex contains $wC_0$. This also implies that $\Pol(T)$ is
contained in
$$\{x\in\mathbb RI\mid\forall\theta\in wC_0,\;
\langle\theta,x\rangle\leq\langle\theta,\dimvec T^w\rangle\}=
\dimvec T^w-w\bigl(\mathbb R_{\geq0}I\bigr).$$

When $\theta$ runs over the Tits cone, it generically belongs to a
chamber, and we have just seen that in this case, the face $P_\theta$
is a vertex. When $\theta$ lies on a facet, $P_\theta$ is an edge
(possibly degenerate). More precisely, if $\theta$ lies on the facet
that separates the chambers $wC_0$ and $ws_iC_0$, with say
$\ell(ws_i)>\ell(w)$, then
$(T_\theta^{\min},T_\theta^{\max})=(T^{ws_i},T^w)$. Results in
\cite{AmiotIyamaReitenTodorov10} and \cite{GeissLeclercSchroer11}
moreover assert that $T^w/T^{ws_i}$ is the direct sum of a finite
number of copies of the $\Lambda$-module $I_w\otimes_\Lambda S_i$.

There is a similar description when $\theta$ is in $-C_T$; here
the submodules $T_w$ of $T$ that come into play are the kernels
of the coevaluation maps $T\to\Hom_\Lambda(I_w,I_w\otimes_\Lambda T)$,
where again $w\in W$.

\subsection{Imaginary edges and partitions (in affine type)}
\label{ss:ImagEdgesPart}
From now on in this introduction, we focus on the case where
$\mathfrak g$ is of symmetric affine type, which in particular
implies $\mathfrak g$ is of untwisted affine type.

The root system for $\mathfrak g$ decomposes into real and imaginary
roots $\Phi=\Phi^{\re}\sqcup(\mathbb Z_{\neq0}\,\delta)$; the real
roots are the conjugate of the simple roots under the Weyl
group action, whereas the imaginary roots are fixed under
this action. The Tits cone is $C_T=\{\theta:\mathbb RI\to\mathbb
R\mid\langle\theta,\delta\rangle>0\}\cup\{0\}$.

We set $\mathfrak t^*=\mathbb RI/\mathbb R\delta$. The projection
$\pi:\mathbb RI\to\mathfrak t^*$ maps $\Phi^{\re}$ onto the
``spherical'' root system $\Phi^s$, whose Dynkin diagram is
obtained from that of $\mathfrak g$ by removing an extending
vertex. The rank of $\Phi^s$ is $r=\dim\mathfrak t^*$, which is
also the multiplicity of the imaginary roots.

The vector space $\mathfrak t$ identifies with the hyperplane
$\{\theta:\mathbb RI\to\mathbb R\mid\langle\theta,\delta\rangle=0\}$
of the dual of $\mathbb RI$. The root system
$\Phi^s\subseteq\mathfrak t^*$ defines an hyperplane arrangement in
$\mathfrak t$, called the spherical Weyl fan. The open cones in this
fan will be called the spherical Weyl chambers. Together, this fan
and the hyperplane arrangement that the real roots define in
$C_T\cup(-C_T)$ make up a (non locally finite) fan in the dual of
$\mathbb RI$, which we call the affine Weyl fan and which we denote
by~$\mathscr W$.

Each set of simple roots in $\Phi^s$ is a basis of $\mathfrak t^*$;
we can then look at the dual basis in $\mathfrak t$, whose elements
are the corresponding fundamental coweights. We denote by $\Gamma$
the set of all fundamental coweights, for all possible choices
of simple roots. Elements in $\Gamma$ are called spherical chamber
coweights; the rays they span are the rays of the spherical Weyl fan.

Now take a $\Lambda$-module $T$. As we saw in the previous section,
the normal cone to $\Pol(T)$ at the vertex $\dimvec T^w$ (respectively,
$\dimvec T_w$) contains $wC_0$ (respectively, $-w^{-1}C_0$).
Altogether, these cones form a dense subset of the dual of $\mathbb RI$:
this leaves no room for other vertices. This analysis also shows that
the normal fan to $\Pol(T)$ is a coarsening of $\mathscr W$.

Thus, the edges of $\Pol(T)$ point in directions orthogonal to
one-codimensional faces of $\mathscr W$, that is, parallel to roots.
In the previous section, we have described the edges that point
in real root directions. We now need to understand the edges that
are parallel to $\delta$. We call these the imaginary edges.

More generally, we are interested in describing the faces
parallel to $\delta$. Let us pick $\theta\in\mathfrak t$
and let us look at the face $P_\theta=\{x\in\Pol(T)\mid
\langle\theta,x\rangle=\psi_{\Pol(T)}(\theta)\}$.
As we saw in section~\ref{ss:FacesHNPol}, this face is the HN
polytope of $T_\theta^{\max}/T_\theta^{\min}$, computed relative
to the category~$\mathscr R_\theta$. It turns out that
$T_\theta^{\max}/T_\theta^{\min}$ and $\mathscr R_\theta$ only
depend on the face $F$ of the spherical Weyl fan to which $\theta$
belongs. We record this fact in the notation by writing
$\mathscr R_F$ for $\mathscr R_\theta$.

We need one more definition: for $\gamma\in\Gamma$, we say that
a $\Lambda$-module is a $\gamma$-core if it belongs to
$\mathscr R_\theta$ for all $\theta\in\mathfrak t$ sufficiently
close to $\gamma$. In other words, the category of $\gamma$-cores
is the intersection of the categories $\mathscr R_C$, taken over
all spherical Weyl chambers $C$ such that $\gamma\in\overline C$.

For each $\nu\in\mathbb NI$, the set of indecomposable modules is
a constructible subset of $\Lambda(\nu)$. It thus makes sense to
ask if the general point of an irreducible subset of $\Lambda(\nu)$
is indecomposable. Similarly, the set of modules that belong to
$\mathscr R_C$ is an open subset of $\Lambda(n\delta)$, so we
may ask if the general point of an irreducible subset of
$\Lambda(n\delta)$ is in $\mathscr R_C$. In section~\ref{ss:Cores},
we will show the following theorems.

\begin{theorem}
\label{th:IntroCore1}
For each integer $n\geq1$ and each $\gamma\in\Gamma$, there is
a unique irreducible component of $\Lambda(n\delta)$ whose general
point is an indecomposable $\gamma$-core.
\end{theorem}

We denote by $I(\gamma,n)$ this component.

\begin{theorem}
\label{th:IntroCore2}
Let $n$ be a positive integer and let $C$ be a spherical Weyl
chamber. There are exactly $r$ irreducible components of
$\Lambda(n\delta)$ whose general point is an indecomposable
module in~$\mathscr R_C$. These components are the $I(\gamma,n)$,
for~$\gamma\in\Gamma\cap\overline C$.
\end{theorem}

In Theorem~\ref{th:IntroCore2}, the multiplicity $r$ of the root
$n\delta$ materializes as a number of irreducible components.

Now let $b\in B(-\infty)$ and pick $\theta$ in a spherical Weyl
chamber $C$. Let $T$ be a general point of $\Lambda_b$ and let
$X=T_\theta^{\max}/T_\theta^{\min}$, an object in $\mathscr R_C$.
Write the Krull-Schmidt decomposition of $X$ as
$X_1\oplus\cdots\oplus X_\ell$, with $X_1$, \dots, $X_\ell$
indecomposable; then each $X_k$ is in $\mathscr R_C$, so
$\dimvec X_k=n_k\delta$ for a certain integer $n_k\geq1$.
Moreover, it follows from Crawley-Boevey and Schr\"oer's theory
of canonical decomposition~\cite{Crawley-BoeveySchroer02} that
each $X_k$ is the general point of an irreducible component
$Z_k\subseteq\Lambda(n_k\delta)$. Using Theorem~\ref{th:IntroCore2},
we then see that each $Z_k$ is a component $I(\gamma_k,n_k)$ for a
certain $\gamma_k\in\Gamma\cap\overline C$. Gathering the integers
$n_k$ according to the coweights $\gamma_k$, we get a tuple of
partitions $(\lambda_\gamma)_{\gamma\in\Gamma\cap\overline C}$.
In this context, we will show that the partition $\lambda_\gamma$
depends only on $b$ and $\gamma$, and not on the Weyl chamber $C$.

We are now ready to give the definition of the MV polytope of $b$:
it is the datum $\widetilde\Pol(b)$ of the HN polytope $\Pol(T)$,
for $T$ general in $\Lambda_b$, together with the family of
partitions $(\lambda_\gamma)_{\gamma\in\Gamma}$ defined above.

\subsection{$2$-faces of MV polytopes}
\label{ss:TwoFaces}
Let us now consider the $2$-faces of our polytopes $\Pol(T)$. Such a
face is certainly of the form
$$P_\theta=\{x\in\Pol(T)\mid\langle\theta,x\rangle=
\psi_{\Pol(T)}(\theta)\},$$
where $\theta$ belongs to a $2$-codimensional face of $\mathscr W$.
There are three possibilities, whether $\theta$ belongs to $C_T$,
$-C_T$ or $\mathfrak t$.

Suppose first that $\theta\in C_T$. Then the root system
$\Phi_\theta=\Phi\cap(\ker\theta)$ is finite of rank $2$, of
type $A_1\times A_1$ or type $A_2$. More precisely, let
$w\in W$ be of minimal length such that $\theta\in w\,\overline{C_0}$
and let $J=\{i\in I\mid\langle w^{-1}\theta,\alpha_i\rangle=0\}$;
then the element $w^{-1}$ maps $\Phi_\theta$ onto the root
system $\Phi_J=\Phi\cap\mathbb RJ$.
The full subgraph of $(I,E)$ defined by $J$ gives rise to a
preprojective algebra~$\Lambda_J$. The obvious surjective
morphism $\Lambda\to\Lambda_J$ induces an inclusion
$\Lambda_J\mmod\hookrightarrow\Lambda\mmod$, whose image is
the category $\mathscr R_{w^{-1}\theta}$. Further, the tilting
ideals $I_w$ provide an equivalence of categories
$$\xymatrix@C=6em{\mathscr R_{w^{-1}\theta}
\ar@<.6ex>[r]^(.56){I_w\otimes_\Lambda?}&
\ar@<.6ex>[l]^(.44){\Hom_\Lambda(I_w,?)}\mathscr R_\theta},$$
whose action on the dimension-vectors is given by $w$. Putting
all this together, we see that $P_\theta$ is the image under $w$
of the HN polytope of the $\Lambda_J$-module
$X=\Hom_\Lambda(I_w,T_\theta^{\max}/T_\theta^{\min})$.
In addition, genericity is preserved in this construction:
if $T$ is a general point in an irreducible component of a
nilpotent variety for $\Lambda$, then $X$ is a general point
in an irreducible component of a nilpotent variety for $\Lambda_J$.
When $\Phi_J$ is of type $A_2$, this implies that the $2$-face
$P_\theta$ obeys the tropical Pl\"ucker relation
from~\cite{Kamnitzer10}.

A similar analysis can be done in the case where $\theta$ is in $-C_T$.
It then remains to handle the case where $\theta\in\mathfrak t$, that
is, where $\theta$ belongs to a face $F$ of codimension one in the
spherical Weyl fan. Here $\Phi_\theta=\Phi\cap(\ker\theta)$ is an
affine root system of type $\widetilde A_1$. The face $F$ separates
two spherical Weyl chambers of $\mathfrak t$, say $C'$ and $C''$, and
there are spherical chamber coweights $\gamma'$ and $\gamma''$ such that
$\Gamma\cap\overline{C'}=\bigl(\Gamma\cap\overline
F\bigr)\sqcup\{\gamma'\}$ and
$\Gamma\cap\overline{C''}=\bigl(\Gamma\cap\overline
F\bigr)\sqcup\{\gamma''\}$.

Choose $\theta'\in C'$ and~$\theta''\in C''$.
Assume that $T$ is the general point of an irreducible component.
As we saw in section~\ref{ss:ImagEdgesPart}, the modules
$T_{\theta'}^{\max}/T_{\theta'}^{\min}$ and
$T_{\theta''}^{\max}/T_{\theta''}^{\min}$ are then described by
tuples of partitions $(\lambda_\gamma)_{\gamma\in\Gamma\cap
\overline{C'}}$ and $(\lambda_\gamma)_{\gamma\in
\Gamma\cap\overline{C''}}$, respectively. Both these modules
are subquotients of $T_\theta^{\max}/T_\theta^{\min}$,
so this latter contains the information about the partitions
$\lambda_\gamma$ for all $\gamma\in(\Gamma\cap\overline F)
\sqcup\{\gamma',\gamma''\}$.

\begin{theorem}
\label{th:IntroImag2Face}
Let $\overline{P_\theta}$ be the polytope obtained by shortening
each imaginary edge of the $2$-face $P_\theta$ by
$\Bigl(\sum_{\gamma\in\Gamma\cap\overline F}|\lambda_\gamma|\Bigr)\delta$.
Then $\overline{P_\theta}$, equipped with the two partitions
$\lambda_{\gamma'}$ and $\lambda_{\gamma''}$, is an MV polytope
of type~$\widetilde A_1$.
\end{theorem}

A partition can be thought of as an MV polytope of type $\widetilde A_0$,
since the generating function for the number of partitions equals the
graded dimension of the upper half of the Heisenberg algebra.
Thus, the family of partitions
$(\lambda_\gamma)_{\gamma\in\Gamma\cap\overline F}$ can be thought
of as an MV polytope of type $(\widetilde A_0)^{r-1}$. We can
therefore regard the datum of the face $P_\theta$ and of the
partitions $(\lambda_\gamma)_{\gamma\in(\Gamma\cap\overline F)
\sqcup\{\gamma',\gamma''\}}$ as an MV polytope of type
$\widetilde A_1\times\widetilde A_0^{r-1}$.

Theorem~\ref{th:IntroImag2Face} will be proved in
section~\ref{ss:MVPolComp}. Our method is to construct an embedding
of $\Pi\mmod$ into $\mathscr R_\theta$, where $\Pi$ is the completed
preprojective algebra of type $\widetilde A_1$; this embedding
depends on $F$ and its essential image is large enough to capture
a dense open subset in the relevant irreducible component of
Lusztig's nilpotent variety. In this construction, we were inspired
by the work of I.~Frenkel et al.\ \cite{FrenkelMalkinVybornov01} who
produced analogous embeddings in the quiver setting.

So the final picture is the following. Let $\mathcal{MV}$ be the
set of all lattice convex polytopes $P$ in $\mathbb RI$, equipped
with a family of partitions $(\lambda_\gamma)_{\gamma\in\Gamma}$,
such that:
\begin{itemize}
\item
The normal fan to $P$ is a coarsening of the Weyl fan $\mathscr W$.
\item
To each spherical Weyl chamber $C$ corresponds a imaginary edge of $P$;
the difference between the two endpoints of this edge is equal to
$\Bigl(\sum_{\gamma\in\Gamma\cap\overline C}|\lambda_\gamma|\Bigr)\delta$.
\item
A $2$-face of $P$ is an MV polytope of type $A_1\times A_1$, $A_2$
or $\widetilde A_1\times\widetilde A_0^{r-1}$; in the type $A_2$
case, this means that its shape obeys the tropical Pl\"ucker relation.
\end{itemize}

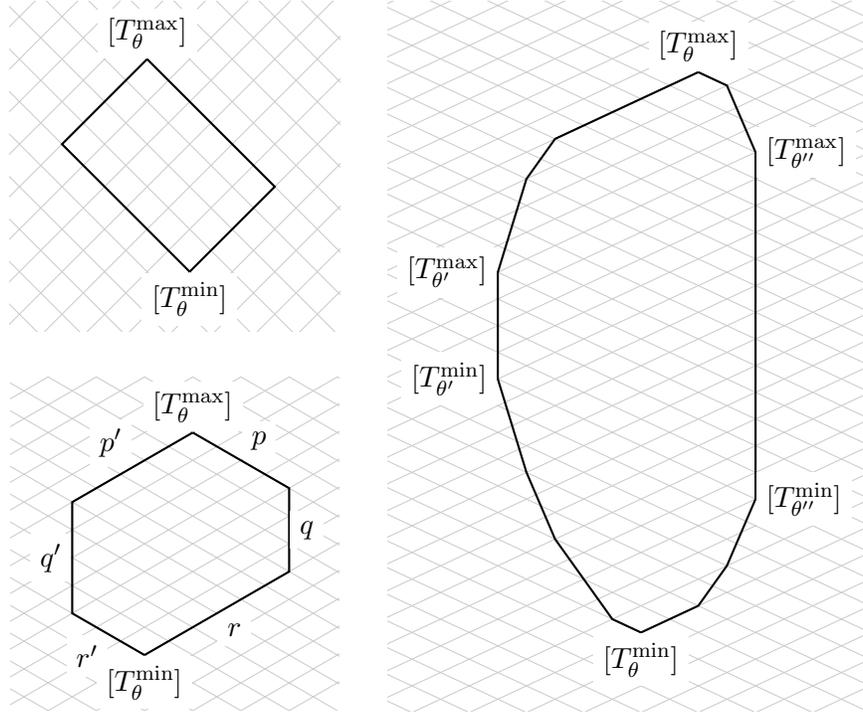
\begin{figure}[!ht]
\begin{center}
\begin{tikzpicture}
\begin{scope}[xshift=.6cm,yshift=5.1cm,scale=.4]
\clip (-6,-2) rectangle (5,9);
\foreach \x in {-8,...,10}
  \draw[color=gray!40] (-45:5) ++(45:\x) -- ++(135:16);
\foreach \y in {-5,...,11}
  \draw[color=gray!40] (-135:8) ++(135:\y) -- ++(45:18);
\draw[thick] (0,0) node[fill=white,below]{$[T_\theta^{\min}]$}
  -- ++(45:4)
  -- ++(135:6) node[fill=white,above]{$[T_\theta^{\max}]$};
\draw[thick] (0,0)
  -- ++ (135:6)
  -- ++(45:4);
\end{scope}
\begin{scope}[scale=.37]
\clip (-4.8,-2) rectangle (7,10);
\foreach \x in {-8,...,15}
  \draw[color=gray!40] (-30:6) ++(30:\x) -- ++(150:19);
\foreach \y in {-5,...,13}
  \draw[color=gray!40] (-150:8) ++(150:\y) -- ++(30:23);
\draw[thick] (0,0) node[fill=white,below]{$[T_\theta^{\min}]$}
  -- ++(30:6) node[fill=white,midway,below right]{$r$}
  -- ++(90:3) node[fill=white,midway,right]{$q$}
  -- ++(150:4) node[fill=white,midway,above right]{$p$}
     node[fill=white,above]{$[T_\theta^{\max}]$};
\draw[thick] (0,0)
  -- ++ (150:3) node[fill=white,midway,below left]{$r'$}
  -- ++(90:4) node[fill=white,midway,left]{$q'$}
  -- ++(30:5) node[fill=white,midway,above left]{$p'$};
\end{scope}
\begin{scope}[xshift=6.6cm,yshift=.3cm,scale=.42]
\clip (-8,-2.5) rectangle (7.2,20);
\foreach \x in {-8,...,28}
  \draw[color=gray!40] (-25:7) ++(25:\x) -- ++(155:35);
\foreach \y in {-7,...,28}
  \draw[color=gray!40] (-155:8) ++(155:\y) -- ++(25:36);
\draw[thick] (0,0) node[fill=white,below]{$[T_\theta^{\min}]$}
  -- ++(25:2)
  -- ++($(25:2)+(155:1)$)
  -- ++($(25:3)+(155:2)$)
     node[fill=white,right]{$[T_{\theta''}^{\min}]$}
  -- ++($(25:13)+(155:13)$)
     node[fill=white,right]{$[T_{\theta''}^{\max}]$}
  -- ++($(25:2)+(155:3)$)
  -- ++(155:1) node[fill=white,above]{$[T_\theta^{\max}]$};
\draw[thick] (0,0)
  -- ++(155:1)
  -- ++($(25:2)+(155:4)$)
  -- ++($(25:2)+(155:3)$)
  -- ++($(25:3)+(155:4)$) node[fill=white,left]{$[T_{\theta'}^{\min}]$}
  -- ++($(25:4)+(155:4)$) node[fill=white,left]{$[T_{\theta'}^{\max}]$}
  -- ++($(25:4)+(155:3)$)
  -- ++($(25:2)+(155:1)$)
  -- ++(25:5);
\end{scope}
\end{tikzpicture}
\end{center}
\caption{Examples of $2$-faces of affine MV polytopes. These faces
are of type $A_1\times A_1$ (top left), $A_2$ (bottom left), and
$\widetilde A_1$ (right). In type $A_2$, the tropical Pl\"ucker
relation is $q'=\min(p,r)$. Note that the edges are parallel to
root directions. In type $\widetilde A_1$, the edges parallel to
the imaginary root $\delta$ (displayed here vertically) must be
decorated with partitions $\lambda_{\gamma'}$ and
$\lambda_{\gamma''}$ which encode the structure of the modules
$T_{\theta'}^{\max}/T_{\theta'}^{\min}$ and
$T_{\theta''}^{\max}/T_{\theta''}^{\min}$ (notation as in the
text above). The lengths of all the edges and the partitions
are subject to relations which are the analogs of the tropical
Pl\"ucker relations; the displayed polygon endowed with the
partitions $\lambda_{\gamma'}=(3,1)$ and
$\lambda_{\gamma''}=(4,2,1,1,1,1,1,1,1)$ satisfies these relations.}
\label{fi:Twofaces}
\end{figure}

At the end of section~\ref{ss:ImagEdgesPart}, we
associated an element $\widetilde\Pol(b)$ of $\mathcal{MV}$
to each $b\in B(-\infty)$.

\begin{theorem}
\label{th:IntroMV}
The map $\widetilde\Pol:B(-\infty)\to\mathcal{MV}$ is bijective.
\end{theorem}

Theorem~\ref{th:IntroMV} will be proved in section~\ref{ss:PfThIntroMV}.
In a companion paper~\cite{BaumannDunlapKamnitzerTingley12}, we
will provide a combinatorial description of MV polytopes of type
$\widetilde A_1$. With that result in hand, the above conditions
provide an explicit characterization of the polytopes
$\widetilde\Pol(b)$.

\subsection{Lusztig data}
\label{ss:LuszData}
As explained at the end of section~\ref{ss:Crystals}, for a finite
dimensional $\mathfrak g$, the MV polytope $\Pol(b)$ of an element
$b\in B(-\infty)$ geometrically encodes all the Lusztig data of $b$.

In more detail, let $N$ be the number of positive roots. Each reduced
decomposition of the longest element $w_0$ of $W$ provides a PBW
basis of the quantum group $U_q(\mathfrak n_+)$, which goes to the
basis $B(-\infty)$ at the limit $q\to0$. To an element $b\in
B(-\infty)$, one can therefore associate many PBW monomials, one
for each PBW basis. In other words, one can associate to $b$ many
elements of $\mathbb N^N$, one for each reduced decomposition of
$w_0$. These elements in $\mathbb N^N$ are called the Lusztig data
of $b$. A reduced decomposition of $w_0$ specifies a path in the
1-skeleton of $\Pol(b)$ that connects the top vertex to the bottom
one, and the corresponding Lusztig datum materializes as the lengths
of the edges of this path.

With this in mind, we now explain that when $\mathfrak g$ is
of affine type, our MV polytopes $\widetilde\Pol(b)$ provide a
fair notion of Lusztig data.

To this aim, we first note that a reasonable analog of the
reduced decompositions of $w_0$ is certainly the notion of
``total reflection order'' (Dyer) or ``convex order'' (Ito), see
\cite{CelliniPapi98,Ito01}. By definition, this is a total order
$\preccurlyeq$ on $\Phi_+$ such that
$$\bigl(\alpha+\beta\in\Phi_+\;\text{ and }\;\alpha\preccurlyeq
\beta\bigr)\ \Longrightarrow\
\alpha\preccurlyeq\alpha+\beta\preccurlyeq\beta.$$
(Unfortunately, the convexity relation implies that
$m\delta\preccurlyeq n\delta$ for any positive integers $m$ and $n$.
We therefore have to accept that $\preccurlyeq$ is only a preorder;
this blemish is however limited to the imaginary roots.)

A convex order $\preccurlyeq$ splits the positive real roots in two
parts: those that are greater than $\delta$ and those that are smaller.
One easily shows that the projection $\pi:\mathbb RI\to\mathfrak t^*$
maps $\{\beta\in\Phi_+\mid\beta\succ\delta\}$ onto a positive
system of~$\Phi^s$. Thus, there exists $\theta\in\mathfrak t$
such that
\begin{equation}
\label{eq:ConvOrderDelta}
\forall\beta\in\Phi_+^{\re},\quad
\beta\succ\delta\ \Longleftrightarrow\
\langle\theta,\beta\rangle>0.
\end{equation}

Given such a convex order $\preccurlyeq$, we will construct a
functorial filtration $(T_{\succcurlyeq\alpha})_{\alpha\in\Phi_+}$
on each finite dimensional $\Lambda$-module $T$, such that each
$\dimvec T_{\succcurlyeq\alpha}$ is a vertex of $\Pol(T)$. The family of
dimension-vectors $(\dimvec T_{\succcurlyeq\alpha})_{\alpha\in\Phi_+}$
are the vertices along a path in the $1$-skeleton of $\Pol(T)$ connecting
the top vertex and bottom vertices. The lengths of the edges in this
path form a family of natural numbers $(n_\alpha)_{\alpha\in\Phi_+}$,
defined by the relation
$\dimvec T_{\succcurlyeq\alpha}/T_{\succ\alpha}=n_\alpha\alpha$.
Further, if we choose $\theta\in\mathfrak t$ satisfying
\eqref{eq:ConvOrderDelta}, then $T_{\succcurlyeq\delta}=T_\theta^{\max}$
and $T_{\succ\delta}=T_\theta^{\min}$.

Fix $b\in B(-\infty)$ and take a general point $T$ in $\Lambda_b$.
Besides the family $(n_\alpha)_{\alpha\in\Phi_+}$ of natural numbers
mentioned just above, we can construct a tuple of partitions
$(\lambda_\gamma)_{\gamma\in\Gamma\cap\overline C}$ by applying the
analysis carried after Theorem~\ref{th:IntroCore2} to the module
$T_\theta^{\max}/T_\theta^{\min}$, where $C$ is the spherical Weyl
chamber containing $\theta$. To $b$, we can thus associate the pair
$\Omega_\preccurlyeq(b)$ consisting in the two families
$(n_\alpha)_{\alpha\in\Phi_+^{\re}}$ and
$(\lambda_\gamma)_{\gamma\in\Gamma\cap\overline C}$. All this
information can be read from $\widetilde\Pol(b)$. We call
$\Omega_\preccurlyeq(b)$ the Lusztig datum of $b$ in direction
$\preccurlyeq$.

Let us denote by $\mathcal P$ the set of all partitions and by
$\mathbb N^{(\Phi_+^{\re})}$ the set of finitely supported families
$(n_\alpha)_{\alpha\in\Phi_+^{\re}}$ of non-negative integers.

\begin{theorem}
\label{th:IntroLusDat}
The map $\Omega_\preccurlyeq:B(-\infty)\to\mathbb N^{(\Phi_+^{\re})}
\times\mathcal P^{\Gamma\cap\overline C}$ is bijective.
\end{theorem}

Theorem~\ref{th:IntroLusDat} will be proved in
section~\ref{ss:LusDatComp}. Let us conclude by a few remarks.
\begin{enumerate}
\item
The MV polytope $\widetilde\Pol(b)$ contains the information of
all Lusztig data of $b$, for all convex orders. This is in complete
analogy with the situation in the case where $\mathfrak g$ is
finite dimensional. The conditions on the $2$-faces given in the
definition of $\mathcal{MV}$ say how the Lusztig datum of $b$ varies
when the convex order changes; they can be regarded as the analog
in the affine type case of Lusztig's piecewise linear bijections.
\item
The knowledge of a single Lusztig datum of $b$, for just one
convex order, allows one to reconstruct the irreducible component
$\Lambda_b$. This fact is indeed an ingredient of the proof of
injectivity in Theorem~\ref{th:IntroMV}.
\item
Through the bijective map $\widetilde\Pol$, the set $\mathcal{MV}$
acquires the structure of a crystal, isomorphic to $B(-\infty)$.
This structure can be read from the Lusztig data. Specifically,
if $\alpha_i$ is the smallest element of the order $\preccurlyeq$,
then $\varphi_i(b)$ is the $\alpha_i$-coordinate of
$\Omega_\preccurlyeq(b)$, and the operators $\tilde e_i$ and
$\tilde f_i$ act by incrementing or decrementing this coordinate.
\item
As mentioned at the beginning of section~\ref{ss:PrepMod}, Beck
in~\cite{Beck94}, Beck and Nakajima in~\cite{BeckNakajima04},
and Ito in~\cite{Ito10} construct PBW bases of $U_q(\mathfrak n_+)$
for $\mathfrak g$ of affine type. An element in one of these
bases is a monomial in root vectors, the product being computed
according to a convex order $\preccurlyeq$. To describe a monomial,
one needs an integer for each real root $\alpha$ and a $r$-tuple
of integers for each imaginary root $n\delta$, so in total,
monomials in a PBW basis are indexed by
$\mathbb N^{(\Phi_+^{\re})}\times\mathcal P^r$. Moreover, such
a PBW basis goes to $B(-\infty)$ at the limit $q\to0$. (This
fact has been established in \cite{BeckChariPressley99} for
Beck's bases, and the result can probably be extended to Ito's
more general bases by using \cite{Lusztig96} or \cite{Saito94}.)
In the end, we get a bijection between $B(-\infty)$ and
$\mathbb N^{(\Phi_+^{\re})}\times\mathcal P^r$. We expect that
this bijection is our map $\Omega_\preccurlyeq$.
\end{enumerate}

\begin{figure}[ht]
\begin{center}
\begin{tikzpicture}
\begin{scope}[scale=.6]
\coordinate(O) at (0,0);
\coordinate(I) at (.2,.3);		
\coordinate(J) at (-1.2,.3);	
\coordinate(K) at (1,.3);		
\coordinate(v1) at ($(O)+0*(I)+0*(J)+0*(K)$);
\coordinate(v2) at ($(O)+1*(I)+0*(J)+0*(K)$);
\coordinate(v3) at ($(O)+3*(I)+2*(J)+0*(K)$);
\coordinate(v4) at ($(O)+3*(I)+4*(J)+0*(K)$);
\coordinate(v5) at ($(O)+1*(I)+4*(J)+0*(K)$);
\coordinate(v6) at ($(O)+0*(I)+3*(J)+0*(K)$);
\coordinate(v7) at ($(O)+1*(I)+4*(J)+2*(K)$);
\coordinate(v8) at ($(O)+1*(I)+2*(J)+2*(K)$);
\coordinate(v9) at ($(O)+6*(I)+5*(J)+3*(K)$);
\coordinate(v10) at ($(O)+9*(I)+9*(J)+7*(K)$);
\coordinate(v11) at ($(O)+7*(I)+7*(J)+4*(K)$);
\coordinate(v12) at ($(O)+6*(I)+7*(J)+3*(K)$);
\coordinate(v13) at ($(O)+7*(I)+9*(J)+5*(K)$);
\coordinate(v14) at ($(O)+6*(I)+9*(J)+5*(K)$);
\coordinate(v15) at ($(O)+6*(I)+9*(J)+7*(K)$);
\coordinate(v16) at ($(O)+7*(I)+9*(J)+8*(K)$);
\coordinate(v17) at ($(O)+9*(I)+9*(J)+8*(K)$);
\coordinate(v18) at ($(O)+6*(I)+7*(J)+7*(K)$);
\coordinate(v19) at ($(O)+7*(I)+7*(J)+7*(K)$);
\coordinate(v20) at ($(O)+6*(I)+5*(J)+5*(K)$);
\filldraw[fill=gray!30,opacity=.8,line join=round]
  (v17) -- (v19) -- (v20) -- cycle;
\filldraw[fill=gray!20,opacity=.8,line join=round]
  (v17) -- (v20) -- (v9) -- (v10) -- cycle;
\filldraw[fill=gray!40,opacity=.8,line join=round]
  (v2) -- (v20) -- (v9) -- (v3) -- cycle;
\filldraw[fill=gray!40,opacity=.8,line join=round]
  (v1) -- (v2) -- (v3) -- (v4) -- (v5) -- (v6) -- cycle;
\filldraw[fill=gray!20,opacity=.8,line join=round]
  (v4) -- (v12) -- (v13) -- (v14) -- (v5) -- cycle;
\filldraw[fill=gray!30,opacity=.8,line join=round]
  (v3) -- (v9) -- (v11) -- (v12) -- (v4) -- cycle;
\filldraw[fill=gray!27,opacity=.8,line join=round]
  (v10) -- (v11) -- (v12) -- (v13) -- cycle;
\filldraw[fill=gray!23,opacity=.8,line join=round]
  (v9) -- (v10) -- (v11) -- cycle;
\filldraw[fill=gray!18,opacity=.8,line join=round]
  (v10) -- (v13) -- (v14) -- (v15) -- (v16) -- (v17) -- cycle;
\filldraw[fill=gray!45,opacity=.8,line join=round]
  (v1) -- (v6) -- (v7) -- (v8) -- cycle;
\filldraw[fill=gray!38,opacity=.8,line join=round]
  (v8) -- (v7) -- (v15) -- (v16) -- (v18) -- cycle;
\filldraw[fill=gray!45,opacity=.8,line join=round]
  (v8) -- (v18) -- (v19) -- (v20) -- (v2) -- (v1) -- cycle;
\filldraw[fill=gray!28,opacity=.8,line join=round]
  (v17) -- (v16) -- (v18) -- (v19) -- cycle;
\filldraw[fill=gray!25,opacity=.8,line join=round]
  (v5) -- (v14) -- (v15) -- (v7) -- (v6) -- cycle;
\draw[ultra thick,line join=round,line cap=butt]
  (v1) -- (v8) -- (v7) -- (v15) -- (v16) -- (v17);
\end{scope}
\begin{scope}[xshift=4.6cm,scale=.6]
\coordinate(O) at (0,0);
\coordinate(I) at (-1.2,.3);	
\coordinate(J) at (.8,.3);		
\coordinate(K) at (.4,.3);		
\coordinate(v1) at ($(O)+0*(I)+0*(J)+0*(K)$);
\coordinate(v2) at ($(O)+1*(I)+0*(J)+0*(K)$);
\coordinate(v3) at ($(O)+3*(I)+2*(J)+0*(K)$);
\coordinate(v4) at ($(O)+3*(I)+4*(J)+0*(K)$);
\coordinate(v5) at ($(O)+1*(I)+4*(J)+0*(K)$);
\coordinate(v6) at ($(O)+0*(I)+3*(J)+0*(K)$);
\coordinate(v7) at ($(O)+1*(I)+4*(J)+2*(K)$);
\coordinate(v8) at ($(O)+1*(I)+2*(J)+2*(K)$);
\coordinate(v9) at ($(O)+6*(I)+5*(J)+3*(K)$);
\coordinate(v10) at ($(O)+9*(I)+9*(J)+7*(K)$);
\coordinate(v11) at ($(O)+7*(I)+7*(J)+4*(K)$);
\coordinate(v12) at ($(O)+6*(I)+7*(J)+3*(K)$);
\coordinate(v13) at ($(O)+7*(I)+9*(J)+5*(K)$);
\coordinate(v14) at ($(O)+6*(I)+9*(J)+5*(K)$);
\coordinate(v15) at ($(O)+6*(I)+9*(J)+7*(K)$);
\coordinate(v16) at ($(O)+7*(I)+9*(J)+8*(K)$);
\coordinate(v17) at ($(O)+9*(I)+9*(J)+8*(K)$);
\coordinate(v18) at ($(O)+6*(I)+7*(J)+7*(K)$);
\coordinate(v19) at ($(O)+7*(I)+7*(J)+7*(K)$);
\coordinate(v20) at ($(O)+6*(I)+5*(J)+5*(K)$);
\filldraw[fill=gray!20,opacity=.8,line join=round]
  (v17) -- (v19) -- (v20) -- cycle;
\filldraw[fill=gray!20,opacity=.8,line join=round]
  (v17) -- (v20) -- (v9) -- (v10) -- cycle;
\filldraw[fill=gray!30,opacity=.8,line join=round]
  (v2) -- (v20) -- (v9) -- (v3) -- cycle;
\filldraw[fill=gray!50,opacity=.8,line join=round]
  (v1) -- (v6) -- (v7) -- (v8) -- cycle;
\filldraw[fill=gray!30,opacity=.8,line join=round]
  (v8) -- (v18) -- (v19) -- (v20) -- (v2) -- (v1) -- cycle;
\filldraw[fill=gray!20,opacity=.8,line join=round]
  (v17) -- (v16) -- (v18) -- (v19) -- cycle;
\filldraw[fill=gray!40,opacity=.8,line join=round]
  (v8) -- (v7) -- (v15) -- (v16) -- (v18) -- cycle;
\draw[ultra thick,line join=round,line cap=butt]
  (v1) -- (v8) -- (v7);
\filldraw[fill=gray!45,opacity=.8,line join=round]
  (v1) -- (v2) -- (v3) -- (v4) -- (v5) -- (v6) -- cycle;
\filldraw[fill=gray!38,opacity=.8,line join=round]
  (v4) -- (v12) -- (v13) -- (v14) -- (v5) -- cycle;
\filldraw[fill=gray!45,opacity=.8,line join=round]
  (v5) -- (v14) -- (v15) -- (v7) -- (v6) -- cycle;
\filldraw[fill=gray!32,opacity=.8,line join=round]
  (v3) -- (v9) -- (v11) -- (v12) -- (v4) -- cycle;
\filldraw[fill=gray!24,opacity=.8,line join=round]
  (v10) -- (v13) -- (v14) -- (v15) -- (v16) -- (v17) -- cycle;
\filldraw[fill=gray!22,opacity=.8,line join=round]
  (v10) -- (v11) -- (v12) -- (v13) -- cycle;
\filldraw[fill=gray!20,opacity=.8,line join=round]
  (v9) -- (v10) -- (v11) -- cycle;
\draw[ultra thick,line join=round,line cap=butt]
  (v7) -- (v15) -- (v16) -- (v17);
\end{scope}
\end{tikzpicture}
\end{center}
\caption{Two views of the same affine MV polytope $\widetilde P(b)$
of type $\widetilde A_2$. The thick line goes successively through
the points $\nu_0=0$, $\nu_1=\nu_0+(\alpha_0+2\alpha_1+2\alpha_2)$,
$\nu_2=\nu_1+2\alpha_1$, $\nu_3=\nu_2+5\delta$,
$\nu_4=\nu_3+(\alpha_0+\alpha_2)$, $\nu_5=\nu_4+2\alpha_0$.
The length of the edges of this line, together with the two partitions
$\lambda_{\varpi_1}=(1,1)$ and $\lambda_{\varpi_2}=(2,1)$, form the
Lusztig datum $\Omega_{\preccurlyeq}(b)$ relative to any convex order
$\preccurlyeq$ such that
$\alpha_0\prec\alpha_0+\alpha_2\prec\delta
\prec\alpha_1\prec\alpha_0+2\alpha_1+2\alpha_2$.
The other vertices were calculated using the conditions on the
$2$-faces. The MV polytope $\widetilde P(b)$ includes the data of
$\lambda_\gamma$ for all spherical chamber coweights $\varpi$,
and in this example, the rest of this decoration is given by
$\lambda_{s_1\varpi_1}=(1,1)$, $\lambda_{s_2s_1\varpi_1}=(0)$,
$\lambda_{s_2\varpi_2}= (2,1)$ and $\lambda_{s_1s_2\varpi_2}=(2,1)$.}
\label{fi:AffMVPolA2}
\end{figure}
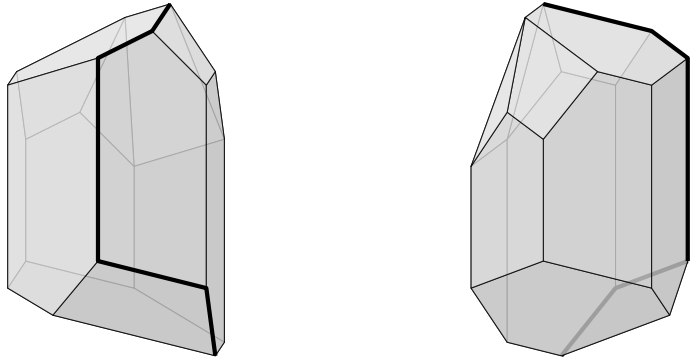

\subsection{Plan of the paper}
\label{ss:PlanPaper}
Section~\ref{se:CombBiconv} recalls combinatorial notions and
facts related to root systems. We emphasize the notion of
biconvex subsets, which is crucial to the study of convex orders
and to the definition of the functorial filtration
$(T_{\succcurlyeq\alpha})_{\alpha\in\Phi_+}$ mentioned in
section~\ref{ss:LuszData}.

Section~\ref{se:TorHNPoly} is devoted to generalities about
HN polytopes in abelian categories.

In section~\ref{se:RecPrepAlg}, we recall known facts about
preprojective algebras and Lusztig's nilpotent varieties. We also
prove that cutting a $\Lambda$-module according to a torsion
pair is an operation that preserves genericity.

In section~\ref{se:TiltTheoPrepAlg}, we exploit the tilting theory
on $\Lambda\mmod$ to define and study the submodules $T^w$ and $T_w$
mentioned in section~\ref{ss:TitsTilt}. An important difference
with the works of Iyama, Reiten et al.\ and of Gei\ss, Leclerc
and Schr\"oer is the fact that we are interested not only in the
small slices that form the categories $\Sub(\Lambda/I_w)$
(notation of Iyama, Reiten et al.) or $\mathcal C_w$ (notation
of Gei\ss, Leclerc and Schr\"oer), but also at controlling the
remainder. Moreover, we track the tilting theory at the level
of the irreducible components of Lusztig's nilpotent varieties and
interpret the result in term of crystal operations.

In section~\ref{se:HallFunc}, we construct embeddings of $\Pi\mmod$
into $\Lambda\mmod$, where $\Pi$ is the completed preprojective
algebra of type $\widetilde A_1$. The data needed to define such
an embedding is a pair $(S,R)$ of rigid orthogonal bricks in
$\Lambda\mmod$ satisfying
$\dim\Ext^1_\Lambda(S,R)=\dim\Ext^1_\Lambda(R,S)=2$. The key
ingredient in the construction is the $2$-Calabi-Yau property
of $\Lambda\mmod$.

The final section~\ref{se:PrepAlgAff} deals with the specifics
of the affine type case. All the results concerning the imaginary
edges, the cores, or the partitions are stated and proved there.

\subsection{Thanks}
\label{ss:Thanks}
We thank Claire Amiot for suggesting to us that the reflection
functors in \cite{BaumannKamnitzer12} are related to those in
\cite{AmiotIyamaReitenTodorov10}, which allowed us to take the
current literature
\cite{BuanIyamaReitenScott09,GeissLeclercSchroer11,SekiyaYamaura13}
into account in section~\ref{se:TiltTheoPrepAlg}.
We also thank Thomas Dunlap for sharing his ideas about
affine MV polytopes and for providing us with his PhD thesis
\cite{Dunlap10}. We thank Alexander Braverman, Bernhard Keller,
Bernard Leclerc, and Dinakar Muthiah for very helpful discussions.
Finally, we thank two anonymous referees for thorough and insightful
reports which led to significant improvements in the presentation.

P.\ B.\ acknowledges support from the ANR,
project~ANR-09-JCJC-0102-01.
J.\ K.\ acknowledges support from NSERC.
P.\ T.\ acknowledges support from the NSF,
grants~DMS-0902649, DMS-1162385 and~DMS-1265555.

\subsection{Summary of the main notations}
\label{ss:SummMainNot}
$\mathbb N=\{0,1,2,\ldots\}$.\\
$\mathcal P$ the set of partitions.\\
$\mathbf K(\mathscr A)$ the Grothendieck group of $\mathscr A$, an
essentially small abelian category.\\
$\Irr\mathscr A$ the set of isomorphism classes of simple objects in
$\mathscr A$.\\
$\Pol(T)\subseteq\mathbf K(\mathscr A)_{\mathbb R}$ the HN polytope
of an object $T\in\mathscr A$.\\
$P_\theta=\{x\in P\mid\langle\theta,x\rangle=m\}$ the face of a HN
polytope $P$, where $\theta\in(\mathbf K(\mathscr A)_{\mathbb R})^*$
and $m$ is the maximal value of $\theta$ on $P$.\\
$K$ the base field for representation of quivers and preprojective
algebras.\\
$(I,E)$ a finite graph, wihout loops (encoding a symmetric generalized
Cartan matrix).\\
$\mathfrak g$ the corresponding symmetric Kac-Moody algebra.\\
$\mathfrak n_+$ the upper nilpotent subalgebra of $\mathfrak g$.\\
$\Omega$ an orientation of $(I,E)$ (thus $Q=(I,\Omega)$ is a quiver).\\
$H=\Omega\sqcup\overline\Omega$ the set of edges of the double quiver
$\overline Q$.\\
$s,t:H\to I$ the source and target maps.\\
$\Lambda$ the completed preprojective algebra of $Q$.\\
$\Lambda\mmod$ the category of finite dimensional left $\Lambda$-modules.\\
$\Phi$ the root system of $\mathfrak g$.\\
$\{\alpha_i\mid i\in I\}$ the standard basis of $\Phi$.\\
$\Phi=\Phi_+\sqcup\Phi_-$ the positive and negative roots with respect
to this basis.\\
$W=\langle s_i\mid i\in I\rangle$ the Weyl group.\\
$\ell:W\to\mathbb N$ the length function.\\
$\mathbb ZI=\bigoplus_{i\in I}\mathbb Z\alpha_i$ the root lattice.\\
$\mathbb RI=\mathbb ZI\otimes_{\mathbb Z}\mathbb R$, the
$\mathbb R$-vector space with basis $(\alpha_i)_{i\in I}$.\\
$(\mathbb RI)^*$ the dual vector space.\\
$\omega_i\in(\mathbb RI)^*$ the $i$-th coordinate on $\mathbb RI$;
thus $(\omega_i)_{i\in I}$ is the dual basis of $(\alpha_i)_{i\in I}$.\\
$(\,,\,):\mathbb ZI\times\mathbb ZI\to\mathbb Z$ the $W$-invariant
symmetric bilinear form (real roots have square length $2$).\\
$C_0=\{\theta\in(\mathbb RI)^*\mid\langle\theta,\alpha_i\rangle>0\}$
the dominant Weyl chamber.\\
$C_T=\bigcup_{w\in W}\overline{C_0}$ the Tits cone.\\
$F_J=\{\theta\in(\mathbb RI)^*\mid\forall j\in
J,\;\langle\theta,\alpha_j\rangle=0\text{ and }\forall i\in
I\setminus J,\;\langle\theta,\alpha_i\rangle>0\}$, for $J\subseteq I$.\\
$\Phi_J$ and $W_J$, the root subsystem and the parabolic subgroup
defined by $J\subseteq I$.\\
$w_J$ the longest element in $W_J$, when the latter is finite.\\
$\height:\mathbb ZI\to\mathbb Z$ the linear form such that
$\height(\alpha_i)=1$ for each $i\in I$.\\
$N_w=\Phi_+\cap w\Phi_-$, for $w\in W$; thus
$N_w=\{s_{i_1}\cdots s_{i_{k-1}}\alpha_{i_k}\mid1\leq k\leq\ell\}$
for any reduced decomposition $w=s_{i_1}\cdots s_{i_\ell}$.\\
$\Pi$ the preprojective algebra of type $\widetilde A_1$.

In the case of an affine root system:\\
$\delta$ the positive primitive imaginary root.\\
$\mathfrak t^*=\mathbb RI/\mathbb R\delta$.\\
$\pi:\mathbb RI\to\mathfrak t^*$ the projection modulo
$\mathbb R\delta$.\\
$\Phi^s=\pi(\Phi^{\re})$ the spherical (finite) root system.\\
$\iota:\Phi^s\to\Phi_+^{\re}$ the ``minimal'' lift, a right inverse
of $\pi$.\\
$r=\dim\mathfrak t^*$ the rank of $\Phi^s$.\\
$\mathfrak t=\{\theta\in(\mathbb RI)^*\mid\langle\theta,\delta\rangle=0\}$
the dual of $\mathfrak t^*$.\\
$\Gamma\subseteq\mathfrak t$ the set of all spherical chamber coweights.\\
$\mathscr W$ the Weyl fan on $(\mathbb RI)^*$, completed on
$\mathfrak t$ by the spherical Weyl fan.\\
$Q^\vee\subseteq\mathfrak t$ the coroot lattice, spanned over
$\mathbb Z$ by the elements $(\alpha_i,?)$.\\
$t_\lambda\in W$ the translation, for $\lambda\in Q^\vee$;
thus $t_\lambda(\nu)=\nu-\langle\lambda,\nu\rangle\delta$ for
each $\nu\in\mathbb RI$.\\
$W_0=W/Q^\vee$ the image of $W$ in $\GL(\mathfrak t^*)$.\\
$\mathscr V=\{A\subseteq\Phi_+\mid A\text{ biconvex}\}$.

And after having chosen an extending vertex $0$ in the extended
Dynkin diagram:\\
$I_0=I\setminus\{0\}$ the vertices of the (finite type) Dynkin diagram.\\
$\{\pi(\alpha_i)\mid i\in I_0\}$ a preferred system of simple
roots for $\Phi^s$.\\
$(\varpi_i)_{i\in I_0}$ the spherical fundamental coweights,
a basis of $\mathfrak t$.\\
$C_0^s=\sum_{i\in I_0}\mathbb R_{>0}\varpi_i$ the dominant spherical
Weyl chamber.

Geometry:\\
The set of irreducible components of a topological space $X$ is
denoted by $\Irr X$. If $Z$ is an irreducible topological space,
then we say that a propriety $P(x)$ depending on a point $x\in Z$
holds for $x$ general in $Z$ if the set of points of $Z$ at which
$P$ holds true contains a dense open subset of $Z$. We sometimes
extend this vocabulary by simply saying ``let $x$ be a general point
in $Z$''; in this case, it is understood that we plan to impose
finitely many such conditions $P$.

\section{Combinatorics of root systems and of MV polytopes}
\label{se:CombBiconv}
In this section, we introduce our notations and recall general
results about root systems and biconvex subsets. Starting from
section~\ref{ss:SetupAffTyp} onwards, we focus on the case of
an affine root system.

\subsection{General setup}
\label{ss:GenSetup}
Let $(I,E)$ be a finite graph, without loops: here $I$ is the set of
vertices and $E$ is the set of edges. We denote by $\mathbb ZI$ the
free abelian group on $I$ and we denote its canonical basis by
$\{\alpha_i\mid i\in I\}$. We endow it with the symmetric bilinear
form $(\,,\,):\mathbb ZI\times\mathbb ZI\to\mathbb Z$, given by
$(\alpha_i,\alpha_i)=2$ for any $i$, and for $i\neq j$,
$(\alpha_i,\alpha_j)$ is the negative of the number of edges between
the vertices $i$ and $j$ in the graph $(I,E)$. The Weyl group $W$ is
the subgroup of $\GL(\mathbb ZI)$ generated by the simple reflections
$s_i:\alpha_j\mapsto\alpha_j-(\alpha_j,\alpha_i)\alpha_i$; this is in
fact a Coxeter system, whose length function is denoted by $\ell$.
Lastly, we denote by $\mathbb NI$ the set of all linear combinations
of the $\alpha_i$ with coefficients in $\mathbb N$ and we denote by
$\height:\mathbb ZI\to\mathbb Z$ the linear form that maps each
$\alpha_i$ to~$1$.

The matrix with entries $(\alpha_i,\alpha_j)$ is a symmetric
generalized Cartan matrix, hence it gives rise to a Kac-Moody algebra
$\mathfrak g$ and a root system $\Phi$. The latter is a $W$-stable
subset of $\mathbb ZI$, which can be split into positive and negative
roots $\Phi=\Phi_+\sqcup\Phi_-$ and into real and imaginary roots
$\Phi=\Phi^{\re}\sqcup\Phi^{\im}$.

Given a subset $J\subseteq I$, we can look at the root system
$\Phi_J=\Phi\cap\linspan_{\mathbb Z}\{\alpha_j\mid j\in J\}$. Its
Weyl group is the parabolic subgroup $W_J=\langle s_j\mid j\in
J\rangle$ of $W$. If $W_J$ is finite, then it has a longest element,
which we denote by $w_J$. An element $u\in W$ is called
$J$-reduced on the right if $\ell(us_j)>\ell(u)$ for each $j\in J$.
If $u$ is $J$-reduced on the right, then $\ell(uv)=\ell(u)+\ell(v)$
for all $v\in W_J$. Each right coset of $W_J$ in $W$ contains a
unique element that is $J$-reduced on the right.

The Weyl group acts on $\mathbb RI$ and on its dual $(\mathbb RI)^*$.
The dominant chamber $C_0$ and the Tits cone $C_T$ are the convex
cones in $(\mathbb RI)^*$ defined as
$$C_0=\{\theta\in(\mathbb RI)^*\mid\forall
i\in I,\;\langle\theta,\alpha_i\rangle>0\}\quad\text{and}\quad
C_T=\bigcup_{w\in W}w\,\overline{C_0}.$$
The closure $\overline C_0$ is the disjoint union of faces
$$F_J=\{\theta\in(\mathbb RI)^*\mid\forall j\in
J,\;\langle\theta,\alpha_j\rangle=0\text{ and }\forall i\in
I\setminus J,\;\langle\theta,\alpha_i\rangle>0\},$$
for $J\subseteq I$. The stabilizer of any point in $F_J$ is precisely
the parabolic subgroup $W_J$. Thus
$$\overline C_0=\bigsqcup_{J\subseteq I}F_J\quad\text{and}\quad
C_T=\bigsqcup_{J\subseteq I}\ \bigsqcup_{w\in W/W_J}wF_J.$$
The disjoint union on the right endows $C_T$ with the structure
of a (non locally finite) fan, which we call the Tits fan.

To an element $w\in W$, we associate the subset
$N_w=\Phi_+\cap w\Phi_-$. If $w=s_{i_1}\cdots s_{i_\ell}$ is a
reduced decomposition, then
$$N_w=\{s_{i_1}\cdots s_{i_{k-1}}\alpha_{i_k}\mid 1\leq k\leq\ell\}.$$
The following result is well-known (see for instance
Remark~$\clubsuit$ in~\cite{CelliniPapi98}).
\begin{lemma}
\label{le:CombiCox}
For $(u,v)\in W^2$, the following three properties are equivalent:
$$\ell(u)+\ell(v)=\ell(uv),\qquad N_u\subseteq N_{uv},\qquad
N_{u^{-1}}\cap N_v=\varnothing.$$
\end{lemma}

\begin{corollary}
\label{co:NwJred}
Let $J\subseteq I$ and let $w\in W$. If $w$ is $J$-reduced on the
right, then $N_{w^{-1}}\cap\Phi_J=\varnothing$.
\end{corollary}
\begin{proof}
Let $J\subseteq I$ and let $w\in W$ be such that
$N_{w^{-1}}\cap\Phi_J\neq\varnothing$. Then there exists
$\beta\in\Phi_J\cap\Phi_+$ such that $w\beta\in\Phi_-$.
Since $\beta$ is a nonnegative linear combination of the roots
$\alpha_j$ for $j\in J$, it follows that there exists $j\in J$
such that $w\alpha_j\in\Phi_-$. For this $j$, we have
$N_{s_j}\subseteq N_{w^{-1}}$, whence
$\ell(s_j)+\ell(s_jw^{-1})=\ell(w^{-1})$ by Lemma~\ref{le:CombiCox}.
Therefore $w$ is not $J$-reduced on the right.
\end{proof}

\subsection{Biconvex sets (general type)}
\label{ss:BiconvSetsGenType}
A subset $A\subseteq\Phi$ is said to be clos if the conditions
$\alpha\in A$, $\beta\in A$, $\alpha+\beta\in\Phi$ imply
$\alpha+\beta\in A$ (see \cite{Bourbaki68}, chapitre 6, \S1,
n\textordmasculine\;7, D\'efinition~4). A subset $A\subseteq\Phi_+$
is said to be biconvex if both $A$ and $\Phi_+\setminus A$ are clos.
We denote by $\mathscr V$ the set of all biconvex subsets of
$\Phi_+$ and endow it with the inclusion order.

\begin{other}{Examples}
\label{ex:BiconvexSets}
\begin{enumerate}
\item
\label{it:BSa}
An increasing union or a decreasing intersection of biconvex subsets
is itself biconvex.
\item
\label{it:BSb}
Each finite biconvex subset of $\Phi_+$ consists of real roots
and is a $N_w$, with $w\in W$ (see \cite{CelliniPapi98},
Proposition~3.2). For convenience, we will say that a biconvex
set $A$ is cofinite if its complement $\Phi_+\setminus A$ is finite.
Given $w\in W$, we set $A_w=N_{w^{-1}}$ and $A^w=\Phi_+\setminus N_w$.
Thus the map $w\mapsto A_w$ (respectively, $w\mapsto A^w$) is a
bijection from $W$ onto the set of finite (respectively, cofinite)
biconvex subsets of $\Phi_+$.
\item
\label{it:BSc}
Each $\theta\in(\mathbb RI)^*$ gives birth to two biconvex subsets
$$A_\theta^{\min}=\{\alpha\in\Phi_+\mid\langle\theta,
\alpha\rangle>0\}\quad\text{and}\quad A_\theta^{\max}=
\{\alpha\in\Phi_+\mid\langle\theta,\alpha\rangle\geq0\}.$$
\end{enumerate}
\end{other}

\begin{other}{Remark}
\label{rk:TrueDefBiconv}
Define a positive root system as a subset $X\subseteq\Phi$ such
that $\Phi=X\sqcup(-X)$ and that the convex cone spanned in
$\mathbb RI$ by $X$ is acute. Denote by $\widetilde{\mathscr V}$
the set of all positive root systems in $\Phi$. If $X$ is a
positive root system, then $A=X\cap\Phi_+$ is biconvex and $X$
can be recovered from $A$ by the formula
$X=A\sqcup(-(\Phi_+\setminus A))$. Thus the map
$X\mapsto X\cap\Phi_+$ from $\widetilde{\mathscr V}$ to $\mathscr V$
is well-defined and injective. The image of this map is the set
$\mathscr V'$ of all subsets $A\subseteq\Phi_+$ such that the
convex cones spanned by $A$ and by $\Phi_+\setminus A$ intersect
only at the origin. Lemma~\ref{le:BiconvPosRS} below shows that
$\mathscr V'=\mathscr V$ whenever $\Phi$ is of finite or affine type.
\end{other}

\begin{proposition}
\label{pr:ATheta}
Let $J\subseteq I$, let $\theta\in F_J$, and let $w\in W$.
Assume $w$ is $J$-reduced on the right. Then
$A^w=A_{w\theta}^{\max}$ and $A_w=A_{-w^{-1}\theta}^{\min}$. In
addition, if $W_J$ is finite, then $A^{ww_J}=A_{w\theta}^{\min}$
and $A_{w_Jw}=A_{-w^{-1}\theta}^{\max}$.
\end{proposition}
\begin{proof}
Let $J$, $\theta$, $w$ as in the statement of the proposition.

We have $A^w=\{\alpha\in\Phi_+\mid w^{-1}\alpha\in\Phi_+\}$ and
$A_{w\theta}^{\max}=\{\alpha\in\Phi_+\mid\langle\theta,w^{-1}\alpha
\rangle\geq0\}$. The inclusion $A^w\subseteq A_{w\theta}^{\max}$ is
straightforward. To show the reverse inclusion, we take
$\alpha\in\Phi_+\setminus A^w$, that is, $\alpha\in N_w$. Then
$\beta=-w^{-1}\alpha$ is in $N_{w^{-1}}$, in particular
$\beta\in\Phi_+$, but $\beta\notin\Phi_J$ by Corollary~\ref{co:NwJred},
and so $\langle\theta,\beta\rangle>0$, which means that
$\alpha\notin A_{w\theta}^{\max}$. We conclude that
$A^w=A_{w\theta}^{\max}$.

Suppose now that $W_J$ is finite. Then
$A^{ww_J}=\{\alpha\in\Phi_+\mid w_Jw^{-1}\alpha\in\Phi_+\}$ and
$A_{w\theta}^{\min}=\{\alpha\in\Phi_+\mid\langle\theta,
w_Jw^{-1}\alpha\rangle>0\}$. The inclusion $A_{w\theta}^{\min}
\subseteq A^{ww_J}$ is straightforward. To show the reverse inclusion,
we take $\alpha\in A^{ww_J}\setminus A_{w\theta}^{\min}$, if possible.
Then $w_Jw^{-1}\alpha$ necessarily belongs to $\Phi_+\cap\Phi_J$,
and so does $\beta=-w^{-1}\alpha$. Then $\beta\in
N_{w^{-1}}\cap\Phi_J$, which contradicts Corollary~\ref{co:NwJred}.
We conclude that $A^{ww_J}=A_{w\theta}^{\min}$.

The last two equalities $A_w=A_{-w^{-1}\theta}^{\min}$ and
$A_{w_Jw}=A_{-w^{-1}\theta}^{\max}$ are proved in a similar fashion.
\end{proof}

\subsection{Setup in the affine type}
\label{ss:SetupAffTyp}
In the rest of section~\ref{se:CombBiconv}, we will focus on the
case where the root system $\Phi$ is of affine type. Then there
exists $\delta\in\mathbb ZI$ such that $\Phi_+^{\im}=\mathbb
Z_{>0}\,\delta$. We set $\mathfrak t^*=\mathbb RI/\mathbb R\delta$
and we denote the natural projection by $\pi:\mathbb RI\to\mathfrak
t^*$. Then $\Phi^s=\pi(\Phi^{\re})$ is a finite root system
in~$\mathfrak t^*$, called the spherical root system.

The dual vector space $\mathfrak t$ of $\mathfrak t^*$ is
identified with $\{\theta\in(\mathbb RI)^*\mid\langle\theta,
\delta\rangle=0\}$. The linear forms $x\mapsto(\alpha_i,x)$
on $\mathbb RI$ span a lattice in $\mathfrak t$, called the
coroot lattice; we denote it by $Q^\vee$. The Weyl group leaves
$\delta$ invariant, hence acts on $\mathfrak t^*$. The kernel
of this action consists of translations $t_\lambda$, for
$\lambda\in Q^\vee$. The translation $t_\lambda$ acts on $\mathbb RI$
by $t_\lambda(\nu)=\nu-\langle\lambda,\nu\rangle\,\delta$ and acts
on $(\mathbb RI)^*$ by $t_\lambda(\theta)=\theta+\langle\theta,
\delta\rangle\lambda$. We denote by $W_0$ the quotient of $W$
by this subgroup of translations; it can be viewed as a subgroup
of $\GL(\mathfrak t^*)$, indeed as the Weyl group of $\Phi^s$.

A basis of the root system $\Phi^s$ is in particular a basis of
$\mathfrak t^*$, and the elements in the dual basis are called the
spherical fundamental coweights. We define a spherical chamber
coweight as an element of $\mathfrak t$ that is conjugate under $W_0$
to a spherical fundamental coweight, and we denote by $\Gamma$ the
set of all spherical chamber coweights. The root system $\Phi^s$ defines
a hyperplane arrangement in $\mathfrak t$, called the spherical Weyl
fan. The open cones in this fan will be called the spherical Weyl
chambers. The rays of this fan are spanned by the spherical chamber
coweights.

The Tits cone is $C_T=\{\theta\in(\mathbb RI)^*\mid
\langle\theta,\delta\rangle>0\}\cup\{0\}$. Thus $(\mathbb RI)^*$ is
covered by $C_T$, $-C_T$ and~$\mathfrak t$. Gathering the faces of
the Tits fan, their opposite, and the faces of the spherical Weyl fan,
we get a (non-locally finite) fan on $(\mathbb RI)^*$. We call it the
affine Weyl fan and denote it by~$\mathscr W$. The cones of this fan
are the equivalence classes of the relation $\sim$ on $(\mathbb RI)^*$
defined in the following way: one says that $x\sim y$ if for each
root $\alpha\in\Phi$, the two real numbers $\langle x,\alpha\rangle$
and $\langle y,\alpha\rangle$ have the same sign.

\begin{figure}[ht]
\begin{center}
\begin{tikzpicture}
\begin{scope}[scale=.7]
\coordinate (u) at (1.7,.2);
\coordinate (v) at (-1,.8);
\coordinate (w) at ($(0,0)-(u)-(v)$);
\coordinate (O) at (0,0);
\coordinate (A) at (-.1,2.4);
\coordinate (B) at ($(A)-(u)$);
\coordinate (C) at ($(A)+(v)$);
\coordinate (D) at ($(A)-(w)$);
\coordinate (E) at ($(A)+(v)-(w)$);
\coordinate (F) at ($(A)-(u)+(v)$);
\coordinate (G) at ($(A)+2*(v)$);
\coordinate (AT) at ($(O)!1.71!(A)$);
\coordinate (BT) at ($(O)!1.65!(B)$);
\coordinate (CT) at ($(O)!1.68!(C)$);
\coordinate (DT) at ($(O)!1.6!(D)$);
\coordinate (ET) at ($(O)!1.48!(E)$);
\coordinate (FT) at ($(O)!1.5!(F)$);
\coordinate (GT) at ($(O)!1.45!(G)$);
\draw[thick] (O) ++(-3.8,-1.5) -- ++(0:8.5) -- ++(70:4)
  node[right]{$\delta=0$};
\draw[thick] (A) ++(-3.8,-1.5) -- ++(0:8.5) -- ++(70:4)
  node[right]{$\delta=1$};
\draw (O) node{$\bullet$};
\draw (ET) ++(0,.5) node{};
\draw ($(O)-1.6*(u)$) -- ($(O)+2.2*(u)$);
\draw ($(O)-1.4*(v)$) -- ($(O)+.7*(v)$);
\draw ($(O)-.6*(w)$) -- ($(O)+1.2*(w)$);
\draw[thin] (O) -- (AT);
\draw[thin] (O) -- (BT);
\draw[thin] (O) -- (DT);
\draw[thin] (O) -- (FT);
\draw[dashed,thin] (O) -- (CT);
\draw[dashed,thin] (O) -- (ET);
\draw[dashed,thin] (O) -- (GT);
\draw[thin] (AT) to[out=-175,in=15] (BT)
  to[out=145,in=-30] (FT) to[out=51,in=-137] (GT)
  to[out=15,in=-175] (ET) to[out=-18,in=152] (DT)
  to[out=-138,in=53] (AT);
\draw[thin] (AT) to[out=130,in=-25] (CT)
  to[out=155,in=-12] (GT);
\draw[thin] (BT) to[out=63,in=-137] (CT)
  to[out=43,in=-156] (ET);
\draw[thin] (FT) to[out=30,in=-168] (CT)
  to[out=12,in=175] (DT);
\draw ($(A)-1.2*(u)$) -- ($(A)+2.5*(u)$);
\draw ($(C)-.2*(u)$) -- ($(C)+3.2*(u)$);
\draw ($(A)-2*(w)-.2*(u)$) -- ($(A)-2*(w)+1.4*(u)$);
\draw ($(A)-(v)-.2*(u)$) -- ($(A)-(v)+1.6*(u)$);
\draw ($(B)-1.6*(w)$) -- ($(B)+.2*(w)$);
\draw ($(A)-2.2*(w)$) -- ($(A)+.6*(w)$);
\draw ($(A)+(u)-2.2*(w)$) -- ($(A)+(u)+1.2*(w)$);
\draw ($(A)+2*(u)-1.2*(w)$) -- ($(A)+2*(u)+1.2*(w)$);
\draw ($(B)-.3*(v)$) -- ($(B)+.3*(v)$);
\draw ($(A)-1.3*(v)$) -- ($(A)+1.3*(v)$);
\draw ($(A)+(u)-1.3*(v)$) -- ($(A)+(u)+1.6*(v)$);
\draw ($(A)+2*(u)-.6*(v)$) -- ($(A)+2*(u)+2.3*(v)$);
\draw ($(A)+3*(u)+.7*(v)$) -- ($(A)+3*(u)+2.3*(v)$);
\end{scope}
\end{tikzpicture}
\end{center}
\caption{The upper half of the affine Weyl fan of type
$\widetilde A_2$. The intersection with the affine hyperplane
$\{\theta\in(\mathbb RI)^*\mid\langle\theta,\delta\rangle=1\}$
draws the familiar pattern of alcoves. On the hyperplane
$\{\theta\in(\mathbb RI)^*\mid\langle\theta,\delta\rangle=0\}$,
one can see the spherical Weyl fan.}
\label{fi:AffWeylFanA2}
\end{figure}
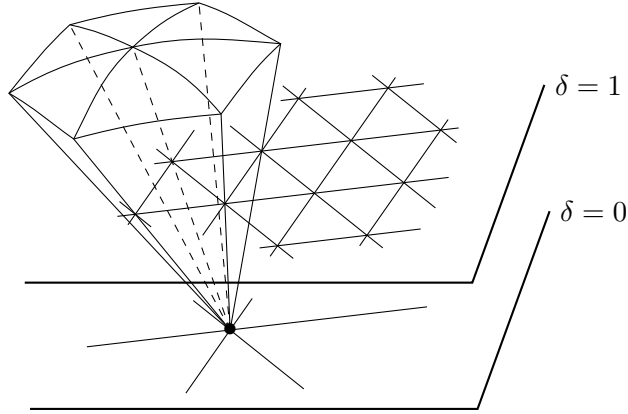

For each $\alpha\in\Phi^s$, we denote by $\iota(\alpha)\in\Phi_+^{\re}$
the unique positive real root such that $\pi(\iota(\alpha))=\alpha$ and
$\iota(\alpha)-\delta\notin\Phi_+^{\re}$. Thus
$\iota:\Phi^s\to\Phi_+^{\re}$ is the ``minimal'' right inverse to $\pi$.

It is often convenient to embed the spherical root system $\Phi^s$
in the affine root system $\Phi$. To do that, we choose an extending
vertex $0$ in $I$ and we set $I_0=I\setminus\{0\}$. Then the spherical
Weyl group $W_0$ can be identified with the parabolic subgroup
$\langle s_i\mid i\in I_0\rangle$ of $W$. Further,
$\{\pi(\alpha_i)\mid i\in I_0\}$ is a basis of the spherical
root system $\Phi^s$, whence a dominant spherical Weyl chamber
$C_0^s=\{\theta\in\mathfrak t\mid\forall i\in I_0,\
\langle\theta,\alpha_i\rangle>0\}$. The highest root $\widetilde\alpha$
of $\Phi^s$ relative to this set of simple roots satisfies
$\delta=\alpha_0+\iota(\widetilde\alpha)$. We denote by
$\{\varpi_i\mid i\in I_0\}$ the basis of $\mathfrak t$ dual to
the basis $\{\pi(\alpha_i)\mid i\in I_0\}$ of $\mathfrak t^*$.

\subsection{Biconvex sets (affine type)}
\label{ss:BiconvSetsAffType}
One nice feature of the affine type is the following key result, which is
a direct application of Theorem~3.12 in~\cite{CelliniPapi98}.

\begin{proposition}
\label{pr:ApproxBiconv}
Let $A\subseteq\Phi_+$ be a biconvex subset. If $\delta\notin A$, then
$A$ is the union of an increasing sequence of finite biconvex subsets.
If $\delta\in A$, then $A$ is the intersection of a decreasing sequence
of cofinite biconvex subsets.
\end{proposition}

\begin{other}{Example}
\label{ex:ALambdaMinMax}
Let $\lambda\in Q^\vee$. With the notation of
Remark~\ref{ex:BiconvexSets}~\ref{it:BSc}, we then have
$$A_\lambda^{\min}=\bigcup_{n\in\mathbb N}A_{t_{n\lambda}}\quad\text{and}
\quad A_\lambda^{\max}=\bigcap_{n\in\mathbb N}A^{t_{n\lambda}};$$
these equalities readily follow from the formula
$t_{n\lambda}(\alpha)=\alpha-n\langle\lambda,\alpha\rangle\delta$.
\end{other}

\begin{lemma}
\label{le:ProjBiconv}
Let $A\subseteq\Phi_+$ be a biconvex subset such that $\delta\notin A$
and let $X=\pi(A)$. Then $X$ is contained in a positive root system of
$\Phi^s$ and $\iota(X)\subseteq A$.
\end{lemma}
\begin{proof}
Since $A$ is clos, so is $X$. Since $\delta\notin A$, we furthermore
have $X\cap(-X)=\varnothing$. By \cite{Bourbaki68}, chapitre~6, \S1,
n\textordmasculine\;7, Proposition~22, $X$ is contained in a positive
root system of $\Phi^s$. Lastly, let $\alpha\in X$, and choose
$\beta\in A\cap\pi^{-1}(\alpha)$ of minimal height. Then
$\beta-\delta$ is not in $A$. It is not in $\Phi_+\setminus A$
either, for this latter is clos and contains $\delta$ but not
$\beta$. Therefore $\beta-\delta\notin\Phi_+$, which means that
$\beta=\iota(\alpha)$. We have shown that $\iota(X)\subseteq A$.
\end{proof}

\begin{lemma}
\label{le:ApproxAdjBiconv}
Let $\alpha\in\Phi_+^{\re}$ and let $A$ and $B$ be two biconvex
subsets such that $B=A\sqcup\{\alpha\}$. We assume that
$\delta\notin A$. Then, for each finite subset $X\subseteq A$,
there are finite biconvex subsets $A'\subseteq A$ and $B'\subseteq B$
such that $X\subseteq A'$ and $B'=A'\sqcup\{\alpha\}$.
\end{lemma}
\begin{proof}
Since $A$ is the increasing union of finite biconvex subsets, one
can find a biconvex $A_0\subseteq A$ that contains $X$. Similarly,
one can find a finite biconvex $B_0\subseteq B$ that contains
$A_0\cup\{\alpha\}$. By Example~\ref{ex:BiconvexSets}~\ref{it:BSb},
we can write $A_0=N_u$ and $B_0=N_{uv}$, with $(u,v)\in W^2$.
Lemma~\ref{le:CombiCox} says then that $\ell(uv)=\ell(u)+\ell(v)$.
Let us write a reduced decomposition $s_{i_1}\cdots s_{i_\ell}$ for
$v$. There exists $k$ such that $\alpha=us_{i_1}\cdots s_{i_{k-1}}
\alpha_{i_k}$. We then take $A'=N_{us_{i_1}\cdots s_{i_{k-1}}}$
and $B'=N_{us_{i_1}\cdots s_{i_k}}$.
\end{proof}

One can of course obtain a similar statement in the case where
$\delta\in B$ by taking complements in $\Phi_+$.

\begin{lemma}
\label{le:DeltaAdj}
\begin{enumerate}
\item
\label{it:DAa}
Let $A$ and $B$ be two biconvex subsets such that $B=A\sqcup(\mathbb
Z_{>0}\delta)$. Then there is a positive system $X\subseteq\Phi^s$
such that $A=\{\alpha\in\Phi_+\mid\pi(\alpha)\in X\}$.
\item
\label{it:DAb}
Let $A\subseteq B$ be two biconvex subsets. Suppose that
$A$ is finite and that $B$ is cofinite. Then there is a
positive system $X\subseteq\Phi^s$ such that
$A\subseteq\{\alpha\in\Phi_+\mid\pi(\alpha)\in X\}\subseteq B$.
\end{enumerate}
\end{lemma}
\begin{proof}
Let us first show \ref{it:DAa}. We take $A$ and $B$ as in the
statement to be proved. Let $X=\pi(A)$ and $Y=\pi(\Phi_+\setminus B)$.
Certainly, $\Phi^s=\pi(\Phi_+^{\re})=X\cup Y$. In addition,
$X$ and $Y$ are disjoint, because otherwise the inclusions
$\iota(X)\subseteq A$ and $\iota(Y)\subseteq\Phi_+\setminus B$
given by Lemma~\ref{le:ProjBiconv} would force $A$ and
$\Phi_+\setminus B$ to share a common element. Lastly,
$X$ and $Y$ are clos. By \cite{Bourbaki68}, chapitre~6, \S1,
n\textordmasculine\;7, Corollaire~1 to Proposition~20, $X$ is a
positive system in $\Phi^s$. Assertion~\ref{it:DAa} then follows
from the observation that $A\subseteq\pi^{-1}(X)$ and
$\pi(\Phi_+\setminus A)\subseteq Y\cup\{0\}$.

Now we consider \ref{it:DAb}. By
Example~\ref{ex:BiconvexSets}~\ref{it:BSb}, there
are $(u,v)\in W^2$ such that $A=A_u=N_{u^{-1}}$ and
$B=A^v=\Phi_+\setminus N_v$. The condition $A\subseteq B$
means that $N_{u^{-1}}\cap N_v=\varnothing$, so
$\ell(uv)=\ell(u)+\ell(v)$ by Lemma~\ref{le:CombiCox}.
By Lemma~\ref{le:ProjBiconv}, $\pi(N_{uv})$ is contained
in a positive root system of $\Phi^s$, say~$Y$.
Since $\pi:\mathbb RI\to\mathfrak t^*$ is $W$-equivariant,
$\pi(u^{-1}N_{uv})$ is contained in $u^{-1}Y$.
Let $X=-u^{-1}Y$. From the equality
$u^{-1}N_{uv}=u^{-1}(N_u\sqcup uN_v)=(-N_{u^{-1}})\sqcup N_v$,
we deduce that $\pi(N_{u^{-1}})\subseteq X$ and that
$\pi(N_v)\cap X=\varnothing$. Therefore
$N_{u^{-1}}\subseteq\{\alpha\in\Phi_+\mid\pi(\alpha)\in X\}
\subseteq\Phi_+\setminus N_v$.
\end{proof}

\begin{lemma}
\label{le:BiconvPosRS}
For any biconvex set $A$, the convex cones spanned by $A$ and
by $\Phi_+\setminus A$ intersect only at the origin.
\end{lemma}
\begin{proof}
Taking complements in $\Phi_+$, we can assume that $\delta\notin A$.
Suppose that the convex cones spanned by $A$ and $\Phi_+\setminus
A$ share a common nonzero element $x$. This $x$ can be expressed
as a non-negative linear combination of a finite family of elements
of $A$, so using Proposition~\ref{pr:ApproxBiconv}, we can find
a finite biconvex subset $B\subseteq A$ such that $x$ belongs to
the convex cone spanned by $B$. Further there exists
$\theta\in(\mathbb RI)^*$ such that $B=A_\theta^{\min}$,
by Example~\ref{ex:BiconvexSets}~\ref{it:BSb} and
Proposition~\ref{pr:ATheta}. We then have $\langle\theta,x\rangle>0$.
On the other hand, $x$ belongs to the convex cone spanned by
$\Phi_+\setminus B=A_{-\theta}^{\max}$, and so
$\langle-\theta,x\rangle\geq0$. This contradiction shows that
the convex cones spanned by $A$ and $\Phi_+\setminus A$ do not
share any nonzero element.
\end{proof}

\subsection{Convex orders}
\label{ss:ConvOrder}
One motivation for studying biconvex subsets comes from the
notion of ``convex order'' on $\Phi_+$. Specifically, a preorder
$\preccurlyeq$ on $\Phi_+$ is called a convex order if for all
$(\alpha,\beta)\in\Phi_+^2$, the three following conditions hold:
\begin{gather*}
\alpha\preccurlyeq\beta\;\text{ or }\;\beta\preccurlyeq\alpha,\\[4pt]
\bigl(\alpha+\beta\in\Phi_+\;\text{ and }\;\alpha\preccurlyeq
\beta\bigr)\ \Longrightarrow\
\alpha\preccurlyeq\alpha+\beta\preccurlyeq\beta,\\[4pt]
\bigl(\alpha\preccurlyeq\beta\;\text{ and }\;
\beta\preccurlyeq\alpha\bigr)\ \Longleftrightarrow\
\alpha\text{ and }\beta\text{ are proportional.}
\end{gather*}
In this section, we restrict to affine type. In this case, in
the last condition above, $\alpha$ and $\beta$ are proportional
if and only if they are equal or they are both imaginary.

A terminal section for a convex order $\preccurlyeq$ is a subset
$A\subseteq\Phi_+$ such that
$$\bigl(\alpha\in A\;\text{ and }\;\alpha\preccurlyeq\beta\bigr)\
\Longrightarrow\ \beta\in A.$$
We denote the set of terminal sections of $\preccurlyeq$ by
$\mathscr U(\preccurlyeq)$. The following result, implicit in
\cite{CelliniPapi98,Ito01}, provides the link between biconvex
subsets and convex orders. We leave its (routine) proof as an
exercise for the reader.

\begin{lemma}
\label{le:OrderBiconv}
For each convex order $\preccurlyeq$, the set
$\mathscr U(\preccurlyeq)$ is a maximal totally
ordered subset of $\mathscr V$. The datum of
$\mathscr U(\preccurlyeq)$ completely determines
$\preccurlyeq$.
\end{lemma}

\begin{other}{Remarks}
\label{rk:RkConvexOrder}
\begin{enumerate}
\item
\label{it:RCOa}
It is known that any biconvex subset is the terminal section
of a convex order (Corollary~3.13 in~\cite{CelliniPapi98}).
\item
\label{it:RCOb}
Let us say that a pair $(A,B)$ of biconvex subsets is adjacent if
$A\subsetneq B$ and if there does not exist a biconvex subset $C$ such
that $A\subsetneq C\subsetneq B$. Each $(w,i)\in W\times I$ with
$\ell(ws_i)>\ell(w)$ gives such a pair, namely $(N_w,N_{ws_i})$;
indeed, one here observes that $N_{ws_i}=N_w\sqcup\{w\alpha_i\}$,
so there is no room between $N_w$ and $N_{ws_i}$. Using
Lemma~\ref{le:CombiCox}, one easily shows that any adjacent pair
of finite biconvex subsets is of this form, so the notion of
adjacent biconvex subset generalizes the covering relation for the
weak Bruhat order.
\item
\label{it:RCOc}
Given a real positive root $\alpha\in\Phi_+^{\re}$, let us say
that a pair $(A,B)$ of biconvex subsets is $\alpha$-adjacent if
$B=A\sqcup\{\alpha\}$. Let us say that $(A,B)$ is $\delta$-adjacent
if $B=A\sqcup\bigl(\mathbb Z_{>0}\delta\bigr)$. We conjecture that
a pair $(A,B)$ of biconvex subsets is adjacent (in the sense of
\ref{it:RCOb} above) if and only if there is a root $\beta$ (real
or imaginary) such that $(A,B)$ is $\beta$-adjacent. This conjecture
seems reasonable in view of our current understanding of biconvex
subsets, but we were not able to extract it from the papers
\cite{CelliniPapi98,Ito05}. If it is correct, then Lemma
\ref{le:OrderBiconv} admits a converse, and ``maximal totally
ordered subset of $\mathscr V$'' would be a notion equivalent to
that of ``convex order''. In any case, Zorn's lemma shows that
any totally ordered subset of $\mathscr V$ can be completed to
a maximal one.
\item
\label{it:RCOd}
Let $\preccurlyeq$ be a convex order. The terminal sections
$$\{\beta\in\Phi_+\mid\beta\succ\delta\}\quad\text{and}
\quad\{\beta\in\Phi_+\mid\beta\succcurlyeq\delta\}$$
satisfy the assumptions of Lemma~\ref{le:DeltaAdj}~\ref{it:DAa},
so there is a positive system $X\subseteq\Phi^s$ such that
$\{\beta\in\Phi_+\mid\beta\succ\delta\}=\{\beta\in\Phi_+\mid
\pi(\beta)\in X\}$. This fact was announced in
section~\ref{ss:LuszData}, see equation~\eqref{eq:ConvOrderDelta}.
\end{enumerate}
\end{other}

\begin{other}{Examples}
\label{ex:ExConvexOrder}
\begin{enumerate}
\item
\label{it:ECOa}
Let us consider the type $\widetilde A_1$. As is customary, we use
$I=\{0,1\}$. Then
$$\Phi_+=\{\alpha_0+n\delta,\alpha_1+n\delta,(n+1)\delta\mid
n\in\mathbb N\}.$$
There are exactly two convex orders on $\Phi_+$. One of them is
$$\alpha_1\prec\alpha_1+\delta\prec\alpha_1+2\delta\prec\cdots\prec\delta
\prec\cdots\prec\alpha_0+2\delta\prec\alpha_0+\delta\prec\alpha_0,$$
the other is the opposite order.
\item
\label{it:ECOb}
A linear form $\theta\in(\mathbb RI)^*$ defines a preorder
on $\Phi_+$, as follows: we say that $\alpha\preccurlyeq\beta$ if
$\langle\theta,\alpha\rangle/\height(\alpha)\leq\langle\theta,
\beta\rangle/\height(\beta)$. For $\theta$ general enough (outside
countably many hyperplanes), this preorder is a convex order.
\end{enumerate}
\end{other}

\subsection{GGMS polytopes in affine type}
\label{ss:GGMSPol}
To a non-empty compact convex subset $K\subseteq\mathbb RI$, one
associates its support function $\psi_K:(\mathbb RI)^*\to\mathbb R$,
defined by $\psi_K(\theta)=\max(\theta(K))$. One can reconstruct
$K$ from the datum of $\psi_K$. If $P$ is a convex polytope, then
$\psi_P$ is piecewise linear; the maximal regions of linearity
are closed cones that cover $(\mathbb RI)^*$. The relative interiors
of these cones and of their faces form the normal fan $\mathscr N_P$
of $P$. In addition, each face of $P$ is of the form
$P_\theta=\{x\in P\mid\langle\theta,x\rangle=\psi_P(\theta)\}$
for some $\theta\in(\mathbb RI)^*$.

We say that a convex lattice polytope $P\subseteq\mathbb RI$ is
GGMS if each open cone of the affine Weyl fan $\mathscr W$ is
contained in an open cone of the normal fan $\mathscr N_P$.
In other words, we ask that for each open cone $C\in\mathscr W$,
there is a vertex $x$ of $P$ such that $P_\theta=\{x\}$ for all
$\theta\in C$. It follows that an edge of a GGMS polytope always
points in a root direction. As we will see later in this section,
a convex lattice polytope $P$ is GGMS if and only if $\mathscr N_P$
is a coarsening of $\mathscr W$, in the sense that each cone of
$\mathscr N_P$ is the union of cones of $\mathscr W$.

Let us fix a GGMS polytope $P$. If $A$ is a finite or cofinite
biconvex subset, then there is a unique open cone $C\in\mathscr W$
such that $A=A_\theta^{\min}=A_\theta^{\max}$ for each
$\theta\in C$, by Example~\ref{ex:BiconvexSets} and
Proposition~\ref{pr:ATheta}, and we denote by $\mu_P(A)$ the
vertex $x$ of $P$ such that $P_\theta=\{x\}$ for each
$\theta\in C$.

We now want to extend this definition to all biconvex subsets
$A$. As mentioned above, the edges of a GGMS polytope $P$
point in root directions. Let us denote by
$E_P\subseteq(\Phi_+^{\re}\sqcup\{\delta\})$ the
(finite) set of all these directions. Furthermore, let us
denote the symmetric difference between two sets $A$ and $B$
by $A\Delta B=(A\setminus B)\cup(B\setminus A)$.

\begin{lemma}
\label{le:DiffVertices}
Let $P$ be a GGMS polytope and let $A$ and $B$ be finite or
cofinite biconvex subsets. Then there is a family of nonnegative
integers $(n_\alpha)_{\alpha\in A\Delta B}$ such that
$n_\alpha=0$ if $\alpha\notin E_P$ and
$$\mu_P(B)-\mu_P(A)=\sum_{\alpha\in B\setminus A}n_\alpha\alpha-
\sum_{\alpha\in A\setminus B}n_\alpha\alpha.$$
\end{lemma}
\begin{proof}
We choose $\theta_0$ and $\theta_1$ in open cones of $\mathscr W$
such that $A=\{\alpha\in\Phi_+\mid\langle\theta_0,\alpha\rangle>0\}$
and $B=\{\alpha\in\Phi_+\mid\langle\theta_1,\alpha\rangle>0\}$.
By moving $\theta_0$ and $\theta_1$ if necessary, we may assume
that the segment $[\theta_0,\theta_1]$ does not meet any cone
of codimension $2$ in the normal fan $\mathscr N_P$. We consider
$\theta(t)=(1-t)\theta_0+t\theta_1$. As $t$ varies from $0$ to
$1$, the face $P_{\theta(t)}$ is generally a vertex of $P$,
occasionally an edge, but never a face of higher dimension.
The vertices and edges found in this way form a path in the
$1$-skeleton of $P$ from $\mu_P(A)$ to $\mu_P(B)$. Each edge
traversed by this path points in the direction of a root
$\alpha$ such that
$\langle\theta_0,\alpha\rangle<0<\langle\theta_1,\alpha\rangle$,
so that either $\alpha\in\Phi_+\cap(B\setminus A)$ or
$-\alpha\in\Phi_+\cap(A\setminus B)$, and moreover the length of
this edge is an integral multiple of $\alpha$, because $P$ is a
lattice polytope.
\end{proof}

With the notation of the lemma, we have $\mu_P(A)=\mu_P(B)$ as
soon as $A\Delta B$ does not meet $E_P$. We can thus extend $\mu_P$
to a map from all of $\mathscr V$ to the set of vertices of $P$
as follows. If $A$ is a biconvex subset such that $\delta\notin A$,
then we set $\mu_P(A)=\mu_P(B)$, where $B$ is any finite biconvex subset
such that $A\cap E_P\subseteq B\subseteq A$; the result does not
depend on the choice of $B$, because the set of all possible $B$
is filtered. Similarly, if $A$ is a biconvex subset such that
$\delta\in A$, then we set $\mu_P(A)=\mu_P(B)$, where $B$ is
any cofinite biconvex subset such that
$A\subseteq B\subseteq(A\cup(\Phi_+\setminus E_P))$.
With these conventions, Lemma~\ref{le:DiffVertices}
trivially extends to any biconvex subsets $A$ and~$B$.

Recall the biconvex subsets $A_\theta^{\min}$ and $A_\theta^{\max}$
from Example~\ref{ex:BiconvexSets}~\ref{it:BSc}.

\begin{proposition}
\label{pr:SuppGGMS}
The support function $\Psi_P$ of a GGMS polytope $P$ is given by
$$\psi_P(\theta)=\langle\theta,\mu_P(A_\theta^{\min})\rangle
=\langle\theta,\mu_P(A_\theta^{\max})\rangle,$$
for all $\theta\in(\mathbb RI)^*$.
\end{proposition}
\begin{proof}
Let $P$ be a GGMS polytope, let $\theta\in(\mathbb RI)^*$,
and let $B$ be any biconvex subset. By definition,
$\langle\theta,\alpha\rangle>0$ for each
$\alpha\in A_\theta^{\min}\setminus B$, and
$\langle\theta,\alpha\rangle\leq0$ for each
$\alpha\in B\setminus A_\theta^{\min}$.
Lemma~\ref{le:DiffVertices} then implies that
$\langle\theta,\mu_P(A_\theta^{\min})-\mu_P(B)\rangle\geq0$.
Since each vertex of $P$ can be written as a $\mu_P(B)$,
it follows that $\langle\theta,\mu_P(A_\theta^{\min})\rangle$ is
the supremum of $\theta$ on $P$, whence the first equality.
The second equality also directly follows
Lemma~\ref{le:DiffVertices}.
\end{proof}

The biconvex subsets $A_\theta^{\min}$ and $A_\theta^{\max}$
only depend on the cone of $\mathscr W$ to which $\theta$ belongs.
From Proposition~\ref{pr:SuppGGMS}, we then deduce that the support
function of a GGMS polytope $P$ is linear on each cone of
$\mathscr W$. Therefore the normal fan $\mathscr N_P$ is a
coarsening of $\mathscr W$, as announced earlier in this section.

The fact that $\mathscr N_P$ is a coarsening of $\mathscr W$
restricts the shape of the $2$-faces of $P$. Specifically,
given~$\theta$ in a codimension $2$ face of $\mathscr W$,
two cases can happen: either $\pm\theta\in C_T$, and then
$\Phi\cap(\ker\theta)$ is a finite root system of type
$A_1\times A_1$ or $A_2$; or $\theta$ belongs to a facet of
the spherical Weyl fan, and then $\Phi\cap(\ker\theta)$ is
an affine root system of type $\widetilde A_1$. In both cases,
$P_\theta$ is a GGMS polytope of the same type as
$\Phi\cap(\ker\theta)$. In the former case, we thus say
that $P_\theta$ is a $2$-face of finite type; in the latter,
we say that $P_\theta$ is a $2$-face of affine type.

Now fix a GGMS polytope $P$ and a convex order $\preccurlyeq$ on
$\Phi_+$. For $\alpha\in\Phi_+^{\re}\sqcup\{\delta\}$, look at
$$A=\{\beta\in\Phi_+\mid\beta\succ\alpha\}\quad\text{and}\quad
B=\{\beta\in\Phi_+\mid\beta\succcurlyeq\alpha\}.$$
These are biconvex subsets such that $B=A\sqcup\{\alpha\}$,
if $\alpha$ is real, or $B=A\sqcup\bigl(\mathbb Z_{>0}\delta\bigr)$,
if $\alpha=\delta$. In either case, $[\mu_P(A),\mu_P(B)]$ is an edge
of $P$ (possibly degenerate) which points in the direction of
$\alpha$, so we may write $\mu_P(B)-\mu_P(A)=n_\alpha\alpha$,
where $n_\alpha$ is a non-negative integer (nonzero only if
$\alpha\in E_P$). For any $A\in\mathscr U(\preccurlyeq)$, we then have
$$\mu_P(A)=\sum_{\alpha\in A\cap(\Phi_+^{\re}\sqcup\{\delta\})}
n_\alpha\alpha.$$

The collection of numbers $(n_\alpha)$ will be called the Lusztig
datum of $P$ in direction $\preccurlyeq$. We will however later
decorate our GGMS polytopes in order to refine the information
carried by $n_\delta$, taking into account all the imaginary
roots and their multiplicities.

\section{Torsion pairs and Harder-Narasimhan polytopes}
\label{se:TorHNPoly}
In this section we will study general facts about torsion pairs and
Harder-Narasimhan polytopes.
We consider an essentially small abelian category $\mathscr A$ such
that all objects have finite length. This assumption ensures that
the Grothendieck group $\mathbf K(\mathscr A)$ is a free abelian group,
with basis the set of isomorphism classes of simple objects. As usual,
we denote by $[T]$ the class in $\mathbf K(\mathscr A)$ of an object
$T\in\mathscr A$. Our subcategories will always be full subcategories.

\subsection{Torsion pairs}
\label{ss:TorPairs}
Following~\cite{Assem90}, a torsion pair in $\mathscr A$ is
a pair $(\mathscr T,\mathscr F)$ of two subcategories, called the
torsion class and the torsion-free class, that
satisfy the following two axioms:
\begin{description}
\item[(T1)]
$\Hom_{\mathscr A}(X,Y)=0$ for each
$(X,Y)\in\mathscr T\times\mathscr F$.
\item[(T2)]
Each object $T\in\mathscr A$ has a subobject~$X$
such that $(X,T/X)\in\mathscr T\times\mathscr F$.
\end{description}
Axiom (T1) forces the subobject $X$ in (T2) to be the largest
subobject of $T$ that belongs to $\mathscr T$, and a fortiori to
be unique; this $X$ is called the torsion subobject of $T$ with
respect to the torsion pair $(\mathscr T,\mathscr F)$.

An equivalent set of axioms are the two requirements:
\begin{description}
\item[(T'1)]
$\mathscr T=\{X\in\mathscr A\mid\forall Y\in\mathscr F,\
\Hom(X,Y)=0\}$.
\item[(T'2)]
$\mathscr F=\{Y\in\mathscr A\mid\forall X\in\mathscr T,\
\Hom(X,Y)=0\}$.
\end{description}
With this second formulation, it is clear that $\mathscr T$ is closed
under taking quotients and extensions and that $\mathscr F$ is closed
under taking subobjects and extensions.

Given two torsion pairs $(\mathscr T',\mathscr F')$ and
$(\mathscr T'',\mathscr F'')$, we write
$(\mathscr T',\mathscr F')\preccurlyeq(\mathscr T'',\mathscr F'')$
if the following three equivalent conditions hold:
$$\mathscr T'\subseteq\mathscr T'',\qquad
\mathscr F'\supseteq\mathscr F'',\qquad
\mathscr T'\cap\mathscr F''=\{0\}.$$
In this case, each object $T\in\mathscr A$ is endowed with
a three-step filtration $0\subseteq X'\subseteq X''\subseteq T$,
where $X'$ and $X''$ are the torsion subobjects of $T$ with respect
to $(\mathscr T',\mathscr F')$ and $(\mathscr T'',\mathscr F'')$,
respectively. Since $\mathscr F'$ is stable under taking subobjects
and $\mathscr T''$ is stable under taking quotients, we have
$(X',X''/X',T/X'')\in(\mathscr T',\mathscr F'\cap\mathscr
T'',\mathscr F'')$.

A typical example of torsion pair is obtained by the following
construction, directly translated from the well-known theories
of Harder-Narasimhan filtrations and stability conditions
\cite{Shatz77,Reineke03,Rudakov97}. Fix a group homomorphism
$\theta:\mathbf K(\mathscr A)\to\mathbb R$ and define five subcategories
$\mathscr I_\theta$, $\overline{\mathscr I}_\theta$, $\mathscr P_\theta$,
$\overline{\mathscr P}_\theta$ and $\mathscr R_\theta$ of $\mathscr A$:
\begin{itemize}
\item
An object $T$ is in $\mathscr I_\theta$ (respectively,
$\overline{\mathscr I}_\theta$) if any nonzero
quotient $X$ of $T$ satisfies $\theta([X])>0$ (respectively,
$\theta([X])\geq0$).
\item
An object $T$ is in $\mathscr P_\theta$ (respectively,
$\overline{\mathscr P}_\theta$) if any nonzero
subobject $X$ of $T$ satisfies $\theta([X])<0$ (respectively,
$\theta([X])\leq0$).
\item
An object $T$ is in $\mathscr R_\theta$ if
$\theta([T])=0$ and any nonzero subobject $X$ of $T$
satisfies $\theta([X])\leq0$.
\end{itemize}

The objects in the category $\mathscr R_\theta$ are
called $\theta$-semistable~\cite{King94}. Note that
$\mathscr R_\theta=\overline{\mathscr I}_\theta\cap
\overline{\mathscr P}_\theta$.

\begin{proposition}
\label{pr:PRITheta}
Both $(\mathscr I_\theta,\overline{\mathscr P}_\theta)$ and
$(\overline{\mathscr I}_\theta,\mathscr P_\theta)$ are torsion
pairs in $\mathscr A$. The category $\mathscr R_\theta$ is an
abelian subcategory of $\mathscr A$.
\end{proposition}
\begin{proof}
Let us first prove that
$(\mathscr I_\theta,\overline{\mathscr P}_\theta)$ is a torsion pair.
The axiom (T1) is obvious, so we have to prove the axiom (T2).

We first show that $\mathscr I_\theta$ is closed under extensions.
Let $0\to T'\to T\xrightarrow fT''\to0$ be a short exact sequence
with $T'$ and $T''$ in $\mathscr I_\theta$ and let $g:T\to X$ be
an epimorphism. The pushout of $(f,g)$ then exhibits $X$ as the
extension of a quotient $X''$ of $T''$ by a quotient $X'$ of $T'$.
By assumption, $\theta([X'])$ and $\theta([X''])$ are both nonnegative,
so $\theta([X])\geq0$. Moreover, equality holds only if both $X'$
and $X''$ are zero, thus only if $X=0$.

Now let $T\in\mathscr A$. Our assumption of finite length allows
us to pick a maximal element $X$ among the subobjects of $T$ that
belong to $\mathscr I_\theta$. Suppose that $T/X$ is not in
$\overline{\mathscr P}_\theta$. Then it contains a subobject $Y$
such that $\theta([Y])>0$, and we may assume that $Y$ has been
chosen minimal with this property. Certainly, $Y$ does not belong
to $\mathscr I_\theta$; otherwise, the extension of $Y$ by $X$
inside $T$ would belong to $\mathscr I_\theta$, contradicting the
maximality of $X$. So $Y$ has a nonzero quotient $Y/Z$ such that
$\theta([Y/Z])\leq0$. Since $Z$ is a subobject of $T/X$ properly
contained in $Y$, the minimality of $Y$ requires
$\theta([Z])\leq0$. We thus reach a contradiction, namely
$0\geq\theta([Z])+\theta([Y/Z])=\theta([Y])>0$. Therefore
$T/X\in\overline{\mathscr P}_\theta$, which establishes (T2).

We have thus shown that
$(\mathscr I_\theta,\overline{\mathscr P}_\theta)$ is a torsion pair.
The proof for $(\overline{\mathscr I}_\theta,\mathscr P_\theta)$ is
similar. The fact that $\mathscr R_\theta$ is an abelian subcategory
is well-known.
\end{proof}

Since $(\mathscr I_\theta,\overline{\mathscr P}_\theta)
\preccurlyeq(\overline{\mathscr I}_\theta,\mathscr P_\theta)$,
these two torsion pairs endow each object $T\in\mathscr A$
with a three-step filtration $0\subseteq T_\theta^{\min}\subseteq
T_\theta^{\max}\subseteq T$. The quotient
$T_\theta^{\max}/T_\theta^{\min}$ belongs to $\mathscr R_\theta=
\overline{\mathscr I}_\theta\cap\overline{\mathscr P}_\theta$.

\begin{proposition}
\label{pr:TminTmax}
Let $\theta:\mathbf K(\mathscr A)\to\mathbb R$ be a group homomorphism
and let $T\in\mathscr A$. Then
$$\theta([T_\theta^{\min}])=\theta([T_\theta^{\max}])\geq\theta([X])$$
for any subobject $X\subseteq T$. Equality holds if only if
$T_\theta^{\min}\subseteq X\subseteq T_\theta^{\max}$ and
$X/T_\theta^{\min}$ is $\theta$-semistable.
\end{proposition}
\begin{proof}
We adopt the notation of the statement. Let $X$ be a subobject of
$T$. Since $T_\theta^{\min}\in\mathscr I_\theta$, we have
$\theta([T_\theta^{\min}/(X\cap T_\theta^{\min})])\geq0$, with
equality only if $T_\theta^{\min}\subseteq X$. Since
$T/T_\theta^{\max}\in\mathscr P_\theta$, we have
$\theta([(X+T_\theta^{\max})/T_\theta^{\max}])\leq0$, with
equality only if $X\subseteq T_\theta^{\max}$. Lastly, we note
that $(X\cap T_\theta^{\max})/(X\cap T_\theta^{\min})$ is a
subobject of $T_\theta^{\max}/T_\theta^{\min}$; since the
latter is in $\overline{\mathscr P}_\theta$, we have
$\theta([(X\cap T_\theta^{\max})/(X\cap T_\theta^{\min})])\leq0$,
with equality if and only if $(X\cap T_\theta^{\max})/(X\cap
T_\theta^{\min})$ is $\theta$-semistable. The result now follows
from the Grassmann relation
$[X+T_\theta^{\max}]+[X\cap T_\theta^{\max}]=[X]+[T_\theta^{\max}]$.
\end{proof}

\subsection{Harder-Narasimhan polytopes}
\label{ss:HNPoly}
We set $\mathbf K(\mathscr A)_{\mathbb R}=\mathbf K(\mathscr
A)\otimes_{\mathbb Z}\mathbb R$. We view this $\mathbb R$-vector
space as the inductive limit of its finite dimensional subspaces;
it is thus a locally convex topological vector space. Linear forms
on this vector space are automatically continuous. We denote the
canonical pairing between $\mathbf K(\mathscr A)_{\mathbb R}$ and
its dual $(\mathbf K(\mathscr A)_{\mathbb R})^*$ by angle brackets.
We may regard a linear form
$\theta\in(\mathbf K(\mathscr A)_{\mathbb R})^*$ as a group
homomorphism $\theta:\mathbf K(\mathscr A)\to\mathbb R$, whence
the torsion pairs $(\mathscr I_\theta,\overline{\mathscr P}_\theta)$
and $(\overline{\mathscr I}_\theta,\mathscr P_\theta)$ and the
abelian subcategory $\mathscr R_\theta$.

Given an object $T\in\mathscr A$, there are finitely many classes
$[X]$ of subobjects $X\subseteq T$. The convex hull in
$\mathbf K(\mathscr A)_{\mathbb R}$ of all these points is a convex
lattice polytope. We call it the Harder-Narasimhan polytope of $T$
and we denote it by $\Pol(T)$. The support function $\psi_{\Pol(T)}$
of this polytope is defined as the function on
$(\mathbf K(\mathscr A)_{\mathbb R})^*$ that maps a linear form
$\theta$ to its maximum on $\Pol(T)$. As in section~\ref{ss:GGMSPol},
$P_\theta=\{x\in\Pol(T)\mid\langle\theta,x\rangle=\psi_{\Pol(T)}(\theta)\}$
is a face of~$\Pol(T)$.

The inclusion $i:\mathscr R_\theta\subseteq\mathscr A$ is an
exact functor, so it induces a group homomorphism
$\mathbf K(i):\mathbf K(\mathscr R_\theta)\to\mathbf K(\mathscr A)$
and a corresponding linear map $\mathbf K(i)_{\mathbb R}$.
Proposition~\ref{pr:TminTmax} then receives the following
interpretation.
\begin{corollary}
\label{co:FacesHN}
Let $\theta\in(\mathbf K(\mathscr A)_{\mathbb R})^*$. Let us
denote by $Q$ the HN polytope of $T_\theta^{\max}/T_\theta^{\min}$,
regarded as an object of $\mathscr R_\theta$ (so that
$Q\subseteq\mathbf K(\mathscr R_\theta)_{\mathbb R}$). Then
$$P_\theta=[T_\theta^{\min}]+\mathbf K(i)_{\mathbb R}(Q).$$
\end{corollary}
Thus the face $P_\theta$ of $\Pol(T)$ is the HN polytope
of $T_\theta^{\max}/T_\theta^{\min}$, computed relative to the
category $\mathscr R_\theta$, and shifted by $[T_\theta^{\min}]$.

A consequence of this observation is that $[T_\theta^{\min}]$ and
$[T_\theta^{\max}]$ are vertices of $\Pol(T)$. Another noteworthy
consequence is the following rigidity property: if $x$ is a vertex
of $\Pol(T)$, then $T$ has a unique subobject $X$ such that $[X]=x$.

We may also interpret the categories $\mathscr I_\theta$,
$\overline{\mathscr I}_\theta$, etc., in terms of HN polytopes.
For instance, an object $T$ belongs to $\overline{\mathscr I}_\theta$
if and only if $T=T_\theta^{\max}$, hence if and only if the top
vertex $[T]$ of $\Pol(T)$ lies on the face defined by $\theta$.

Given $T$ and $\theta$, the subfaces of $P_\theta$ are obtained
by perturbing slightly $\theta$. The following result states this
formally.
\begin{proposition}
\label{pr:SubfacPert}
Let $(\theta,\eta)\in\Hom_{\mathbb Z}(\mathbf K(\mathscr A),
\mathbb R)^2$, let $T\in\mathscr A$, let
$X=T_\theta^{\max}/T_\theta^{\min}$, and let
$i:\mathscr R_\theta\subseteq\mathscr A$ be the inclusion functor.
Let $0\subseteq X_{\eta\circ\mathbf K(i)}^{\min}\subseteq
X_{\eta\circ\mathbf K(i)}^{\max}\subseteq X$ be the filtration
of $X$, regarded as an object in $\mathscr R_\theta$, relative
to the group homomorphism $\eta\circ\mathbf K(i)\in\Hom_{\mathbb
Z}(\mathbf K(\mathscr R_\theta),\mathbb R)$. Call
$T_\theta^{\min}\subseteq T'\subseteq T''\subseteq T_\theta^{\max}$
the pull-back of this filtration by the canonical epimorphism
$T_\theta^{\max}\to X$. Then for $m$ large enough,
$T'=T_{m\theta+\eta}^{\min}$ and $T''=T_{m\theta+\eta}^{\max}$.
\end{proposition}
\begin{proof}
The classes in $\mathbf K(\mathscr A)$ of subquotients of $T$ are finitely
many. Pick $m$ large enough so that, for all subquotients $Z$ of $T$,
$$\theta([Z])>0\ \Longrightarrow\ (m\theta+\eta)([Z])>0
\quad\text{and}\quad
\theta([Z])<0\ \Longrightarrow\ (m\theta+\eta)([Z])<0.$$

Each nonzero quotient $Y$ of $T_\theta^{\min}$ satisfies
$\theta([Y])>0$, hence satisfies $(m\theta+\eta)([Y])>0$.
Therefore $T_\theta^{\min}\in\mathscr I_{m\theta+\eta}$,
and so $T_\theta^{\min}\subseteq T_{m\theta+\eta}^{\min}$.
The quotient $U=T_{m\theta+\eta}^{\min}/T_\theta^{\min}$
belongs to $\mathscr I_{m\theta+\eta}$ and to
$\overline{\mathscr P}_\theta$, so we have
$(m\theta+\eta)([U])\geq0$ and $\theta([U])\leq0$, which
forces $\theta([U])=0$.

In a similar fashion, we see that
$T_{m\theta+\eta}^{\max}\subseteq T_\theta^{\max}$ and
$\theta([T_\theta^{\max}/T_{m\theta+\eta}^{\max}])=0$.
We conclude that
$$T_\theta^{\min}\subseteq T_{m\theta+\eta}^{\min}\subseteq
T_{m\theta+\eta}^{\max}\subseteq T_\theta^{\max}$$
and that the subquotients of this filtration belong to
$\mathscr R_\theta$.

Reducing modulo $T_\theta^{\min}$, we get a three-step
filtration $0\subseteq X'\subseteq X''\subseteq X$ of
$X=T_\theta^{\max}/T_\theta^{\min}$, viewed as an object of
$\mathscr R_\theta$. Any nonzero quotient $Y$ of $X'$ in
$\mathscr R_\theta$ is a nonzero quotient of
$T_{m\theta+\eta}^{\min}$ in $\mathscr A$ such that $\theta([Y])=0$,
and so $\eta([Y])=(m\theta+\eta)([Y])>0$. Therefore $X'$ belongs to
the subcategory $\mathscr I_{\eta\circ\mathbf K(i)}$ of
$\mathscr R_\theta$. One checks in a similar fashion that
$X''/X'\in\mathscr R_{\eta\circ\mathbf K(i)}$ and
$X/X''\in\mathscr P_{\eta\circ\mathbf K(i)}$. We conclude that
$X'=X_{\eta\circ\mathbf K(i)}^{\min}$ and
$X''=X_{\eta\circ\mathbf K(i)}^{\max}$.
\end{proof}

Let us now compare this construction to the more usual notion
of Harder-Narasimhan filtration. To define the latter, we need
to fix a pair $(\eta,\theta)\in\Hom_{\mathbb Z}(\mathbf K(\mathscr
A),\mathbb R)^2$ such that $\eta([T])>0$ for each nonzero object
$T$. The slope of a nonzero object $T\in\mathscr A$ is defined
as $\mu(T)=\theta([T])/\eta([T])$ and an object $T$ is called
semistable if it is zero or if $\mu(X)\leq\mu(T)$ for any nonzero
subobject $X\subseteq T$. It can then be shown that any object
$T\in\mathscr A$ has a finite filtration
\begin{equation}
\label{eq:HNFilt}
0=T_0\subset T_1\subset\cdots\subset T_{\ell-1}\subset T_\ell=T
\end{equation}
whose subquotients are nonzero and semistable, with moreover
$\mu(T_k/T_{k-1})$ decreasing with $k$ (see for instance
\cite{Reineke03}, close to the present context). This filtration
is unique and is called the Harder-Narasimhan filtration of $T$.

Given $a\in\mathbb R$, the torsion subobject of $T$ with respect
to the torsion pair $(\mathscr I_{\theta-a\eta},\overline{\mathscr
P}_{\theta-a\eta})$ is $T_k$, where $k$ is the largest index such
that $\mu(T_k/T_{k-1})>a$. In our former notation, this means that
$T_k=T_{\theta-a\eta}^{\min}$; in particular, $[T_k]$ is a vertex
of $\Pol(T)$. One can even be more precise: the linear map
$\varphi:\mathbf K(\mathscr A)_{\mathbb R}\to\mathbb R^2$ given
in coordinates as $(\eta,\theta)$ projects $\Pol(T)$ to a convex
polygon of the plane, and the upper ridge of this polygon is
the polygonal line going successively through the points
$\varphi([T_k])$, for $0\leq k\leq\ell$. We leave the proof of
this fact to the reader.

The polygonal line just obtained is what Shatz calls the HN polygon
\cite{Shatz77}. Thus our HN polytopes are a multidimensional
analog of those HN polygons; they simply take into account the
existence of a whole space of stability conditions. There may
well exist sensible adaptations of this notion of HN polytope to
other contexts where spaces of stability conditions have been
defined (see for instance~\cite{Bridgeland09}).

\begin{other}{Remarks}
\label{rk:HNPol}
\begin{enumerate}
\item
\label{it:HNPa}
In Corollary~\ref{co:FacesHN}, the map $\mathbf K(i)_{\mathbb R}$
induces an actual loss of information. For example, in our study of
imaginary edges (sections~\ref{ss:ImagEdgesPart} and~\ref{ss:Cores}),
the category $\mathscr R_\theta$ has infinitely many simple objects;
since they all have the same dimension-vector, their classes have the
same image by $\mathbf K(i)_{\mathbb R}$.
\item
\label{it:HNPb}
The HN polytope of the direct sum of two objects is the Minkowski sum
of the HN polytopes of the two objects.
\end{enumerate}
\end{other}

\subsection{Nested families of torsion pairs}
\label{ss:NesTor}
The subobjects of $T$ that appear in the HN filtration
\eqref{eq:HNFilt} are the torsion subobjects with respect to the
torsion pairs $(\mathscr I_{\theta+a\eta},\overline{\mathscr
P}_{\theta+a\eta})$, as $a$ varies over $\mathbb R$. Observe that,
in the notation of section~\ref{ss:TorPairs},
$$\forall(a,b)\in\mathbb R^2,\qquad a\leq b\ \Longrightarrow\
(\mathscr I_{\theta+a\eta},\overline{\mathscr
P}_{\theta+a\eta})\preccurlyeq(\mathscr I_{\theta+b\eta},
\overline{\mathscr P}_{\theta+b\eta}).$$

This prompts the following definition: a nested
family of torsion pairs is the datum of a family
$(\mathscr T_a,\mathscr F_a)_{a\in A}$ of torsion
pairs, indexed by a totally ordered set $A$, such that
$$\forall(a,b)\in A^2,\qquad a\leq b\ \Longrightarrow\
(\mathscr T_a,\mathscr F_a)\preccurlyeq(\mathscr T_b,\mathscr F_b).$$
This definition is certainly less general than Rudakov's
study \cite{Rudakov97} but is sufficient for our
purposes.

A nested family of torsion pairs
$(\mathscr T_a,\mathscr F_a)_{a\in A}$ in $\mathscr A$ induces
a non-decreasing filtration $(T_a)_{a\in A}$ on any object
$T\in\mathscr A$: simply define $T_a$ as the torsion subobject
of $T$ with respect to $(\mathscr T_a,\mathscr F_a)$. As
already shown in section~\ref{ss:TorPairs}, the object $T_b/T_a$
is in $\mathscr F_a\cap\mathscr T_b$ whenever $a\leq b$.

\section{Background on preprojective algebras}
\label{se:RecPrepAlg}
\subsection{Basic definitions}
\label{ss:BasicDef}
We fix a base field $K$, which we assume for convenience to be
algebraically closed of characteristic $0$. As in section~\ref{ss:GenSetup},
we fix a graph $(I,E)$, where $I$ is the set of vertices and $E$ the set
of edges. We denote by $H$ the set of oriented edges of this graph. Thus
each edge in $E$ gives birth to two oriented edges in $H$, and $H$ comes
with a source map $s:H\to I$, a target map $t:H\to I$ and a fixed-point
free involution $a\mapsto\overline a$ such that $s(a)=t(\overline a)$ for
each $a\in H$.

An orientation is a subset $\Omega\subset H$ such that
$H=\Omega\sqcup\overline\Omega$. Such an orientation yields a quiver
$Q=(I,\Omega,s,t)$, and then $\overline Q=(I,H,s,t)$
is the double quiver of $Q$. We set $\varepsilon(a)=1$ if
$a\in\Omega$ and $\varepsilon(a)=-1$ if $a\notin\Omega$.

Let $K\overline Q$ be the path algebra of
$\overline Q$. The linear span
$\mathbf S=\linspan_K(e_i)_{i\in I}$ of the lazy paths is a
commutative semisimple subalgebra of $K\overline Q$.
The linear span $\mathbf A=\linspan_K(a)_{a\in H}$ of the paths of
length one is an $\mathbf S$-$\mathbf S$-bimodule.
Then $K\overline Q$ is the tensor algebra
$T_{\mathbf S}\mathbf A$. For $i\in I$, set
$$\rho_i=\sum_{\substack{a\in H\\s(a)=i}}\varepsilon(a)\overline aa,$$
the so-called preprojective relation at vertex $i$.
The linear span $\mathbf R=\linspan_K(\rho_i)_{i\in I}$ is a
$\mathbf S$-$\mathbf S$-subbimodule of $K\overline Q$.

By definition, the preprojective algebra of $Q$ is the quotient of
$K\overline Q$ by the ideal generated by~$\mathbf R$. This
is an augmented algebra over $\mathbf S$. Its completion with respect
to the augmentation ideal is called the completed preprojective algebra
and is denoted by $\Lambda_Q$. For brevity, we will generally drop
the $Q$ in the notation $\Lambda_Q$. Completing has the effect that
the augmentation ideal becomes the Jacobson radical; thus $\Lambda$
quotiented by its Jacobson radical is isomorphic to $\mathbf S$, and
the simple $\Lambda$-modules are just the simple $\mathbf S$-modules,
namely the one dimensional modules $S_i$. We denote by $\Lambda\mmod$
the category of finite dimensional left $\Lambda$-modules.

The involution $a\mapsto\overline a$ on the set $H$ of oriented
edges induces an anti-automorphism of $\Lambda$. If $M$ is a finite
dimensional left $\Lambda$-module, then we denote by $M^*$ the dual
module $\Hom_K(M,K)$, viewed as a left module by means of this
anti-automorphism.

Occasionally, we will have to write $\Lambda$-modules in a
concrete fashion. Our notation is as follows. An $\mathbf S$-module
$M$ is an $I$-graded vector space $M=\bigoplus_{i\in I}M_i$, where
$M_i=e_iM$; we define the dimension-vector of $M$ to be the element
$\dimvec M=\sum_{i\in I}(\dim M_i)\alpha_i$ in $\mathbb NI$.
A $K\overline Q$-module $M$ is the datum of an $\mathbf S$-module
and of linear maps $M_a:M_{s(a)}\to M_{t(a)}$ for each $a\in H$.

Since the simple $\Lambda$-modules are the modules $S_i$, concentrated
at a vertex of the quiver, it is natural to present a special notation
designed to analyze a $\Lambda$-module $M$ locally around a vertex $i$.
Specifically, we break the datum of $M$ in two parts: the first part
consists of the vector spaces $M_j$ for $j\neq i$ and of the linear
maps between them; the second part consists of the vector spaces and
of the linear maps that appear in the diagram
$$\bigoplus_{\substack{a\in H\\s(a)=i}}M_{t(a)}
\xrightarrow{\ (M_{\overline a})\ }M_i\xrightarrow{(\varepsilon(a)M_a)}
\bigoplus_{\substack{a\in H\\s(a)=i}}M_{t(a)}.$$
Here $(M_{\overline a})$ denotes a column-matrix, whose lines are
indexed by $\{a\in H\mid s(a)=i\}$; likewise, $(\varepsilon(a)M_a)$
denotes a row-matrix.

For brevity, we will write this diagram as
\begin{equation}
\label{eq:LocalDesc}
\widetilde M_i\xrightarrow{M_{\iin(i)}}M_i
\xrightarrow{M_{\out(i)}}\widetilde M_i.
\end{equation}
With this notation, the preprojective relation at $i$ is
$M_{\iin(i)}M_{\out(i)}=0$.

We define the $i$-socle of $M$ as the largest submodule of $M$
that is isomorphic to a direct sum of copies of $S_i$; we denote
it by $\soc_iM$. It is concentrated at the vertex $i$ and it
can be identified with the vector space $\ker M_{\out(i)}$.
Likewise, the largest quotient of $M$ that is isomorphic to a
direct sum of copies of $S_i$ is called the $i$-head of $M$ and
is denoted by $\hd_iM$. It is concentrated at the vertex $i$
and can be identified with the vector space $\coker M_{\iin(i)}$.

The assignment $[M]\mapsto\dimvec M$ provides an isomorphism
between the Grothendieck group $\mathbf K(\Lambda\mmod)$ and
the root lattice $\mathbb ZI$. Crawley-Boevey's formula
(Lemma~1 in~\cite{Crawley-Boevey00}) gives a module-theoretic
meaning to the bilinear form on the root lattice:
\begin{equation}
\label{eq:CrawleyBoeveyForm}
\dim\Hom_\Lambda(M,N)+\dim\Hom_\Lambda(N,M)-
\dim\Ext^1_\Lambda(M,N)=\bigl(\dimvec M,\dimvec N\bigr)
\end{equation}
for any finite dimensional $\Lambda$-modules $M$ and $N$.

\begin{other}{Remark}
\label{rk:PolDual}
The HN polytope $\Pol(T)$ of an object $T\in\Lambda\mmod$ lives
in $\mathbf K(\Lambda\mmod)_{\mathbb R}\cong\mathbb RI$. Since the
duality exchanges submodules and quotients and leaves the
dimension-vector unchanged, $\Pol(T^*)$ is the image of
$\Pol(T)$ under the involution $x\mapsto\dimvec T-x$ of
$\mathbb RI$. We leave it to the reader to check the equalities
$\mathscr I_{-\theta}=(\mathscr P_\theta)^*$,
$\overline{\mathscr I}_{-\theta}=(\overline{\mathscr
P}_\theta)^*$ and $\mathscr R_{-\theta}=(\mathscr R_\theta)^*$, for
any $\theta\in\Hom_{\mathbb Z}(\mathbf K(\Lambda\mmod),\mathbb R)$.
\end{other}

\subsection{Projective resolutions}
\label{ss:ProjRes}
We now recall Gei\ss, Leclerc and Schr\"oer's description of
the extension groups in the category $\Lambda\mmod$ (see
\cite{GeissLeclercSchroer07}, section~8).

Consider the complex of $\Lambda$-bimodules
\begin{equation}
\label{eq:GLSComp1}
\Lambda\otimes_{\mathbf S}\mathbf R\otimes_{\mathbf S}\Lambda
\xrightarrow{d_1}
\Lambda\otimes_{\mathbf S}\mathbf A\otimes_{\mathbf S}\Lambda
\xrightarrow{d_0}
\Lambda\otimes_{\mathbf S}\mathbf S\otimes_{\mathbf S}\Lambda
\to\Lambda\to0,
\end{equation}
where the map on the right is multiplication in $\Lambda$,
where for each $a\in H$
$$d_0(1\otimes a\otimes1)=a\otimes e_{s(a)}\otimes1-1\otimes
e_{t(a)}\otimes a,$$
and where for each $i\in I$
$$d_1(1\otimes\rho_i\otimes1)=\sum_{\substack{a\in H\\s(a)=i}}
\varepsilon(a)(\overline a\otimes a\otimes1+1\otimes\overline
a\otimes a).$$
Then \eqref{eq:GLSComp1} is the beginning of a projective
resolution of~$\Lambda$, by~\cite{GeissLeclercSchroer07},
Lemma~8.1.1.

Given $M,N\in\Lambda\mmod$, one can apply
$\Hom_\Lambda(?\otimes_\Lambda M,N)$ to \eqref{eq:GLSComp1}.
One then obtains the complex
\begin{equation}
\label{eq:GLSComp2}
0\to\bigoplus_{i\in I}\Hom_K(M_i,N_i)\xrightarrow{d^0_{M,N}}
\bigoplus_{a\in H}\Hom_K(M_{s(a)},N_{t(a)})\xrightarrow{d^1_{M,N}}
\bigoplus_{i\in I}\Hom_K(M_i,N_i),
\end{equation}
where
$$d^0_{M,N}:(f_i)_{i\in I}\mapsto
\bigl(N_af_{s(a)}-f_{t(a)}M_a\bigr)_{a\in H}$$
and
$$d^1_{M,N}:(g_a)_{a\in H}\mapsto
\left(\sum_{\substack{a\in H\\s(a)=i}}
\varepsilon(a)\bigl(N_{\overline a}g_a+g_{\overline a}M_a\bigr)\right)
\raisebox{-20pt}{$\scriptstyle i\in I$}.$$
Thus for $k\in\{0,1\}$, the extension group $\Ext^k_\Lambda(M,N)$
can be identified to the cohomology groups in degree~$k$ of the
complex \eqref{eq:GLSComp2}.

In \cite{GeissLeclercSchroer07}, section~8.2, Gei\ss, Leclerc and
Schr\"oer explain that in this identification, the bilinear map
$$\tau_1:\left(\bigoplus_{a\in H}\Hom_K(M_{s(a)},N_{t(a)})\right)\times
\left(\bigoplus_{a\in H}\Hom_K(N_{s(a)},M_{t(a)})\right)\to K$$
defined by
$$\tau_1\bigl((g_a),(h_a)\bigr)=\sum_{i\in I}\Tr\left(
\sum_{\substack{a\in H\\s(a)=i}}\varepsilon(a)g_{\overline a}h_a\right)$$
induces a non-degenerate pairing between $\Ext^1_\Lambda(M,N)$ and
$\Ext^1_\Lambda(N,M)$. Note that because of the cyclicity of the
trace and of the presence of the signs $\varepsilon(a)$, this
pairing $\tau_1$ is antisymmetric.

One sees likewise that the bilinear map
$$\tau_2:\left(\bigoplus_{i\in I}\Hom_K(M_i,N_i)\right)\times
\left(\bigoplus_{i\in I}\Hom_K(N_i,M_i)\right)\to K$$
defined by
$$\tau_2\bigl((f_i),(h_i)\bigr)=\sum_{i\in I}\Tr(f_ih_i)$$
induces a non-degenerate pairing between $\coker d^1_{M,N}$ and
$\Hom_\Lambda(N,M)$.

\subsection{Lusztig's nilpotent varieties}
\label{ss:LuszNilpVar}
Given a dimension-vector $\nu=\sum_{i\in I}\nu_i\alpha_i$ in
$\mathbb NI$, we set
$$\Rep_K(\overline Q,\nu)=\bigoplus_{a\in H}
\Hom_K(K^{\nu_{s(a)}},K^{\nu_{t(a)}}).$$
A structure of $\Lambda$-module on the $I$-graded vector space
$\bigoplus_{i\in I}K^{\nu_i}$ is specified by linear maps
$T_a:K^{\nu_{s(a)}}\to K^{\nu_{t(a)}}$, for each $a\in H$,
that is, by a point $T=(T_a)_{a\in H}$ in
$\Rep_K(\overline Q,\nu)$. To have an action of the
completed preprojective algebra, we must moreover impose the
preprojective relations and the nilpotency condition. These
equations define a subvariety
$$\Lambda(\nu)\subseteq\Rep_K\bigl(\overline Q,\nu\bigr),$$
called the affine variety of representations of $\Lambda$ or
Lusztig's nilpotent variety.

The group $G(\nu)=\prod_{i\in I}\GL_{\nu_i}(K)$ acts by conjugation
on $\Rep_K\bigl(\overline Q,\nu\bigr)$. This action
preserves the nilpotent variety $\Lambda(\nu)$. Then any isomorphism
class of a $\Lambda$-module of dimension-vector $\nu$ can be regarded
as a $G(\nu)$-orbit in~$\Lambda(\nu)$.

In~\cite{Lusztig91}, Lusztig shows that $\Lambda(\nu)$ is a
Lagrangian subvariety of $\Rep_K\bigl(\overline Q,\nu\bigr)$;
in particular, all the irreducible components of $\Lambda(\nu)$ have
dimension $\dim\bigl(\Rep_K\bigl(\overline Q,\nu\bigr)\bigr)/2$.
A straightforward calculation shows that this common dimension is
\begin{equation}
\label{eq:DimNilpVar}
\dim\Lambda(\nu)=\dim G(\nu)-(\nu,\nu)/2.
\end{equation}

A $\Lambda$-module $M$ is called rigid if $\Ext^1_\Lambda(M,M)=0$.
An immediate consequence of Crawley-Boevey's formula (see Corollary~3.15
in \cite{GeissLeclercSchroer06}) is that the closure of the
$G(\nu)$-orbit through a point $T\in\Lambda(\nu)$ is an irreducible
component if and only if $T$ is a rigid module.

As in the introduction, we denote by
$\mathfrak B(\nu)=\Irr\Lambda(\nu)$ the set of irreducible
components of the nilpotent variety and we define
$\mathfrak B=\bigsqcup_{\nu\in\mathbb NI}\mathfrak B(\nu)$.
This set $\mathfrak B$ is endowed with a crystal structure,
defined by Lusztig (\cite{Lusztig90b}, section~8), which we
quickly recall.

The weight of an element $Z\in\mathfrak B(\nu)$ is $\wt Z=\nu$. The
number $\varphi_i(Z)$ is the dimension of the $i$-head of a general point
$T\in Z$. The number $\varepsilon_i(Z)$ can then be found be the general
formula $\varphi_i(Z)-\varepsilon_i(Z)=\langle\alpha_i^\vee,\wt Z\rangle$.
The operators $\tilde e_i$ and $\tilde f_i$ add or remove a copy
of $S_i$ in general position at the top of a module $T\in Z$. In other
words, the relationship $Z'=\tilde e_iZ$ corresponds to extensions
$0\to T\to T'\to S_i\to0$ as general as possible: if $T$ runs over a
dense open subset of $Z$, then $T'$ will also run over a dense open
subset of $Z'$, and vice versa. The duality $*$ corresponds to the
involution $Z\mapsto Z^*$ on $\mathfrak B$, which preserves the weight.
We refer the reader to the literature for the formal definitions.

Kashiwara and Saito show in \cite{KashiwaraSaito97} that
the crystal $\mathfrak B$ is isomorphic to the crystal
$B(-\infty)$ of $U_q(\mathfrak n_+)$. This isomorphism is
canonical, because the only endomorphism of the crystal
$B(-\infty)$ is the identity. For $b\in B(-\infty)$, we write
$\Lambda_b$ for the image of $b$ under this isomorphism.
Moreover, the crystal $B(-\infty)$ is endowed with an involution,
usually also denoted by $*$ (see~\cite{Kashiwara95}, \S8.3),
and the proof in \cite{KashiwaraSaito97} shows that the
isomorphism $B(-\infty)\to\mathfrak B$ commutes with both
involutions --- thus, the use of the same notation $*$ does
not lead to any confusion.

\subsection{The canonical decomposition of a component}
\label{ss:CanonDec}
In this section, we quickly recall Crawley-Boevey and Schr\"oer's
results~\cite{Crawley-BoeveySchroer02} on the canonical
decomposition of irreducible components of module varieties,
specializing their results to the case of nilpotent varieties.

Let $\nu'$ and $\nu''$ be two dimension-vectors.
The function $(T',T'')\mapsto\dim\Hom_\Lambda(T',T'')$ on
$\Lambda(\nu')\times\Lambda(\nu'')$ is upper semicontinuous.
Given $Z'\in\Irr\Lambda(\nu')$ and $Z''\in\Irr\Lambda(\nu'')$,
we denote its minimum on $Z'\times Z''$ by $\hom_\Lambda(Z',Z'')$.
Then $\hom_\Lambda(Z',Z'')=\dim\Hom_\Lambda(T',T'')$
for $(T',T'')$ general in~$Z'\times Z''$.

Likewise, the function $(T',T'')\mapsto\dim\Ext^1_\Lambda(T',T'')$
on $\Lambda(\nu')\times\Lambda(\nu'')$ is upper semicontinuous.
Given $Z'\in\Irr\Lambda(\nu')$ and $Z''\in\Irr\Lambda(\nu'')$,
we denote its minimum on $Z'\times Z''$ by $\ext^1_\Lambda(Z',Z'')$.
Then $\ext^1_\Lambda(Z',Z'')=\dim\Ext^1_\Lambda(T',T'')$
for $(T',T'')$ general in~$Z'\times Z''$.

Let $n\geq1$, let $\nu_1$, \dots, $\nu_n$ be dimension-vectors,
and for $1\leq j\leq n$, let $Z_j\in\Irr\Lambda(\nu_j)$. Set
$\nu=\nu_1+\cdots+\nu_n$ and denote by $Z_1\oplus\cdots\oplus Z_n$
the set of all modules in $\Lambda(\nu)$ that are isomorphic to
a direct sum $T_1\oplus\cdots\oplus T_n$, with $T_j\in Z_j$ for all
$1\leq j\leq n$. This is an irreducible subset of $\Lambda(\nu)$.
Its closure $\overline{Z_1\oplus\cdots\oplus Z_n}$ is an
irreducible component of $\Lambda(\nu)$ if and only if
$\ext^1_\Lambda(Z_j,Z_k)=0$ for all $j\neq k$.

Conversely, for any $Z\in\Irr\Lambda(\nu)$, there exists $n$,
$\nu_j$ and $Z_j$ as above such that the general point in $Z_j$
is an indecomposable $\Lambda$-module and
$$Z=\overline{Z_1\oplus\cdots\oplus Z_n}.$$
Furthermore, the $Z_j$ are unique up to permutation. This is
called the canonical decomposition of $Z$.

\subsection{Torsion pairs in $\Lambda\mmod$}
\label{ss:TorLambda}
Here we consider the constructions of section~\ref{se:TorHNPoly}
in the category $\Lambda\mmod$. Since we are primarily interested
in the crystal $B(-\infty)$, we need to make sure that our constructions
go down to the level of irreducible components of the nilpotent varieties.

\begin{proposition}
\label{pr:PolConstr}
Let $\nu\in\mathbb NI$ be a dimension-vector and let
$\theta\in(\mathbb RI)^*$.
\begin{enumerate}
\item
\label{it:PCa}
For each $\xi\in\mathbb NI$, the set of all $T\in\Lambda(\nu)$
that contain a submodule of dimension-vector $\xi$ is closed.
\item
\label{it:PCb}
There are finitely many polytopes $\Pol(T)$, for
$T\in\Lambda(\nu)$. For each polytope $P\subseteq\mathbb RI$,
the set $\{T\in\Lambda(\nu)\mid\Pol(T)=P\}$ is constructible.
\item
\label{it:PCc}
For each category $\mathscr C$ among $\mathscr I_\theta$,
$\overline{\mathscr I}_\theta$, $\mathscr P_\theta$,
$\overline{\mathscr P}_\theta$ and $\mathscr R_\theta$,
the subset $\{T\in\Lambda(\nu)\mid T\in\mathscr C\}$ is open
in $\Lambda(\nu)$.
\end{enumerate}
\end{proposition}
\begin{proof}
When we view a point $T\in\Lambda(\nu)$ as a $\Lambda$-module,
we tacitly agree that the underlying $I$-graded vector space
of this module is $\bigoplus_{i\in I}K^{\nu_i}$. Let $X$ be
the set of all its $I$-graded vector subspaces
$\bigoplus_{i\in I}V_i$ of dimension-vector $\xi$; as a
product of Grassmannians, $X$ is naturally endowed with the
structure of a smooth projective variety. The incidence variety
$Y$ consisting of all pairs $(T,V)\in\Lambda(\nu)\times X$ such
that $T_a(V_{s(a)})\subseteq V_{t(a)}$ for all $a\in H$ is closed.
The first projection $Y\to\Lambda(\nu)$ is therefore a projective
morphism, so is proper. Its image is therefore closed, which
shows \ref{it:PCa}.

Let $R=(\mathbb NI)\cap(\nu-\mathbb NI)$; this is a finite set.
If $T\in\Lambda(\nu)$, then the dimension-vectors of the submodules
of $T$ form a subset $S(T)$ of $R$. Assertion~\ref{it:PCa} says
that $\{T\in\Lambda(\nu)\mid\xi\in S(T)\}$ is closed for each
$\xi\in R$. This implies that $\{T\in\Lambda(\nu)\mid S(T)=S\}$
is locally closed for each subset $S\subseteq R$. Gathering these
locally closed subsets according to the convex hull of $S$, we
obtain assertion~\ref{it:PCb}.

According to a remark following Corollary~\ref{co:FacesHN}, a
point $T\in\Lambda(\nu)$ belongs to $\overline{\mathscr I}_\theta$
if and only if $\nu$ lies on the face of $\Pol(T)$ defined by
$\theta$. This condition means that $S(T)$ does not meet
$\{\xi\in R\mid\langle\theta,\xi\rangle>\langle\theta,\nu\rangle\}$.
Thus assertion~\ref{it:PCa} exhibits $\{T\in\Lambda(\nu)\mid
T\in\overline{\mathscr I}_\theta\}$ as a finite intersection of
open subsets of $\Lambda(\nu)$. This shows the case $\mathscr
C=\overline{\mathscr I}_\theta$ in assertion~\ref{it:PCc}. The
other cases are dealt with in a similar fashion.
\end{proof}

Now let us fix a torsion pair $(\mathscr T,\mathscr F)$ in
$\Lambda\mmod$. All torsion submodules mentioned hereafter in
this section are taken with respect to it. We make the following
assumption:
\begin{description}
\item[(O)]
For each $\nu\in\mathbb NI$, both sets $\{T\in\Lambda(\nu)\mid
T\in\mathscr T\}$ and $\{T\in\Lambda(\nu)\mid T\in\mathscr F\}$
are open.
\end{description}

Under this assumption, it is legitimate to consider the set
$\mathfrak T(\nu)$ (respectively, $\mathfrak F(\nu)$) of all
irreducible components of $\Lambda(\nu)$ whose general point
belongs to $\mathscr T$ (respectively, $\mathscr F$). We then
get two subsets
$$\mathfrak T=\bigsqcup_{\nu\in\mathbb NI}\mathfrak T(\nu)
\quad\text{and}\quad
\mathfrak F=\bigsqcup_{\nu\in\mathbb NI}\mathfrak F(\nu)$$
of $\mathfrak B$. Our aim now is to construct a
bijection $\Xi:\mathfrak T\times\mathfrak F\to\mathfrak B$
that reflects at the component level the decomposition of
$\Lambda$-modules provided by $(\mathscr T,\mathscr F)$.

Let $(\nu_t,\nu_f)\in(\mathbb NI)^2$. We set $\nu=\nu_t+\nu_f$
and define $\Lambda^{\mathscr T}(\nu_t)=\{T_t\in\Lambda(\nu_t)\mid
T_t\in\mathscr T\}$ and $\Lambda^{\mathscr
F}(\nu_f)=\{T_f\in\Lambda(\nu_f)\mid T_f\in\mathscr F\}$.
We define $\Theta(\nu_t,\nu_f)$ as the set of all tuples
$(T,T_t,T_f,f,g)$ such that $T\in\Lambda(\nu)$,
$(T_t,T_f)\in\Lambda^{\mathscr T}(\nu_t)\times\Lambda^{\mathscr F}(\nu_f)$,
and $0\to T_t\xrightarrow fT\xrightarrow gT_f\to0$ is an exact
sequence in $\Lambda\mmod$. This is a quasi-affine algebraic
variety. We can then form the diagram
\begin{equation}
\label{eq:TorIrrComp}
\Lambda^{\mathscr T}(\nu_t)\times\Lambda^{\mathscr F}(\nu_f)
\xleftarrow p\Theta(\nu_t,\nu_f)\xrightarrow q\Lambda(\nu)
\end{equation}
in which $p$ and $q$ are the obvious projections.

\begin{lemma}
\label{le:TorFiber}
The map $p$ is a locally trivial fibration with a smooth and
connected fiber of dimension $\dim G(\nu)-(\nu_t,\nu_f)$.
The image of $q$ is the set of all points $T\in\Lambda(\nu)$
whose torsion submodule has dimension-vector $\nu_t$. The
non-empty fibers of $q$ are isomorphic to $G(\nu_t)\times G(\nu_f)$.
\end{lemma}
\begin{proof}
The statements concerning $q$ are obvious, so we only have to deal
with $p$.

The points $(T_a)$, $(T_{t,a})$ and $(T_{f,a})$, chosen in the
nilpotent varieties $\Lambda(\nu)$, $\Lambda^{\mathscr T}(\nu_t)$
and $\Lambda^{\mathscr F}(\nu_f)$, define $\Lambda$-module structures
on the vector spaces $\bigoplus_{i\in I}K^{\nu_i}$,
$\bigoplus_{i\in I}K^{\nu_{t,i}}$ and $\bigoplus_{i\in I}K^{\nu_{f,i}}$.

Let first consider the complex~\eqref{eq:GLSComp2} from section
\ref{ss:ProjRes}, with the $\Lambda$-modules $T_f$ and $T_t$ in
place of $M$ and $N$. The maps $d^0_{T_f,T_t}$ and $d^1_{T_f,T_t}$
of the complex depend on the datum of the arrows
$(T_{t,a})\in\Lambda^{\mathscr T}(\nu_t)$ and
$(T_{f,a})\in\Lambda^{\mathscr F}(\nu_f)$, but the spaces of the
complex depend only on $\nu_f$ and $\nu_t$. The map $d^0_{T_f,T_t}$
has rank $\dim\Hom_{\mathbf S}(T_f,T_t)-\dim\Hom_\Lambda(T_f,T_t)$.
In addition, $\Ext^1_\Lambda(T_f,T_t)\cong\ker d^1_{T_f,T_t}/\im
d^0_{T_f,T_t}$. Using Crawley-Boevey's formula
\eqref{eq:CrawleyBoeveyForm} and using the axiom (T1) of torsion
pairs, we easily compute
$$\dim\ker d^1_{M,N}=\dim\Ext^1_\Lambda(T_f,T_t)+\rk d^0_{T_f,T_t}
=\dim\Hom_{\mathbf S}(T_f,T_t)-(\nu_f,\nu_t).$$
Remarkably, this dimension depends only on $\nu_f$ and $\nu_t$,
and not on the datum of the arrows $(T_{t,a})$ and $(T_{f,a})$.

Let $E$ be the set of all exact sequences $0\to T_t\xrightarrow
fT\xrightarrow gT_f\to0$ of $I$-graded vector spaces. This is
a homogeneous space for the group $G(\nu)$ and the stabilizer of a
point $(f,g)$ is $\{\id+fhg\mid h\in\Hom_{\mathbf S}(T_f,T_t)\}$.
It is thus a smooth connected variety of dimension
$\dim G(\nu)-\dim\Hom_{\mathbf S}(T_f,T_t)$.

The fiber of $p$ over a point $((T_{t,a}),(T_{f,a}))\in
\Lambda^{\mathscr T}(\nu_t)\times\Lambda^{\mathscr F}(\nu_f)$
consists of the datum of $(f,g)\in E$ and of $(T_a)\in\Lambda(\nu)$,
with a compatibility condition between the two. The datum of
$(f,g)$ corresponds to a trivial fiber bundle over
$\Lambda^{\mathscr T}(\nu_t)\times\Lambda^{\mathscr F}(\nu_f)$
with fiber $E$. Let us now examine how $(T_a)$ can be chosen
when $((T_{t,a}),(T_{f,a}))$ and $(f,g)$ are given.

Once chosen an $I$-graded complementary subspace of $\ker g$ in the
vector space $T$, the set of possible choices for $(T_a)$ is
isomorphic to $\ker d^1_{T_f,T_t}$; moreover, the isomorphism
depends smoothly on $(f,g)$. The linear map $d^1_{T_f,T_t}$ depends
smoothly on $((T_{t,a}),(T_{f,a}))$ and has constant rank, as we
have seen above, so its kernel depends smoothly on
$((T_{t,a}),(T_{f,a}))$. In this fashion, we eventually see that
the set of possible choices for $(T_a)$ depends smoothly on
$((T_{t,a}),(T_{f,a}),f,g)$. Choosing trivializations where needed,
we conclude that $p$ is a locally trivial fibration.

Finally, the dimension of the fibers of $p$ is the sum of two
contributions, namely $\dim E$ and $\dim\ker d^1_{T_f,T_t}$.
We find $\dim G(\nu)-(\nu_f,\nu_t)$, as announced.
\end{proof}

Let $(Z_t,Z_f)\in\Irr\Lambda^{\mathscr T}(\nu_t)\times\Irr
\Lambda^{\mathscr F}(\nu_f)$. In view of Lemma~\ref{le:TorFiber},
$p^{-1}(Z_t\times Z_f)$ is an irreducible component of
$\Theta(\nu_t,\nu_f)$. Then $Z=q(p^{-1}(Z_t\times Z_f))$ is an
irreducible subset of $\Lambda(\nu)$ which, by e.g.\ I, \S8,
Theorem~3 in \cite{Mumford99}, has dimension
$$\dim(Z_t\times Z_f)+(\dim G(\nu)-(\nu_t,\nu_f))-
\dim(G(\nu_t)\times G(\nu_f)).$$
Equation~\eqref{eq:DimNilpVar} shows that this dimension is equal
to that of $\Lambda(\nu)$, so $\overline Z\in\Irr\Lambda(\nu)$.

This construction defines a map $(\overline{Z_t},\overline{Z_f})
\mapsto\overline Z$ from $\mathfrak T(\nu_t)\times\mathfrak F(\nu_f)$
to $\mathfrak B(\nu)$. Gluing these maps for all possible
$(\nu_t,\nu_f)$, we eventually get a map
$\Xi:\mathfrak T\times\mathfrak F\to\mathfrak B$.

\begin{theorem}
\label{th:TorIrrComp}
The map $\Xi:\mathfrak T\times\mathfrak F\to\mathfrak B$ is
bijective.
\end{theorem}
\begin{proof}
Given $\nu_t$ and $\nu_f$, $\Xi$ is a bijection from
$\mathfrak T(\nu_t)\times\mathfrak F(\nu_f)$ onto the set of
irreducible components of $\overline{q(\Theta(\nu_t,\nu_f))}$.

Now we take $\nu\in\mathbb NI$. We consider the diagrams
\eqref{eq:TorIrrComp} for all $\nu_t$ and $\nu_f$ such that
$\nu_t+\nu_f=\nu$. In this fashion, we split $\Lambda(\nu)$
according to the dimension-vector of the torsion submodule:
$$\Lambda(\nu)=\bigsqcup_{\nu_t+\nu_f=\nu}q(\Theta(\nu_t,\nu_f)).$$
Each piece of this partition is constructible, therefore an
irreducible component of $\Lambda(\nu)$ is contained in one
and only one closure $\overline{q(\Theta(\nu_t,\nu_f))}$.

Taking the union, we see that $\Xi$ defines a bijection from
$\bigsqcup_{\nu_t+\nu_f=\nu}(\mathfrak T(\nu_t)\times\mathfrak
F(\nu_f))$ onto $\mathfrak B(\nu)$.
\end{proof}

The construction of $\Xi$ implies that, if $T$ is a general
point of $\overline Z$, then the torsion submodule $X$ of $T$
has dimension-vector $\nu_t$ and the point $(X,T/X)$ is general
in $\overline{Z_t}\times\overline{Z_f}$. To see this, take a
$G(\nu_t)$-invariant dense open subset $U_t\subseteq Z_t$ and a
$G(\nu_f)$-invariant dense open subset of $U_f\subseteq Z_f$.
Then $p^{-1}(U_t\times U_f)$ is dense in $p^{-1}(Z_t\times Z_f)$.
The subset $q(p^{-1}(U_t\times U_f))$ is thus dense in the
irreducible set $\overline Z$, and is constructible by
Chevalley's theorem, so it contains a dense open subset $U$ of
$\overline Z$. By construction, if $T$ belongs to $U$, then
$\dimvec X=\nu_t$ and $(X,T/X)$ belongs to the prescribed open
subset $U_t\times U_f$, as desired.

Suppose now that we are given two torsion pairs
$(\mathscr T',\mathscr F')$ and $(\mathscr T'',\mathscr F'')$
in $\Lambda\mmod$ that both satisfy the openness condition (O).
They give rise to subsets $\mathfrak T'$, $\mathfrak F'$,
$\mathfrak T''$ and $\mathfrak F''$ of $\mathfrak B$ and to
bijections $\Xi':\mathfrak T'\times\mathfrak F'\to\mathfrak B$
and $\Xi'':\mathfrak T''\times\mathfrak F''\to\mathfrak B$.

\begin{proposition}
\label{pr:NesTorIrrComp}
Assume that $(\mathscr T',\mathscr F')\preccurlyeq(\mathscr
T'',\mathscr F'')$. Then the map $\Xi'$ restricts to a bijection
$\mathfrak T'\times(\mathfrak F'\cap\mathfrak T'')\to\mathfrak T''$,
the map $\Xi''$ restricts to a bijection
$(\mathfrak F'\cap\mathfrak T'')\times\mathfrak F''\to\mathfrak F'$,
and we have a commutative diagram
$$\xymatrix@C=5em{\mathfrak T'\times(\mathfrak F'\cap\mathfrak
T'')\times\mathfrak F''\ar[r]^(.6){\Xi'\times\id}
\ar[d]_{\id\times\Xi''}&
\mathfrak T''\times\mathfrak F''\ar[d]^{\Xi''}\\
\mathfrak T'\times\mathfrak F'\ar[r]_{\Xi'}&\mathfrak B.}$$
\end{proposition}
\begin{proof}
Let $(Z_1,Z_2)\in\mathfrak T'\times\mathfrak F'$ and set
$Z=\Xi'(Z_1,Z_2)$. Let $T$ be a general point of $Z$ and
let $X$ be the torsion submodule of $T$ with respect to
$(\mathscr T',\mathscr F')$. Then $X\in\mathscr T''$ and
the point $(X,T/X)$ is general in $Z_1\times Z_2$. Since a
torsion class is stable under taking quotients and extensions,
$T$ belongs to $\mathscr T''$ if and only is $T/X$ does.
This means that $Z$ belongs to $\mathfrak T''$ if and only
if $Z_2$ does. Thus $\Xi'$ restricts to a bijection
$\mathfrak T'\times(\mathfrak F'\cap\mathfrak T'')\to\mathfrak T''$,
as announced.

One shows that $\Xi''$ restricts to a bijection
$(\mathfrak F'\cap\mathfrak T'')\times\mathfrak
F''\to\mathfrak F'$ in a similar fashion.

Now let $(Z_1,Z_2,Z_3)\in\mathfrak T'\times(\mathfrak F'\cap
\mathfrak T'')\times\mathfrak F''$. Set $Z_4=\Xi'(Z_1,Z_2)$ and
$Z=\Xi''(Z_4,Z_3)$. Let $T$ be a general point of $Z$ and let
$X'$ and $X''$ be the torsion submodules of $T$ with respect to
$(\mathscr T',\mathscr F')$ and $(\mathscr T'',\mathscr F'')$,
respectively. Then the point $(X'',T/X'')$ is general in
$Z_4\times Z_3$. Since $X'$ is the torsion submodule of $X''$
with respect to $(\mathscr T',\mathscr F')$, the point
$(X',X''/X',T/X'')$ is general in $Z_1\times Z_2\times Z_3$.

A similar reasoning shows that $(X',X''/X',T/X'')$ is also
general in $\widetilde Z_1\times\widetilde Z_2\times\widetilde Z_3$,
where $(\widetilde Z_1,\widetilde Z_2,\widetilde Z_3)=
(\Xi'\circ(\id\times\Xi''))^{-1}(Z)$. Therefore a point
can be general in $Z_1\times Z_2\times Z_3\strut$ and in
$\widetilde Z_1\times\widetilde Z_2\times\widetilde Z_3$
at the same time. Then necessarily $(Z_1,Z_2,Z_3)=(\widetilde
Z_1,\widetilde Z_2,\widetilde Z_3)$, which establishes the
commutativity.
\end{proof}

The torsion pairs
$(\mathscr I_\theta,\overline{\mathscr P}_\theta)$ and
$(\overline{\mathscr I}_\theta,\mathscr P_\theta)$ satisfy
the assumption (O), thanks to
Proposition~\ref{pr:PolConstr}~\ref{it:PCc}. Applying
Proposition~\ref{pr:NesTorIrrComp} to them, we get a bijection
$$\Xi_\theta:\mathfrak I_\theta\times\mathfrak R_\theta\times
\mathfrak P_\theta\to\mathfrak B,$$
where $\mathfrak I_\theta$, $\mathfrak R_\theta$ and
$\mathfrak P_\theta$ are the subsets of $\mathfrak B$,
consisting of components whose general points belong to
$\mathscr I_\theta$, $\mathscr R_\theta$ and
$\mathscr P_\theta$, respectively.

We also note that Proposition~\ref{pr:NesTorIrrComp} can be
generalized in an obvious fashion to any finite nested sequence
$$(\mathscr T_0,\mathscr F_0)\preccurlyeq\cdots\preccurlyeq(\mathscr
T_\ell,\mathscr F_\ell)$$
of torsion pairs that satisfy (O).

\section{Tilting theory on preprojective algebras}
\label{se:TiltTheoPrepAlg}
\subsection{Reflection functors}
\label{ss:ReflFunc}
Let $I_i=\Lambda(1-e_i)\Lambda$; in other words, let $I_i$ be the
annihilator of the simple $\Lambda$-module $S_i$. Amplifying their
previous work \cite{IyamaReiten08}, Iyama and Reiten, in a joint
paper~\cite{BuanIyamaReitenScott09} with Buan and Scott, show
that $I_i$ is a tilting $\Lambda$-module of projective dimension
at most one and that~$\End_\Lambda(I_i)\cong\Lambda$ (at least,
when no connected component of $(I,E)$ is of Dynkin type).

This certainly invites us to look at the endofunctors
$\Sigma_i=\Hom_\Lambda(I_i,?)$ and $\Sigma_i^*=I_i\otimes_\Lambda?$
of the category $\Lambda\mmod$. It turns out that these
functors can be described in a very explicit fashion.

Recall that we locally depict a $\Lambda$-module $M$ around the
vertex $i$ by the diagram~\eqref{eq:LocalDesc}.

\begin{proposition}
\label{pr:DescReflFunc}
\begin{enumerate}
\item
\label{it:DRFa}
The module $\Sigma_iM$ is obtained by replacing~\eqref{eq:LocalDesc}
with
$$\widetilde M_i\xrightarrow{\overline M_{\out(i)}M_{\iin(i)}}
\ker M_{\iin(i)}\hookrightarrow\widetilde M_i,$$
where the map $\overline M_{\out(i)}:M_i\to\ker M_{\iin(i)}$ is
induced by $M_{\out(i)}$.
\item
\label{it:DRFb}
The module $\Sigma_i^*M$ is obtained by replacing~\eqref{eq:LocalDesc}
with
$$\widetilde M_i\twoheadrightarrow\coker M_{\out(i)}
\xrightarrow{M_{\out(i)}\overline M_{\iin(i)}}\widetilde M_i,$$
where the map $\overline M_{\iin(i)}:\coker M_{\out(i)}\to M_i$ is
induced by $M_{\iin(i)}$.
\end{enumerate}
\end{proposition}
\begin{proof}
Applying the functor $S_i\otimes_\Lambda?$ to the resolution
\eqref{eq:GLSComp1} and changing the right arrow by a sign, one
finds the following exact sequence of right $\Lambda$-modules:
\begin{equation}
\label{eq:GLSComp3}
K\rho_i\otimes_{\mathbf S}\Lambda\xrightarrow{\partial_1}
\bigoplus_{\substack{a\in H\\s(a)=i}}
K\overline a\otimes_{\mathbf S}\Lambda\xrightarrow{\partial_0}
e_iI_i\to0,
\end{equation}
where
$$\partial_1(\rho_i\otimes1)=\sum_{\substack{a\in H\\s(a)=i}}
\varepsilon(a)\overline a\otimes a\qquad\text{and}\qquad
\partial_0(\overline a\otimes1)=\overline a.$$
The sequence obtained by applying $?\otimes_\Lambda M$ to
\eqref{eq:GLSComp3} can be identified with
$$M_i\xrightarrow{M_{\out(i)}}\widetilde M_i\to
e_iI_i\otimes_\Lambda M\to0.$$
Using the decomposition $I_i=(1-e_i)\Lambda\oplus e_iI_i$, one
can subsequently identify $\Sigma_i^*M=I_i\otimes_\Lambda M$
with the vector space described in Statement~\ref{it:DRFb}.

Let us check the equality $(\Sigma_i^*M)_{\out(i)}=M_{\out(i)}
\overline M_{\iin(i)}$. Let $x\in\coker M_{\out(i)}$. It can be
represented by an element $(x_b)\in\widetilde M_i$, which, in
the identification
$$\widetilde M_i\cong\left(\bigoplus_{\substack{b\in H\\s(b)=i}}
K\overline b\otimes_{\mathbf S}\Lambda\right)\otimes_\Lambda M,\quad
\text{corresponds to}\quad\sum_{\substack{b\in H\\s(b)=i}}
\overline b\otimes x_b.$$
In the $\Lambda$-module $\Sigma_i^*M$, an arrow $a$ that starts
at $i$ maps $x$ to
$$\sum_{\substack{b\in H\\s(b)=i}}(a\overline b)\otimes x_b=
\sum_{\substack{b\in H\\s(b)=i}}M_aM_{\overline b}x_b,$$
where the left-hand side lives in $(1-e_i)\Lambda\otimes_\Lambda
M$. Therefore
$$(\Sigma_i^*M)_{\out(i)}(x)=
\left(\rule{0pt}{30pt}\varepsilon(a)\bigl(\Sigma_i^*M\bigr)_a(x)\right)
\raisebox{-17pt}{$\scriptstyle\substack{a\in H\\s(a)=i}$}=
\left(\sum_{\substack{b\in H\\s(b)=i}}
\varepsilon(a)M_aM_{\overline b}x_b\right)\raisebox{-17pt}{$\scriptstyle
\substack{a\in H\\s(a)=i}$}=M_{\out(i)}\overline M_{\iin(i)}(x).$$

One checks similarly that $(\Sigma_i^*M)_{\iin(i)}$ is the canonical
map $\widetilde M_i\to\coker M_{\out(i)}$, which concludes the
proof of~\ref{it:DRFb}. The proof of~\ref{it:DRFa} is similar, with
two differences: one starts with the exact sequence obtained by
applying $?\otimes_\Lambda S_i$ to the resolution~\eqref{eq:GLSComp1},
and one has to change the position of the signs $\varepsilon(a)$
(Remark~2.4 in~\cite{BaumannKamnitzer12} explains that this change
is without consequences).
\end{proof}

These mutually adjoint functors $\Sigma_i$ and $\Sigma_i^*$ are
called reflection functors. The concrete description afforded by
the proposition yields several important properties that they enjoy:
\begin{itemize}
\item
Adjunction morphisms can be inserted in functorial short exact
sequences
\begin{equation}
\label{eq:FuncSes}
0\to\soc_i\to\id\to\Sigma_i\Sigma_i^*\to0\quad\text{and}\quad
0\to\Sigma_i^*\Sigma_i\to\id\to\hd_i\to0
\end{equation}
(see Proposition~2.5 in~\cite{BaumannKamnitzer12}).
\item
They are exchanged by the $*$-duality; in other words,
$\Sigma_i^*T^*\cong(\Sigma_i T)^*$ for all finite dimensional
$\Lambda$-module~$T$.
\item
They induce Kashiwara and Saito's crystal reflections $S_i$ and
$S_i^*$ on $\mathfrak B$ (see section \ref{ss:TiltCrysOp}).
\item
The operation of restricting a representation of $\Lambda$ to a
representation of the quiver $Q$ intertwines the functors
$\Sigma_i$ and $\Sigma_i^*$ with the traditional
Bernstein-Gelfand-Ponomarev reflection functors (see
Proposition~7.1 in~\cite{BaumannKamnitzer12}).
\end{itemize}

\begin{other}{Remark}
\label{rk:HistReflFunc}
These functors $\Sigma_i$ and $\Sigma_i^*$ were introduced by
Iyama and Reiten in~\cite{IyamaReiten10} by means of the ideals
$I_i$, and also, independently, by the first two authors
in~\cite{BaumannKamnitzer12} by the explicit description of
Proposition~\ref{pr:DescReflFunc}. The link between the two
constructions was suggested to us by Amiot.
\end{other}

\subsection{The tilting ideals $I_w$}
\label{ss:TiltIdIw}
Reflection functors satisfy the braid relations, so it is natural
to study products of reflection functors computed according to
reduced decompositions of elements in $W$. We now look for an
analog of the exact sequences \eqref{eq:FuncSes} for such a product.

To simplify the presentation, we consider in this section the case
where no connected component of the diagram $(I,E)$ is of Dynkin type.
This allows us to rely on the following result, due to Buan, Iyama,
Reiten and Scott (section~II.1 in~\cite{BuanIyamaReitenScott09})
We will however argue in section~\ref{ss:DynkinType} that almost all
the results presented here hold true in general.

\begin{theorem}
\label{th:BuanIyamaReitenScott}
\begin{enumerate}
\item
\label{it:BIRSa}
If $s_{i_1}\cdots s_{i_\ell}$ is a reduced decomposition, then the
multiplication in $\Lambda$ gives rise to an isomorphism of bimodules
$I_{i_1}\otimes_\Lambda\cdots\otimes_\Lambda I_{i_\ell}\to
I_{i_1}\cdots I_{i_\ell}$.
\item
\label{it:BIRSb}
Under the same assumption, the product $I_{i_1}\cdots I_{i_\ell}$
depends only on $w=s_{i_1}\cdots s_{i_\ell}$; we can thus denote it
by $I_w$. It has finite codimension in $\Lambda$.
\item
\label{it:BIRSc}
Each $I_w$ is a tilting $\Lambda$-bimodule of projective dimension at
most~$1$ and $\End_\Lambda(I_w)\cong\Lambda$.
\item
\label{it:BIRSd}
If $\ell(ws_i)>\ell(w)$, then $\Tor_1^\Lambda(I_w,S_i)=0$.
If $\ell(s_iw)>\ell(w)$, then $\Ext^1_\Lambda(I_w,S_i)=0$.
\end{enumerate}
\end{theorem}

In view of Theorem~\ref{th:BuanIyamaReitenScott}~\ref{it:BIRSc}, it
is natural to apply Brenner and Butler's theorem. For that, we
define categories
\begin{xalignat*}2
\mathscr T_w&=\{T\mid I_w\otimes_\Lambda T=0\},&
\mathscr T^w&=\{T\mid\Ext^1_\Lambda(I_w,T)=0\},\\
\mathscr F_w&=\{T\mid\Tor^1_\Lambda(I_w,T)=0\},&
\mathscr F^w&=\{T\mid\Hom_\Lambda(I_w,T)=0\}.
\end{xalignat*}

\begin{theorem}
\label{th:BrennerButler}
\begin{enumerate}
\item
\label{it:BBa}
The pair $(\mathscr T^w,\mathscr F^w)$ is a torsion pair in
$\Lambda\mmod$. For each $\Lambda$-module $T$, the evaluation map
$I_w\otimes_\Lambda\Hom_\Lambda(I_w,T)\to T$ is injective and
its image is the torsion submodule of $T$ with respect to
$(\mathscr T^w,\mathscr F^w)$.
\item
\label{it:BBb}
The pair $(\mathscr T_w,\mathscr F_w)$ is a torsion pair in
$\Lambda\mmod$. For each $\Lambda$-module $T$, the coevaluation map
$T\to\Hom_\Lambda(I_w,I_w\otimes_\Lambda T)$ is surjective and
its kernel is the torsion submodule of $T$ with respect to
$(\mathscr T_w,\mathscr F_w)$.
\item
\label{it:BBc}
There are mutually inverse equivalences
$$\xymatrix@C=6em{\mathscr F_w\ar@<.6ex>[r]^{I_w\otimes_\Lambda?}&
\ar@<.6ex>[l]^{\Hom_\Lambda(I_w,?)}\mathscr T^w}.$$
\end{enumerate}
\end{theorem}
\begin{proof}
See \cite{Assem90}, in particular the lemma in section~1.6,
the corollary in section~1.9, and the theorem and its
corollary in section~2.1.
\end{proof}

Given $w\in W$ and $T\in\Lambda\mmod$, we will denote by $T^w$
and $T_w$ the torsion submodules of $T$ with respect to
$(\mathscr T^w,\mathscr F^w)$ and $(\mathscr T_w,\mathscr F_w)$,
respectively.

If $u,v\in W$ are such that $\ell(u)+\ell(v)=\ell(uv)$, then
$I_{uv}\cong I_u\otimes_\Lambda I_v$. It immediately follows that
$$\mathscr F^{uv}\supseteq\mathscr F^u\quad\text{and}\quad
\mathscr T_v\subseteq\mathscr T_{uv},$$
which can be written
\begin{equation}
\label{eq:TorsTheoMono1}
(\mathscr T^{uv},\mathscr F^{uv})\preccurlyeq
(\mathscr T^u,\mathscr F^u)\quad\text{and}\quad
(\mathscr T_v,\mathscr F_v)\preccurlyeq
(\mathscr T_{uv},\mathscr F_{uv}).
\end{equation}

\begin{other}{Remarks}
\label{rk:ProbInfDim}
\begin{enumerate}
\item
\label{it:PIDa}
The reader may here object that we apply a theorem valid for finite
dimensional algebras to an infinite dimensional framework. In fact,
there is no difficulty. The non-obvious point is to show that the
functors $\Hom_\Lambda(I_w,?)$ and $I_w\otimes_\Lambda?$ preserve
the category of finite dimensional $\Lambda$-modules. In the
case where $w$ is a simple reflection, this follows from
Proposition~\ref{pr:DescReflFunc}. The general case follows by
composition.
\item
\label{it:PIDb}
By Theorem~\ref{th:BrennerButler}~\ref{it:BBc}, any module
$T\in\mathscr T^w$ is isomorphic to a module of the form
$I_w\otimes_\Lambda X$. Conversely, one easily checks that
if $T$ is of the form $I_w\otimes_\Lambda X$, then the evaluation
map $I_w\otimes_\Lambda\Hom_\Lambda(I_w,T)\to T$ is surjective,
so $T$ belongs to $\mathscr T^w$ by Theorem~\ref{th:BrennerButler}
\ref{it:BBa}. Therefore $\mathscr T^w$ is the essential image of
the functor $I_w\otimes_\Lambda?$. Likewise, one shows that
$\mathscr F_w$ is the essential image of $\Hom_\Lambda(I_w,?)$.
\end{enumerate}
\end{other}

\begin{other}{Examples}
\label{ex:TorsTheo}
\begin{enumerate}
\item
\label{it:TTa}
The case where $w$ is a simple reflection has been dealt with
in the previous section: the functorial short exact sequences
\eqref{eq:FuncSes} imply that the torsion submodule of $T$ with
respect to $(\mathscr T_{s_i},\mathscr F_{s_i})$ is the $i$-socle
of $T$ and that the torsion-free quotient of $T$ with respect to
$(\mathscr T^{s_i},\mathscr F^{s_i})$ is the $i$-head of $T$.
Therefore
\begin{xalignat*}2
\mathscr T_{s_i}&=\add S_i,&
\mathscr T^{s_i}&=\{T\mid\hd_iT=0\},\\
\mathscr F_{s_i}&=\{T\mid\soc_iT=0\},&
\mathscr F^{s_i}&=\add S_i,
\end{xalignat*}
where $\add S_i$ is the additive closure of $S_i$ in $\Lambda\mmod$.
These equalities were also obtained by Sekiya and Yamaura (Lemma~2.23
in~\cite{SekiyaYamaura13}).
\item
\label{it:TTb}
Let us generalize the first example. Let $J\subseteq I$ and let $Q_J$
be the full subquiver $Q$ with set of vertices $J$. The preprojective
algebra $\Lambda_J$ is a quotient of $\Lambda$. The kernel of
the natural morphism $\Lambda\to\Lambda_J$ is the ideal
$I_J=\Lambda\bigl(1-\sum_{j\in J}e_j\bigr)\Lambda$.
The pull-back functor allows to identify $\Lambda_J\mmod$ with
the full subcategory
$$\{M\in\Lambda\mmod\mid e_jM\neq0\ \Rightarrow j\in J\}$$
of $\Lambda\mmod$. Moreover, a $\Lambda$-module $T$ belongs to
$\Lambda_J\mmod$ if and only if $\Hom_\Lambda(I_J,T)=0$.
In fact, if this equality holds, then $I_JT=0$, hence $T$ is a
$\Lambda/I_J$-module. Conversely, if $T\in\Lambda_J\mmod$, then
$$\Ext^1_\Lambda(\Lambda_J,T)\cong\Ext^1_{\Lambda_J}(\Lambda_J,T)=0,$$
for $\Lambda_J\mmod$ is closed under extensions; since the first
arrow in the short exact sequence
$$\Hom_\Lambda(\Lambda_J,T)\hookrightarrow\Hom_\Lambda(\Lambda,T)\to
\Hom_\Lambda(I_J,T)\to\Ext^1_\Lambda(\Lambda_J,T)=0$$
is an isomorphism, we get $\Hom_\Lambda(I_J,T)=0$.

Assume now that $Q_J$ is of Dynkin type. Then Theorem~II.3.5 in
\cite{BuanIyamaReitenScott09} says that $I_J$ is the ideal $I_{w_J}$,
where $w_J$ is the longest element in the parabolic subgroup
$W_J=\langle s_j\mid j\in J\rangle$ of $W$. We conclude that
$\mathscr F^{w_J}=\Lambda_J\mmod$. Further, by duality, we also
have $\mathscr T_{w_J}=\Lambda_J\mmod$.

Thus the torsion-free part of a module $T$ with respect to
the torsion pair $(\mathscr T^{w_J},\mathscr F^{w_J})$ is the
largest quotient of $T$ that belongs to $\Lambda_J\mmod$, and
the torsion part of $T$ with respect to the torsion pair
$(\mathscr T_{w_J},\mathscr F_{w_J})$ is the largest submodule
of $T$ that belongs to $\Lambda_J\mmod$.
\item
\label{it:TTc}
We have already mentioned that the reflection functors $\Sigma_i$
and $\Sigma_i^*$ are exchanged by the $*$-duality. By composition,
we obtain $(I_w\otimes_\Lambda T)^*=\Hom_\Lambda(I_{w^{-1}},T^*)$.
This readily implies
$$\mathscr F^w=(\mathscr T_{w^{-1}})^*\quad\text{and}\quad
\mathscr T^w=(\mathscr F_{w^{-1}})^*.$$
\item
\label{it:TTd}
We will explain in Examples~\ref{ex:CompIR} and~\ref{ex:CompGLS}
that $\mathscr F^w$ is Buan, Iyama, Reiten and Scott's category
$\Sub(\Lambda/I_w)$ \cite{BuanIyamaReitenScott09} and that
$\mathscr T_w$ is Gei\ss, Leclerc and Schr\"oer's category
$\mathcal C_w$ \cite{GeissLeclercSchroer11}.
\end{enumerate}
\end{other}

We conclude this section by showing that the equivalence of
categories described in Theorem~\ref{th:BrennerButler}~\ref{it:BBc}
can be broken into pieces according to any reduced decomposition
of $w$.
\begin{proposition}
\label{pr:CompEquiv}
Let $(u,v,w)\in W^3$ be such that $\ell(uvw)=\ell(u)+\ell(v)+\ell(w)$.
Then one has mutually inverse equivalences
$$\xymatrix@C=6em{\mathscr F_{uv}\cap\mathscr T^w
\ar@<.6ex>[r]^{I_v\otimes_\Lambda?}&
\ar@<.6ex>[l]^{\Hom_\Lambda(I_v,?)}\mathscr F_u\cap\mathscr T^{vw}}.$$
\end{proposition}
\begin{proof}
Certainly, $I_v\otimes_\Lambda?$ maps $\mathscr T^w$, the essential
image of $I_w\otimes_\Lambda?$, to $\mathscr T^{vw}$, the essential
image of $I_{vw}\otimes_\Lambda?$. On the other side, any module
$T\in\mathscr F_{uv}$ belongs to the essential image of
$$\Hom_\Lambda(I_{uv},?)=\Hom_\Lambda(I_v,\Hom_\Lambda(I_u,?)),$$
hence is isomorphic to $\Hom_\Lambda(I_v,X)$, with $X\in\mathscr F_u$.
By Theorem~\ref{th:BrennerButler}~\ref{it:BBa}, $I_v\otimes_\Lambda T$
is a submodule of $X$, hence belongs to $\mathscr F_u$. To sum up,
$I_v\otimes_\Lambda?$ maps $\mathscr T^w$ to $\mathscr T^{vw}$ and
maps $\mathscr F_{uv}$ to $\mathscr F_u$, hence maps
$\mathscr F_{uv}\cap\mathscr T^w$ to $\mathscr F_u\cap\mathscr T^{vw}$.

One shows in a dual fashion that $\Hom_\Lambda(I_v,?)$ maps
$\mathscr F_u\cap\mathscr T^{vw}$ to $\mathscr F_{uv}\cap\mathscr T^w$.
\end{proof}

Under the assumptions of the proposition, one thus has a chain of
equivalences of categories
$$\xymatrix@C=6em{\mathscr F_{uvw}
\ar@<.6ex>[r]^(.45){I_w\otimes_\Lambda?}&
\ar@<.6ex>[l]^(.55){\Hom_\Lambda(I_w,?)}\mathscr F_{uv}\cap\mathscr T^w
\ar@<.6ex>[r]^{I_v\otimes_\Lambda?}&
\ar@<.6ex>[l]^{\Hom_\Lambda(I_v,?)}\mathscr F_u\cap\mathscr T^{vw}
\ar@<.6ex>[r]^(.55){I_u\otimes_\Lambda?}&
\ar@<.6ex>[l]^(.45){\Hom_\Lambda(I_u,?)}\mathscr T^{uvw}}.$$

\begin{corollary}
\label{co:ReflDimVec}
Let $w\in W$ and let $T\in\mathscr F_w$. Then
$$\dimvec I_w\otimes_\Lambda T=w(\dimvec T).$$
\end{corollary}
\begin{proof}
The particular case where $w$ is a simple reflection is Lemma~2.5
in~\cite{AmiotIyamaReitenTodorov10}; it can also be proved in a more
elementary way with the help of Proposition~\ref{pr:DescReflFunc}.
The general case is obtained by writing $w$ as a product of $\ell(w)$
simple reflections.
\end{proof}

\subsection{Layers and stratification}
\label{ss:LayStrat}
In section~10 of~\cite{GeissLeclercSchroer11}, Gei\ss, Leclerc and
Schr\"oer construct a filtration of the objects in $\mathscr T_w$.
In section~2 of \cite{AmiotIyamaReitenTodorov10}, Amiot, Iyama,
Reiten and Todorov construct a filtration of the finite dimensional
module $\Lambda/I_w$ by layers. Our aim in this section is to show
that these two constructions are $*$-dual to each other and are
defined by the torsion pairs associated to the tilting
ideals~$I_w$. These results will help us later to identify the
categories $\mathscr T^w$, $\mathscr F^w$, $\mathscr T_w$ and
$\mathscr F_w$ with the categories $\mathscr I_\theta$ and
$\mathscr P_\theta$ from section~\ref{ss:TorPairs}.

We begin by introducing the layers. The next lemma is identical
to Corollary~9.3 in~\cite{GeissLeclercSchroer11}, to Theorem~2.6
in~\cite{AmiotIyamaReitenTodorov10}, and to Lemma~5.11
in~\cite{SekiyaYamaura13}. We cannot resist offering a fourth
proof of this basic result.
\begin{lemma}
\label{le:Layers}
Let $(w,i)\in W\times I$ be such that $\ell(ws_i)>\ell(w)$.
Then the simple module $S_i$ belongs to $\mathscr F_w$. Moreover, the
module $I_w\otimes_\Lambda S_i$ belongs to $\mathscr F^{ws_i}$
and has dimension-vector $w\alpha_i$.
\end{lemma}
\begin{proof}
We proceed by induction on the length of $w$. The case $w=1$ is obvious.
Assume $\ell(w)>0$ and write $w=s_jv$ with $\ell(v)<\ell(w)$. Applying
$\Hom_\Lambda(I_v,?)$ to the short exact sequence
$$0\to\soc_j(I_v\otimes_\Lambda S_i)\to I_v\otimes_\Lambda S_i\to
\Sigma_j\Sigma_j^*(I_v\otimes_\Lambda S_i)\to0$$
leads to
$$0\to\Hom_\Lambda(I_v,S_j)^{\oplus n}\to
\Hom_\Lambda(I_v,I_v\otimes_\Lambda S_i)\to
\Hom_\Lambda(I_w,I_w\otimes_\Lambda S_i)\to0$$
with $n=\dim\soc_j(I_v\otimes_\Lambda S_i)$, by
Theorem~\ref{th:BuanIyamaReitenScott}~\ref{it:BIRSd}. Applied to
$(v,i)$, the inductive hypothesis says that the middle term is
isomorphic to $S_i$, hence is simple. Applied to $(v^{-1},j)$,
the inductive hypothesis, together with
Example~\ref{ex:TorsTheo}~\ref{it:TTc} and Corollary~\ref{co:ReflDimVec},
implies that the left term has dimension-vector $n\,v^{-1}(\alpha_j)$.
Moreover, the assumption $\ell(ws_i)>\ell(w)$ means that
$v^{-1}(\alpha_j)\neq\alpha_i$. We conclude that necessarily $n=0$
and that the right term is isomorphic to $S_i$. Therefore $S_i$
belongs to the essential image of $\Hom_\Lambda(I_w,?)$, that is,
to $\mathscr F_w$. We also see that
$$\Hom_\Lambda(I_{ws_i},I_w\otimes_\Lambda S_i)\cong
\Sigma_iS_i=0,$$
which means that $I_w\otimes_\Lambda S_i$ belongs to
$\mathscr F^{ws_i}$. The last claim follows from
Corollary~\ref{co:ReflDimVec}.
\end{proof}

The following result is in essence due to Amiot, Iyama, Reiten and
Todorov~\cite{AmiotIyamaReitenTodorov10} and to Gei\ss, Leclerc and
Schr\"oer~\cite{GeissLeclercSchroer11}.
\begin{theorem}
\label{th:GLSFilt}
\begin{enumerate}
\item
\label{it:GLSFa}
Let $(w,i)\in W\times I$ be such that $\ell(ws_i)>\ell(w)$ and
let $T\in\Lambda\mmod$. Then $T^w/T^{ws_i}$ is the largest quotient
of $T^w$ isomorphic to a direct sum of copies of the module
$I_w\otimes_\Lambda S_i$.
\item
\label{it:GLSFb}
Let $w=s_{i_1}\cdots s_{i_\ell}$ be a reduced decomposition. Then
a module $T\in\Lambda\mmod$ belongs to $\mathscr F^w$ if and only
if it has a filtration whose subquotients are modules of the
form $I_{s_{i_1}\cdots s_{i_{k-1}}}\otimes_\Lambda S_{i_k}$ with
$1\leq k\leq\ell$.
\item
\label{it:GLSFc}
Let $w=s_{i_1}\cdots s_{i_\ell}$ be a reduced decomposition. Then
a module $T\in\Lambda\mmod$ belongs to $\mathscr T^w$ if and only
if there is no epimorphism $T\twoheadrightarrow I_{s_{i_1}\cdots
s_{i_{k-1}}}\otimes_\Lambda S_{i_k}$ with $1\leq k\leq\ell$.
\end{enumerate}
\end{theorem}
\begin{proof}
Let $w$, $i$ and $T$ as in~\ref{it:GLSFa} and set
$X=\Hom_\Lambda(I_w,T)$.
Applying $I_w\otimes_\Lambda?$ to the short exact sequence
$$0\to\Sigma_i^*\Sigma_iX\to X\to\hd_iX\to0$$
and using $\Tor_1^\Lambda(I_w,S_i)=0$
(Theorem~\ref{th:BuanIyamaReitenScott}~\ref{it:BIRSd}), we get
$$0\to I_{ws_i}\otimes_\Lambda\Hom_\Lambda(I_{ws_i},T)\to
I_w\otimes_\Lambda\Hom_\Lambda(I_w,T)\to I_w\otimes_\Lambda
S_i^{\oplus m}\to0,$$
where $m=\dim\hd_iX$. Therefore $T^w/T^{ws_i}$ has the desired form.

To finish the proof of~\ref{it:GLSFa}, it remains to show the
maximality. Before that, we look at Statements~\ref{it:GLSFb}
and~\ref{it:GLSFc}. For $1\leq k\leq\ell$, we set
$L_k=I_{s_{i_1}\cdots s_{i_{k-1}}}\otimes_\Lambda S_{i_k}$.

Let $T\in\Lambda\mmod$. What has already been showed from
Statement~\ref{it:GLSFa} implies that $T/T^w$ has a filtration
whose subquotients are all isomorphic to some $L_k$.
If $T$ belongs to $\mathscr F^w$, then $T=T/T^w$ has such a
filtration: this shows the necessity of the condition proposed
in Statement~\ref{it:GLSFb}.
If $T$ has no quotient isomorphic to a module $L_k$, then
$T=T^w$, and therefore $T$ belongs to $\mathscr T^w$: this shows
the sufficiency of the condition proposed in Statement~\ref{it:GLSFc}.

By Lemma~\ref{le:Layers}, $L_k$ belongs to
$\mathscr F^{s_{i_1}\cdots s_{i_k}}$, hence to $\mathscr F^w$.
Any iterated extension of modules in the set
$\{L_k\mid1\leq k\leq\ell\}$ therefore belongs to $\mathscr F^w$,
for $\mathscr F^w$ is closed under extensions: this shows the
sufficiency of the condition in Statement~\ref{it:GLSFb}. On the
other hand, a module in $\mathscr T^w$ cannot have a nonzero map
to a module in $\mathscr F^w$, hence has no quotient isomorphic to
a module $L_k$: this shows the necessity of the condition in
Statement~\ref{it:GLSFc}.

Now let us go back to Statement~\ref{it:GLSFa}, resuming the proof
where we left it. We choose a reduced decomposition
$w=s_{i_1}\cdots s_{i_{\ell-1}}$ and set $i_\ell=i$.
For $1\leq k\leq\ell$, we set
$L_k=I_{s_{i_1}\cdots s_{i_{k-1}}}\otimes_\Lambda S_{i_k}$.

Assume the existence of a short exact sequence of the form
$0\to Z\to T^w\to L_\ell^{\oplus n}\to0$, where $n$ is a
positive integer and $Z$ does not belong to $\mathscr T^w$.
Then $Z$ has a quotient $Z/Y$ isomorphic to some $L_k$, with $k<\ell$,
and we have an extension $0\to L_k\to T^w/Y\to L_\ell^{\oplus n}\to0$.
By Lemma~\ref{le:Layers}, $\dimvec L_k$ and $\dimvec L_\ell$ are
distinct real roots, hence are not proportional; therefore
$\dimvec T^w/Y$ is not a multiple of $\dimvec L_\ell$. Let us set
$U=T^w/Y$. This module belongs to $\mathscr T^w$, because $T^w$
does, so $U=U^w$; it also belongs to $\mathscr F^{ws_i}$, because
$L_k$ and $L_\ell$ do and because $\mathscr F^{ws_i}$ is closed
under extensions, so $U^{ws_i}=0$. Therefore $U=U^w/U^{ws_i}$,
which implies (by the first part of the proof applied to $U$)
that $U$ is a direct sum of copies of $L_\ell$. We thus reach
a contradiction, and conclude that the assumption at the
beginning of this paragraph is wrong.

Consider a short exact sequence $0\to Z\to T^w\to L_\ell^{\oplus n}\to0$.
The preceding paragraph says that $Z$ belongs to $\mathscr T^w$.
Further, we note that $\Hom_\Lambda(I_w,T^w)=\Hom_\Lambda(I_w,T)=X$,
by Theorem~\ref{th:BrennerButler} and
Remark~\ref{rk:ProbInfDim}~\ref{it:PIDb}. Applying
$\Hom_\Lambda(I_w,?)$ to the short exact sequence, we then get
$0\to\Hom_\Lambda(I_w,Z)\to X\to S_i^{\oplus n}\to0$. Therefore,
by definition of the $i$-head, $\Hom_\Lambda(I_w,Z)$ contains
$\Sigma_i^*\Sigma_iX$, and moreover the quotient
$\Hom_\Lambda(I_w,Z)/\Sigma_i^*\Sigma_iX$ is a direct sum
of copies of $S_i$. Applying $I_w\otimes_\Lambda?$ to the inclusion
$\Sigma_i^*\Sigma_iX\hookrightarrow\Hom_\Lambda(I_w,Z)$ and using
Theorem~\ref{th:BuanIyamaReitenScott}~\ref{it:BIRSd}, we
conclude that $Z=I_w\otimes_\Lambda\Hom_\Lambda(I_w,Z)$
contains $T^{ws_i}=I_w\otimes_\Lambda\Sigma_i^*\Sigma_iX$.
\end{proof}

This theorem admits a dual version, which we state for reference.

\begin{theorem}
\label{th:GLSFiltDual}
\begin{enumerate}
\item
\label{it:GLSFDa}
Let $(w,i)\in W\times I$ be such that $\ell(s_iw)>\ell(w)$ and
let $T\in\Lambda\mmod$. Then $T_{s_iw}/T_w$ is the largest submodule
of $T/T_w$ isomorphic to a direct sum of copies of the module
$\Hom_\Lambda(I_w,S_i)$.
\item
\label{it:GLSFDb}
Let $w=s_{i_\ell}\cdots s_{i_1}$ be a reduced decomposition. Then
a module $T\in\Lambda\mmod$ belongs to $\mathscr T_w$ if and only
if it has a filtration whose subquotients are modules of the
form $\Hom_\Lambda(I_{s_{i_{k-1}}\cdots s_{i_1}},S_{i_k})$ with
$1\leq k\leq\ell$.
\item
\label{it:GLSFDc}
Let $w=s_{i_\ell}\cdots s_{i_1}$ be a reduced decomposition. Then
a module $T\in\Lambda\mmod$ belongs to $\mathscr F_w$ if and only
if there is no monomorphism $\Hom_\Lambda(I_{s_{i_{k-1}}\cdots
s_{i_1}},S_{i_k})\hookrightarrow T$ with $1\leq k\leq\ell$.
\end{enumerate}
\end{theorem}

As already mentioned, most of the results presented above in this
section can be found in papers by Iyama, Reiten et al.\ and by
Gei\ss, Leclerc and Schr\"oer. The aim of the next three examples is
to explain some connections in more detail.

\begin{other}{Example}
\label{ex:CompBKAIRT}
Let us denote by $\{\omega_i\mid i\in I\}$ the set of fundamental
weights of the root system $\Phi$; these weights are elements of
a representation of $W$ which contains $\mathbb RI$, and for all
$(i,j)\in I^2$, we have $s_j\omega_i=\omega_i-\delta_{ij}\alpha_i$,
where $\delta_{ij}$ is Kronecker's symbol. Let us now fix
$(i,w)\in I\times W$. By \cite{BaumannKamnitzer12}, Theorem~3.1~(ii),
there exists a unique $\Lambda$-module $N(-w\omega_i)$ whose
dimension-vector is $\omega_i-w\omega_i$ and whose socle is $0$
or $S_i$. (The paper~\cite{BaumannKamnitzer12} mainly deals
with the case where $\mathfrak g$ is finite dimensional, but the
constructions in sections~2 and~3 are valid in general, with the
exception of Proposition~3.6.) The aim of this example is to show that
$N(-w\omega_i)\cong((\Lambda/I_w)\otimes_\Lambda\Lambda e_i)^*$, where
$e_i$ is the lazy path at vertex $i$ (see section~\ref{ss:BasicDef}).
For that, we consider a reduced decomposition $w=s_{j_1}\cdots
s_{j_\ell}$. The filtration $\Lambda\supset I_{s_{j_1}}\supset
I_{s_{j_1}s_{j_2}}\supset\cdots\supset I_{s_{j_1}\cdots
s_{j_{\ell-1}}}\supset I_w$ induces a filtration of
$\Lambda/I_w$, whose $k$-th subquotient is the layer
$$I_{s_{j_1}\cdots s_{j_{k-1}}}/I_{s_{j_1}\cdots s_{j_k}}\cong
I_{s_{j_1}\cdots s_{j_{k-1}}}\otimes_\Lambda(\Lambda/I_{j_k})
\cong I_{s_{j_1}\cdots s_{j_{k-1}}}\otimes_\Lambda S_{j_k}$$
studied by Amiot, Iyama, Reiten and Todorov
\cite{AmiotIyamaReitenTodorov10}. Tensoring this filtration
on the right with the projective $\Lambda$-module $\Lambda e_i$
kills all the subquotients with $j_k\neq i$, so by
Lemma~\ref{le:Layers},
$$\dimvec(\Lambda/I_w)\otimes_\Lambda\Lambda e_i=
\sum_{\substack{1\leq k\leq\ell\\[2pt]j_k=i}}
s_{j_1}\cdots s_{j_{k-1}}\alpha_{j_k}=\omega_i-w\omega_i.$$
In addition, the head of $\Lambda e_i$ is equal to $S_i$, so the
head of $(\Lambda/I_w)\otimes_\Lambda\Lambda e_i$ is $0$ or $S_i$.
The dual of $(\Lambda/I_w)\otimes_\Lambda\Lambda e_i$ therefore
satisfies the two conditions that characterize $N(-w\omega_i)$.
\end{other}

\begin{other}{Example}
\label{ex:CompIR}
Theorem~\ref{th:GLSFilt}~\ref{it:GLSFb} and Corollary~2.9 in
\cite{IyamaReiten10} imply that our category $\mathscr F^w$ is equal
to $\Sub(\Lambda/I_w)$, the full subcategory of $\Lambda\mmod$ whose
objects are the modules that can be embedded in a finite direct sum
of copies of $\Lambda/I_w$. (To make sure that the assumptions of
the statements in~\cite{IyamaReiten10} are fulfilled, observe that
a module $T\in\mathscr F^w$ satisfies $\Hom_\Lambda(I_w,T)=0$
by definition, whence $I_wT=0$, so $T$ can be seen as a
$\Lambda/I_w$-module.)
\end{other}

\begin{other}{Example}
\label{ex:CompGLS}
We now compare Theorem~\ref{th:GLSFilt} to Gei\ss, Leclerc
and Schr\"oer's stratification. We fix a reduced decomposition
$w=s_{i_\ell}\cdots s_{i_1}$. Then Gei\ss, Leclerc and Schr\"oer
define modules $V_k$ and $M_k$ for $1\leq k\leq\ell$ (sections~2.4
and~9 of~\cite{GeissLeclercSchroer11}). By construction, $V_k$
is a submodule of the injective hull of $S_{i_k}$, hence has
socle $0$ or $S_{i_k}$; moreover, the dimension-vector of
$V_k$ is $\omega_{i_k}-s_{i_1}\cdots s_{i_k}\omega_{i_k}$, by
Corollary~9.2 in~\cite{GeissLeclercSchroer11}. Comparing with
Example~\ref{ex:CompBKAIRT}, we see that
$$V_k\cong N(-s_{i_1}\cdots s_{i_k}\omega_{i_k})\cong
((\Lambda/I_{s_{i_1}\cdots s_{i_k}})\otimes_\Lambda\Lambda e_{i_k})^*.$$
If $\{s\in\{1,\dots,k-1\}\mid i_s=i_k\}$ is not empty,
then $k^-$ is defined to be the largest element in this set and
$M_k$ is defined to be the quotient $V_k/V_{k^-}$; note here that
$$V_{k^-}\cong N(-s_{i_1}\cdots s_{i_{k^-}}\omega_{i_{k^-}})=
N(-s_{i_1}\cdots s_{i_{k-1}}\omega_{i_k})\cong((\Lambda/I_{s_{i_1}
\cdots s_{i_{k-1}}})\otimes_\Lambda\Lambda e_{i_k})^*.$$
Otherwise, $M_k$ is defined to be $V_k$. In both case, $M_k$ is
the dual of
$$(I_{s_{i_1}\cdots s_{i_{k-1}}}/I_{s_{i_1}\cdots s_{i_k}})
\otimes_\Lambda\Lambda e_{i_k}\cong I_{s_{i_1}\cdots s_{i_{k-1}}}
\otimes_\Lambda(\Lambda/I_{i_k})\otimes_\Lambda\Lambda e_{i_k}\cong
I_{s_{i_1}\cdots s_{i_{k-1}}}\otimes_\Lambda S_{i_k}.$$
By Example~\ref{ex:TorsTheo}~\ref{it:TTc}, this gives
$M_k\cong\Hom_\Lambda(I_{s_{i_{k-1}}\cdots s_{i_1}},S_{i_k})$.
Further, Gei\ss, Leclerc and Schr\"oer consider the module
$V=V_1\oplus\cdots V_\ell$ and the category $\mathcal C_w=\Fac(V)$,
the full subcategory of $\Lambda\mmod$ whose objects are the
homomorphic images of a direct sum of copies of $V$. Comparing
Lemma~10.2 in~\cite{GeissLeclercSchroer11} with
Theorem~\ref{th:GLSFiltDual}~\ref{it:GLSFDb}, we see that
$\mathcal C_w$ is our category $\mathscr T_w$ and that the
stratification constructed in section~10 of~\cite{GeissLeclercSchroer11}
on a $\Lambda$-module $T\in\mathcal C_w$ coincides with the
filtration by the submodules $T_{s_{i_k}\cdots s_{i_1}}$. Combining
this with the conclusion of Example~\ref{ex:CompIR}, we additionally
obtain that $\mathcal C_w=\Fac((\Lambda/I_{w^{-1}})^*)$, in agreement
with Theorem~2.8~(iv) in~\cite{GeissLeclercSchroer11}.
\end{other}

Theorem~\ref{th:GLSFilt} has the following noteworthy consequence.
\begin{proposition}
\label{pr:TorsTheoMono2}
Let $(u,v)\in W^2$ such that $\ell(u)+\ell(v)=\ell(uv)$.
Then $(\mathscr T_u,\mathscr F_u)\preccurlyeq
(\mathscr T^v,\mathscr F^v)$.
\end{proposition}
\begin{proof}
We want to show that $\mathscr T_u\cap\mathscr F^v=\{0\}$.
Assume the contrary and choose $T\neq0$ of minimal dimension
in the intersection. Write $u=s_{i_\ell}\cdots s_{i_1}$ and
$v=s_{j_1}\cdots s_{j_m}$. By
Theorem~\ref{th:GLSFilt}~\ref{it:GLSFb}, $T$ has a
filtration with subquotients of the
form $I_{s_{j_1}\cdots s_{j_{q-1}}}\otimes_\Lambda S_{j_q}$.
Any subquotient of this filtration belongs to $\mathscr F^v$,
and the top one also belongs to $\mathscr T_u$ since a
torsion class is closed under taking quotients. The minimality of
$\dim T$ imposes then that the filtration has just one step.
Dually, $T$ has a filtration with subquotients of the form
$\Hom_\Lambda(I_{s_{i_{p-1}}\cdots s_{i_1}},S_{i_p})$, and minimality
impose again that this filtration has just one step. We end up
with an isomorphism
$$\Hom_\Lambda(I_{s_{i_{p-1}}\cdots s_{i_1}},S_{i_p})\cong
I_{s_{j_1}\cdots s_{j_{q-1}}}\otimes_\Lambda S_{j_q}.$$
Taking dimension-vectors, we get $s_{i_1}\cdots s_{i_{p-1}}\alpha_{i_p}
=s_{j_1}\cdots s_{j_{q-1}}\alpha_{j_q}$, by Lemma~\ref{le:Layers}.
This contradicts the assumption $\ell(u)+\ell(v)=\ell(uv)$.
\end{proof}

We conclude this section with a proposition that slightly refines
Proposition~3.2 in~\cite{IyamaReiten10}.
\begin{proposition}
\label{pr:ShiftF^w}
Assume that $\ell(u)+\ell(v)=\ell(uv)$. Then one has equivalences
of categories
$$\xymatrix@C=6em{\mathscr F^v\ar@<.6ex>[r]^(.43){I_u\otimes_\Lambda?}&
\ar@<.6ex>[l]^(.57){\Hom_\Lambda(I_u,?)}\mathscr F^{uv}\cap\mathscr T^u}
\qquad\text{and}\qquad
\xymatrix@C=6em{\mathscr T_{uv}\cap\mathscr F_v\ar@<.6ex>[r]^(.59){I_v
\otimes_\Lambda?}&\ar@<.6ex>[l]^(.41){\Hom_\Lambda(I_v,?)}\mathscr T_u}.$$
\end{proposition}
\begin{proof}
If $T\in\mathscr F^v$, then $T\in\mathscr F_u$, by
Proposition~\ref{pr:TorsTheoMono2}. Thus
$T=T/T_u\cong\Hom_\Lambda(I_u,I_u\otimes_\Lambda T)$, whence
$$\Hom_\Lambda(I_{uv},I_u\otimes_\Lambda T)\cong
\Hom_\Lambda(I_v,\Hom_\Lambda(I_u,I_u\otimes_\Lambda
T))\cong\Hom_\Lambda(I_v,T)=0,$$
which shows that $I_u\otimes_\Lambda T\in\mathscr F^{uv}$.
Conversely, if $T\in\mathscr F^{uv}$, then
$$\Hom_\Lambda(I_v,\Hom_\Lambda(I_u,T))=\Hom_\Lambda(I_{uv},T)=0,$$
hence $\Hom_\Lambda(I_u,T)\in\mathscr F^v$. The
first pair of equivalences then follows from
Theorem~\ref{th:BrennerButler}~\ref{it:BBc}.

The second equivalence is proved in a similar fasion (or follows
by duality).
\end{proof}

\subsection{Tilting structure and HN polytopes in $\Lambda\mmod$}
\label{ss:TiltPol}
We now relate all this material about the reflection
functors and the torsion pairs $(\mathscr T^w,\mathscr F^w)$ and
$(\mathscr T_w,\mathscr F_w)$ to the categories $\mathscr I_\theta$,
$\mathscr R_\theta$, etc., and to the HN polytopes defined in
section~\ref{se:TorHNPoly}.

The following result is almost identical to Theorem~3.5 in
\cite{SekiyaYamaura13}; we however prove the key point in a
different fashion.
\begin{theorem}
\label{th:SekiyaYamaura}
Let $\theta:\mathbb ZI\to\mathbb R$ be a group homomorphism
and let $i\in I$. If $\langle\theta,\alpha_i\rangle>0$, then
$\Sigma_i$ and $\Sigma_i^*$ induce mutually inverse equivalences
$$\xymatrix@C=3.4em{\mathscr R_\theta\ar@<.6ex>[r]^{\Sigma_i^*}&
\ar@<.6ex>[l]^{\Sigma_i}\mathscr R_{s_i\theta}}.$$
\end{theorem}
\begin{proof}
The assumption $\langle\theta,\alpha_i\rangle>0$ forbids a
module $T\in\mathscr R_\theta$ to have a submodule isomorphic
to $S_i$, so $\mathscr R_\theta\subseteq\mathscr F_{s_i}$ by
Example~\ref{ex:TorsTheo}~\ref{it:TTa}. Likewise,
$\mathscr R_{s_i\theta}\subseteq\mathscr T^{s_i}$.

Let $T\in\mathscr R_\theta$. Then $T\in\mathscr F_{s_i}$
and $\langle\theta,\dimvec T\rangle=0$. Using
Corollary~\ref{co:ReflDimVec} with $w=s_i$, we obtain
$\langle s_i\theta,\dimvec\Sigma_i^*T\rangle=0$.

To show that $\Sigma_i^*T$ is $s_i\theta$-semistable, it
remains to show that $\langle s_i\theta,\dimvec X\rangle\geq0$
for any quotient $X$ of $\Sigma_i^*T$. In this aim, consider a
surjective morphism $f:\Sigma_i^*T\to X$. The functor $\Sigma_i$
modifies only the vector space at vertex $i$, so the cokernel of
$\Sigma_if$ is concentrated at this vertex. There exists thus a
natural integer $n$ such that $\dimvec\coker(\Sigma_if)=n\alpha_i$.
Note now that not only $\Sigma_i^*T$, but also $X$ belong to
$\mathscr T^{s_i}$, for a torsion class is closed under taking
quotients. By Corollary~\ref{co:ReflDimVec} again, we have
$\dimvec\Sigma_iX=s_i\dimvec X$. We therefore have
$s_i\dimvec X=\dimvec\im(\Sigma_if)+n\alpha_i$. Since $T$ is
$\theta$-semistable, $\langle\theta,\dimvec\im(\Sigma_if)\rangle\geq0$.
We eventually find that $\langle s_i\theta,\dimvec X\rangle\geq
n\langle\theta,\alpha_i\rangle\geq0$, as desired.

We thus see that $\Sigma_i^*$ maps $\mathscr R_\theta$ to
$\mathscr R_{s_i\theta}$. A dual reasoning shows that $\Sigma_i$
maps $\mathscr R_{s_i\theta}$ to $\mathscr R_\theta$. The
theorem now follows from Theorem~\ref{th:BrennerButler}~\ref{it:BBc}
and from the inclusions $\mathscr R_\theta\subseteq\mathscr F_{s_i}$
and $\mathscr R_{s_i\theta}\subseteq\mathscr T^{s_i}$.
\end{proof}

Recall from section~\ref{ss:GenSetup} the definition of the dominant
cone $\overline C_0$ and of the Tits cone $C_T=\bigcup_{w\in
W}w\overline C_0$. A subset $J\subseteq I$ gives rise to a
face $F_J\subseteq\overline C_0$, to a parabolic subgroup
$W_J=\langle s_j\mid j\in J\rangle$ of $W$, and to a root
system $\Phi_J\subseteq\Phi$. In the case $W_J$ is finite, we
denote its longest element by $w_J$. Recall also that for $w\in W$,
we set $N_w=\Phi_+\cap w\Phi_-$ and that for any reduced
decomposition $w=s_{i_1}\cdots s_{i_\ell}$, we have
$$N_w=\{s_{i_1}\cdots s_{i_{k-1}}\alpha_{i_k}\mid1\leq k\leq\ell\}.$$

\begin{theorem}
\label{th:TiltPIR}
Let $J\subseteq I$, let $\theta\in F_J$, and let $w\in W$. We
assume that $w$ is $J$-reduced on the right, that is,
$\ell(ws_j)>\ell(w)$ for each $j\in J$.
\begin{enumerate}
\item
\label{it:TPIRa}
The category $\mathscr R_\theta$ coincides with the subcategory
$\Lambda_J\mmod$. Moreover, there are mutually inverse equivalences
$$\xymatrix@C=6em{\mathscr R_\theta\ar@<.6ex>[r]^{I_w\otimes_\Lambda?}&
\ar@<.6ex>[l]^{\Hom_\Lambda(I_w,?)}\mathscr R_{w\theta}}.$$
\item
\label{it:TPIRb}
We have $(\mathscr T^w,\mathscr F^w)=(\overline{\mathscr
I}_{w\theta},\mathscr P_{w\theta})$.
\item
\label{it:TPIRc}
If $W_J$ is finite, then have $(\mathscr T^{ww_J},\mathscr
F^{ww_J})=(\mathscr I_{w\theta},\overline{\mathscr P}_{w\theta})$.
\end{enumerate}
\end{theorem}
\begin{proof}
Let $J$, $\theta$, $w$ as in the statement of the theorem.

Given $T\in\Lambda\mmod$, the condition $\langle\theta,
\dimvec T\rangle=0$ is necessary for $T$ to be in $\mathscr
R_\theta$. It is also sufficient, because any quotient module
$X$ of $T$ satisfies $\langle\theta,\dimvec X\rangle\geq0$ by
the dominance of $\theta$. The equality
$\mathscr R_\theta=\Lambda_J\mmod$ then follows from
Example~\ref{ex:TorsTheo}~\ref{it:TTb}.

Since $\theta$ is in $F_J$, it takes a positive value at each root
in $\Phi_+\setminus\Phi_J$. Corollary~\ref{co:NwJred} then ensures
that $\theta$ take positive values on $N_{w^{-1}}$. Choosing a
reduced decomposition $w=s_{i_1}\cdots s_{i_\ell}$, we obtain
$\langle s_{i_{k+1}}\cdots s_{i_\ell}\theta,\alpha_{i_k}
\rangle>0$ for each $1\leq k\leq\ell$. Using
Theorem~\ref{th:SekiyaYamaura}, we get a chain
of equivalences of categories
$$\xymatrix@C=4.5em{\mathscr R_\theta
\ar@<.6ex>[r]^{\Sigma_{i_\ell}^*}&
\ar@<.6ex>[l]^{\Sigma_{i_\ell}}
\mathscr R_{s_{i_\ell}\theta}
\ar@<.6ex>[r]^(.55){\Sigma_{i_{\ell-1}}^*}&
\ar@<.6ex>[l]^(.45){\Sigma_{i_{\ell-1}}}
\cdots
\ar@<.6ex>[r]^(.4){\Sigma_{i_2}^*}&
\ar@<.6ex>[l]^(.6){\Sigma_{i_2}}
\mathscr R_{s_{i_\ell}\cdots s_{i_2}\theta}
\ar@<.6ex>[r]^(.6){\Sigma_{i_1}^*}&
\ar@<.6ex>[l]^(.4){\Sigma_{i_1}}
\mathscr R_{w\theta}}.$$
By composition, we get assertion~\ref{it:TPIRa}.

Let $T\in\mathscr T^w$ and let $X$ be a quotient of $T$; then
$X\in\mathscr T^w$. By Theorem~\ref{th:BrennerButler}~\ref{it:BBc}
and Corollary~\ref{co:ReflDimVec}, $\dimvec X$ is of the form
$w\nu$ with $\nu\in\mathbb NI$, and so $\langle w\theta,\dimvec
X\rangle=\langle\theta,\nu\rangle\geq0$. This proves that
$T\in\overline{\mathscr I}_{w\theta}$.

Let $T\in\mathscr F^w$ and let $X\subseteq T$ be a nonzero
submodule; then $X$ is in $\mathscr F^w$. Theorem~\ref{th:GLSFilt}
\ref{it:GLSFb} and Lemma~\ref{le:Layers} then imply that
$\dimvec X$ is a nontrivial $\mathbb N$-linear combination of
elements in $N_w$. In addition, $w\theta$ takes negative values on
$N_w$, for $\theta$ takes positive values on $N_{w^{-1}}=-w^{-1}N_w$,
as we have seen during the course of the proof of assertion
\ref{it:TPIRa}. Therefore $\langle w\theta,\dimvec X\rangle<0$.
This reasoning shows that $T\in\mathscr P_{w\theta}$.

We have established that $\mathscr T^w\subseteq\overline{\mathscr
I}_{w\theta}$ and that $\mathscr F^w\subseteq\mathscr P_{w\theta}$.
This implies assertion~\ref{it:TPIRb}.

We now prove assertion~\ref{it:TPIRc}, assuming that $W_J$ is
finite.

Consider first $T\in\mathscr T^{ww_J}$ and take a nonzero quotient
$X$ of $T$. Then $X$ also belongs to $\mathscr T^{ww_J}$, and by
Proposition~\ref{pr:CompEquiv}, we can write
$X=I_w\otimes_\Lambda Y$, with $Y\in\mathscr T^{w_J}$.
Since $X\neq0$, we have $Y\neq0$, hence $Y\notin\mathscr F^{w_J}$.
By Example~\ref{ex:TorsTheo}~\ref{it:TTb}, this means that
$\dimvec Y$ is not in $\mathbb NJ$, the set of $\mathbb N$-linear
combinations of elements in $\{\alpha_j\mid j\in J\}$.
Therefore $\langle w\theta,\dimvec X\rangle=\langle\theta,\dimvec
Y\rangle>0$. This proves that $T\in\mathscr I_{w\theta}$.

Now let $T\in\mathscr F^{ww_J}$ and take a submodule $X\subseteq T$.
Then $X\in\mathscr F^{ww_J}$ as well. Theorem~\ref{th:GLSFilt}
\ref{it:GLSFb} and Lemma~\ref{le:Layers} then imply that $\dimvec X$
is a $\mathbb N$-linear combination of elements in $N_{ww_J}$.
Certainly, $ww_J\theta$ takes nonpositive values on $N_{ww_J}$,
so $\langle ww_J\theta,\dimvec X\rangle\leq0$. Observing that
$w_J\theta=\theta$, we conclude that $T\in\overline{\mathscr
P}_{w\theta}$.

We have established that $\mathscr T^{ww_J}\subseteq\mathscr
I_{w\theta}$ and that $\mathscr F^{ww_J}\subseteq\overline{\mathscr
P}_{w\theta}$, whence assertion~\ref{it:TPIRc}.
\end{proof}

\begin{other}{Remarks}
\label{rk:ComTPIR}
\begin{enumerate}
\item
\label{it:CTPIRa}
In the context of assertions~\ref{it:TPIRb} and~\ref{it:TPIRc}
of Theorem~\ref{th:TiltPIR}, we have $T^w=T_{w\theta}^{\max}$ and
$T^{ww_J}=T_{w\theta}^{\min}$ for any $\Lambda$-module $T$.
This shows in particular that each $\dimvec T^w$ is a vertex
of the HN polytope $\Pol(T)$.
\item
\label{it:CTPIRb}
With the help of Remark~\ref{rk:PolDual} and
Example~\ref{ex:TorsTheo}~\ref{it:TTc}, we can complement the
statement of Theorem~\ref{th:TiltPIR} as follows:
$$(\mathscr T_{w^{-1}},\mathscr F_{w^{-1}})=
(\mathscr I_{-w\theta},\overline{\mathscr P}_{-w\theta})
\quad\text{and}\quad
(\mathscr T_{w_Jw^{-1}},\mathscr F_{w_Jw^{-1}})=
(\overline{\mathscr I}_{-w\theta},\mathscr P_{-w\theta}).$$
\item
\label{it:CTPIRc}
Combining~\ref{it:CTPIRb} with Proposition~\ref{pr:PolConstr}
\ref{it:PCc} and Remark~\ref{ex:CompGLS}, we get another
proof of the openness statement in Remark~14.2
of~\cite{GeissLeclercSchroer11}.
\item
\label{it:CTPIRd}
In the case where $W_J$ is finite, assertion~\ref{it:TPIRa} of
Theorem~\ref{th:TiltPIR} is a particular case of Proposition
\ref{pr:ShiftF^w}, since $\mathscr R_{w\theta}=\overline{\mathscr
I}_{w\theta}\cap\overline{\mathscr P}_{w\theta}=\mathscr
T^w\cap\mathscr F^{ww_J}$ and
$\mathscr R_\theta=\Lambda_J\mmod=\mathscr F^{w_J}$.
\end{enumerate}
\end{other}

Our first corollary to Theorem~\ref{th:TiltPIR} shows that the
position of the facets of our polytopes can be computed as the
dimension of homomorphism spaces.
\begin{corollary}
\label{co:FacetsViaHom}
Let $J$, $\theta$ and $w$ be as in Theorem~\ref{th:TiltPIR}.
Suppose that $\theta$ is integral, that is, each
$\langle\theta,\alpha_i\rangle$ is an integer. Set
$N(w\theta)=\bigoplus_{i\in I}(I_w\otimes_\Lambda\Lambda
e_i)^{\oplus\langle\theta,\alpha_i\rangle}$ and denote
by $N(-w\theta)$ the dual of
$\bigoplus_{i\in I}((\Lambda/I_w)\otimes_\Lambda
\Lambda e_i)^{\oplus\langle\theta,\alpha_i\rangle}$.
Then for any $\Lambda$-module $T$,
$$\dim\Hom_\Lambda(N(\pm w\theta),T)=\psi_{\Pol(T)}(\pm w\theta).$$
\end{corollary}
\begin{proof}
The case $+w\theta$ comes from the computation
\begin{align*}
\dim\Hom_\Lambda(N(w\theta),T)
&=\sum_{i\in I}\langle\theta,\alpha_i\rangle
\dim\Hom_\Lambda(\Lambda e_i,\Hom_\Lambda(I_w,T))\\
&=\sum_{i\in I}\langle\theta,\alpha_i\rangle
\dim e_i\Hom_\Lambda(I_w,T)\\
&=\langle\theta,\dimvec\Hom_\Lambda(I_w,T)\rangle\\
&=\langle w\theta,\dimvec T^w\rangle\\
&=\langle w\theta,\dimvec T_{w\theta}^{\max}\rangle\\
&=\psi_{\Pol(T)}(w\theta).
\end{align*}
Here the first equality is adjunction, the fourth is
Corollary~\ref{co:ReflDimVec}, and the fifth is
Remark~\ref{rk:ComTPIR}~\ref{it:CTPIRa}.

Now applying the functor $\Hom_\Lambda(?,T^*)$ to the short exact
sequence $0\to I_w\otimes_\Lambda\Lambda e_i\to\Lambda e_i\to
(\Lambda/I_w)\otimes_\Lambda\Lambda e_i\to0$, we get a long
exact sequence
\begin{align*}
0\to\Hom_\Lambda((\Lambda/I_w)\otimes_\Lambda\Lambda e_i,T^*)\to
&\Hom_\Lambda(\Lambda e_i,T^*)\\
\to&\Hom_\Lambda(I_w\otimes_\Lambda\Lambda e_i,T^*)\to
\Ext^1_\Lambda((\Lambda/I_w)\otimes_\Lambda\Lambda e_i,T^*)\to0.
\end{align*}
Taking dimensions, and using Crawley-Boevey's formula
\eqref{eq:CrawleyBoeveyForm} and Example~\ref{ex:CompBKAIRT}, we
obtain
\begin{align*}
\dim&\Hom_\Lambda(T^*,(\Lambda/I_w)\otimes_\Lambda\Lambda e_i)\\
&=\dim\Hom_\Lambda(I_w\otimes_\Lambda\Lambda e_i,T^*)
-\dim\Hom_\Lambda(\Lambda e_i,T^*)+
(\dimvec(\Lambda/I_w)\otimes_\Lambda\Lambda e_i,\dimvec T^*)\\
&=\dim\Hom_\Lambda(I_w\otimes_\Lambda\Lambda e_i,T^*)
-\dim e_iT^*+(\omega_i-w\omega_i,\dimvec T^*).
\end{align*}
Multiplying by $\langle\theta,\alpha_i\rangle$ and summing over
$i\in I$ gives then
\begin{align*}
\dim\Hom_\Lambda(N(-w\theta),T)
&=\sum_{i\in I}\langle\theta,\alpha_i\rangle
\dim\Hom_\Lambda(T^*,(\Lambda/I_w)\otimes_\Lambda\Lambda e_i)\\
&=\dim\Hom_\Lambda(N(w\theta),T^*)-\langle\theta,\dimvec T^*\rangle
+\langle\theta-w\theta,\dimvec T^*\rangle\\
&=\psi_{\Pol(T^*)}(w\theta)-\langle w\theta,\dimvec T\rangle.
\end{align*}
Remark~\ref{rk:PolDual} shows that the right-hand side is
$\psi_{\Pol(T)}(-w\theta)$.
\end{proof}

Our second corollary compares the normal fan of an HN polytope to
the Tits fan.
\begin{corollary}
\label{co:SuppFunTitsFan}
The support function of the HN polytope of a $\Lambda$-module is
linear on each face $wF_J$ of the Tits cone.
\end{corollary}
\begin{proof}
Let $T$ be a $\Lambda$-module. The support function of $\Pol(T)$
is given by $\theta\mapsto\langle\theta,\dimvec T_\theta^{\max}
\rangle$. However, $T_\theta^{\max}=T^w$ if $\theta\in wF_J$,
with $w$ chosen $J$-reduced on the right.
\end{proof}

Using Remark~\ref{rk:PolDual}, we see that the support function
of an HN polytope is also linear on each face of the opposite
$-C_T$ of the Tits cone.

The third corollary describes the face of an HN polytope defined
by a linear form in the Tits cone.
\begin{corollary}
\label{co:FiniteFaces}
Let $J$, $\theta$ and $w$ as in the statement of
Theorem~\ref{th:TiltPIR} and let $T\in\Lambda\mmod$.
Set $X=\Hom_\Lambda(I_w,T_{w\theta}^{\max}/T_{w\theta}^{\min})$.
Then $X$ is in the subcategory $\Lambda_J\mmod$ and
$$\{x\in\Pol(T)\mid\langle w\theta,x\rangle=\psi_{\Pol(T)}(w\theta)\}
=\dimvec T_{w\theta}^{\min}+w\Pol(X).$$
\end{corollary}
\begin{proof}
This follows by combining Corollary~\ref{co:FacesHN},
Theorem~\ref{th:TiltPIR}~\ref{it:TPIRa},
Corollary~\ref{co:ReflDimVec} and
Example~\ref{ex:TorsTheo}~\ref{it:TTb}.
\end{proof}

In the case $J=\varnothing$, that is, if $\theta$ belongs to the
open cone $C_0$, then $T^w=T_{w\theta}^{\min}=T_{w\theta}^{\max}$,
$\mathscr I_{w\theta}=\overline{\mathscr I}_{w\theta}=\mathscr T^w$,
$\mathscr P_{w\theta}=\overline{\mathscr P}_{w\theta}=\mathscr F^w$,
and $\mathscr R_{w\theta}\cong\mathscr R_\theta=\Lambda_J\mmod=\{0\}$.
This case corresponds of course to a vertex of $\Pol(T)$.

In the case where $J$ has just one element, say $i$, then $w_J=s_i$,
$(T_{w\theta}^{\min},T_{w\theta}^{\max})=(T^{ws_i},T^w)$,
$\mathscr R_\theta=\Lambda_J\mmod=\add S_i$, and
$\mathscr R_{w\theta}=\add(I_w\otimes_\Lambda S_i)$. This case
corresponds to an edge of $\Pol(T)$ that points in the direction
$\dimvec(I_w\otimes_\Lambda S_i)=w\alpha_i$.

The case where $J$ contains just two vertices $i$ and $j$ linked by
a single edge is more interesting. Here $w_J=s_is_js_i=s_js_is_j$,
$(T_{w\theta}^{\min},T_{w\theta}^{\max})=(T^{ww_J},T^w)$, and
$\mathscr R_\theta=\Lambda_J\mmod$ has four indecomposables.
This case corresponds to a $2$-face of type $A_2$. We will come
back to this case soon: if $T$ is a general point in an irreducible
component $Z$, then the shape of this $2$-face will be constrained
by the tropical Pl\"ucker relations.

\subsection{Tilting structure and crystal operations}
\label{ss:TiltCrysOp}
Thanks to Theorem~\ref{th:TorIrrComp}, we know that under a suitable
openness condition (O), each torsion pair $(\mathscr T,\mathscr F)$
gives rise to a bijection $\Xi:\mathfrak T\times\mathfrak
F\to\mathfrak B$. The aim of this section is to show that in the
case of the torsion pair $(\mathscr T^w,\mathscr F^w)$, this
bijection can be described by elementary operations on the
crystal $\mathfrak B$. Given $\nu\in\mathbb NI$, both sets
$\{T\in\Lambda(\nu)\mid T\in\mathscr T^w\}$ and
$\{T\in\Lambda(\nu)\mid T\in\mathscr F^w\}$ are open (Proposition
\ref{pr:PolConstr}~\ref{it:PCc} and Theorem~\ref{th:TiltPIR}
\ref{it:TPIRb}), so we can define the subsets $\mathfrak T^w$
and $\mathfrak F^w$ of $\mathfrak B$ formed by irreducible
components whose general point belongs to $\mathscr T^w$ and
$\mathscr F^w$. We define $\mathfrak T_w$ and $\mathfrak F_w$
in a similar fashion.

The first ingredient is the particular case where $w$ is a simple
reflection $s_i$. By Example~\ref{ex:TorsTheo}~\ref{it:TTa},
$\mathfrak F^{s_i}=\bigsqcup_{n\in\mathbb N}\mathfrak
B(n\alpha_i)$ is in bijection with $\mathbb N$; moreover, by
the definition of the crystal structure on $\mathfrak B$
(see section~\ref{ss:LuszNilpVar}), we have
$\mathfrak T^{s_i}=\{Z\in\mathfrak B\mid\varphi_i(Z)=0\}$. Under
the identification $\mathfrak F^{s_i}\cong \mathbb N$, the
bijection $\mathfrak T^{s_i}\times\mathfrak F^{s_i}\to\mathfrak B$
becomes the map $(Z,n)\mapsto\tilde e_i^nZ$. The inverse of this
map is $Z\mapsto(\tilde f_i^{\max}Z,\varphi_i(Z))$, where
$\tilde f_i^{\max}\,b=(\tilde f_i)^{\varphi_i(b)}\,b$.

To go further, we need to understand how the equivalence of categories
provided by Theorem~\ref{th:BrennerButler}~\ref{it:BBc} relates to
crystal operations. Twisting the usual crystal operators by the
involution~$*$, one defines the starred operators on $B(-\infty)$
as follows:
$$\tilde e_i^*:b\mapsto (\tilde e_i\,b^*)^*,\qquad
\tilde f_i^*:b\mapsto(\tilde f_i\,b^*)^*,\qquad
(\tilde f_i^*)^{\max}:b\mapsto(\tilde f_i^{\max}\,b^*)^*.$$
In \cite{Saito94}, Corollary~3.4.8 (see also \cite{KashiwaraSaito97},
section~8.2), Saito defines mutually inverse bijections
$$\xymatrix@C=3.5em{\{b\in B(-\infty)\mid\varphi_i(b)=0\}
\ar@<.6ex>[r]^{S_i}&\ar@<.6ex>[l]^{S_i^*}
\{b\in B(-\infty)\mid\varphi_i(b^*)=0\}}$$
by the rules
$S_i(b)=(\tilde e_i)^{\varepsilon_i(b^*)}(\tilde f_i^*)^{\max}\,b$
and
$S_i^*(b)=(\tilde e_i^*)^{\varepsilon_i(b)}\tilde f_i^{\max}\,b$.
This definition ensures that if $\varphi_i(b)=0$, then
$\wt S_i(b)=s_i(\wt b)$. For convenience, we extend $S_i$ and
$S_i^*$ on $B(-\infty)$ by setting $\sigma_ib=S_i(\tilde f_i^{\max}b)$
and $\sigma_i^*b=S_i^*((\tilde f_i^*)^{\max}b)$. By transport
through the bijection $B(-\infty)\cong\mathfrak B$, we can view the
maps $\sigma_i$ and $\sigma_i^*$ as maps from $\mathfrak B$ to itself.
In view of Example~\ref{ex:TorsTheo}~\ref{it:TTa}, $\sigma_i$ and
$\sigma_i^*$ restrict to mutually inverse bijections
$$\xymatrix@C=3.5em{\mathfrak T^{s_i}
\ar@<.6ex>[r]^{\sigma_i}&\ar@<.6ex>[l]^{\sigma_i^*}
\mathfrak F_{s_i}}.$$

\begin{proposition}
\label{pr:CrysBBEquiv}
\begin{enumerate}
\item
\label{it:CBBEa}
Let $\nu\in\mathbb NI$, let $i\in I$, let $Z\in\mathfrak F_{s_i}(\nu)$
and let $Z'=\sigma_i^*(Z)$, an element in $\mathfrak T^{s_i}(s_i\nu)$.
Let $U=\{T\in Z\mid\soc_iT=0\}$ and $U'=\{T'\in Z'\mid\hd_iT'=0\}$.
Let $\Theta$ be the set of triples $(T,T',h)$, such that
$(T,T')\in U\times U'$ and $h:T'\to\Sigma_i^*T$ is an isomorphism.
Then the first projection $\Theta\to U$ and the second one
$\Theta\to U'$ are locally trivial fibrations with a smooth and
connected fiber.
\item
\label{it:CBBEb}
Let $(u,v,i)\in W^2\times I$ be such that
$\ell(us_iv)=\ell(u)+\ell(v)+1$. Then
$\mathfrak F^v\subseteq\mathfrak F_{us_i}$ and
$\mathfrak F^{s_iv}\subseteq\mathfrak F_u$. In addition,
$\sigma_i$ and $\sigma_i^*$ restrict to mutually inverse bijections
$$\xymatrix@C=3.5em{\mathfrak F_u\cap\mathfrak T^{s_iv}
\ar@<.6ex>[r]^{\sigma_i}&\ar@<.6ex>[l]^{\sigma_i^*}
\mathfrak F_{us_i}\cap\mathfrak T^v}
\quad\text{and}\quad
\xymatrix@C=3.5em{\mathfrak F^{s_iv}\cap\mathfrak T^{s_i}
\ar@<.6ex>[r]^(.6){\sigma_i}&\ar@<.6ex>[l]^(.4){\sigma_i^*}
\mathfrak F^v}.$$
\end{enumerate}
\end{proposition}
\begin{proof}
Assertion~\ref{it:CBBEa} is \cite{BaumannKamnitzer12}, Theorem~5.3.
It is the precise way of stating that if $T$ is a general point in $Z$,
then $\Sigma_i^*T$ ``belongs'' to $Z'$ and is general in $Z'$. (The
quotes around ``belongs'' reflects the fact that a point of $Z'$ can
only be isomorphic to $\Sigma_i^*T$, and not equal to it.)
Assertion~\ref{it:CBBEb} follows then from
Propositions~\ref{pr:CompEquiv}, \ref{pr:TorsTheoMono2}
and~\ref{pr:ShiftF^w}.
\end{proof}

Let $(w,i)\in W\times I$ be such that $\ell(ws_i)>\ell(w)$.
By Proposition~\ref{pr:ShiftF^w} and
Example~\ref{ex:TorsTheo}~\ref{it:TTa}, we have
$\mathscr F^{ws_i}\cap\mathscr T^w=\add(I_w\otimes_\Lambda S_i)$,
so $\mathfrak F^{ws_i}\cap\mathfrak T^w$ is in bijection with
$\mathbb N$: to an integer $n$ corresponds the closure in
$\Lambda(w\alpha_i)$ of the orbit representing the module
$I_w\otimes_\Lambda S_i^{\oplus n}$.

Now choose a finite sequence $\mathbf i=(i_1,\ldots,i_\ell)$
such that $s_{i_1}\cdots s_{i_\ell}$ is a reduced decomposition.
Set $(\mathscr T_0,\mathscr F_0)=(\Lambda\mmod,\{0\})$, and
for $1\leq k\leq\ell$, set $(\mathscr T_k,\mathscr F_k)=
(\mathscr T^{s_{i_1}\cdots s_{i_k}},\mathscr F^{s_{i_1}\cdots
s_{i_k}})$. Then
$$(\mathscr T_\ell,\mathscr F_\ell)\preccurlyeq\cdots
\preccurlyeq(\mathscr T_0,\mathscr F_0).$$
The generalization of Proposition~\ref{pr:NesTorIrrComp} to
finite sequences provides a bijection
$$\Omega_{\mathbf i}:\mathfrak B\to\prod_{k=1}^\ell(\mathfrak
F_k\cap\mathfrak T_{k-1})\times\mathfrak T_\ell.$$
(Here we have used the inverse map to that defined in Proposition
\ref{pr:NesTorIrrComp} and have reversed the order of the factors.)
Under the identification $\mathfrak F_k\cap\mathfrak
T_{k-1}\cong\mathbb N$, this bijection $\Omega_{\mathbf i}$ can be
expressed in terms of the crystal operations in the following way.

\begin{proposition}
\label{pr:LusParTop}
Let $Z\in\mathfrak B$. Set $Z'=(\sigma_{i_1}^*\cdots\sigma_{i_\ell}^*
\sigma_{i_\ell}\cdots\sigma_{i_1})(Z)$ and
$n_k=\varphi_{i_k}(\sigma_{i_{k-1}}\cdots\sigma_{i_1}Z)$
for $1\leq k\leq\ell$. Then
$\Omega_{\mathbf i}(Z)=(n_1,\ldots,n_\ell,Z')$.
\end{proposition}
\begin{proof}
Set $i=i_1$ and $\mathbf{i'}=(i_2,\ldots,i_\ell)$.
For $2\leq k\leq\ell$, set $(\mathscr T'_k,\mathscr F'_k)=
(\mathscr T^{s_{i_2}\cdots s_{i_k}},\mathscr F^{s_{i_2}\cdots
s_{i_k}})$. By Propositions~\ref{pr:CompEquiv} and~\ref{pr:ShiftF^w},
$\Sigma_i$ and $\Sigma_i^*$ restrict to equivalences of categories
$$\xymatrix@C=3.5em{\mathscr T^{s_i}\cap\mathscr T_k
\ar@<.6ex>[r]^{\Sigma_i}&\ar@<.6ex>[l]^{\Sigma_i^*}
\mathscr F_{s_i}\cap\mathscr T'_k}
\quad\text{and}\quad
\xymatrix@C=3.5em{\mathscr T^{s_i}\cap\mathscr F_k
\ar@<.6ex>[r]^{\Sigma_i}&\ar@<.6ex>[l]^{\Sigma_i^*}
\mathscr F_{s_i}\cap\mathscr F'_k}.$$

Let $T\in Z$ be a general point and let $X$ be its torsion
submodule with respect to $(\mathscr T_1,\mathscr F_1)$.
Then $X$ is the top step in the filtration of $T$ defined
by our nested sequence of torsion pairs. The other
modules define a filtration of $X$, whose image by $\Sigma_i$
is the filtration on $\Sigma_iX$ defined by the nested sequence
$$(\mathscr T'_\ell,\mathscr F'_\ell)\preccurlyeq\cdots
\preccurlyeq(\mathscr T'_2,\mathscr F'_2).$$

Now $X$ is a general point of $\tilde f_i^{\max}Z$, so $\Sigma_iX$
is a general point of $\sigma_i(Z)$. By induction, we then have
$\Omega_{\mathbf{i'}}(\sigma_i(Z))=(n_2,\ldots,n_\ell,Z'')$,
where $Z''=\sigma_{i_2}^*\cdots\sigma_{i_\ell}^*\sigma_{i_\ell}
\cdots\sigma_{i_2}(\sigma_iZ)$. On the other hand, we have
$T/X\cong S_i^{\oplus n_1}$, because $X$ is the kernel of the
canonical map $T\to\hd_iT$ and $n_1$ is the dimension of the
$i$-head of $T$. The result now follows from
Proposition~\ref{pr:CrysBBEquiv}~\ref{it:CBBEb}.
\end{proof}

\begin{other}{Remarks}
\label{rk:RemLusParTop}
\begin{enumerate}
\item
\label{it:RLPTa}
Set $w=s_{i_1}\cdots s_{i_\ell}$. Obviously, $\Omega_{\mathbf i}$
induces a bijection between $\mathfrak F^w$ and $\mathbb N^l$. Up
to duality, this bijection corresponds to the parameterization
defined by Gei\ss, Leclerc and Schr\"oer
(\cite{GeissLeclercSchroer11}, Proposition~14.5), as one sees
from our discussion in Example~\ref{ex:CompGLS}. In this context,
Proposition~\ref{pr:LusParTop} has been independently proved by
Jiang~\cite{Jiang13}.
\item
\label{it:RLPTb}
Consider the case where $\mathfrak g$ is finite dimensional and
$s_{i_1}\cdots s_{i_\ell}=w_0$. Then $\Omega_{\mathbf i}$ induces
a bijection between $\mathfrak B$ and $\mathbb N^l$. For
$b\in B(-\infty)$, the procedure to compute
$\Omega_{\mathbf i}(\Lambda_b)$ given in
Proposition~\ref{pr:LusParTop} coincides with Saito's method to
determine the Lusztig datum of $b$ in direction $\mathbf i$
(compare with the proof of \cite{Saito94}, Lemma~4.1.3).
So in this case, $\Omega_{\mathbf i}$ gives the usual Lusztig
data. This gives us an incentive to use $\Omega_{\mathbf i}$ in
order to define Lusztig data in the affine type case as well;
we will pursue this road in section~\ref{ss:LusDatComp}.
\end{enumerate}
\end{other}

We now examine how $\Omega_{\mathbf i}$ depends on $\mathbf i$,
when $s_{i_1}\cdots s_{i_\ell}$ is fixed in $W$. Matsumoto's lemma
instructs us to look what happens under a braid move, so let us
locate a subword of the form $(i,j,i)$ in $\mathbf i$, where $i$
and $j$ are linked with a single edge in the graph $(I,E)$. Let
us denote by $m+1$ the index at which this subword begins and let
us denote by $\mathbf j$ the result of the substitution of
$(i,j,i)$ by $(j,i,j)$ in $\mathbf i$.
\begin{proposition}
\label{pr:LusPLBij1}
Let $Z\in\mathfrak B$ and write $\Omega^{\mathbf i}(Z)=(n_1,\ldots,
n_m,p,q,r,n_{m+4},\ldots,n_\ell,Z')$. Define $p'$, $q'$ and $r'$ by
the formulas
$$q'=\min(p,r),\quad p'+q'=q+r,\quad q'+r'=p+q.$$
Then $\Omega^{\mathbf j}(Z)=(n_1,\ldots,
n_m,p',q',r',n_{m+4},\ldots,n_\ell,Z')$.
\end{proposition}
\begin{proof}
We adopt the notation of the statement of the proposition.
Let us set $J=\{i,j\}$.

The nested sequences of torsion pairs defined by $\mathbf i$ and
$\mathbf j$ differ only in the places $m+1$, $m+2$ and $m+3$, so
$\Omega_{\mathbf i}(Z)$ and $\Omega_{\mathbf j}(Z)$ differ only
there. Moreover, thanks to Proposition~\ref{pr:LusParTop}, we can
reduce to the case where $m=0$ by replacing $Z$ by
$\sigma_{i_m}\cdots\sigma_{i_1}(Z)$. Finally, we can also assume
without loss of generality that $\ell=3$ and that
$Z\in\mathfrak F^{w_J}$.

By Example~\ref{ex:TorsTheo}~\ref{it:TTb}, the category
$\mathscr F^{w_J}$ is isomorphic to the category of
representations of the preprojective algebra $\Lambda_J$, of type
$A_2$. It is well-known that this category has four indecomposable
objects, namely two simple objects $S_i$ and $S_j$ and two objects
$T_i$ and $T_j$ of Loewy length $2$, obtained as the middle terms
of non-split extensions $0\to S_j\to T_i\to S_i\to0$ and
$0\to S_i\to T_j\to S_j\to0$; the structure of these modules
$T_i$ and $T_j$ can be represented pictorially as follows:
$$T_i=\left(\bsm i&&\\&\searrow&\\&&j\esm\right)\quad\text{and}\quad
T_j=\left(\bsm&&j\\&\swarrow&\\i&&\esm\right).$$
Thus, there exists $(a,b,c,d)\in\mathbb N^4$ such that a
general point in $Z$ is isomorphic to
$$S_i^{\oplus a}\oplus S_j^{\oplus b}\oplus
T_i^{\oplus c}\oplus T_j^{\oplus d},$$
with moreover $\min(a,b)=0$, because a module which has
$S_i\oplus S_j$ as summand cannot be general. One then easily
computes
$$\Omega^{(i,j,i)}(Z)=(a+c,d,b+c,\{0\})\qquad\text{and}\qquad
\Omega^{(j,i,j)}(Z)=(b+d,c,a+d,\{0\}).$$
Setting $p=a+c$, $q=d$ and $r=b+c$, one checks that $p'=b+d$,
$q'=c$, $r'=a+d$, which shows the desired result.
\end{proof}

Likewise, let us locate a subword of the form $(i,j)$ in $\mathbf i$,
where $i$ and $j$ are not linked in the graph $(I,E)$.
Let us denote by $m+1$ the index at which this subword begins and
let us denote by $\mathbf j$ the result of the substitution of
$(i,j)$ by $(j,i)$ in $\mathbf i$.
\begin{proposition}
\label{pr:LusPLBij2}
Let $Z\in\mathfrak B$ and write $\Omega^{\mathbf i}(Z)=(n_1,\ldots,
n_m,p,q,n_{m+3},\ldots,n_\ell,Z')$. Then
$\Omega^{\mathbf j}(Z)=(n_1,\ldots,
n_m,q,p,n_{m+3},\ldots,n_\ell,Z')$.
\end{proposition}
The proof is similar to that of Proposition~\ref{pr:LusPLBij1},
save that we now deal with a preprojective algebra $\Lambda_J$
of type $A_1\times A_1$.

For a fixed $w\in W$, let $\mathscr X(w)$ be the set of all
tuples $\mathbf i$ that represent a reduced decomposition of
$w$. Lusztig's piecewise linear bijections
$R_{\mathbf i}^{\mathbf j}:\mathbb N^\ell\to\mathbb N^\ell$
(\cite{Lusztig90a}, section~2.1) can be defined here just
as in the case where $w$ is the longest element in a finite $W$.
Since any two elements in $\mathscr X(w)$ can be related by a
sequence of braid and of commutation relations,
Propositions~\ref{pr:LusPLBij1} and~\ref{pr:LusPLBij2}
say that the numerical parts of $\Omega_{\mathbf i}$ and
$\Omega_{\mathbf j}$ are related by $R_{\mathbf i}^{\mathbf j}$.

To conclude, let us consider again Proposition~\ref{pr:LusParTop}
and choose a general point $T\in Z$. The integers $n_1$, \dots,
$n_\ell$ are equal to the lengths of the edges of $\Pol(T)$ along
the path
$$\dimvec T,\quad\dimvec T^{s_{i_1}},\quad\dimvec
T^{s_{i_1}s_{i_2}},\quad\ldots,\quad\dimvec T^{s_{i_1}\cdots
s_{i_\ell}}.$$
Now look at Proposition~\ref{pr:LusPLBij1}, set $J=\{i,j\}$,
choose $\theta\in F_J$, and set $u=s_{i_1}\cdots s_{i_m}$; then
$u$ is $J$-reduced on the right. Using
Theorem~\ref{th:TiltPIR}~\ref{it:TPIRa} and
Corollary~\ref{co:SuppFunTitsFan}, one can show that the vertices
on the face of $\Pol(T)$ defined by $u\theta$ are the six vertices
$\dimvec T^{uv}$, with $v\in W_J$. The relation given in
Proposition~\ref{pr:LusPLBij1} constrains the lengths of the edges
of this face; it is equivalent to the tropical Pl\"ucker relations
of~\cite{Kamnitzer10}.

\subsection{The finite type case}
\label{ss:DynkinType}
Our main focus of interest in this paper concerns the case where
$\Phi$ is of affine type, to which section~\ref{ss:TiltIdIw}
directly applies. In the finite type case, the ideals $I_w$ are
not tilting of projective dimension at most~$1$ anymore. They
nevertheless exist, so we can define the full subcategories
\begin{align*}
\mathscr T^w&=\text{essential image of }I_w\otimes_\Lambda?,\\
\mathscr F^w&=\text{kernel of }\Hom_\Lambda(I_w,?),\\
\mathscr T_w&=\text{kernel of }I_w\otimes_\Lambda?,\\
\mathscr F_w&=\text{essential image of }\Hom_\Lambda(I_w,?)
\end{align*}
of $\Lambda\mmod$. With this definition, all the results in
sections~\ref{ss:ReflFunc} hold unchanged, except of course
Theorem~\ref{th:BuanIyamaReitenScott}~\ref{it:BIRSc}.

To show this, one can adopt the method of Iyama, Reiten and their
collaborators, namely, one chooses an embedding of the Dynkin diagram
into a non-Dynkin one. One then get a natural surjective morphism
from the preprojective algebra $\widehat\Lambda$ of non-Dynkin type
onto the preprojective algebra $\Lambda$ of Dynkin type, and thus a
natural embedding of $\Lambda\mmod$ as a full subcategory of
$\widehat\Lambda\mmod$. This subcategory is abelian and closed under
extensions. Each $i\in I$ yields then an ideal $I_i$ of $\Lambda$
and an ideal $\widehat I_i$ of $\widehat\Lambda$. By
Proposition~\ref{pr:DescReflFunc}, the functors
$I_i\otimes_\Lambda?$ and $\widehat I_i\otimes_{\widehat\Lambda}?$
(respectively, $\Hom_\Lambda(I_i,?)$ and
$\Hom_{\widehat\Lambda}(\widehat I_i,?)$) coincide on
$\Lambda\mmod$.

Moreover, the Weyl group $W$ of the Dynkin diagram embeds as a
parabolic subgroup of the Weyl group $\widehat W$ of the larger
diagram. Given $w\in W$, one can then define $I_w$ as the bimodule
$I_{i_1}\otimes_\Lambda\cdots\otimes_\Lambda I_{i_\ell}$ by
choosing a reduced decomposition $w=s_{i_1}\cdots s_{i\ell}$.
Since the functors $I_w\otimes_\Lambda?$ and $\widehat
I_w\otimes_{\widehat\Lambda}?$ coincide on $\Lambda\mmod$, we deduce
the Dynkin case of Theorem~\ref{th:BuanIyamaReitenScott}~\ref{it:BIRSd}
from the non-Dynkin case. The proof of~\cite{BuanIyamaReitenScott09},
Proposition~II.1.5 then immediately implies
Theorem~\ref{th:BuanIyamaReitenScott}~\ref{it:BIRSa}. Using
that moreover the functors $\Hom_\Lambda(I_w,?)$ and
$\Hom_{\widehat\Lambda}(\widehat I_w,?)$ coincide on $\Lambda\mmod$,
we deduce the Dynkin case from the non-Dynkin case in
Theorem~\ref{th:BrennerButler}, Examples~\ref{ex:TorsTheo}
\ref{it:TTa} and~\ref{it:TTb}, Propositions~\ref{pr:CompEquiv}
and Corollary~\ref{co:ReflDimVec}.

Finally, to see that Examples~\ref{ex:TorsTheo}~\ref{it:TTc}
and~\ref{it:TTd} hold true in the Dynkin case, one can apply
Example~\ref{ex:CompIR}. One may here moreover note that
the surjection $\widehat\Lambda\to\Lambda$ induces an isomorphism
$\widehat\Lambda/\widehat I_w\cong\Lambda/I_w$ (this follows,
for instance, by an obvious dimension argument based on Amiot,
Iyama, Reiten and Todorov's filtration, see section~2 of
\cite{AmiotIyamaReitenTodorov10}).

Another pecularity of the finite type case is the fact that the
Tits cone $C_T$ fills the whole dual space of $\mathbb RI$.
By Theorem~\ref{th:TiltPIR}, any torsion pair
$(\overline{\mathscr I}_\theta,\mathscr P_\theta)$ or
$(\mathscr I_\theta,\overline{\mathscr P}_\theta)$
is of the form $(\mathscr T_w,\mathscr F_w)$, so these latter
are enough to completely describe the polytopes $\Pol(T)$. This
fact, combined with Corollary~\ref{co:FacetsViaHom}, shows
that the definition of $\Pol(T)$ given in \cite{BaumannKamnitzer12}
is identical to the definition used in the present paper.

In addition, the intersection between $C_T$ and its opposite is
not $\{0\}$ anymore, so it is possible to write
$-u^{-1}\eta=v\theta$ with $(u,v)\in W^2$ and both $\eta$ and
$\theta$ in $C_0$. In this case, we have
\begin{equation}
\label{eq:DynkinTopBot}
(\mathscr T^v,\mathscr F^v)=(\mathscr I_{v\theta},\mathscr
I_{v\theta})=(\mathscr I_{-u^{-1}\eta},\mathscr P_{-u^{-1}\eta})
=(\mathscr T_u,\mathscr F_u),
\end{equation}
by Theorem~\ref{th:TiltPIR}~\ref{it:TPIRb} and
Remark~\ref{rk:ComTPIR}~\ref{it:CTPIRb}.

A last peculiarity of the finite type case was mentioned
above in Remark~\ref{rk:RemLusParTop}~\ref{it:RLPTb}.

\section{The Hall functors}
\label{se:HallFunc}
In this section, we assume we are given two orthogonal rigid bricks
$S$ and $R$ in $\Lambda\mmod$ such that $\dim\Ext^1(S,R)=2$. In other
words, we assume that $S$ and $R$ are two finite dimensional
$\Lambda$-modules such that
\begin{equation}
\label{eq:DatumSR}
\begin{aligned}
\End_\Lambda(S)&=\End_\Lambda(R)=K,\\
\Ext^1_\Lambda(S,S)&=\Ext^1_\Lambda(R,R)=0,
\end{aligned}
\qquad\qquad
\begin{aligned}
\Hom_\Lambda(S,R)=\Hom_\Lambda(R,S)&=0,\\
\dim\Ext^1_\Lambda(S,R)=2.&
\end{aligned}
\end{equation}

We fix $\xi$ and $\eta$ in
$\bigoplus_{a\in H}\Hom_K(S_{s(a)},R_{t(a)})$
and $\hat\xi$ and $\hat\eta$ in
$\bigoplus_{a\in H}\Hom_K(R_{s(a)},S_{t(a)})$
such that, in the notation of section~\ref{ss:ProjRes},
\begin{xalignat}2
\label{eq:ApproxCoh1}
d^1_{S,R}(\xi)&=d^1_{S,R}(\eta)=0,&
d^1_{R,S}(\hat\xi)&=d^1_{R,S}(\hat\eta)=0,\\[4pt]
\label{eq:ApproxCoh2}
\tau_1(\xi,\hat\xi)&=\tau_1(\eta,\hat\eta)=1,&
\tau_1(\xi,\hat\eta)&=\tau_1(\eta,\hat\xi)=0.
\end{xalignat}

Thus $(\xi,\eta)$ can be regarded as a basis of
$\Ext^1_\Lambda(S,R)$ and $(\hat\xi,\hat\eta)$ can be
regarded as the dual basis of $\Ext^1_\Lambda(R,S)$.

\subsection{A combinatorial lemma}
\label{ss:CombiLem}
We denote by $\Pi$ the completed preprojective algebra of type
$\widetilde A_1$. This is the completed preprojective algebra
of the Kronecker quiver
$$\xymatrix@C=3em{0\ar@/^/[r]^\alpha\ar@/_/[r]_\beta&1.}$$
It contains orthogonal idempotents $e_0$ and $e_1$ and arrows
$\alpha,\beta\in e_1\Pi e_0$ and $\overline\alpha,\overline\beta\in
e_0\Pi e_1$. We denote its augmentation ideal by $J$.

For the following, it is useful to think of the algebra $\Pi$
as the quotient
$$\mathbf S\langle\langle\,\alpha,\beta,\overline\alpha,
\overline\beta\,\rangle\rangle/(\alpha\overline\alpha+
\beta\overline\beta,\overline\alpha\alpha+\overline\beta\beta)$$
of the ring of non-commutative formal power series in four
variables $\alpha$, $\beta$, $\overline\alpha$, $\overline\beta$
with coefficients in the commutative semisimple algebra
$\mathbf S=Ke_0\oplus Ke_1$.

\begin{lemma}
\label{le:CombiLem}
The image of the linear map
$$C:\bigl(e_1\Pi e_0\bigr)^2\times\bigl(e_0\Pi e_1\bigr)^2\to
e_0\Pi e_0\times e_1\Pi e_1$$
given by
$$C(x,y,\hat x,\hat y)=(-\hat x\alpha-\hat y\beta-\overline\alpha x-
\overline\beta y,\;x\overline\alpha+y\overline\beta+\alpha\hat x+
\beta\hat y)$$
contains:
\begin{itemize}
\item
any element of the form $(uv-vu,0)$, where
$(u,v)\in\bigl(e_0\Pi e_0\bigr)^2$;
\item
any element of the form $(0,uv-vu)$, where
$(u,v)\in\bigl(e_1\Pi e_1\bigr)^2$;
\item
any element of the form $(uv,-vu)$, where
$(u,v)\in e_0\Pi e_1\times e_1\Pi e_0$.
\end{itemize}
\end{lemma}
\begin{proof}
Let $u$ and $v$ be two words of even length in the alphabet
$\bigl\{\alpha,\beta,\overline\alpha,\overline\beta\bigr\}$, in which
barred and non-barred letters alternate, and which start with a
barred letter. Thus $u$ and $v$ define elements in $e_0\Pi e_0$. We
write $u=\overline{c_1}c_2\cdots\overline{c_{2\ell-1}}c_{2\ell}$,
with $c_k\in\{\alpha,\beta\}$. For $1\leq k\leq\ell$, we set
$$m_k=c_{2k}\overline{c_{2k+1}}\cdots c_{2\ell}\;v\;\overline{c_1}
c_2\cdots c_{2k-2}\qquad\text{and}\qquad
n_k=\overline{c_{2k+1}}c_{2k+2}\cdots c_{2\ell}\;v\;\overline{c_1}
c_2\cdots\overline{c_{2k-1}}.$$
Then $(uv-vu,0)$
is the image of the element
$$\left(
-\sum_{\substack{1\leq k\leq\ell\\[3pt]
\overline{c_{2k-1}}=\overline\alpha}}m_k,\quad
-\sum_{\substack{1\leq k\leq\ell\\[3pt]
\overline{c_{2k-1}}=\overline\beta}}m_k,\quad
\sum_{\substack{1\leq k\leq\ell\\[3pt]c_{2k}=\alpha}}n_k,\quad
\sum_{\substack{1\leq k\leq\ell\\[3pt]c_{2k}=\beta}}n_k\right)$$
by our linear map. This shows that the elements of the first kind
belong to the image of our map. The two other cases are similar.
\end{proof}

\subsection{A universal lifting}
\label{ss:UnivLift}
Given a $K$-vector space $V$, two elements $X$ and $Y$ in
$V\otimes_K\Pi$, and an integer $k\geq0$, generalizing a
standard notation, we will write $X\equiv Y\bmod J^k$ to
express that $X-Y$ belongs to $V\otimes_K J^k$.

For $a\in H$, we set $\dot S_a=S_a\otimes e_0$, where
$S_a$ is the linear map attached to the arrow $a$ in the
$\Lambda$-module $S$; this $\dot S_a$ is an element in
$\Hom_K(S_{s(a)},S_{t(a)})\otimes_Ke_0\Pi e_0$. Likewise,
we set $\dot R_a=R_a\otimes e_1$; this is an element in
$\Hom_K(R_{s(a)},R_{t(a)})\otimes_Ke_1\Pi e_1$.

We set $P=\xi\otimes\alpha+\eta\otimes\beta$ and
$Q=\hat\xi\otimes\overline\alpha+\hat\eta\otimes\overline\beta$;
these are elements in
$$\bigoplus_{a\in H}\Hom_K(S_{s(a)},R_{t(a)})
\otimes_Ke_1\Pi e_0\qquad\text{and}\qquad
\bigoplus_{a\in H}\Hom_K(R_{s(a)},S_{t(a)})
\otimes_Ke_0\Pi e_1,$$
respectively. For $a\in H$, we write $P_a$ for the component of
$P$ in $\Hom_K(S_{s(a)},R_{t(a)})\otimes_Ke_1\Pi e_0$ and
$Q_a$ for the component of $Q$ in $\Hom_K(R_{s(a)},S_{t(a)})
\otimes_Ke_0\Pi e_1$.

\begin{lemma}
\label{le:ApproxCoh}
There are elements
\begin{gather*}
S^{(\infty)}\in\bigoplus_{a\in H}\Hom_K(S_{s(a)},S_{t(a)})
\otimes_Ke_0\Pi e_0,\qquad
R^{(\infty)}\in\bigoplus_{a\in H}\Hom_K(R_{s(a)},R_{t(a)})
\otimes_Ke_1\Pi e_1,\\[4pt]
P^{(\infty)}\in\bigoplus_{a\in H}\Hom_K(S_{s(a)},R_{t(a)})
\otimes_Ke_1\Pi e_0,\qquad
Q^{(\infty)}\in\bigoplus_{a\in H}\Hom_K(R_{s(a)},S_{t(a)})
\otimes_Ke_0\Pi e_1
\end{gather*}
such that for each $a\in H$,
$$\begin{pmatrix}S^{(\infty)}_a&Q^{(\infty)}_a\\
P^{(\infty)}_a&R^{(\infty)}_a\end{pmatrix}
\equiv\begin{pmatrix}\dot S_a&Q_a\\[2pt]P_a&\dot R_a\end{pmatrix}
\quad\bmod\ \begin{pmatrix}J^2&J^3\\[2pt]J^3&J^2\end{pmatrix},$$
and for each $i\in I$,
$$\sum_{\substack{a\in H\\s(a)=i}}\varepsilon(a)
\begin{pmatrix}S^{(\infty)}_{\overline a}&Q^{(\infty)}_{\overline a}\\
P^{(\infty)}_{\overline a}&R^{(\infty)}_{\overline a}\end{pmatrix}
\begin{pmatrix}S^{(\infty)}_a&Q^{(\infty)}_a\\
P^{(\infty)}_a&R^{(\infty)}_a\end{pmatrix}=0.$$
\end{lemma}
\begin{proof}
The desired elements will be constructed as the limit in the $J$-adic
topology of elements $S^{(k)}$, $R^{(k)}$, $P^{(k)}$ and $Q^{(k)}$
such that
$$\begin{pmatrix}S^{(k+1)}&Q^{(k+1)}\\
P^{(k+1)}&R^{(k+1)}\end{pmatrix}
\equiv\begin{pmatrix}S^{(k)}&Q^{(k)}\\
P^{(k)}&R^{(k)}\end{pmatrix}\quad
\bmod\ \begin{pmatrix}J^{2k+2}&J^{2k+3}\\
J^{2k+3}&J^{2k+2}\end{pmatrix}.$$
For $k=0$, we define
$$S^{(0)}\in\bigoplus_{a\in H}\Hom_K(S_{s(a)},S_{t(a)})
\otimes_Ke_0\Pi e_0$$
by gathering the elements $\dot S_a$. We similarly define $R^{(0)}$
and we set $P^{(0)}=P$ and $Q^{(0)}=Q$. The conditions we impose at
step $k$ are
\begin{equation}
\label{eq:ApproxCoh3}
\sum_{\substack{a\in H\\s(a)=i}}\varepsilon(a)
\begin{pmatrix}S^{(k)}_{\overline a}&Q^{(k)}_{\overline a}\\
P^{(k)}_{\overline a}&R^{(k)}_{\overline a}\end{pmatrix}
\begin{pmatrix}S^{(k)}_a&Q^{(k)}_a\\
P^{(k)}_a&R^{(k)}_a\end{pmatrix}\equiv0\quad
\bmod\ \begin{pmatrix}J^{2k+2}&J^{2k+3}\\
J^{2k+3}&J^{2k+2}\end{pmatrix}
\end{equation}
for each $i\in I$,
\begin{equation}
\label{eq:ApproxCoh4}
\dot\tau_1\Bigl(S^{(k)},S^{(k)}\Bigr)+\dot\tau_1\Bigl(Q^{(k)},P^{(k)}\Bigr)
\equiv0\quad\bmod\ J^{2k+4}
\end{equation}
and
\begin{equation}
\label{eq:ApproxCoh5}
\dot\tau_1\Bigl(R^{(k)},R^{(k)}\Bigr)+\dot\tau_1\Bigl(P^{(k)},Q^{(k)}\Bigr)
\equiv0\quad\bmod\ J^{2k+4}.
\end{equation}
In~\eqref{eq:ApproxCoh4} and~\eqref{eq:ApproxCoh5}, the symbol
$\dot\tau_1$ represents $\tau_1\otimes\id_\Pi$, the map obtained
from the bilinear form $\tau_1$ defined in section~\ref{ss:ProjRes}
by extending the scalars from $K$ to $\Pi$. Similarly, we will
denote the map $\tau_2\otimes\id_\Pi$ by the symbol $\dot\tau_2$.

Thanks to the preprojective relations for the $\Lambda$-modules $S$ and
$R$ and to equations~\eqref{eq:ApproxCoh1} and \eqref{eq:ApproxCoh2},
the conditions~\eqref{eq:ApproxCoh3}--\eqref{eq:ApproxCoh5} are
fulfilled at step $k=0$.

Let us assume that $S^{(k)}$, $R^{(k)}$, $P^{(k)}$ and $Q^{(k)}$
have been constructed. We set
$$f_i=\sum_{\substack{a\in H\\s(a)=i}}\varepsilon(a)
\Bigl(S^{(k)}_{\overline a}S^{(k)}_a+Q^{(k)}_{\overline
a}P^{(k)}_a\Bigr)$$
and we regard $(f_i)$ as a formal series in $\alpha$, $\beta$,
$\overline\alpha$, $\overline\beta$ with coefficients in
$\bigoplus_{i\in I}\Hom_K(S_i,S_i)$ and valuation at least
$2k+2$. Then, thanks to~\eqref{eq:ApproxCoh4}, we have
$$\dot\tau_2\bigl((f_i),(\id_{S_i}\otimes e_0)\bigr)\in J^{2k+4};$$
in other words, the coefficient in $(f_i)$ of any monomial of
degree less than $2k+4$ is $\tau_2$-orthogonal to $\Hom_\Lambda(S,S)$.
We conclude that modulo $J^{2k+4}$, $(f_i)$ belongs $\im d^1_{S,S}$.
Therefore there exists
$$\widetilde S\in\bigoplus_{a\in H}\Hom_K(S_{s(a)},S_{t(a)})
\otimes_Ke_0\Pi e_0$$
of valuation at least $2k+2$ such that
$d^1_{S,S}\bigl(\widetilde S\bigr)\equiv(f_i)\ \bmod J^{2k+4}$.
We then set $S^{(k+1)}=S^{(k)}-\widetilde S$, and the upper
left corner of \eqref{eq:ApproxCoh3} is satisfied at step $k+1$.
One similarly finds first $R^{(k+1)}$, and next $P^{(k+1)}$ and
$Q^{(k+1)}$, that satisfy~\eqref{eq:ApproxCoh3} at step $k+1$.
However, the elements $P^{(k+1)}$ and $Q^{(k+1)}$ obtained in
this way are not the final ones, since they do not yet satisfy
\eqref{eq:ApproxCoh4} and~\eqref{eq:ApproxCoh5}.

The left-hand sides of \eqref{eq:ApproxCoh4}
and~\eqref{eq:ApproxCoh5} are the components of
$$D=\biggl(\dot\tau_1\Bigl(S^{(k+1)},S^{(k+1)}\Bigr)+
\dot\tau_1\Bigl(Q^{(k+1)},P^{(k+1)}\Bigr),\
\dot\tau_1\Bigl(R^{(k+1)},R^{(k+1)}\Bigr)+
\dot\tau_1\Bigl(P^{(k+1)},Q^{(k+1)}\Bigr)\biggr).$$
Thus $D\in e_0\Pi e_0\times e_1\Pi e_1$ and its two components
have valuations at least $2k+4$. Because of the cyclicity of the
trace and of the presence of the sign $\varepsilon(a)$ in
$$\dot\tau_1\Bigl(S^{(k+1)},S^{(k+1)}\Bigr)=\sum_{i\in I}
\Tr\left(\sum_{\substack{a\in H\\s(a)=i}}\varepsilon(a)
\Bigl(S^{(k+1)}_{\overline a}S^{(k+1)}_a\Bigr)\right),$$
the first term in the first component of $D$ is a
linear combination of elements of the kind $uv-vu$, with
$(u,v)\in\bigl(e_0\Pi e_0\bigr)^2$. Likewise, we see that the
contributions to $D$ of $R^{(k+1)}$, and of $P^{(k+1)}$ and $Q^{(k+1)}$,
are linear combination of elements of the second and third kind in
the statement of Lemma~\ref{le:CombiLem}. Therefore $D$ belongs
to the image of the map $C$, so we can find $(x,y)\in(e_1\Pi e_0)^2$
and $(\hat x,\hat y)\in(e_0\Pi e_1)^2$ of valuation at least $2k+3$
such that $D=C(x,y,\hat x,\hat y)$. If we now correct $P^{(k+1)}$ and
$Q^{(k+1)}$ by subtracting $\xi\otimes x+\eta\otimes y$ from the
former and $\hat\xi\otimes\hat x+\hat\eta\otimes\hat y$ from the
latter, then $D$ will vanish modulo $J^{2k+6}$, thanks
to~\eqref{eq:ApproxCoh2}, while the condition~\eqref{eq:ApproxCoh3}
remains satisfied, thanks to~\eqref{eq:ApproxCoh1}.
\end{proof}

\subsection{Construction of the Hall functors}
\label{ss:ConsHallFunc}
We are now in a position to define a fully faithful exact
functor $\mathscr H:\Pi\mmod\to\Lambda\mmod$.

Let $V$ be a $\Pi$-module. We define an $I$-graded vector space
$M=\bigoplus_{i\in I}M_i$ by
$$M_i=\bigl(S_i\otimes_KV_0\bigr)\oplus\bigl(R_i\otimes_KV_1\bigr).$$
Now let $a\in H$. Then $S^{(\infty)}_a$ is an element of
$\Hom_K(S_{s(a)},S_{t(a)})\otimes_Ke_0\Pi e_0$, hence can be
seen as a matrix with entries in $e_0\Pi e_0$. We can evaluate
these entries in the representation $V$; the result belongs to
$$\Hom_K(S_{s(a)},S_{t(a)})\otimes_K\Hom_K(V_0,V_0)=
\Hom_K(S_{s(a)}\otimes_KV_0,S_{t(a)}\otimes_KV_0).$$
Evaluating $R^{(\infty)}_a$, $P^{(\infty)}_a$ and $Q^{(\infty)}_a$
in a similar way, we map
$$\begin{pmatrix}S^{(\infty)}_a&Q^{(\infty)}_a\\
P^{(\infty)}_a&R^{(\infty)}_a\end{pmatrix}$$
to an element in $\Hom_K(M_{s(a)},M_{t(a)})$. The conditions
imposed in Lemma~\ref{le:ApproxCoh} assert that we then get
a $\Lambda$-module $M$.

Since we have worked with a universal formula (the same for
all $\Pi$-modules $V$), the assignment $V\mapsto M$ defines
a functor $\mathscr H$, which moreover is exact.
Let $\langle S,R\rangle$ be the smallest abelian, closed
under extensions, subcategory of $\Lambda\mmod$ that contains
the isomorphism classes of $S$ and~$R$.

\begin{theorem}
\label{th:HallFunc}
The functor $\mathscr H$ induces an equivalence of categories
between $\Pi\mmod$ and $\langle S,R\rangle$.
\end{theorem}
\begin{proof}
The category $\langle S,R\rangle$ has only (up to isomorphism)
two simple objects, namely $S$ and $R$, for these latter are
orthogonal bricks. In view of \cite{GabrielRoiter92}, Lemma~11.7,
it thus suffices to show that for any simple $\Pi$-modules $L$
and $L'$, the induced homomorphism $\Ext^k_\Pi(L,L')\to
\Ext^k_{\langle S,R\rangle}(\mathscr H(L),\mathscr H(L'))$ is
bijective for $k\in\{0,1\}$ and injective for $k=2$. We can here
replace the extension spaces in $\langle S,R\rangle$ by the
extension spaces in $\Lambda\mmod$: this does not change the
$\Ext^0$ nor the $\Ext^1$, for $\langle S,R\rangle$ is full and
closed under extensions; and if the injectivity condition holds
for $\Ext^2_{\Lambda}$, it will a fortiori holds for
$\Ext^2_{\langle S,R\rangle}$.

Let us call $W_0$ and $W_1$ the two simple $\Pi$-modules,
concentrated at vertices $0$ and $1$ respectively; then
$\mathscr H(W_0)=S$ and $\mathscr H(W_1)=R$. Obviously,
$$\End(W_0)=\End(W_1)=K\qquad\text{and}\qquad
\Hom(W_0,W_1)=\Hom(W_1,W_0)=0,$$
so the condition is fulfilled for $k=0$.

The $\Pi$-modules $T_\alpha$ and $T_\beta$ with dimension-vector
$(1,1)$ obtained by letting the arrows of $\Pi$ act by
$$(\alpha,\beta,\overline\alpha,\overline\beta)\mapsto(1,0,0,0)
\qquad\text{and}\qquad
(\alpha,\beta,\overline\alpha,\overline\beta)\mapsto(0,1,0,0)$$
are extensions of $W_0$ by $W_1$.
We denote their extension classes in $\Ext^1_\Pi(W_0,W_1)$ by
$\alpha$ and $\beta$, respectively. The extension classes of
$\mathscr H(T_\alpha)$ and $\mathscr H(T_\beta)$ are $\xi$ and
$\eta$. Thus, the induced homomorphism
$\Ext^1_\Pi(W_0,W_1)\to\Ext^1_\Lambda(S,R)$ maps the basis
$(\alpha,\beta)$ of the first space to the basis $(\xi,\eta)$
of the second space; it is therefore bijective. We check in a
similar way the other cases for $k=1$.

The equality $\tau_1(\xi,\hat\xi)=1$ implies that the Yoneda
product $\xi\hat\xi\in\Ext^2_\Lambda(R,R)$ does not vanish.
The induced homomorphism $\Ext^2_\Pi(W_1,W_1)\to\Ext^2_\Lambda(R,R)$
maps $\alpha\overline\alpha$ to $\xi\hat\xi$, so it cannot be zero. It
is thus injective, for $\Ext^2_\Pi(W_1,W_1)$ is one dimensional.
The other cases for $k=2$ are treated in like manner.
\end{proof}

We here note that a proof for Lemma~11.7 in~\cite{GabrielRoiter92}
can be found in~\cite{Kimura07}, Proposition~3.4.3.

\subsection{Irreducible components}
\label{ss:HFIrrComp}
We now study the consequences of the existence of a Hall functor
at the level of irreducible components of the nilpotent varieties.

Let $\mu=(\mu_0,\mu_1)$ be a dimension-vector for $\Pi$ and set
$\nu=\mu_0\dimvec S+\mu_1\dimvec R$. We denote by $\Pi(\mu)$
the nilpotent variety for $\Pi$ and by
$\Lambda_{\langle S,R\rangle}(\nu)$ the set of all points in
$\Lambda(\nu)$ that belong to $\langle S,R\rangle$. In addition,
we define $\Omega(\mu)$ to be the set of all triples $(V,M,f)$
such that $V\in\Pi(\mu)$, $M\in\Lambda(\nu)$ and
$f:\mathscr H(V)\to M$ is an isomorphism of $\Lambda$-modules.
We can then form the diagram
\begin{equation}
\label{eq:HFIrrComp}
\Pi(\mu)\xleftarrow p\Omega(\mu)\xrightarrow q\Lambda(\nu)
\end{equation}
in which $p$ and $q$ are the first and second projection.
Obviously, $p$ is a principal $G(\nu)$-bundle, the image of
$q$ is $\Lambda_{\langle S,R\rangle}(\nu)$, and each non-empty
fiber of $q$ is isomorphic to $G(\mu)$.

\begin{proposition}
\label{pr:HFIrrComp}
The subset $\Lambda_{\langle S,R\rangle}(\nu)$ is constructible
and all its irreducible components have full dimension in
$\Lambda(\nu)$. The diagram \eqref{eq:HFIrrComp} induces a
bijection between the irreducible components of $\Pi(\mu)$ and
the irreducible components of $\Lambda(\nu)$ whose general point
belongs to $\Lambda_{\langle S,R\rangle}(\nu)$.
\end{proposition}
\begin{proof}
Combining the conditions~\eqref{eq:DatumSR} with
equation~\eqref{eq:CrawleyBoeveyForm}, we get
$$(\mu,\mu)=(\nu,\nu),$$
where $(\,,\,)$ in the left-hand side is the bilinear form on
$\mathbf K(\Pi\mmod)$ and $(\,,\,)$ in the right-hand side is the
bilinear form on $\mathbf K(\Lambda\mmod)\cong\mathbb ZI$. In
view of~\eqref{eq:DimNilpVar}, this translates to
$$\dim G(\mu)-\dim\Pi(\mu)=\dim G(\nu)-\dim\Lambda(\nu).$$
The proposition now results from general results of algebraic
geometry, similar to those used in section~\ref{ss:TorLambda}.
\end{proof}

\section{Preprojective algebras of affine type}
\label{se:PrepAlgAff}
From now on, $\mathfrak g$ is of symmetric affine type. We will apply
the theory we have been developing to finally obtain our affine MV
polytopes. We will use the notation concerning affine roots systems
discussed in section~\ref{ss:SetupAffTyp}.

\subsection{Torsion pairs associated to biconvex subsets}
\label{ss:TorBiconv}
From Corollary~\ref{co:SuppFunTitsFan}, it follows that the HN
polytope $P=\Pol(T)$ of a finite dimensional $\Lambda$-module
$T$ is GGMS. The construction in section~\ref{ss:GGMSPol} then
assigns a vertex $\mu_P(A)$ of $P$ to each biconvex subset
$A\subseteq\Phi_+$. As noticed after Corollary~\ref{co:FacesHN},
$\mu_P(A)$ is the dimension vector of a unique submodule $T_A$
of $T$. Our aim in this section is to construct a torsion pair
$(\mathscr T(A),\mathscr F(A))$ with respect to which $T_A$ is
the torsion submodule of $T$, for all $\Lambda$-module $T$.

As in section~\ref{ss:BiconvSetsGenType}, each $w\in W$ gives rise
to two biconvex subsets: a finite one, namely $A_w=N_{w^{-1}}$,
and a cofinite one, $A^w=\Phi_+\setminus N_w$
(see Example~\ref{ex:BiconvexSets}~\ref{it:BSb}). We set
$$(\mathscr T(A_w),\mathscr F(A_w))=(\mathscr T_w,\mathscr F_w)
\quad\text{and}\quad
(\mathscr T(A^w),\mathscr F(A^w))=(\mathscr T^w,\mathscr F^w).$$
In this fashion, we associate a torsion pair in $\Lambda\mmod$ to
each finite or cofinite biconvex set. (By \eqref{eq:DynkinTopBot},
this construction is unambiguous in the finite type case, where
biconvex sets are both finite and cofinite.)

Using Lemma~\ref{le:CombiCox}, \eqref{eq:TorsTheoMono1} and
Proposition~\ref{pr:TorsTheoMono2}, we deduce the following
monotonicity property: if $A\subseteq B$, then $(\mathscr
T(A),\mathscr F(A))\preccurlyeq(\mathscr T(B),\mathscr F(B))$.
We also note the following interpretation of
Example~\ref{ex:TorsTheo}~\ref{it:TTc}: if $B=\Phi_+\setminus A$,
then $(\mathscr T(B),\mathscr F(B))=(\mathscr F(A)^*,\mathscr T(A)^*)$.

By Proposition~\ref{pr:ApproxBiconv}, every biconvex subset
$A$ is either the increasing union of finite biconvex subsets or
the decreasing intersection of cofinite biconvex subsets. In the
former case, we set
$$\mathscr T(A)=\bigcup_{\substack{\text{$B$ finite biconvex}\\[1pt]
B\subseteq A}}\mathscr T(B)\qquad\text{and}\qquad
\mathscr F(A)=\bigcap_{\substack{\text{$B$ finite biconvex}\\[1pt]
B\subseteq A}}\mathscr F(B).$$
Then $(\mathscr T(A),\mathscr F(A))$ is a torsion pair. The
axiom~(T1) is indeed easily verified. To check~(T2), we take
$T\in\Lambda\mmod$. Each finite biconvex subset $B\subseteq A$
provides a submodule $T_B\subseteq T$ such that
$(T_B,T/T_B)\in\mathscr T(B)\times\mathscr F(B)$. The monotonicity
property implies that the map $B\mapsto T_B$ is non-decreasing, so
for dimension reasons, the family $(T_B)$ has an maximal element, say
$T_{B_0}$. Then for any $B_1\supseteq B_0$, we have $T_{B_1}=T_{B_0}$,
which shows that $(T_{B_0},T/T_{B_0})\in\mathscr T({B_1})\times\mathscr
F({B_1})$. We conclude that $(T_{B_0},T/T_{B_0})\in\mathscr
T(A)\times\mathscr F(A)$, as desired.

When $A$ is the intersection of cofinite biconvex subsets, we set
$$\mathscr T(A)=\bigcap_{\substack{\text{$B$ finite biconvex}\\[1pt]
B\supseteq A}}\mathscr T(B)\qquad\text{and}\qquad
\mathscr F(A)=\bigcup_{\substack{\text{$B$ finite biconvex}\\[1pt]
B\supseteq A}}\mathscr F(B).$$
We then have a torsion pair $(\mathscr T(A),\mathscr F(A))$
associated in a monotonous way to each biconvex subset~$A$, and
for $B=\Phi_+\setminus A$, we have $(\mathscr T(B),\mathscr
F(B))=(\mathscr F(A)^*,\mathscr T(A)^*)$.

\begin{proposition}
\label{pr:TorsionBiconv}
Let $T$ be a finite-dimensional $\Lambda$-module $T$, let
$P$ be its HN polytope, let $A$ be a biconvex subset, and
let $T_A$ be the torsion submodule of $T$ with respect to
the torsion pair $(\mathscr T(A),\mathscr F(A))$. Then
$\mu_P(A)=\dimvec T_A$.
\end{proposition}
\begin{proof}
Suppose first that $A$ is cofinite. Let us choose $\theta\in C_0$
and let $w$ be the unique element in $W$ such that $A=A^w$
(Proposition~\ref{pr:ATheta}). We have $T^w=T_{w\theta}^{\min}
=T_{w\theta}^{\max}$ by Theorem~\ref{th:TiltPIR}. Therefore
the face $P_{w\theta}$ of $P$ is reduced to the vertex
$\dimvec T^w$. The definition in section~\ref{ss:GGMSPol}
then leads to $\mu_P(A)=\dimvec T^w$, which is the desired
equality in our case. The case where $A$ is finite is proved
in a similar fashion, using Remark~\ref{rk:ComTPIR}~\ref{it:CTPIRb}.
The case of a general $A$ then follows by monotonous
approximation.
\end{proof}

Recall from Example~\ref{ex:BiconvexSets}~\ref{it:BSc} that any
$\theta\in(\mathbb RI)^*$ defines two biconvex subsets
$$A_\theta^{\min}=\{\alpha\in\Phi_+\mid\langle\theta,
\alpha\rangle>0\}\quad\text{and}\quad A_\theta^{\max}=
\{\alpha\in\Phi_+\mid\langle\theta,\alpha\rangle\geq0\}.$$

\begin{proposition}
\label{pr:AThetaTor}
For each $\theta\in(\mathbb RI)^*$, we have
$$(\mathscr T(A_\theta^{\min}),\mathscr F(A_\theta^{\min}))=
(\mathscr I_\theta,\overline{\mathscr P}_\theta)\quad\text{and}\quad
(\mathscr T(A_\theta^{\max}),\mathscr F(A_\theta^{\max}))=
(\overline{\mathscr I}_\theta,\mathscr P_\theta).$$
\end{proposition}
\begin{proof}
Let $\theta\in(\mathbb RI)^*$.

Let $T$ be a finite dimensional $\Lambda$-module and let $P$ be its
HN polytope. By definition, $T\in\overline{\mathscr P}_\theta$ if
and only if $T_\theta^{\min}=0$. By Proposition~\ref{pr:TminTmax},
this is equivalent to the equation $\psi_P(\theta)=0$, which holds
if and only if $\langle\theta,\mu_P(A_\theta^{\min})\rangle=0$ by
Proposition~\ref{pr:SuppGGMS}. Since $\mu_P(A_\theta^{\min})$ is
a non-negative linear combination of roots $\alpha$ such that
$\langle\theta,\alpha\rangle>0$ (Lemma~\ref{le:DiffVertices}),
this equation amounts to $\mu_P(A_\theta^{\min})=0$.
By Proposition~\ref{pr:TorsionBiconv}, this holds if and only if
$T\in\mathscr F(A_\theta^{\min})$.

Therefore $\overline{\mathscr P}_\theta=\mathscr F(A_\theta^{\min})$,
whence the first equality. The second follows by $*$-duality.
\end{proof}

The biconvex subsets $A_\theta^{\min}$ and $A_\theta^{\max}$
only depend on the face of $\mathscr W$ to which $\theta$
belongs. By Proposition~\ref{pr:AThetaTor}, the same is true for
the torsion pairs $(\mathscr I_\theta,\overline{\mathscr P}_\theta)$
and $(\overline{\mathscr I}_\theta,\mathscr P_\theta)$, and also
for the category $\mathscr R_\theta=\overline{\mathscr I}_\theta\cap
\overline{\mathscr P}_\theta$. For each face $F$ of $\mathscr W$,
we may thus define categories $\mathscr I_F$, $\overline{\mathscr I}_F$,
etc.\ so that $\mathscr I_F=\mathscr I_\theta$,
$\overline{\mathscr I}_F=\overline{\mathscr I}_\theta$, etc.\
for any $\theta\in F$. Further, we denote by $\mathfrak I_F$,
$\mathfrak R_F$, and $\mathfrak P_F$ the subsets of $\mathfrak B$
consisting of irreducible components whose general point belong
to the categories $\mathscr I_F$, $\mathscr R_F$, and~$\mathscr P_F$,
respectively.

\begin{lemma}
\label{le:ContTPBiconv}
Let $T$ be a finite dimensional $\Lambda$-module and let
$A\subseteq B$ be two biconvex subsets. Denote by $T_A$ and
$T_B$ the torsion submodules of $T$ with respect to the torsion
pairs $(\mathscr T(A), \mathscr F(A))$ and
$(\mathscr T(B),\mathscr F(B))$. If $\dim T<\height\alpha$
for each $\alpha\in B\setminus A$, then $T_A=T_B$.
\end{lemma}
\begin{proof}
Let $P$ be the HN polytope of $T$. By Proposition~\ref{pr:TorsionBiconv},
$\dimvec T_B/T_A=\mu_P(B)-\mu_P(A)$, whence
$\height(\mu_P(B)-\mu_P(A))=\dim T_B/T_A\leq\dim T$. On the other hand,
Lemma~\ref{le:DiffVertices} expresses the weight $\mu_P(B)-\mu_P(A)$
as a non-negative linear combination of roots in $B\setminus A$.
The last assumption in the statement forces this linear combination
to be trivial, which implies that $T_B/T_A=0$.
\end{proof}

\begin{proposition}
\label{pr:TPAdjBiconv}
Let $A$ and $B$ be two biconvex subsets and let
$\alpha\in\Phi_+^{\re}$. Assume that $B=A\sqcup\{\alpha\}$.
Then there is a rigid indecomposable $\Lambda$-module $L(A,B)$
of dimension-vector $\alpha$ such that
$\mathscr F(A)\cap\mathscr T(B)=\add L(A,B)$.
\end{proposition}
\begin{proof}
The case where $A$ and $B$ are both finite follows from
Theorem~\ref{th:GLSFiltDual}~\ref{it:GLSFDa}: if
$A=A_w$ and $B=A_{s_iw}$, then $L(A,B)=\Hom_\Lambda(I_w,S_i)$.
This module is rigid, because $S_i$ is rigid and the equivalence
in Theorem~\ref{th:BrennerButler}~\ref{it:BBc} preserves rigidity.

Assume now that $A$ and $B$ are infinite and that $\delta\notin A$.
By Lemma~\ref{le:ApproxAdjBiconv}, we can then find $A'\subseteq A$
and $B'\subseteq B$ finite biconvex subsets such that
$B'=A'\sqcup\{\alpha\}$. Certainly, $A'$ and $B'$ are not unique
subject to these requirements. We claim however that $L(A',B')$
does not depend on the choice of $A'$ and $B'$.

To see this, consider $A''$ and $B''$ finite biconvex subsets
with $B''=A''\sqcup\{\alpha\}$ and $A''\subseteq A$. We want to
show that $L(A',B')\cong L(A'',B'')$. Without loss of generality,
we may assume that $A''\supseteq A'$. We write $A'=A_u$, $B'=A_{s_iu}$,
$A''=A_{vu}$, $B''=A_{s_jvu}$ with $\ell(vu)=\ell(v)+\ell(u)$ and
$\alpha=u^{-1}\alpha_i=(vu)^{-1}\alpha_j$. Then
$L(A',B')=\Hom_\Lambda(I_u,S_i)$ and $L(A'',B'')=\Hom_\Lambda
(I_{vu},S_j)=\Hom_\Lambda(I_u,\Hom_\Lambda(I_v,S_j))$.
Observing that $\Hom_\Lambda(I_v,S_j)$ has dimension-vector
$v^{-1}\alpha_j=\alpha_i$, we obtain $L(A',B')\cong L(A'',B'')$,
as announced.

Certainly, $L(A',B')\in\mathscr T(B)$. Let $A_0$ be a finite
biconvex subset contained in $A$. By Lemma~\ref{le:ApproxAdjBiconv},
there is a finite biconvex subset $A''\subseteq A$ that contains
$A_0$ and such that $B''=A''\sqcup\{\alpha\}$ is biconvex.
Then $L(A',B')\cong L(A'',B'')$ belongs to $\mathscr F(A'')$,
hence to $\mathscr F(A_0)$. Since $A_0$ was arbitrary, we get
$L(A',B')\in\mathscr F(A)$. It follows that
$\add L(A',B')\subseteq\mathscr F(A)\cap\mathscr T(B)$.

Conversely, let $T\in\mathscr F(A)\cap\mathscr T(B)$. By
Lemma~\ref{le:ApproxAdjBiconv}, there is a biconvex subset
$A''\subseteq A$ that contains
$\{\beta\in A\mid\height\beta\leq\dim T\}$ and such that
$B''=A''\sqcup\{\alpha\}$ is biconvex. By
Lemma~\ref{le:ContTPBiconv}, we have
$T\in\mathscr T(B'')$. Therefore $T$ belongs to
$\mathscr F(A'')\cap\mathscr T(B'')=\add L(A'',B'')$.

Setting $L(A,B)=L(A',B')$, we thus have
$\mathscr F(A)\cap\mathscr T(B)=\add L(A,B)$, as desired.

It remains to deal with the case where $\delta\in A$.
When $A$ and $B$ are both cofinite, the result follows from
Theorem~\ref{th:GLSFilt}~\ref{it:GLSFa}: if $A=A^{ws_i}$
and $B=A^w$, then $L(A,B)=I_w\otimes_\Lambda S_i$. The general
case follows by approximation, in a similar fashion as above.
\end{proof}

We now claim that for any biconvex subset $A$, the torsion
pair $(\mathscr T(A),\mathscr F(A))$ satisfies the openness
condition~(O) of section~\ref{ss:TorLambda}. When $A$ is
finite or cofinite, or more generally when $A$ is of the form
$A_\theta^{\min}$ or $A_\theta^{\max}$, this follows from
Proposition~\ref{pr:PolConstr}~\ref{it:PCc}. The general case
is deduced from this particular case with the help of
Lemma~\ref{le:ContTPBiconv}. For instance in the case
$\delta\notin A$, for each dimension-vector $\nu\in\mathbb NI$,
we can find a finite biconvex subset $A_0\subseteq A$ that
contains $\{\alpha\in A\mid\height\alpha\leq\height\nu\}$; then
$$\{T\in\Lambda(\nu)\mid T\in\mathscr T(A)\}=
\{T\in\Lambda(\nu)\mid T\in\mathscr T(A_0)\}$$
and
$$\{T\in\Lambda(\nu)\mid T\in\mathscr F(A)\}=
\{T\in\Lambda(\nu)\mid T\in\mathscr F(A_0)\};$$
and thus condition~(O) for
$(\mathscr T(A),\mathscr F(A))$ in dimension-vector $\nu$
follows from the condition~(O) for
$(\mathscr T(A_0),\mathscr F(A_0))$.

We may thus apply the results of section~\ref{ss:TorLambda}:
each biconvex subset $A$ defines subsets $\mathfrak T(A)$ and
$\mathfrak F(A)$ of $\mathfrak B$, and we have a bijection
$$\Xi(A):\mathfrak T(A)\times\mathfrak F(A)\to\mathfrak B.$$

More generally if $\mathbf A=(A_0,\ldots,A_\ell)$ is a nondecreasing
list of biconvex subsets, then we have nested torsion pairs
$$(\mathscr T(A_0),\mathscr F(A_0))\preccurlyeq\cdots
\preccurlyeq(\mathscr T(A_\ell),\mathscr F(A_\ell)),$$
whence a bijection
$$\Xi(\mathbf A):\mathfrak T(A_0)\times\prod_{k=1}^\ell
\Bigl(\mathfrak F(A_{k-1})\cap\mathfrak T(A_k)\Bigr)
\times\mathfrak F(A_\ell)\to\mathfrak B.$$

We define the character of a subset
$\mathfrak X\subseteq\mathfrak B$ as the formal series
$$P_{\mathfrak X}(t)=\sum_{\nu\in\mathbb NI}\Card\mathfrak
X(\nu)\;t^\nu$$
in $\mathbb Z[[\,(t^{\alpha_i})_{i\in I}\,]]$,
where $\mathfrak X(\nu)=\mathfrak X\cap\mathfrak B(\nu)$ is the
set of elements of weight $\nu$ in $\mathfrak X$. We denote the
multiplicity of a root $\alpha$ by $m_\alpha$; thus $m_\alpha=1$
if $\alpha$ is real and $m_\alpha=r$ if $\alpha$ is imaginary.

\begin{proposition}
\label{pr:CaracTPBiconv}
Let $A\subseteq B$ be two biconvex subsets. Then
\begin{equation}
\label{eq:ProofCTPB0}
P_{\mathfrak F(A)\cap\mathfrak T(B)}=\prod_{\alpha\in
B\setminus A}\frac1{(1-t^\alpha)^{m_\alpha}}.
\end{equation}
\end{proposition}
\begin{proof}
Let $A\subseteq B$ be two biconvex subsets. The bijection
$\Xi((A,B))$ leads to the equation
$$P_{\mathfrak T(A)}\times P_{\mathfrak F(A)\cap\mathfrak T(B)}
\times P_{\mathfrak F(B)}=P_{\mathfrak B}.$$
Since $\mathfrak B$ indexes a basis for $U(\mathfrak n_+)$, the
series $P_{\mathfrak B}$ is given by the Kostant partition function
\begin{equation}
\label{eq:KostPartFun}
P_{\mathfrak B}(t)=\prod_{\alpha\in\Phi_+}
\frac1{(1-t^\alpha)^{m_\alpha}}.
\end{equation}

For any subset $S\subseteq\Phi_+$ consider the formal power series
$$Q_S=\prod_{\alpha\in S}\frac1{(1-t^\alpha)^{m_\alpha}},$$
and notice that a monomial $t^\nu$ can only occur in $Q_S$ if $\nu$
belongs to the convex cone spanned by $S$. Our aim is to show that
$P_{\mathfrak F(A)\cap\mathfrak T(B)}=Q_{B\setminus A}$.

To simplify the notation, let us set $Q'=Q_A$,
$Q''=Q_{B\setminus A}$, $Q'''=Q_{\Phi_+\setminus B}$,
$R'=P_{\mathfrak T(A)}$, $R''=P_{\mathfrak F(A)\cap\mathfrak T(B)}$
and $R'''=P_{\mathfrak F(B)}$. We then have
$$Q'Q''Q'''=R'R''R'''$$
because each side is equal to $P_\mathfrak B$.

Given a formal power series
$$P=\sum_{\nu\in\mathbb NI}a_\nu t^\nu$$
and a nonnegative integer $n$, we define
$$P_n=\sum_{\substack{\nu\in\mathbb NI\\[1pt]\height\nu=n}}a_\nu t^\nu
\qquad\text{and}\qquad P_{\leq n}=P_0+P_1+\cdots+P_n.$$
We now show the equations
\begin{equation*}
Q'_{\leq n}=R'_{\leq n},\qquad
Q''_{\leq n}=R''_{\leq n},\qquad
Q'''_{\leq n}=R'''_{\leq n}
\tag{E$_n$}
\end{equation*}
by induction on $n$.

Certainly, (E$_0$) holds, because all the series involved have
constant term equal to $1$. Assume that (E$_{n-1}$) holds.
Then
$$Q'Q''Q'''-Q'_{\leq n-1}Q''_{\leq n-1}Q'''_{\leq n-1}=
R'R''R'''-R'_{\leq n-1}R''_{\leq n-1}R'''_{\leq n-1}.$$
After appropriate truncation, this leads to
\begin{equation}
\label{eq:ProofCTPB}
Q'_n+Q''_n+Q'''_n=R'_n+R''_n+R'''_n.
\end{equation}
Now consider a monomial $t^\nu$ that occurs in the right-hand
side of~\eqref{eq:ProofCTPB}. If it appears in $R'_n$, then
there is a $\Lambda$-module in $\mathscr T(A)$ of dimension-vector
$\nu$, and it follows from Lemma~\ref{le:DiffVertices} and
Proposition~\ref{pr:TorsionBiconv} that $\nu$ belong to the
convex cone spanned by $A$. By Lemma~\ref{le:BiconvPosRS},
$\nu$ does not belong to the convex cones spanned by $B\setminus A$
or by $\Phi_+\setminus B$, hence $t^\nu$ cannot occur in $Q''$
nor in $Q'''$. Therefore $t^\nu$ can only appear in $Q'_n$.
Likewise, we see that if $t^\nu$ appears in $R''_n$ (respectively,
$R'''_n$), then it can only appear in $Q''_n$ (respectively, $Q'''_n$).
We conclude that equation~\eqref{eq:ProofCTPB} splits into the
three equations $Q'_n=R'_n$, $Q''_n=R''_n$ and $Q'''_n=R'''_n$,
and thus that (E$_n$) holds.

Therefore (E$_n$) holds for each $n\in\mathbb N$, whence $Q''=R''$.
\end{proof}

In view of its later use, the particular case
$(A,B)=(A_\theta^{\min},A_\theta^{\max})$ with
$\theta\in\mathfrak t$ deserves a special mention.
\begin{corollary}
\label{co:CntIrrCompRF}
Let $\theta\in\mathfrak t$. Then
$$P_{\mathfrak R_\theta}=
\left(\prod_{\substack{\alpha\in\Phi_+^{\re}\\
\langle\theta,\alpha\rangle=0}}\frac1{1-t^\alpha}\right)
\left(\prod_{n\geq1}\frac1{1-t^{n\delta}}\right)^{\!r}.$$
\end{corollary}

\subsection{Simple regular modules}
\label{ss:SimpRegMod}
We come back to the description of the abelian categories $\mathscr R_F$.
Our aim in this section is to get information on their simple objects.

We begin with a general remark: let $F$ and $G$ be two faces
such that $F\subseteq\overline G$. If we pick $\theta\in F$ and
$\eta\in G$, then
$$A_\theta^{\min}\subseteq A_\eta^{\min}\subseteq
A_\eta^{\max}\subseteq A_\theta^{\max},$$
hence
$$(\mathscr I_\theta,\overline{\mathscr P}_\theta)\preccurlyeq
(\mathscr I_\eta,\overline{\mathscr P}_\eta)\preccurlyeq
(\overline{\mathscr I}_\eta,\mathscr P_\eta)\preccurlyeq
(\overline{\mathscr I}_\theta,\mathscr P_\theta);$$
in other words,
$$(\mathscr I_F,\overline{\mathscr P}_F)\preccurlyeq
(\mathscr I_G,\overline{\mathscr P}_G)\preccurlyeq
(\overline{\mathscr I}_G,\mathscr P_G)\preccurlyeq
(\overline{\mathscr I}_F,\mathscr P_F).$$
Therefore
\begin{equation}
\label{eq:CompCatSphFace}
\mathscr I_F\subseteq\mathscr I_G,\quad\mathscr P_F\subseteq\mathscr
P_G\quad\text{and}\quad\mathscr R_F\supseteq\mathscr R_G,
\end{equation}
and for any $\Lambda$-module $T$, we have a filtration
$0\subseteq T_\theta^{\min}\subseteq T_\eta^{\min}\subseteq
T_\eta^{\max}\subseteq T_\theta^{\max}\subseteq T$. The three
subquotients
$$T_\eta^{\min}/T_\theta^{\min}\in\overline{\mathscr P}_\theta
\cap\mathscr I_\eta,\quad
T_\eta^{\max}/T_\eta^{\min}\in\overline{\mathscr P}_\eta
\cap\overline{\mathscr I}_\eta\quad\text{and}\quad
T_\theta^{\max}/T_\eta^{\max}\in\mathscr P_\eta
\cap\overline{\mathscr I}_\theta$$
all belong to $\mathscr R_\theta$; in particular, a simple object
of $\mathscr R_F$ belongs either to $\mathscr I_G$, $\mathscr R_G$
or $\mathscr P_G$.

We denote by $\Irr\mathscr R_F$ the set of simple objects in
$\mathscr R_F$. Recall that two objects $T$ and $U$ in
$\Irr\mathscr R_F$ are said to be linked if there is a finite sequence
$T=X_0$, $X_1$, \dots, $X_n=U$ of objects in $\Irr\mathscr R_F$
such that $\Ext^1_\Lambda(X_{k-1},X_k)\neq0$ for each
$k\in\{1,\ldots,n\}$. (Note here that the groups $\Ext^1$ are the
same computed in $\Lambda\mmod$ and in $\mathscr R_F$, for the latter
is closed under extensions, and that $\Ext^1_\Lambda(X,Y)$ and
$\Ext^1_\Lambda(Y,X)$ are $K$-dual to each other.)
The linkage relation is an equivalence relation. Finally, recall
the map $\iota:\Phi^s\to\Phi_+^{\re}$ constructed in section
\ref{ss:SetupAffTyp} as a right inverse to the projection $\pi$.

\begin{theorem}
\label{th:DescSimpRF}
Let $F$ be a face of the spherical Weyl fan.
\begin{enumerate}
\item
\label{it:DSRFa}
If the dimension-vector of a simple object $T\in\mathscr R_F$ is
a multiple of $\delta$, then $\{T\}$ is a linkage class in
$\Irr\mathscr R_F$, and $T$ belongs to $\mathscr R_C$ for each
spherical Weyl chamber $C$ such that $F\subseteq\overline C$.
\item
\label{it:DSRFb}
The other objects in $\Irr\mathscr R_F$ are rigid. Their
dimension-vectors belong to $\iota(\Phi^s)$. Given
$\alpha\in\iota(\Phi^s)$, there is at most one simple object
in $\mathscr R_F$ of dimension-vector $\alpha$, up to isomorphism.
\end{enumerate}
\end{theorem}
\begin{proof}
Let $T$ be as in~\ref{it:DSRFa}. For any $X\in\Irr\mathscr R_F$
different from $T$, we have $\Hom_\Lambda(T,X)=\Hom_\Lambda(X,T)=0$,
by Schur's lemma applied in the category $\mathscr R_F$, so
$\Ext^1_\Lambda(T,X)=0$ by Crawley-Boevey's
formula~\eqref{eq:CrawleyBoeveyForm}. Therefore $\{T\}$ is
a linkage class. Let $C$ be a spherical Weyl chamber such that
$F\subseteq\overline C$. The assumption on $\dimvec T$ rules out
the possibility that $T\in\mathscr I_C$ or $\mathscr P_C$. We
conclude that $T\in\mathscr R_C$. Thus assertion~\ref{it:DSRFa}
is true.

We now turn to~\ref{it:DSRFb}. Let $T\in\Irr\mathscr R_F$, whose
dimension-vector $\alpha$ is not a multiple of $\delta$. Then
$(\alpha,\alpha)$ is a positive even integer. By Schur's lemma, the
endomorphism algebra of $T$ has dimension~$1$. Using Crawley-Boevey's
formula \eqref{eq:CrawleyBoeveyForm}, we then see that $(\alpha,\alpha)=2$
and $\dim\Ext^1_\Lambda(T,T)=0$. Thus $T$ is a rigid $\Lambda$-module,
and, by Proposition~5.10 in~\cite{Kac90}, $\alpha$ is a real root.

In this context, assume that $\alpha-\delta$ is a positive root.
Then, by Corollary~\ref{co:CntIrrCompRF}, $\Lambda(\alpha-\delta)$ has
an irreducible component whose general point belongs to $\mathscr R_F$.
In particular, there exists $X\in\mathscr R_F$ of dimension-vector
$\alpha-\delta$. But then $(\dimvec X,\dimvec T)=2$, and
\eqref{eq:CrawleyBoeveyForm} gives that $\Hom_\Lambda(T,X)$ or
$\Hom_\Lambda(X,T)$ is nonzero. Since $\dim T>\dim X$, this forbids
$T$ to be simple in $\mathscr R_F$, which contradicts our choice of $T$.
Therefore $\alpha-\delta\notin\Phi_+$, which means that
$\alpha\in\iota(\Phi^s)$.

Lastly, let $T'$ and $T''$ be two simple objects in $\mathscr R_F$
with the same dimension-vector $\alpha\in\iota(\Phi^s)$. Since
$(\dimvec T',\dimvec T'')=2$, \eqref{eq:CrawleyBoeveyForm} gives
that $\Hom_\Lambda(T',T'')$ or $\Hom_\Lambda(T'',T')$ is nonzero.
By Schur's lemma, $T'$ and $T''$ are isomorphic.
\end{proof}

One can show that the simple objects of $\mathscr R_F$ described in
Theorem~\ref{th:DescSimpRF}~\ref{it:DSRFa} always have
dimension-vector $\delta$ (see Corollary~\ref{co:CompDescSimpRF}).

\subsection{The type $\widetilde A_1$}
\label{ss:TypeTildeA1}
Let $\Pi$ be the completed preprojective algebra of the Kronecker
quiver, as in section~\ref{se:HallFunc}. We identify the Grothendieck
group of $\Pi\mmod$ with $\mathbb Z^2$ by writing the dimension-vector
of a $\Pi$-module $V$ as the pair $(\dim V_0,\dim V_1)$.
$$
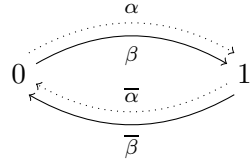
\begin{tikzpicture}
\node (0) at (0,0){$0$};
\node (1) at (3,0){$1$};
\node (0N) at (0,0.2){};
\node (0S) at (0,-0.2){};
\node (1N) at (3,0.2){};
\node (1S) at (3,-0.2){};
\draw[->,dotted] (0N) to[out=30,in=150]
  node[midway,above]{$\scriptstyle\alpha$} (1N);
\draw[->] (0) to[out=30,in=150]
  node[midway,below]{$\scriptstyle\beta$} (1);
\draw[->,dotted] (1) to[out=-150,in=-30]
  node[midway,above]{$\scriptstyle\overline\alpha$} (0);
\draw[->] (1S) to[out=-150,in=-30]
  node[midway,below]{$\scriptstyle\overline\beta$} (0S);
\end{tikzpicture}
$$

As in section~\ref{ss:HFIrrComp}, we denote by $\Pi(\mu)$
the nilpotent variety of type $\widetilde A_1$ for a given
dimension-vector $\mu=(\mu_0,\mu_1)$. A point in $\Pi(\mu)$
is a $4$-tuple of matrices
$T=(T_\alpha,T_\beta,T_{\overline\alpha},T_{\overline\beta})$
which satisfy the equations
$T_\alpha T_{\overline\alpha}+T_\beta T_{\overline\beta}=0$
and $T_{\overline\alpha} T_\alpha+T_{\overline\beta} T_\beta=0$
and the nilpotency condition.

We denote the root system of type $\widetilde A_1$ by $\Delta$,
so $\Delta_+=\Delta_+^{\re}\sqcup(\mathbb Z_{>0}\delta)$,
where $\delta=(1,1)$ is the primitive imaginary root and
$$\Delta_+^{\re}=\{(1,0)+n\delta,(0,1)+n\delta,\mid n\in\mathbb N\}$$
(Note that we use the same letter $\delta$ to denote the
primitive imaginary root in both $\Phi_+$ and~$\Delta_+$; this
will not lead to confusion.)

There are two opposite spherical chamber coweights, namely
$$\gamma'(\mu_0,\mu_1)=\mu_0-\mu_1\quad\text{and}\quad
\gamma''(\mu_0,\mu_1)=\mu_1-\mu_0.$$
The spherical Weyl fan has three faces, namely $\{0\}$ and the
two chambers $\mathbb R_{>0}\gamma'$ and $\mathbb R_{>0}\gamma''$.

Given $n\in\mathbb N$, we denote by $\Pi(n\delta)^\times$ the open
subset of all points $T\in\Pi(n\delta)$ such that the $n\times n$
matrix $T_\alpha$ is invertible. Obviously, the dimension-vector
$\mu$ of a submodule of $T$ satisfies $\mu_0\leq\mu_1$, so $T$ is
$\gamma'$-semistable, and is even $\gamma'$-stable if $n=1$.

We now describe $\Pi(n\delta)^\times$ with the help of an auxiliary
variety. Let
$$Z_n=\{(X,Y)\in M_n(K)^2\mid X\text{ is nilpotent, }XY=YX\},$$
where $M_n(K)$ denotes the algebra of $n\times n$ matrices
over $K$. The group $\GL_n(K)$ acts by conjugation on $Z_n$.
Given a partition $\lambda$ of size $n$, let us denote by
$\mathscr O_\lambda\subseteq M_n(K)$ the adjoint orbit of nilpotent
matrices of Jordan type $\lambda$. The following lemma is due to
I.~Frenkel and Savage (it is a particular case of
\cite{FrenkelSavage03}, Proposition~2.9).

\begin{lemma}
\label{le:RegTildeA1}
\begin{enumerate}
\item
\label{it:RTAa}
The map $f:\Pi(n\delta)^\times\to Z_n$ defined by
$f(T)=(T_{\overline\beta}T_\alpha,T_{\alpha}^{-1}T_\beta)$
is a principal $\GL_n(K)$-bundle.
\item
\label{it:RTAb}
The first projection $Z_n\to M_n(K)$ identifies $Z_n$ with the disjoint
union of the conormal bundles $T^*_{\mathscr O_\lambda}M_n(K)$, where
$\lambda$ runs over the set of all partitions of $n$.
\end{enumerate}
\end{lemma}
\begin{proof}
We begin with~\ref{it:RTAa}. Let $T\in\Pi(n\delta)^\times$.
By definition, $T_{\overline\beta}T_\alpha$ is nilpotent. In addition,
$$(T_{\overline\beta}T_\alpha)(T_{\alpha}^{-1}T_\beta)
=-T_{\overline\alpha}T_\alpha
=-T_{\alpha}^{-1}(T_\alpha T_{\overline\alpha})T_\alpha
=(T_{\alpha}^{-1}T_\beta)(T_{\overline\beta}T_\alpha),$$
thanks to the preprojective equations. So $f$ is well defined.
Now $\GL_n(K)$ acts on $\Pi(n\delta)^\times$ by
$$U\cdot(T_\alpha,T_\beta,T_{\overline\alpha},T_{\overline\beta})=
(UT_\alpha,UT_\beta,T_{\overline\alpha}U^{-1},T_{\overline\beta}U^{-1}).$$
This action is free, for $T_\alpha$ is invertible, and the orbits
of this action are the fibers of $f$. Thus~\ref{it:RTAa} holds true.

For a matrix $X\in M_n(K)$, the tangent space to its adjoint orbit
identifies with the subspace $\{[W,X]\mid W\in M_n(K)\}$. Under the
standard trace duality, the orthogonal of this subspace identifies
with the Lie algebra centralizer of $X$ because for any $Y\in M_n(K)$,
\begin{align*}
[X,Y]=0\ &\Longleftrightarrow\ \Bigl(\forall W\in M_n(K),\
\Tr\bigl(W[X,Y]\bigr)=0\Bigr)\\&\Longleftrightarrow\
\Bigl(\forall W\in M_n(K),\ \Tr\bigl([W,X]Y\bigr)=0\Bigr).
\end{align*}
Assertion~\ref{it:RTAb} follows.
\end{proof}

Thanks to Lemma~\ref{le:RegTildeA1}, we see that the irreducible
components of $\Pi(n\delta)^\times$ are of the form
$f^{-1}\bigl(\overline{T^*_{\mathscr O_\lambda}M_n(K)}\bigr)$. We
denote by $I(\lambda)$ the closure in $\Pi(n\delta)$ of this set;
since $\Pi(n\delta)^\times$ is open, this is an irreducible component
of $\Pi(n\delta)$, whose general point is $\gamma'$-semistable.
When $\lambda$ has just one nonzero part, here $n$, we write simply
$I(n)$ instead of $I((n))$.

\begin{figure}[ht]
\begin{center}
\begin{tikzpicture}
\node (0T) at (0,5){$0$};
\node (1T) at (2.4,4){$1$};
\node (0M) at (0,3){$0$};
\node (1M) at (2.4,2){$1$};
\node (0B) at (0,1){$0$};
\node (1B) at (2.4,0){$1$};
\coordinate (auxT) at (-.8,3);
\coordinate (auxM) at (3,2);
\coordinate (auxB) at (-.8,1);
\coordinate (auxE) at (4,3);
\draw[->,dotted] (0T) to[out=-15,in=145] (1T);
\draw[->] (0T) to[out=-35,in=165] (1T);
\draw[->,dotted] (1T) to[out=-165,in=35] (0M);
\draw[->] (1T) to[out=-145,in=15] (0M);
\draw[->,dotted] (0M) to[out=-15,in=145] (1M);
\draw[->] (0M) to[out=-35,in=165] (1M);
\draw[->,dotted] (1M) to[out=-165,in=35] (0B);
\draw[->] (1M) to[out=-145,in=15] (0B);
\draw[->,dotted] (0B) to[out=-15,in=145] (1B);
\draw[->] (0B) to[out=-35,in=165] (1B);
\draw[->] (0T) to[out=-130,in=100] (auxT) to[out=-80,in=180] (1M);
\draw[->] (0M) to[out=-130,in=100] (auxB) to[out=-80,in=180] (1B);
\draw[->] (0T) to[out=0,in=100] (auxE) to[out=-80,in=30] (1B);
\draw[->,dotted](1T) to[out=-50,in=80] (auxM) to[out=-100,in=0] (0B);
\draw (0T) ++(.6,-.6) node{$\scriptstyle x$};
\draw (0M) ++(.6,-.6) node{$\scriptstyle x$};
\draw (0B) ++(.6,-.6) node{$\scriptstyle x$};
\draw (1T) ++(-1.9,-.4) node{$\scriptstyle-x$};
\draw (1M) ++(-1.9,-.4) node{$\scriptstyle-x$};
\draw (0T) ++(-1.1,-1.6) node{$\scriptstyle x'$};
\draw (0M) ++(-1.1,-1.6) node{$\scriptstyle x'$};
\draw (1T) ++(1,-1.7) node{$\scriptstyle-x'$};
\draw (0T) ++(4.3,-1.8) node{$\scriptstyle x''$};
\end{tikzpicture}
\end{center}
\caption[]{The $\Pi$-module afforded by a general point $T$ of $I(3)$.
On the picture, a digit $0$ or $1$ represents a basis vector of the
corresponding degree, each dotted arrow indicates a nonzero entry
(equal to $1$ unless otherwise indicated) in the matrix that
represents $T_\alpha$ or $T_{\overline\alpha}$, and each plain arrow
indicates a nonzero entry in the matrix that represents $T_\beta$ or
$T_{\overline\beta}$. Things have been arranged so that
$T_{\overline\beta}T_\alpha=\begin{pmatrix}0&1&0\\
0&0&1\\0&0&0\end{pmatrix}$ is a Jordan block of size $3$ and
$T_\alpha^{-1}T_\beta=\begin{pmatrix}x&x'&x''\\0&x&x'\\
0&0&x\end{pmatrix}$ is a point in the commutant
of~$T_{\overline\beta}T_\alpha$. The parameters $x$, $x'$ and
$x''$ can take any value in~$K$.}
\label{fi:I(3)}
\end{figure}
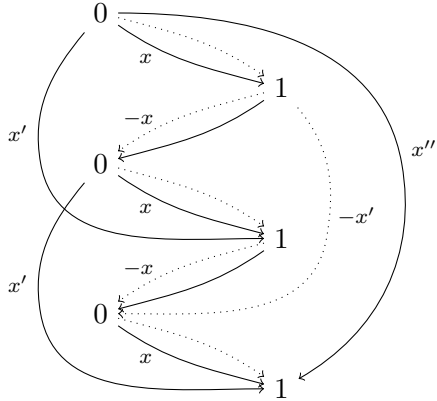

We denote the set of partitions by $\mathcal P$. Recalling the
well-known formula
$$\sum_{\lambda\in\mathcal P}t^{|\lambda|}=\prod_{n\geq1}
\frac1{1-t^n}$$
and applying Corollary~\ref{co:CntIrrCompRF} to $\Pi$ and
$\gamma'$, we see that $\{I(\lambda)\mid\lambda\in\mathcal P\}$
is the full set of elements in~$\mathfrak R_{\gamma'}$.

Recall Crawley-Boevey and Schr\"oer's theory of the canonical
decomposition explained in section~\ref{ss:CanonDec}.

\begin{proposition}
\label{pr:PptiesI}
Let $m$ and $n$ be positive integers and let $\lambda$ be a partition.
\begin{enumerate}
\item
\label{it:PIa}
Any general point $T$ in $I(n)$ is an indecomposable $\Pi$-module.
\item
\label{it:PIb}
We have $\hom_\Pi(I(m),I(n))=\ext^1_\Pi(I(m),I(n))=0$.
\item
\label{it:PIc}
Writing $\lambda=(\lambda_1,\ldots,\lambda_\ell)$, where $\ell$ is the
number of nonzero parts of $\lambda$, the canonical decomposition of
$I(\lambda)$ is
$$I(\lambda)=\overline{I(\lambda_1)\oplus\cdots\oplus I(\lambda_\ell)}.$$
\item
\label{it:PId}
The component $I(\lambda)^*$ is obtained from $I(\lambda)$
by applying the automorphism of $\Pi$ that exchanges the
idempotents $e_0$ and $e_1$, the arrows $\alpha$ and
$\overline\alpha$, and the arrows $\beta$ and $\overline\beta$.
\item
\label{it:PIe}
For any general point $T$ in $I(\lambda)$, we have
$\dim\End_\Pi(T)=|\lambda|$.
\end{enumerate}
\end{proposition}
\begin{proof}
Let $T$ be a general point in $I(n)$. It provides a $\Pi$-module
with dimension-vector $n\delta$. If this module were indecomposable,
then the matrix $T_{\overline\beta}T_\alpha$ would represent an
endomorphism of the vector space $K^n$ which stabilizes two
complementary subspaces. This is however impossible, since
$T_{\overline\beta}T_\alpha$ is a nilpotent matrix of Jordan
type~$(n)$. This shows~\ref{it:PIa}.

The group $G(\delta)$ acts on $I(1)$ with infinitely many orbits,
and a general point in $I(1)$ is a $\gamma'$-stable $\Pi$-module.
So if $S'$ and $S''$ are points in two different orbits, then they are
non-isomorphic simple objects in the category $\mathscr R_{\gamma'}$,
whence $\Hom_\Pi(S',S'')=\Hom_\Pi(S'',S')=0$ by Schur's lemma.
Now let $(T',T'')$ be general in $I(m)\times I(n)$. Then there is
a general point $(S',S'')\in I(1)^2$ such that $T'$ is an $m$-th
iterated extension of $S'$ and $T''$ is an $n$-th iterated extension
of $S''$. Therefore $\Hom_\Pi(T',T'')=\Hom_\Pi(T'',T')=0$, and
using Crawley-Boevey's formula~\eqref{eq:CrawleyBoeveyForm}, we
get $\Ext^1_\Pi(T',T'')=0$. Item~\ref{it:PIb} is~proved.

From~\ref{it:PIb} and from Crawley-Boevey and Schr\"oer's
theory, it follows that for any partition
$\lambda=(\lambda_1,\ldots,\lambda_\ell)$, the
closure $Z(\lambda)=\overline{I(\lambda_1)\oplus\cdots\oplus
I(\lambda_\ell)}$ is an element in $\mathfrak R_{\gamma'}$.
There thus exists a partition $\mu$ such that $Z(\lambda)=I(\mu)$.
Let $T$ be a general point in this irreducible component.
Looking at the Jordan type of the nilpotent matrix of
$T_{\overline\beta}T_\alpha$, we conclude that $\lambda=\mu$,
whence~\ref{it:PIc}.

In order to establish~\ref{it:PId}, it suffices by~\ref{it:PIc}
to consider the case $\lambda=(n)$. Certainly, $I(n)^*$ belongs
to $\mathfrak R_{\gamma''}$, by Remark~\ref{rk:PolDual}. Applying
to it the automorphism defined in the statement, we obtain again
an element of $\mathfrak R_{\gamma'}$, which we can write $I(\mu)$
for a certain partition $\mu$. Now the general point of this
component must be indecomposable, by~\ref{it:PIa}, so $\mu$ has
only one nonzero part. Looking at the dimension-vector, we
conclude that $\mu=n$, as desired. Item~\ref{it:PId} is~proved.

Lastly, let $T$ be a general point in $I(n)$. The commutator of
the matrix $X=T_{\overline\beta}T_\alpha$ in $M_n(K)$ is the algebra
of polynomials on $X$, because $X$ is a nilpotent matrix of Jordan
type $(n)$, so is a commutative algebra of dimension $n$. Therefore
the stabilizer of $f(T)$ in $\GL_n(K)$ coincides with the group of
invertible elements in this algebra; in particular, its dimension
is equal to $n$. The stabilizer of $T$ in $\Pi(n\delta)$ is
isomorphic to this group, and is the group of invertible elements
in the endomorphism algebra $\End_\Pi(T)$. We conclude that
$\dim\End_\Pi(T)=n$. This shows~\ref{it:PIe} in the particular
case where $\lambda$ has just one nonzero part; the general case
then follows from~\ref{it:PIb} and~\ref{it:PIc}.
\end{proof}

Assertion~\ref{it:PId} in this proposition implies
that $\{I(\lambda)^*\mid\lambda\in\mathcal P\}$
is the full set of elements in~$\mathfrak R_{\gamma''}$.

\subsection{Cores (proofs of Theorems~\ref{th:IntroCore1}
and~\ref{th:IntroCore2})}
\label{ss:Cores}
Let $\gamma\in\Gamma$ be a spherical chamber coweight. By
equation~\eqref{eq:CompCatSphFace}, we have
$$\bigcap_{\substack{\text{$F$ face}\\
\gamma\in\overline F}}\mathscr R_F
=\bigcap_{\substack{\text{$C$ Weyl chamber}\\
\gamma\in\overline C}}\mathscr R_C.$$
Objects in this intersection are called $\gamma$-cores.
They form an abelian, closed under extensions, subcategory of
$\Lambda\mmod$. The dimension-vector of a $\gamma$-core is a
multiple of $\delta$.

The following proposition provides an alternative definition
of $\gamma$-cores.
\begin{proposition}
\label{pr:AltDefCore}
A object in $\mathscr R_\gamma$ is a $\gamma$-core if and
only if the dimension-vectors of all its Jordan-H\"older
components are multiples of $\delta$.
\end{proposition}
\begin{proof}
By Theorem~\ref{th:DescSimpRF}~\ref{it:DSRFa}, a simple
$\mathscr R_\gamma$-module whose dimension-vector is a multiple
of $\delta$ is necessarily a $\gamma$-core. The sufficiency of
the condition then follows from the fact that the category
of $\gamma$-cores is closed under extensions.

Conversely, let $T\in\mathscr R_\gamma$ be such that the
dimension-vector of at least one Jordan-H\"older component of
$T$ is not a multiple of $\delta$. We want to show that $T$ is
not a $\gamma$-core. It suffices to show that $T$ has a direct
summand which is not a $\gamma$-core, so without loss of
generality, we can assume that $T$ is indecomposable. Then all
the Jordan-H\"older components of $T$, regarded as an object
of $\mathscr R_\gamma$, belong to the same linkage class; by
Theorem~\ref{th:DescSimpRF}, none of these components have a
dimension-vector multiple of $\delta$. Now consider a simple
object $X$ of $\mathscr R_\gamma$ contained in $T$. Since $X$
belongs to $\mathscr R_\gamma$, its dimension-vector $\dimvec X$
is in the kernel of $\gamma$, without being a multiple of $\delta$,
so there exists $\theta$ near $\gamma$, in a Weyl chamber, such
that $\langle\theta,\dimvec X\rangle>0$. It follows that $T$ is
not in $\mathscr R_\theta$, and therefore is not a $\gamma$-core.
This proves the necessity of the condition given in the lemma.
\end{proof}

Obviously, the $*$-dual of a $\gamma$-core is a $(-\gamma)$-core.
More interesting is the following compatibility between cores
and reflection functors.

\begin{proposition}
\label{pr:CorReflFunc}
Let $\gamma\in\Gamma$, let $i\in I$, and let $T$ be a $\gamma$-core.
If $\langle\gamma,\alpha_i\rangle<0$, then $T\in\mathscr T^{s_i}$
and $\Sigma_iT$ is a $s_i\gamma$-core. If $\langle\gamma,\alpha_i
\rangle>0$, then $T\in\mathscr F_{s_i}$ and $\Sigma_i^*T$ is a
$s_i\gamma$-core. If $\langle\gamma,\alpha_i\rangle=0$,
then $T\cong\Sigma_iT\cong\Sigma_i^*T$.
\end{proposition}
\begin{proof}
The first two claims immediately follow from
Theorem~\ref{th:SekiyaYamaura}, so let us consider the case
where $\langle\gamma,\alpha_i\rangle=0$. Then there exists
$\theta\in\mathfrak t$ close to $\gamma$ such that
$\langle\theta,\alpha_i\rangle>0$, which forbids $S_i$ to
appear as a submodule of $T$, and there exists
$\eta\in\mathfrak t$ close to $\gamma$ such that
$\langle\eta,\alpha_i\rangle<0$, which forbids $S_i$ to appear
as a quotient of $T$. Therefore the $i$-socle and the
$i$-head of $T$ are both trivial. Now recall the diagram
\eqref{eq:LocalDesc}. We rewrite it as
$$\xymatrix@C=3.6em{T_i\ar[r]^{T_{\out(i)}}&\widetilde
T_i\ar[r]^{T_{\iin(i)}}&T_i.}$$
This is a complex, $T_{\out(i)}$ is injective, and $T_{\iin(i)}$ is
surjective. We have $(\alpha_i,\dimvec T)=0$, for $\dimvec T$ is
a multiple of $\delta$, so the dimension of $\widetilde T_i$ is
twice the dimension of $T_i$. Our complex is therefore acyclic.
The isomorphisms $T\cong\Sigma_iT\cong\Sigma_i^*T$ then follow from
Proposition~\ref{pr:DescReflFunc}.
\end{proof}

We will produce $\gamma$-cores by means of Hall functors. According
to section~\ref{se:HallFunc}, we need to construct pairs $(S,R)$ of
$\Lambda$-modules that satisfy \eqref{eq:DatumSR}. We proceed as
follows.

We call a pair $(C,F)$ a flag, if it consists of a spherical Weyl chamber
$C$ and a facet $F$ contained in the closure of $C$. Such a pair
determines a spherical chamber coweight $\gamma_{C/F}$ by the equation
$C=F+\mathbb R_{>0}\,\gamma_{C/F}$. The spherical Weyl group $W_0$
acts on the set of flags.

Take a flag $(C,F)$ and pick $\theta\in F$. Then
$\Phi\cap(\ker\theta)$ is an affine root system of type
$\widetilde A_1$ and $\Phi^s\cap(\ker\theta)$ is a root system
of type $A_1$, consisting of two opposite roots $\pm\alpha$.
By Theorem~\ref{th:DescSimpRF}, the dimension-vectors of the
simple objects in $\mathscr R_F$ belong to
$\{\iota(\alpha),\iota(-\alpha)\}\cup\bigl(\mathbb Z_{>0}\delta\bigr)$.
This implies that an object in $\mathscr R_F$ of dimension-vector
$\iota(\pm\alpha)$ is necessarily simple.

Now Corollary~\ref{co:CntIrrCompRF} guarantees the existence of
a whole irreducible component of $\Lambda(\iota(\pm\alpha))$
whose general point is in $\mathscr R_F$. Therefore there are
objects in $\mathscr R_F$ of dimension-vector $\iota(\pm\alpha)$.
These objects are necessarily simple, and therefore unique up
to isomorphism, by Theorem~\ref{th:DescSimpRF}~\ref{it:DSRFb}.
We denote them by $S_{C,F}$ and $R_{C,F}$, the labels being
adjusted so that
\begin{equation}
\label{eq:NormSRCF}
\langle\gamma_{C/F},\dimvec S_{C,F}\rangle=1\quad\text{and}\quad
\langle\gamma_{C/F},\dimvec R_{C,F}\rangle=-1.
\end{equation}

\begin{other}{Examples}
\label{ex:ScfRcf}
\begin{enumerate}
\item
\label{it:SRa}
One fashion to produce $S_{C,F}$ and $R_{C,F}$ is to use
Proposition~\ref{pr:TPAdjBiconv}: if we set
$A=A_\theta^{\min}=\{\alpha\in\Phi_+\mid
\langle\theta,\alpha\rangle>0\}$, $B'=A\sqcup\{\iota(\alpha)\}$
and $B''=A\sqcup\{\iota(-\alpha)\}$, then both modules $L(A,B')$
and $L(A,B'')$ belong to $\mathscr R_\theta$ and have the
correct dimension-vectors to be $S_{C,F}$ and $R_{C,F}$.
Alternatively, we can consider $L(A',B)$ and $L(A'',B)$,
where $B=A_\theta^{\max}=\{\alpha\in\Phi_+\mid
\langle\theta,\alpha\rangle\geq0\}$, $A'=B\setminus\{\iota(\alpha)\}$
and $A''=B\setminus\{\iota(-\alpha)\}$.
\item
\label{it:SRb}
A facet $F$ separates two chambers, say $C'$ and $C''$. We then
have $(S_{C',F},R_{C',F})=(R_{C'',F},S_{C'',F})$.
\item
\label{it:SRc}
If $(C,F)$ is a flag, then $(-C,-F)$ is also a flag, and we have
$(S_{-C,-F},R_{-C,-F})=((R_{C,F})^*,(S_{C,F})^*)$.
\item
\label{it:SRd}
Let us fix an extending vertex $0\in I$, as explained at the end
of section~\ref{ss:SetupAffTyp}. Then each $i\in I_0$ provides
a flag $(C_0^s,F_{\{i\}})$, where
$$F_{\{i\}}=\{\theta \in\mathfrak t\mid\langle\theta,
\alpha_i\rangle=0,\ \langle\theta,\alpha_j\rangle>0
\text{ for all }j\in I_0\setminus\{i\}\}.$$
We have $\gamma_{C_0^s/F_{\{i\}}}=\varpi_i$ and
$\{\iota(\pm\alpha)\}=\{\alpha_i,\delta-\alpha_i\}$. The
dimension-vector of $S_{C_0^s,F_{\{i\}}}$ is $\alpha_i$, hence
$S_{C_0^s,F_{\{i\}}}=S_i$. We define $R_i=R_{C_0^s,F_{\{i\}}}$.
Since $R_i$ is in $\strut\mathscr R_{F_{\{i\}}}$, it cannot contain
any submodule isomorphic to $S_j$, with $j\in I_0\setminus\{i\}$,
and since it is simple, it cannot contain any submodule isomorphic
to $S_i$ either. Therefore $R_i$ has dimension-vector
$\delta-\alpha_i$ and its socle is $S_0$. By Lemma 2 (2) of
\cite{Crawley-Boevey00}, these two conditions characterize $R_i$.
\item
\label{it:SRe}
Keeping our extending vertex $0\in I$, let $(i,w)\in I_0\times W_0$
be such that $\ell(ws_i)>\ell(w)$ and consider
$(C,F)=(wC_0^s,wF_{\{i\}})$. By Theorem~\ref{th:TiltPIR}~\ref{it:TPIRa},
we have equivalences of categories
$$\xymatrix@C=6em{\mathscr R_{F_{\{i\}}}
\ar@<.6ex>[r]^{I_w\otimes_\Lambda?}&
\ar@<.6ex>[l]^{\Hom_\Lambda(I_w,?)}\mathscr R_{wF_{\{i\}}},}$$
which carries $(S_i,R_i)$ to $(S_{C,F},R_{C,F})$.
\end{enumerate}
\end{other}

\begin{lemma}
\label{le:SRBricks}
The modules $S_{C,F}$ and $R_{C,F}$ satisfy the
conditions~\eqref{eq:DatumSR}.
\end{lemma}

\begin{proof}
The modules $S_{C,F}$ and $R_{C,F}$ are simple objects in
$\mathscr R_F$, so by Schur's lemma, they are orthogonal bricks:
$$\End_\Lambda(S_{C,F})=\End_\Lambda(R_{C,F})=K,\qquad
\Hom_\Lambda(S_{C,F},R_{C,F})=\Hom_\Lambda(R_{C,F},S_{C,F})=0.$$
The remaining equations
$$\Ext^1_\Lambda(S_{C,F},S_{C_,F})=\Ext^1_\Lambda(R_{C,F},R_{C,F})=0,
\qquad\dim\Ext^1_\Lambda(S_{C,F},R_{C,F})=2$$
follow from Crawley-Boevey's formula~\eqref{eq:CrawleyBoeveyForm}.
\end{proof}

We can thus apply the results of section~\ref{se:HallFunc} to the
modules $S_{C,F}$ and $R_{C,F}$. We get a Hall functor
$\mathscr H_{C,F}:\Pi\mmod\to\Lambda\mmod$, which is an
equivalence of categories between $\Pi\mmod$ and the subcategory
$\langle S_{C,F},R_{C,F}\rangle$ of~$\mathscr R_F$.

We denote by $\mathfrak A=\bigsqcup_{\mathbf d\in\mathbb
N^2}\Irr\Pi(\mathbf d)$ the analog of the crystal $\mathfrak B$,
but for the Kronecker quiver. Proposition~\ref{pr:HFIrrComp}
claims that $\mathscr H_{C,F}$ induces a injection
$\mathfrak H_{C,F}:\mathfrak A\to\mathfrak B$, whose image
consists of those components whose general points lie in
$\langle S_{C,F},R_{C,F}\rangle$.

For $\lambda$ a partition, recall the irreducible component
$I(\lambda)\in\mathfrak A$ defined in section~\ref{ss:TypeTildeA1}.
We will see that the general point of $\mathfrak H_{C,F}(I(\lambda))$
is a $\gamma_{C,F}$-core, and moreover that
$\mathfrak H_{C,F}(I(\lambda))$ depends only on $\gamma_{C,F}$
and on $\lambda$.

To establish these results, our strategy is to first focus on the
particular case described in Example~\ref{ex:ScfRcf}~\ref{it:SRd}.
So now let us fix an extending vertex $0\in I$, let us take
$i\in I_0$, and let us set
$I(\varpi_i,1)=\mathfrak H_{C_0^s,F_{\{i\}}}(I(1))$.
As we saw earlier in this section, there is a unique irreducible
component of $\Lambda(\delta-\alpha_i)$ whose general point is in
$\mathscr R_{F_{\{i\}}}$. Further, this component is the closure
in $\Lambda(\delta-\alpha_i)$ of the set of all modules isomorphic
to $R_i$. Denoting it by $Z_i$, we thus have
$I(\varpi_i,1)=\tilde e_iZ_i$. The following proposition is
essentially a reformulation of \cite{Crawley-Boevey00}, Theorem~2.

\begin{proposition}
\label{pr:SimpCoreDom}
\begin{enumerate}
\item
\label{it:SCDa}
One has $\mathfrak R_{C_0^s}(\delta)=\{I(\varpi_i,1)\mid i\in I_0\}$.
\item
\label{it:SCDb}
For each $i\in I_0$, a general point in $I(\varpi_i,1)$ has socle
$S_0$ and head $S_i$.
\item
\label{it:SCDc}
For each $i\in I_0$, a general point in $I(\varpi_i,1)$ is a
$\varpi_i$-core. Conversely, any $\varpi_i$-core of
dimension-vector $\delta$ belongs to $I(\varpi_i,1)$.
\end{enumerate}
\end{proposition}
\begin{proof}
We first note that the socle of a module $T\in\mathscr R_{C_0^s}$
is necessarily concentrated at the vertex $0$. If in addition
$\dimvec T=\delta$, then $\soc T\cong S_0$.

Let $\Lambda_0=\{T\in\Lambda(\delta)\mid T\in\mathscr R_{C_0^s}\}$,
an open subset of $\Lambda(\delta)$ by
Proposition~\ref{pr:PolConstr}~\ref{it:PCc}.
Then $\mathfrak R_{C_0^s}(\delta)$ can be identified with
the set of irreducible components of $\Lambda_0$.

Let $T\in\Lambda_0$ and let $S_i$ in the head of $T$.
There is then a surjective morphism $T\to S_i$. Its kernel
has socle $S_0$ and dimension-vector $\delta-\alpha_i$, so is
isomorphic to $R_i$. We thus have a short exact sequence
$0\to R_i\to T\to S_i\to0$, which shows that $T$ belongs to
$\tilde e_iZ_i=I(\varpi_i,1)$.

Therefore $\Lambda_0$ is covered by the irreducible components
$I(\varpi_i,1)$. Now Corollary~\ref{co:CntIrrCompRF} asserts
that $\Lambda_0$ has $r$ irreducible components.
Assertion~\ref{it:SCDa} follows.

If a point in $\Lambda_0$ has two different simple modules $S_i$
and $S_j$ in its head, then it belongs to both $I(\varpi_i,1)$
and $I(\varpi_j,1)$, and therefore is not general. Therefore the
head of a general point in $I(\varpi_i,1)$ is concentrated at the
vertex $i$. Further, $R_i$ and $S_i$ are non-isomorphic simple objects
in $\mathscr R_{F_{\{i\}}}$, therefore $\Hom_\Lambda(R_i,S_i)=0$,
and so the $i$-head of $R_i$ is trivial. It follows that the
$i$-head of a general point $T$ in $I(\varpi_i,1)$ is
one-dimensional. All this shows~\ref{it:SCDb}.

Now fix $i\in I_0$. Let $T$ be a general point in $I(\varpi_i,1)$.
Then $\langle\varpi_i,\dimvec T\rangle=0$ and \ref{it:SCDb} implies
that $\langle\varpi_i,\dimvec X\rangle<0$ for any proper submodule
$X\subseteq T$. The module $T$ is thus a simple object in
$\mathscr R_{\varpi_i}$. By Theorem~\ref{th:DescSimpRF}~\ref{it:DSRFa},
$T$ is a $\varpi_i$-core.

Conversely, if $T$ is a $\varpi_i$-core of dimension-vector $\delta$,
then its socle must be $S_0$ and only $S_i$ can show up in the head
of $T$. The reasoning used to prove \ref{it:SCDa} applies anew, and
we conclude that $T$ belongs to $I(\varpi_i,1)$.
Assertion~\ref{it:SCDc} is proved.
\end{proof}

\begin{other}{Remark}
\label{rk:McKayCorr}
With the notation of the proof, a point $T$ in
$\Lambda_0\cap I(\varpi_i,1)$ is always the middle term
of a non-split extension $0\to R_i\to T\to S_i\to0$. Using
the fact that $\End_\Lambda(S_i)=\End_\Lambda(R_i)=K$, one shows
that the datum of the isomorphism class of $T$ is equivalent
to the datum of the class of the extension, up to scalar.
In other words, $G(\delta)$-orbits in $\Lambda_0\cap I(\varpi_i,1)$
are in bijection with points in the projective line
$\mathbb P(\Ext^1_\Lambda(S_i,R_i))$. In~\cite{Crawley-Boevey00},
Crawley-Boevey shows that in the moduli space
$\Lambda_0/\!/G(\delta)$ (the quotient here should be understood in
the GIT sense w.r.t.\ a character $\theta\in C_0^s$,
see \cite{King94}, Definition~2.1), these projective lines
intersect as displayed by the edges of the Dynkin diagram.
\end{other}

For $i\in I_0$ and $\lambda\in\mathcal P$, let us set
$I(\varpi_i,\lambda)=\mathfrak H_{C_0^s,F_{\{i\}}}(I(\lambda))$.
As before, we simplify the notation $I(\varpi_i,(n))$ into
$I(\varpi_i,n)$ when $\lambda$ has just one nonzero part,
here $n$. It follows from Theorem~\ref{th:HallFunc} and
Propositions~\ref{pr:HFIrrComp} and~\ref{pr:PptiesI}~\ref{it:PIc}
that if $\lambda=(\lambda_1,\ldots,\lambda_\ell)$, then
\begin{equation}
\label{eq:CanDecCoreDom}
I(\varpi_i,\lambda)=\overline{I(\varpi_i,\lambda_1)\oplus\cdots
\oplus I(\varpi_i,\lambda_\ell)}.
\end{equation}

\begin{lemma}
\label{le:CoreDom}
\begin{enumerate}
\item
\label{it:CDa}
Let $i\in I_0$. For any partition $\lambda$, a general point
in $I(\varpi_i,\lambda)$ is a $\varpi_i$-core.
\item
\label{it:CDb}
Let $(i,n)\in I_0\times\mathbb N$. A general point in
$I(\varpi_i,n)$ is indecomposable in $\Lambda\mmod$.
\item
\label{it:CDc}
For $(i,m)$ and $(j,n)$ in $I_0\times\mathbb N$, we have
$$\hom_\Lambda(I(\varpi_i,m),I(\varpi_j,n))=
\ext^1_\Lambda(I(\varpi_i,m),I(\varpi_j,n))=0.$$
In addition, if $(i,m)\neq(j,n)$, then
$I(\varpi_i,m)\neq I(\varpi_j,n)$.
\end{enumerate}
\end{lemma}
\begin{proof}
Let $n$ be a positive integer. Let us take a general point in
$I(\varpi_i,n)$. It is of the form $\mathscr H_{C_0^s,F_{\{i\}}}(T)$,
where $T$ is a general point in $I(n)$, and there is a general
point $S\in I(1)$ such that $T$ is an $n$-th iterated extension
of $S$. By Proposition~\ref{pr:SimpCoreDom}~\ref{it:SCDc},
$\mathscr H_{C_0^s,F_{\{i\}}}(S)$ is a $\varpi_i$-core.
Thus $\mathscr H_{C_0^s,F_{\{i\}}}(T)$, being an iterated extension
of $\varpi_i$-cores, is itself a $\varpi_i$-core. This shows
\ref{it:CDa} in the case where $\lambda$ has just one nonzero part;
the general case follows then from~\eqref{eq:CanDecCoreDom}.

The functor $\mathscr H_{C_0^s,F_{\{i\}}}$ preserves the
indecomposability, because it is fully faithful and because
direct summands are kernels of idempotent endomorphisms.
Assertion~\ref{it:CDb} is thus a consequence of
Proposition~\ref{pr:PptiesI}~\ref{it:PIa}.

Let $(i,j,m,n)\in(I_0)^2\times\mathbb N^2$, with $i\neq j$.
Let $(T',T'')$ be a general point in $I(\varpi_i,m)\times
I(\varpi_j,n)$. Then there is a general point $(S',S'')\in
I(\varpi_i,1)\times I(\varpi_j,1)$ such that $T'$ is an
$m$-th iterated extension of $S'$ and $T''$ is an $n$-iterated
extension of $S''$. By Proposition~\ref{pr:SimpCoreDom}~\ref{it:SCDb},
the head of $S'$ is concentrated at vertex $i$ and the head of $S''$
is concentrated at vertex $j$. Therefore the simple objects $S'$ and
$S''$ of $\mathscr R_{C_0^s}$ are not isomorphic. By Schur's lemma,
we then have $\Hom_\Lambda(S',S'')=\Hom_\Lambda(S'',S')=0$, whence
$\Hom_\Lambda(T',T'')=\Hom_\Lambda(T'',T')=0$. Using Crawley-Boevey's
formula~\eqref{eq:CrawleyBoeveyForm}, we get
$\Ext^1_\Lambda(T',T'')=0$. Moreover, this argument show that
the head of a general point in $I(\varpi_i,m)$ (respectively,
$I(\varpi_j,n)$) is concentrated at vertex $i$ (respectively, $j$),
whence $I(\varpi_i,m)\neq I(\varpi_j,n)$. All this shows~\ref{it:CDc}
in the case $i\neq j$. When $i=j$, item~\ref{it:CDc} follows from
Proposition~\ref{pr:PptiesI}~\ref{it:PIb}.
\end{proof}

Any spherical chamber coweight $\gamma$ can be written as $w\varpi_i$,
with $(i,w)\in I_0\times W_0$. We denote by $I(\gamma,\lambda)$ the
image of $I(\varpi_i,\lambda)$ by the functor $I_w\otimes ?$, viewed as
operating on irreducible components as in section~\ref{ss:TiltCrysOp}.
Proposition~\ref{pr:CorReflFunc} shows that $I(\gamma,\lambda)$ does
not depend on the choice of $w$, which justifies the notation; it also
shows that a general point of $I(\gamma,\lambda)$ is a $\gamma$-core.
Moreover,
$$I(\gamma,\lambda)=\overline{I(\gamma,\lambda_1)\oplus\cdots\oplus
I(\gamma,\lambda_\ell)}.$$

\begin{theorem}
\label{th:DescRCham}
Let $C$ be a spherical Weyl chamber. Then the map
$$(\lambda_\gamma)\mapsto
\overline{\bigoplus_{\gamma\in\Gamma\cap\overline C}
I(\gamma,\lambda_\gamma)}$$
is a bijection from $\mathcal P^{\Gamma\cap\overline C}$
onto $\mathfrak R_C$.
\end{theorem}
\begin{proof}
Using the reflection functors (specifically,
Theorem~\ref{th:TiltPIR}~\ref{it:TPIRa} and
Propositions~\ref{pr:CompEquiv} and \ref{pr:CrysBBEquiv}), we
can reduce to the case of the dominant chamber $C_0^s$.

Lemma~\ref{le:CoreDom} shows that each
$\overline{\bigoplus_{\gamma\in\Gamma\cap\overline{C_0^s}}
I(\gamma,\lambda_\gamma)}$ is an irreducible component whose general
point belongs to $\mathscr R_{C_0^s}$. The canonical decompositions
of these components show that they are pairwise distinct. The map
described in the statement is thus well defined and injective.
Its surjectivity follows from Corollary~\ref{co:CntIrrCompRF}.
\end{proof}

Theorem~\ref{th:DescRCham} provides the following intrinsic
characterization of $I(\gamma,n)$.
\begin{corollary}
\label{co:IndecCore}
For each $\gamma\in\Gamma$ and each $n\geq1$, $I(\gamma,n)$ is the
unique irreducible component of $\Lambda(n\delta)$ whose general
point is an indecomposable $\gamma$-core.
\end{corollary}

We have now established Theorem~\ref{th:IntroCore1}
and~\ref{th:IntroCore1}: Theorem~\ref{th:IntroCore2} is a special
case of Corollary~\ref{co:IndecCore}, and with this in hand
Theorem~\ref{th:IntroCore2} is exactly Theorem~\ref{th:DescRCham}.

The characterization in Corollary~\ref{co:IndecCore} also proves
that the components $I(\gamma,n)$, and thus also the components
$I(\gamma,\lambda)$, do not depend on the various choices made to
construct them, notably in the construction of the functors
$\mathscr H_{C_0^s,F_{\{i\}}}$ used to define $I(\varpi_i,\lambda)$.
It also implies that $I(-\gamma,n)=I(\gamma,n)^*$, whence
\begin{equation}
\label{eq:CorDual}
I(-\gamma,\lambda)=I(\gamma,\lambda)^*
\end{equation}
for any $(\gamma,\lambda)\in\Gamma\times\mathcal P$.

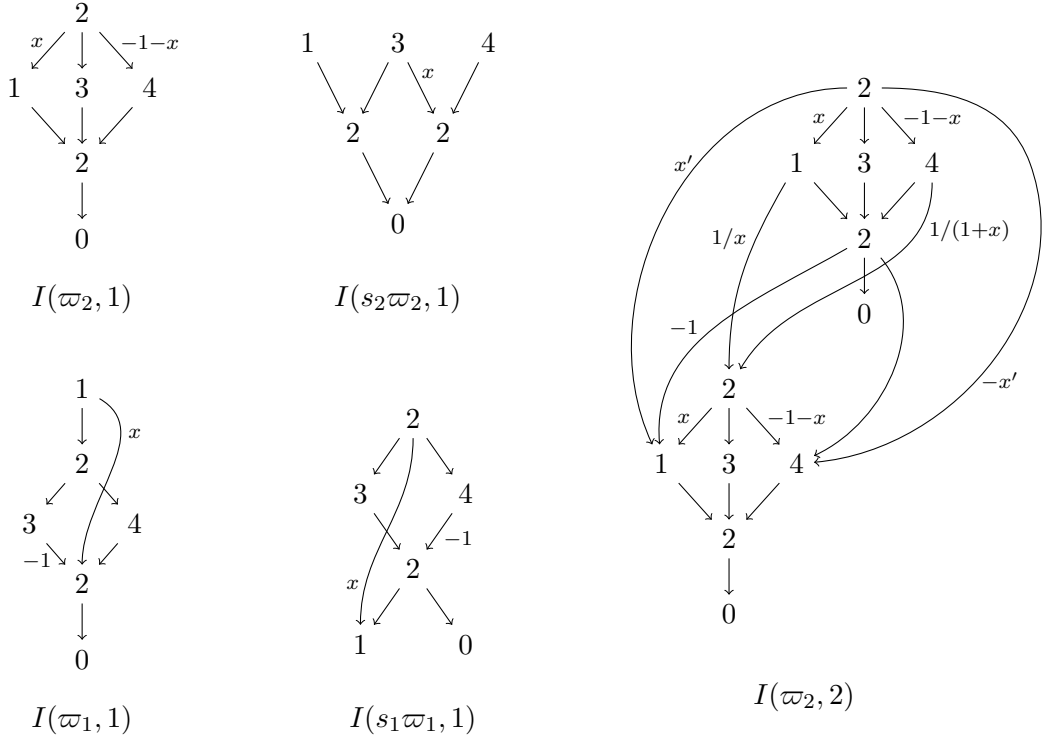
\begin{figure}[ht]
\begin{center}
\begin{tikzpicture}
\begin{scope}
\node (2a) at (0,3){$2$};
\node (1) at (-.9,2){$1$};
\node (3) at (0,2){$3$};
\node (4) at (.9,2){$4$};
\node (2b) at (0,1){$2$};
\node (0) at (0,0){$0$};
\draw[->] (2a) -- (1);
\draw[->] (2a) -- (3);
\draw[->] (2a) -- (4);
\draw[->] (1) -- (2b);
\draw[->] (3) -- (2b);
\draw[->] (4) -- (2b);
\draw[->] (2b) -- (0);
\draw (2a) ++(-.6,-.4) node{$\scriptstyle x$};
\draw (2a) ++(.9,-.4) node{$\scriptstyle-1-x$};
\draw (0,-.8) node{$I(\varpi_2,1)$};
\draw (2a) ++(0,.4) node{};
\end{scope}
\begin{scope}[xshift=4.2cm,yshift=.2cm]
\node (1) at (-1.2,2.4){$1$};
\node (3) at (0,2.4){$3$};
\node (4) at (1.2,2.4){$4$};
\node (2a) at (-.6,1.2){$2$};
\node (2b) at (.6,1.2){$2$};
\node (0) at (0,0){$0$};
\draw[->] (1) -- (2a);
\draw[->] (3) -- (2a);
\draw[->] (3) -- (2b);
\draw[->] (4) -- (2b);
\draw[->] (2a) -- (0);
\draw[->] (2b) -- (0);
\draw (3) ++(.4,-.4) node{$\scriptstyle x$};
\draw (0,-1) node{$I(s_2\varpi_2,1)$};
\end{scope}
\begin{scope}[yshift=-5.6cm]
\node(1) at (0,3.6){$1$};
\node(2a) at (0,2.6){$2$};
\node(3) at (-.7,1.8){$3$};
\node(4) at (.7,1.8){$4$};
\node(2b) at (0,1){$2$};
\node(0) at (0,0){$0$};
\draw[->] (1) -- (2a);
\draw[->] (1) to[out=-30,in=90] (2b);
\draw[->] (2a) -- (3);
\draw[->] (2a) -- (4);
\draw[->] (3) -- (2b);
\draw[->] (4) -- (2b);
\draw[->] (2b) -- (0);
\draw (1) ++(.7,-.6) node{$\scriptstyle x$};
\draw (3) ++(.1,-.5) node{$\scriptstyle-1$};
\draw (0,-.8) node{$I(\varpi_1,1)$};
\end{scope}
\begin{scope}[xshift=4.4cm,yshift=-5.4cm]
\node (2a) at (0,3){$2$};
\node (3) at (-.7,2){$3$};
\node (4) at (.7,2){$4$};
\node (2b) at (0,1){$2$};
\node (1) at (-.7,0){$1$};
\node (0) at (.7,0){$0$};
\draw[->] (2a) -- (3);
\draw[->] (2a) -- (4);
\draw[->] (3) -- (2b);
\draw[->] (4) -- (2b);
\draw[->] (2a) to[out=-90,in=90](1);
\draw[->] (2b) -- (1);
\draw[->] (2b) -- (0);
\draw (4) ++(-.1,-.6) node{$\scriptstyle-1$};
\draw (1) ++(-.1,.8) node{$\scriptstyle x$};
\draw (0,-1) node{$I(s_1\varpi_1,1)$};
\end{scope}
\begin{scope}[xshift=8.6cm,yshift=-5cm]
\node (2aT) at (1.8,7){$2$};
\node (1T) at (.9,6){$1$};
\node (3T) at (1.8,6){$3$};
\node (4T) at (2.7,6){$4$};
\node (2bT) at (1.8,5){$2$};
\node (0T) at (1.8,4){$0$};
\node (2aB) at (0,3){$2$};
\node (1B) at (-.9,2){$1$};
\node (3B) at (0,2){$3$};
\node (4B) at (.9,2){$4$};
\node (2bB) at (0,1){$2$};
\node (0B) at (0,0){$0$};
\coordinate(aux) at (4,6);
\draw[->] (2aT) -- (1T);
\draw[->] (2aT) -- (3T);
\draw[->] (2aT) -- (4T);
\draw[->] (1T) -- (2bT);
\draw[->] (3T) -- (2bT);
\draw[->] (4T) -- (2bT);
\draw[->] (2bT) -- (0T);
\draw[->] (2aB) -- (1B);
\draw[->] (2aB) -- (3B);
\draw[->] (2aB) -- (4B);
\draw[->] (1B) -- (2bB);
\draw[->] (3B) -- (2bB);
\draw[->] (4B) -- (2bB);
\draw[->] (2bB) -- (0B);
\draw[->] (2aT) to[out=180,in=115] (1B);
\draw[->] (2aT) to[out=0,in=110] (aux) to[out=-70,in=5] (4B);
\draw[->] (1T) to[out=-120,in=90] (2aB);
\draw[->] (4T) to[out=-90,in=60] (2aB);
\draw[->] (2bT) to[out=-150,in=95] (1B);
\draw[->] (2bT) to[out=-50,in=25] (4B);
\draw (2aT) ++(-.6,-.4) node{$\scriptstyle x$};
\draw (2aT) ++(.9,-.4) node{$\scriptstyle-1-x$};
\draw (2aB) ++(-.6,-.4) node{$\scriptstyle x$};
\draw (2aB) ++(.9,-.4) node{$\scriptstyle-1-x$};
\draw (2aT) ++(-2.4,-1) node{$\scriptstyle x'$};
\draw (2aT) ++(1.8,-3.9) node{$\scriptstyle-x'$};
\draw (2aB) ++(0,2) node{$\scriptstyle1/x$};
\draw (4T) ++(.5,-.9) node{$\scriptstyle1/(1+x)$};
\draw (2bT) ++(-2.4,-1.2) node{$\scriptstyle-1$};
\draw (1,-1.1) node{$I(\varpi_2,2)$};
\end{scope}
\end{tikzpicture}
\end{center}
\caption{Examples of cores in type $D_4$ (convention: the central
node of the Dynkin diagram is $2$). The isomorphism class of the
general point of $I(\gamma,n)$ depends on $n$ parameters; here
these parameters are called $x$ and $x'$.}
\label{fi:Cores}
\end{figure}

This independence on the choice of the Hall functors leads to
the following result.
\begin{corollary}
\label{co:CoresAgree}
For all flags $(C,F)$ and all partitions $\lambda$,
we have $\mathfrak H_{C,F}(I(\lambda))=I(\gamma_{C,F},\lambda)$.
\end{corollary}

\begin{proof}
Suppose that we are in the situation of Example~\ref{ex:ScfRcf}
\ref{it:SRe} with $(C,F)=(wC_0^s,wF_{\{i\}})$. Then
$\Hom_\Lambda(I_w,?)\circ\mathscr H_{C,F}$ is a Hall functor
built from the datum of $S_i$ and $R_i$, so it can play the
role of $\mathscr H_{C_0^s,F_{\{i\}}}$. Thus by the above observation,
the map on components defined by the functor $\strut\Hom_\Lambda(I_w,?)$
sends $\mathfrak H_{C,F}(I(\lambda))$ to $I(\varpi_i,\lambda)$,
for each partition $\lambda$. This immediately yields
$\mathfrak H_{C,F}(I(\lambda))=I(\gamma_{C,F},\lambda)$.

This analysis applies to exactly half of the flags.
The remaining flags are of the form $(C,F)=(-wC_0^s,-wF_{\{i\}})$;
their cases are deduced from the previous situation by duality.
\end{proof}

Combining the above result with Proposition~\ref{pr:PptiesI}~\ref{it:PId}
and Example~\ref{ex:ScfRcf}~\ref{it:SRb}, we obtain:
\begin{corollary}
\label{co:DualHallFunc}
Let $F$ be a facet and let $C'$ and $C''$ be the two Weyl
chambers that $F$ separates. For any partition $\lambda$, we have
$\mathfrak H_{C',F}(I(\lambda)^*)=I(\gamma_{C''/F},\lambda)$.
\end{corollary}

Our last corollary to Theorem~\ref{th:DescRCham} rounds off
Theorem~\ref{th:DescSimpRF}~\ref{it:DSRFa}.
\begin{corollary}
\label{co:CompDescSimpRF}
Let $F$ be a face of the spherical Weyl fan. Let $T$ be a simple
object of $\mathscr R_F$. If $\dimvec T$ is a multiple of $\delta$,
then in fact $\dimvec T=\delta$.
\end{corollary}
\begin{proof}
Using the reflection functors, we can reduce to the case where
$F$ is contained in the dominant spherical Weyl chamber $C_0^s$.

Let $n$ be a positive integer. Take an irreducible component
$Z\in\mathfrak R_{C_0^s}$ in dimension-vector $n\delta$.
This $Z$ can be written in the form provided by
Theorem~\ref{th:DescRCham}. Combining
Proposition~\ref{pr:PptiesI}~\ref{it:PIe} with
Lemma~\ref{le:CoreDom}, we obtain that $\dim\End_\Lambda(X)=n$
for any general point $X\in Z$. Since the function
$X\mapsto\dim\End_\Lambda(X)$ is upper semicontinuous on
$\Lambda(n\delta)$, we have $\dim\End_\Lambda(X)\geq n$ for
any $X\in Z$. Since $Z$ was arbitrary in $\mathfrak R_{C_0^s}$,
the latter inequality holds for any $X\in\mathscr R_{C_0^s}$
with dimension-vector $n\delta$.

Now take a simple object $T\in\mathscr R_F$ with
dimension-vector~$n\delta$. Then $T\in\mathscr R_{C_0^s}$,
by Theorem~\ref{th:DescSimpRF}~\ref{it:DSRFa}, and
$\dim\End_\Lambda(T)=1$, by Schur's lemma. The result just
established then implies that $1\geq n$.
\end{proof}

\begin{theorem}
\label{th:DescRFacet}
For any flag $(C,F)$, the map
$$(Z,(\lambda_\gamma))\mapsto\overline{\mathfrak H_{C,F}(Z)\oplus
\bigoplus_{\gamma\in\Gamma\cap\overline F}I(\gamma,\lambda_\gamma)}$$
is a bijection from $\mathfrak A\times\mathcal P^{\Gamma\cap\overline F}$
onto $\mathfrak R_F$.
\end{theorem}
\begin{proof}
Let $\gamma\in\Gamma\cap\overline F$ and let $n$ be a positive integer.
Let $T$ be a general point in $I(\gamma,n)$. Then $T$ is the $n$-th
iterated extension of a general point $X\in I(\gamma,1)$. By
Corollary~\ref{co:IndecCore}, $X$ is a $\gamma$-core of
dimension-vector $\delta$, so $X$ is simple in $\mathscr R_\gamma$
by Proposition~\ref{pr:AltDefCore}, so $X$ is simple in $\mathscr R_F$.
By Theorem~\ref{th:DescSimpRF}, $\{X\}$ is a linkage class in
$\Irr\mathscr R_F$, so $X$ is not linked to the modules $S_{C,F}$
and $R_{C,F}$. It follows that $\Ext^1_\Lambda(Y,T)=0$ for any
module $Y$ in the subcategory $\langle S_{C,F},R_{C,F}\rangle$, and
therefore that $\ext^1_\Lambda(\mathfrak H_{C,F}(Z),I(\gamma,n))=0$
for any $Z\in\mathfrak A$. In view of Crawley-Boevey and Schr\"oer's
theory, this implies that our map is well defined.

The argument above also shows that the general point of the
component $I(\gamma,n)$ does not belong to the category
$\langle S_{C,F},R_{C,F}\rangle$, so $I(\gamma,n)$ cannot occur
in the canonical decomposition of $\mathfrak H_{C,F}(Z)$. The
set of indecomposable irreducible components that arise from the
$\mathfrak H_{C,F}(Z)$ is thus disjoint from $\bigl\{I(\gamma,n)\bigm|
\gamma\in\Gamma\cap\overline F,\ n\geq1\bigr\}$. The uniqueness
of the canonical decomposition of an element in $\mathfrak R_F$
implies then that our map is injective.

Finally, we use a counting argument to prove the surjectivity of
our map. Pick $\theta\in F$. If $Z\in\mathfrak A$ has weight
$\mu=(\mu_0,\mu_1)$, then $\mathfrak H_{C,F}(Z)\in\mathfrak B$
has weight $\mathbf K(\mathscr H_{C,F})(\mu)=\mu_0\dimvec
S_{C,F}+\mu_1\dimvec R_{C,F}$. We here note that
$\mathbf K(\mathscr H_{C,F})$ maps the imaginary root
$\delta\in\Delta_+$ to the imaginary root in $\delta\in\Phi_+$;
thus the use of the same notation $\delta$ for both root systems
does not lead to any confusion. Plugging this information into
the character series for $\mathfrak A$, given by the analog for
$\Delta_+$ of \eqref{eq:KostPartFun}, and adding the contribution
of the components $I(\gamma,\lambda_\gamma)$, we get the character
of the image of our map:
$$\prod_{n\in\mathbb N}\left(\frac1{1-t^{\dimvec S_{C,F}+n\delta}}
\times\frac1{1-t^{\dimvec R_{C,F}+n\delta}}\times
\frac1{1-t^{(n+1)\delta}}\right)
\times\prod_{\gamma\in\Gamma\cap\overline F}\left(
\sum_{\lambda_\gamma\in\mathcal P}t^{|\lambda_\gamma|\delta}\right).$$
This is equal to
$$P_{\mathfrak R_\theta}=
\left(\prod_{\substack{\alpha\in\Phi_+^{\re}\\
\langle\theta,\alpha\rangle=0}}\frac1{1-t^\alpha}\right)
\left(\prod_{n\geq1}\frac1{1-t^{n\delta}}\right)^{\!r}$$
(see Corollary~\ref{co:CntIrrCompRF}), which ensures that our map
is surjective.
\end{proof}

Let $(Z,(\lambda_\gamma))\in\mathfrak A\times\mathcal
P^{\Gamma\cap\overline F}$ and denote by $\widetilde Z\in\mathfrak R_F$
its image under the map in Theorem~\ref{th:DescRFacet}. Write each
partition involved as $\lambda_\gamma=(\lambda_{\gamma,1},
\lambda_{\gamma,2},\ldots)$ and denote by $\ell(\lambda_\gamma)$
the number of nonzero parts of $\lambda_\gamma$. Pick general points
$X\in Z$ and $Y_{\gamma,p}\in I(\gamma,\lambda_{\gamma,p})$,
for each $\gamma\in\Gamma\cap\overline F$ and each
$p\in\{1,\ldots,\ell(\lambda_\gamma)\}$.
Then
\begin{equation}
\label{eq:GenPt2Face}
\widetilde X=\mathscr H_{C,F}(X)\oplus\left(\bigoplus_{\gamma\in
\Gamma\cap F}\left(\bigoplus_{1\leq p\leq\ell(\lambda_\gamma)}
Y_{\gamma,p}\right)\right)
\end{equation}
is a general point of $\widetilde Z$. By
Remark~\ref{rk:HNPol}~\ref{it:HNPb} applied to the category
$\mathscr R_F$, the HN polytope of $\widetilde X$ is the Minkowski
sum of the HN polytopes of its summands, that is, the Minkowski sum
of $\mathbf K(\mathscr H_{C,F})_{\mathbb R}(\Pol(X))$ and of
segments that join $0$ to each $[Y_{\gamma,p}]$. Further, if
$i:\mathscr R_F\subseteq\Lambda\mmod$ denotes the inclusion
functor, then $\mathbf K(i)_{\mathbb R}([Y_{\gamma,p}])=p\delta$, so
\begin{equation}
\label{eq:MinkSum2Face}
\mathbf K(i)_{\mathbb R}\bigl(\Pol\bigl(\widetilde X\bigr)\bigr)
=\mathbf K(i\circ\mathscr H_{C,F})_{\mathbb R}(\Pol(X))+
\sum_{\gamma\in\Gamma\cap\overline F}
\bigl[\;0,\;|\lambda_\gamma|\,\delta\;\bigr].
\end{equation}

Pick now $\theta\in F$ and $\Lambda_b\in\mathfrak B$. The bijection
$\Xi_\theta^{-1}$ maps $\Lambda_b$ to say $(\Lambda_{b'},\Lambda_{b''},
\Lambda_{b'''})\in\mathfrak I_F\times\mathfrak R_F\times\mathfrak
P_F$. If $T\in\Lambda_b$ is a general point, then
$T_\theta^{\max}/T_\theta^{\min}$ is a general point in
$\Lambda_{b''}$. Corollary~\ref{co:FacesHN} then says that
the HN polytope of $T_\theta^{\max}/T_\theta^{\min}$, viewed as
an object of $\mathscr R_F$, is the $2$-face of the HN polytope of
$T$ defined by $\theta$. Putting $\widetilde Z=\Lambda_{b''}$ in the
previous paragraph, we then see that this $2$-face can be written as
a Minkowski sum \eqref{eq:MinkSum2Face}. This sum involves
$|\Gamma\cap\overline F|=r-1$ one-dimensional polytopes, which
all point in the direction $\delta$, so we can intuitively regard
it as an MV polytope of type $\widetilde A_1\times\widetilde A_0^{r-1}$
if we equip it with adequate partitions.

We now need to look at these partitions. In particular, we need to
show that the partitions used in Theorem~\ref{th:DescRFacet}
(including those that decorate the polytope $\Pol(X)$) are the same
as the partitions provided by Theorem~\ref{th:DescRCham}, which
decorate $\Pol(T)$.

\subsection{The MV polytope of a component (proof of
Theorem~\ref{th:IntroImag2Face})}
\label{ss:MVPolComp}
For any spherical Weyl chamber $C$, we have bijections
$$\mathfrak I_C\times\mathfrak R_C\times\mathfrak P_C\to\mathfrak B
\quad\text{and}\quad\mathcal P^{\Gamma\cap\overline C}\to\mathfrak
R_C,$$
by Proposition~\ref{pr:NesTorIrrComp} and Theorem~\ref{th:DescRCham}.
Therefore each $\Lambda_b\in\mathfrak B$ provides a tuple of
partitions $(\lambda_\gamma)_{\gamma\in\Gamma\cap\overline C}$.
Concretely, for any $\eta\in C$ and any general point $T$ in
$\Lambda_b$, in the Krull-Schmidt decomposition of
$T_\eta^{\max}/T_\eta^{\min}$, there are as many $\gamma$-cores
of dimension-vector $n\delta$ as parts equal to $n$ in
$\lambda_\gamma$.

A priori, $\lambda_\gamma$ depends on $b$, on $\gamma$ and on $C$,
but in fact it only depends on $b$ and $\gamma$. Our aim now is to
show this fact. To this end we fix $b$ and we indicate the potential
dependence on $C$ by writing $\lambda_\gamma(C)$.

Let us study what happens around a facet $F$. We thus consider a flag
$(C',F)$, whence a functor $\mathscr H_{C',F}$ to which we apply
Theorem~\ref{th:DescRFacet}. Tracing $\Lambda_b\in\mathfrak B$ through
the bijections
$$\mathfrak I_F\times\mathfrak R_F\times\mathfrak P_F\to\mathfrak B
\quad\text{and}\quad\mathfrak A\times\mathcal P^{\Gamma\cap
\overline F}\to\mathfrak R_F$$
given by Proposition~\ref{pr:NesTorIrrComp} and
Theorem~\ref{th:DescRFacet}, we get $(Z,(\lambda_\gamma(F)))\in
\mathfrak A\times\mathcal P^{\Gamma\cap\overline F}$. Concretely,
for any $\theta\in F$ and any general point $T$ in $\Lambda_b$,
we can write
$$T_\theta^{\max}/T_\theta^{\min}\cong\mathscr H_{C,F}(X)\oplus
\left(\bigoplus_{\gamma\in\Gamma\cap F}\left(\bigoplus_{1\leq
p\leq\ell(\lambda_\gamma)}Y_{\gamma,p}\right)\right)$$
where the notations are as in equation~\eqref{eq:GenPt2Face}.
In addition, with the notation of section~\ref{ss:TypeTildeA1}, we
can look at the $\Pi$-module $Y_{\gamma'}^{\max}/Y_{\gamma'}^{\min}$;
this is the general point of an irreducible component $I(\lambda')$.
Likewise, the $\Pi$-module $Y_{\gamma''}^{\max}/Y_{\gamma''}^{\min}$
is the general point of an irreducible component $I(\lambda'')^*$.
Finally, let $C''$ be the other spherical Weyl chamber bordered by~$F$.

\begin{proposition}
\label{pr:RedImagTwoFace}
In the context above,
$$\lambda_{\gamma_{C'/F}}(C')=\lambda',\quad
\lambda_{\gamma_{C''/F}}(C'')=\lambda'',$$
and for each $\gamma\in\Gamma\cap\overline F$,
$$\lambda_\gamma(C')=\lambda_\gamma(C'')=\lambda_\gamma(F).$$
\end{proposition}
\begin{proof}
Take $m$ large enough and consider $\eta=m\theta+\gamma_{C'/F}$.

We first note that $\gamma'=\eta\circ\mathbf K(\mathscr H_{C',F})$.
In fact, $\mathbf K(\mathscr H_{C',F})$ maps a dimension-vector
$\mu$ in $\mathbf K(\Pi\mmod)$ to the element $\mu_0\dimvec S_{C',F}+
\mu_1\dimvec R_{C',F}$ of $\mathbf K(\Lambda\mmod)$.
Using~\eqref{eq:NormSRCF}, we then see that $\eta\circ\mathbf
K(\mathscr H_{C',F})$ maps $\mu$ to $\mu_0-\mu_1$, as does
$\gamma'$ (see section~\ref{ss:TypeTildeA1}).

Now $\eta$ is an element of $C'$, so the module
$T_\eta^{\max}/T_\eta^{\min}$ bears the information about the
partitions $\lambda_\gamma(C')$. Proposition~\ref{pr:SubfacPert}
explains how this module can be obtained from
$T_\theta^{\max}/T_\theta^{\min}$. In the process, the summands
$Y_{\gamma,p}$ stay unchanged, because they already belong to
$\mathscr R_{C'}$. By contrast, $\mathscr H_{C',F}(X)$
is truncated to its subquotient
$X'=\mathscr H_{C',F}(X_{\gamma'}^{\max}/X_{\gamma'}^{\min})$,
which is a general point of the component
$\mathfrak H_{C',F}(I(\lambda'))=I(\gamma_{C'/F},\lambda')$.
We thus have
$$\lambda_\gamma(C')=\begin{cases}
\lambda_\gamma(F)&\text{if }\gamma\in\Gamma\cap\overline F,\\
\lambda'&\text{if }\gamma=\gamma_{C'/F}.\end{cases}$$
The partitions $\lambda_\gamma(C'')$ are computed in a similar
fashion, using Corollary~\ref{co:DualHallFunc} at the last step.
\end{proof}

Proposition~\ref{pr:RedImagTwoFace} asserts in particular that
$\lambda_\gamma(C')=\lambda_\gamma(C'')$ if $C'$ and $C''$
are two adjacent spherical Weyl chambers whose closures contain
$\gamma$. This implies that $\lambda_\gamma(C)$ is independent
of $C$, assuming of course that $\gamma\in\overline C$.

To an irreducible component $\Lambda_b\in\mathfrak B$, we may thus
associate a family of partitions $(\lambda_\gamma)\in\mathcal P^\Gamma$.
In addition, by Proposition~\ref{pr:PolConstr}~\ref{it:PCb},
$\Lambda_b$ contains a dense open subset on which the map
$T\mapsto\Pol(T)$ is constant. As in the introduction, we denote by
$\widetilde\Pol(b)$ the datum of this constant value $\Pol(T)$ and
of the family of partitions $(\lambda_\gamma)$.

We have shown that $\widetilde\Pol(b)$ belongs to the set
$\mathcal{MV}$ of decorated lattice polytopes defined in
section~\ref{ss:TwoFaces}:
\begin{itemize}
\item
The polytope is GGMS, by Corollary~\ref{co:SuppFunTitsFan},
so its normal fan is coarser than~$\mathscr W$.
\item
Given a spherical Weyl chamber $C$, the partitions $(\lambda_\gamma)$
for $\gamma\in\overline C$ describe the Krull-Schmidt decomposition of
$T_\theta^{\max}/T_\theta^{\min}$, where $T$ is general in $\Lambda_b$
and $\theta\in C$. Therefore
$$\left(\sum_{\gamma\in\Gamma\cap\overline C}|\lambda_\gamma|\right)
\delta=\dimvec T_\theta^{\max}/T_\theta^{\min}.$$
\item
Each $2$-face of finite type is constrained by the relations in
Propositions~\ref{pr:LusPLBij1} and~\ref{pr:LusPLBij2}, and each
$2$-face of affine type is an MV polytope of type
$\widetilde A_1\times\widetilde A_0^{r-1}$, as explained at the
end of section~\ref{ss:Cores} and in the discussion following
Proposition~\ref{pr:RedImagTwoFace}.
\end{itemize}

The last point in the list above proves Theorem~\ref{th:IntroImag2Face}.

\subsection{Lusztig data (proof of Theorem~\ref{th:IntroLusDat})}
\label{ss:LusDatComp}
We are now ready to prove Theorem~\ref{th:IntroLusDat}. However,
first we need to precisely define the map $\Omega_{\preceq}$
(i.e.\ the map taking an element of $B(-\infty)$ to its Lusztig
data) in affine type.

Fix a convex order $\preccurlyeq$. Given a weight
$\nu\in\mathbb NI$, we set $E_\nu=\bigl\{\alpha\in\Phi_+^{\re}
\sqcup\{\delta\}\bigm|\height\alpha\leq\height\nu\bigr\}$.
Enumerate the elements in $E_\nu$ in decreasing order:
$\beta_1\succ\beta_2\succ\cdots\succ\beta_\ell$. For $1\leq k\leq\ell$,
set $A_k=\{\alpha\in\Phi_+\mid\alpha\succ\beta_k\}$ and
$B_k=\{\alpha\in\Phi_+\mid\alpha\succcurlyeq\beta_k\}$. These
biconvex subsets provide a nested family of torsion pairs
(here we write only the torsion classes):
$$\{0\}\subseteq\mathscr T(A_1)\subseteq\mathscr T(B_1)\subseteq
\mathscr T(A_2)\subseteq\cdots\subseteq\mathscr T(A_\ell)\subseteq
\mathscr T(B_\ell)\subseteq\Lambda\mmod.$$

On a $\Lambda$-module $T$, this induces a filtration
\begin{equation}
\label{eq:FilConvOrd}
0\subseteq T_1\subseteq\overline T_1\subseteq T_2\subseteq
\cdots\subseteq T_\ell\subseteq\overline T_\ell\subseteq T,
\end{equation}
with $\overline T_k/T_k\in\mathscr F(A_k)\cap\mathscr T(B_k)$.
If $\dimvec T=\nu$, then by Lemma~\ref{le:ContTPBiconv} the only
jumps in the filtration~\eqref{eq:FilConvOrd} occur between a $T_k$
and the corresponding $\overline T_k$. If moreover $T$ is a general
point in an irreducible component $Z\in\mathfrak B(\nu)$, then by
Proposition~\ref{pr:NesTorIrrComp} each subquotient $\overline T_k/T_k$
will be a general point in an irreducible component
$Z_k\in\mathfrak F(A_k)\cap\mathfrak T(B_k)$, and furthermore the map
\begin{equation}
\label{eq:BijLuszData}
\mathfrak B(\nu)\to\left\{(Z_1,\ldots,Z_\ell)\in
\prod_{k=1}^\ell\Bigl(\mathfrak F(A_k)\cap\mathfrak T(B_k)\Bigr)
\left|\;\sum_{k=1}^\ell\wt Z_k=\nu\right.\right\}
\end{equation}
is a bijection.

By Proposition~\ref{pr:TPAdjBiconv}, if $\beta_k$ is real, then
$\mathscr F(A_k)\cap\mathscr T(B_k)=\add L(A_k,B_k)$, where
$L(A_k,B_k)$ is a rigid $\Lambda$-module of dimension-vector
$\beta_k$. Therefore $\mathfrak F(A_k)\cap\mathfrak T(B_k)$ is
in one-to-one correspondence with $\mathbb N$: to a natural
number $n$ corresponds the closure in $\Lambda(n\beta_k)$ of
the orbit that represents $L(A_k,B_k)^{\oplus n}$.

On the other hand, if $\beta_k=\delta$, then we can find a spherical
Weyl chamber $C$ such that $(A_k,B_k)=(A_\theta^{\min},A_\theta^{\max})$
for $\theta\in C$ (Lemma~\ref{le:DeltaAdj}~\ref{it:DAa}),
and then $\mathscr F(A_k)\cap\mathscr T(B_k)=\mathscr R_C$
(Proposition~\ref{pr:AThetaTor}), whence
$\mathfrak F(A_k)\cap\mathfrak T(B_k)=\mathfrak R_C$. Further,
Theorem~\ref{th:DescRCham} provides a bijection between
$\mathfrak R_C$ and~$\mathcal P^{\Gamma\cap\overline C}$.

The bijection \eqref{eq:BijLuszData} can thus be rewritten as
$$\mathfrak B(\nu)\to\left\{((n_\beta),(\lambda_\gamma))\in\mathbb
N^{E_\nu\cap\Phi_+^{\re}}\times\mathcal P^{\Gamma\cap\overline C}
\left|\;\sum_{\beta\in E_\nu\cap\Phi_+^{\re}}n_\beta\beta+
\left(\sum_{\gamma\in\Gamma\cap\overline C}|\lambda_\gamma|\right)
\delta=\nu\right.\right\}.$$
Letting $\nu$ run over $\mathbb NI$ and assembling the resulting
maps, we get a bijection
$$\Omega_\preccurlyeq:\mathfrak B\to\mathbb N^{(\Phi_+^{\re})}\times
\mathcal P^{\Gamma\cap\overline C},$$
proving Theorem~\ref{th:IntroLusDat}.

One may here observe that the map $\Omega_{\mathbf i}$
constructed in section~\ref{ss:TiltCrysOp} gives the beginning of
$\Omega_\preccurlyeq$ when the smallest roots for $\preccurlyeq$
are, in order
$$\alpha_{i_1},\ s_{i_1}\alpha_{i_2},\ s_{i_1}s_{i_2}\alpha_{i_3},\
\ldots$$
In addition, Remark~\ref{rk:RemLusParTop}~\ref{it:RLPTb} and
Propositions~\ref{pr:LusPLBij1} and~\ref{pr:LusPLBij2} justify
referring to the bijection $\Omega_{\mathbf i}$ as the partial
Lusztig datum in direction $\mathbf i$. We are thus led to regard
$\Omega_\preccurlyeq$ as the Lusztig datum in direction~$\preccurlyeq$.
A further justification of this terminology is the fact that the
components of $\Omega_\preccurlyeq(b)$ can be read as the lengths
and the decorations of the edges on the path in the $1$-skeleton
of $\widetilde\Pol(b)$ defined by $\preccurlyeq$.

\subsection{Proof of Theorem~\ref{th:IntroMV}}
\label{ss:PfThIntroMV}
\trivlist
\item[\hskip\labelsep{\itshape Proof of the injectivity of
$\widetilde\Pol$.}]\upshape
Let us choose a convex order $\preccurlyeq\,$, and let $C$
be the spherical Weyl chamber whose positive root system is
$\pi(\{\beta\in\Phi_+^{\re}\mid\beta\succ\delta\})$.

An element $\widetilde P\in\mathcal{MV}$ is the datum
of a GGMS polytope $P$ and of a family of partitions
$(\lambda_\gamma)\in\mathcal P^\Gamma$, subject to certain
conditions. To $P$ and $\preccurlyeq$, the construction at the end
of section~\ref{ss:GGMSPol} associates a finitely supported family
of non-negative integers $(n_\alpha)$, indexed by $\Phi_+^{\re}$;
specifically, $n_\alpha$ is the length of the edge parallel to
$\alpha$ on the path in the $1$-skeleton of $P$ defined by
$\preccurlyeq$. Adding to this datum the partitions
$\lambda_\gamma$ for all $\gamma$ in $\Gamma\cap\overline C$, we
get an element $\widehat\Omega_\preccurlyeq(\widetilde P)\in\mathbb
N^{(\Phi_+^{\re})}\times\mathcal P^{\Gamma\cap\overline C}$.

Proposition~\ref{pr:TorsionBiconv} shows that the resulting map
$\widehat\Omega_\preccurlyeq$ is compatible with the map
$\Omega_\preccurlyeq$ constructed in the previous section, in the
sense that the diagram
\begin{equation}
\label{eq:CompLusDat}
\raisebox{25pt}{\xymatrix@C=1em@R=3em{\mathfrak
B\ar[rr]^{\widetilde\Pol}\ar[dr]_(.4){\Omega_\preccurlyeq}
&&\mathcal{MV}\ar[dl]^(.4){\widehat\Omega_\preccurlyeq}\\&
\mathbb N^{(\Phi_+^{\re})}\times\mathcal P^{\Gamma\cap\overline C}&}}
\end{equation}
commutes. The injectivity of $\widetilde\Pol$ then follows
from the injectivity of $\Omega_\preccurlyeq$.
\nobreak\noindent$\square$
\endtrivlist

\trivlist
\item[\hskip\labelsep{\itshape Proof of the surjectivity of
$\widetilde\Pol$.}]\upshape
As explained in Example~\ref{ex:ExConvexOrder}~\ref{it:ECOb},
a sufficiently general $\theta\in(\mathbb RI)^*$ defines a convex
order. The precise condition is that $\theta$ avoids the countably
many hyperplanes
$$H_{\alpha,\beta}=\{\theta\in(\mathbb RI)^*\mid\theta(\alpha)/
\height(\alpha)=\theta(\beta)/\height(\beta)\},$$
where $(\alpha,\beta)$ runs over pairs of non-proportional roots.
We denote the collection of all these hyperplanes by $\mathscr X$.
For simplicity, we will denote by $\widehat{\Omega}_\theta$
the map $\widehat\Omega_\preccurlyeq$ relative to the order
$\preccurlyeq$ defined by such a $\theta$.

Consider an MV polytope $\widetilde P=(P,(\lambda_\gamma))$.
Every vertex (respectively, every imaginary edge) of $P$ is a
face $P_\theta$, and $\theta\in(\mathbb RI)^*$ can certainly be
chosen outside all the hyperplanes in the collection~$\mathscr X$.
The position of this vertex (respectively, the partitions
relative to this imaginary edge) is then determined
by $\widehat\Omega_\theta(\widetilde P)$. Thus $\widetilde P$
is determined by the datum of
$\widehat\Omega_\theta(\widetilde P)$ for all possible~$\theta$.

Now fix an element $\eta_0\in(\mathbb RI)^*$ outside the
hyperplanes in the collection $\mathscr X$. We will show that
for any general $\eta_1\in(\mathbb RI)^*$ and any MV polytope
$\widetilde P$, the datum $\widehat\Omega_{\eta_1}(\widetilde P)$
can unambiguously be determined from
$\widehat\Omega_{\eta_0}(\widetilde P)$ by a
rule which depends on $\widetilde P$ only through its weight.
Together with the argument in the previous paragraph, this
establishes the injectivity of the map $\widehat\Omega_{\eta_0}$.
Chasing in the diagram~\eqref{eq:CompLusDat} relative to the order
$\preccurlyeq$ defined by $\eta_0$, we can then deduce the desired
result from the surjectivity of the map $\Omega_\preccurlyeq$.

Let $\widetilde P=(P,(\lambda_\gamma))$ be an MV polytope,
let $\nu$ be the weight of $\widetilde P$, and as in
section~\ref{ss:LusDatComp}, let
$E_\nu=\bigl\{\alpha\in\Phi_+^{\re}\sqcup\{\delta\}\bigm|
\height\alpha\leq\height\nu\bigr\}$. The datum
$\widehat\Omega_\preccurlyeq(\widetilde P)$
depends on the choice of the convex order $\preccurlyeq$
only through the restriction of $\preccurlyeq$ to $E_\nu$.
Therefore, as a function of $\theta$, the datum
$\widehat\Omega_\theta(\widetilde P)$ only changes when
$\theta$ crosses an hyperplane $H\in\mathscr X_\nu$, where
$\mathscr X_\nu$ denotes the collection of all hyperplanes
$H_{\alpha,\beta}$ with $(\alpha,\beta)\in(E_\nu)^2$.

Pick $\eta_1\in(\mathbb RI)^*$ outside all the hyperplanes of
the collection $\mathscr X$, and let $(\eta_t)$ be a piecewise
linear path in $(\mathbb RI)^*$ that connects $\eta_0$ to $\eta_1$,
general enough so that $\eta_t$ lies on a hyperplane of the
collection $\mathscr X_\nu$ at only finitely many times
$0<t_1<\cdots<t_m<1$, and never lies at once on two such
hyperplanes. Set $t_0=0$ and $t_{m+1}=1$. For a fixed
$j\in\{0,\ldots,m\}$, the data $\widehat\Omega_{\eta_t}(\widetilde P)$
for $t\in(t_j,t_{j+1})$ are all one and the same. We denote
this datum by $\widehat\Omega_j$ and observe that
$\widehat\Omega_{\eta_0}(\widetilde P)=\widehat\Omega_0$
and $\widehat\Omega_{\eta_1}(\widetilde P)=\widehat\Omega_m$.
For $j\in\{1,\ldots,m\}$, the data $\widehat\Omega_{j-1}$ and
$\widehat\Omega_j$ record the lengths and the partitions along
two paths in the $1$-skeleton of $P$ which coincide almost
everywhere, the only difference between these paths being that they
may traverse the opposite sides of the $2$-face $P_\theta$, where
$\theta=\eta_{t_j}$. But $\widetilde P$ is an MV polytope, so its
$2$-faces are constrained: by Propositions~\ref{pr:LusPLBij1}
and~\ref{pr:LusPLBij2} in the case where $P_\theta$ is
a $2$-face of finite type, and by Theorem~\ref{th:DescRFacet} and
Proposition~\ref{pr:RedImagTwoFace} in the case where $P_\theta$ is
a $2$-face of affine type, the datum $\widehat\Omega_{j-1}$ determines
$\widehat\Omega_j$. Thus $\widehat\Omega_{\eta_1}(\widetilde P)$
is determined by $\widehat\Omega_{\eta_0}(\widetilde P)$, as
announced.
\nobreak\noindent$\square$
\endtrivlist

\appendix{Restriction to the tame quiver}
The path algebra $KQ$ of an acyclic quiver $Q$ can be seen as a
subalgebra of the completed preprojective algebra $\Lambda_Q$ of
$Q$. In our present situation of an extended Dynkin diagram, $Q$
is tame, so its representation theory is very well understood,
thanks to the work of Dlab and Ringel. In this appendix, we discuss
our constructions in terms of the representation theory of $Q$.

We begin with a refinement to Theorem~\ref{th:DescSimpRF}
in the case where $F$ is a minimal face, that is, the ray generated
by a spherical chamber coweight $\gamma$. For $(\mu,\nu)\in(\mathbb ZI)^2$,
we write $\mu\geq\nu$ if $\mu-\nu\in\mathbb NI$.

\begin{proposition}
\label{pr:SumDimVecSimp}
Let $F$ be a ray of the spherical Weyl fan and let $L$
be a linkage class of simple objects in $\mathscr R_F$.
Then $\sum_{S\in L}\dimvec S\leq\delta$.
\end{proposition}
\begin{proof}
Theorem~\ref{th:DescSimpRF} distinguishes two kinds of simple objects,
described in its assertions~\ref{it:DSRFa} and~\ref{it:DSRFb}.
For the objects of the first kind, the desired property is proved
in Corollary~\ref{co:CompDescSimpRF}. In the sequel of this proof,
we consider the other case, when the dimension-vectors of objects
in $L$ belong to $\iota(\Phi^s)$.

We choose an extending vertex in $I$ and we set $I_0=I\setminus\{0\}$.
The spherical root system $\Phi^s$ is then endowed with a basis,
namely $\{\pi(\alpha_i)\mid i\in I_0\}$, whence a positive system
$\Phi^s_+$, and a dominant spherical Weyl chamber $C_0^s$.
As in section~\ref{ss:SetupAffTyp}, we identify the spherical Weyl
group $W_0$ with the parabolic subgroup
$\langle s_i\mid i\in I_0\rangle$ of $W$.

First consider the case where $F\subseteq\overline{C_0^s}$. Then
$F$ is spanned by a certain spherical fundamental coweight $\varpi_i$,
with~$i\in I_0$.

Given a connected component $J$ of $I_0\setminus\{i\}$, we can look
at the root system $\Phi_J=\Phi\cap\mathbb ZJ$. This root system is
finite and irreducible and comes with a natural basis, so it has a
largest root $\widetilde\alpha_J$. By the dual statement of
\cite{Crawley-Boevey00}, Lemma~2 (2), there is a unique
$\Lambda$-module with socle $S_0$ and dimension-vector
$\delta-\widetilde\alpha_J$; we denote it by $R_J$.

We claim that the head of $R_J$ is isomorphic to $S_i$. In fact,
$S_0$ does not occur in the head of $R_J$; otherwise, $S_0$ would
be a direct factor of $R_J$ (because it occurs in the socle of $R_J$
and its Jordan-H\"older multiplicity in $R_J$ is one), which is ruled
out by the fact that $R_J$ is indecomposable (the socle of $R_J$ is
simple) of dimension-vector $\neq\alpha_0$. If $S_j$ occurs in
the head of $R_J$, then we can produce a $\Lambda$-module $X$
with socle $S_0$ and dimension-vector $\dimvec R_J-\alpha_j$; the
latter is then a root, by \cite{Crawley-Boevey00}, Lemma~2 (1),
and therefore $\widetilde\alpha_J+\alpha_j$ is a root; this forces
$j=i$. If $S_i$ occurred twice in the head of $R_J$,
then $\dimvec R_J-2\alpha_i$ would be a root, so
$\widetilde\alpha_J+2\alpha_i$ would be a root, which is
impossible because the root system $\Phi^s$ is simply laced.

Next we claim that $R_J$ is a simple object in $\mathscr R_F$. To
prove that, it suffices to show that $R_J$ is $\varpi_i$-stable, in
other words, that $\langle\varpi_i,\dimvec R_J\rangle=0$ and
that $\langle\varpi_i,\dimvec(R_J/X)\rangle>0$ for all proper
submodules $X$ of $R_J$. The first equation comes from the fact
that $\Phi_J$ is contained in $\ker\varpi_i$. To prove the second
equation, we observe that $S_0$ is not a Jordan-H\"older
component of $R_J/X$ (because $X$ contains the unique copy of
$S_0$ in $R_J$), so the simple components in $R_J/X$ are $S_j$
with $j\in I_0$, and $S_i$ appears at least once in $R_J/X$.

In addition, the modules $S_j$, for $j\in I_0\setminus\{i\}$, are
also simple objects in $\mathscr R_F$. We now claim that the modules
$R_J$ and $S_j$ are all the simple objects in $\mathscr R_F$
whose dimension-vectors are in $\iota(\Phi^s)$.

Indeed, let $T\in\Irr\mathscr R_F$ such that $\dimvec T\in\iota(\Phi^s_+)$.
The vector space $T_0$ attached to the extending vertex is thus zero.
The vector space $T_i$ attached to $i$ is then also zero, for
$\varpi_i(\dimvec T)=0$. Thus $T$ is an iterated extension of the
modules $S_j$ with $j\in I_0\setminus\{i\}$. Since all these modules
belong to $\mathscr R_F$, we conclude that $T$ is one of these $S_j$.

On the other hand, let $T\in\Irr\mathscr R_F$ such that
$\beta=\dimvec T$ belongs to $\iota(\Phi^s_-)$. The simplicity of
$T$ forbids any $S_j$ with $j\in I_0\setminus\{i\}$ to appear in
the socle or in the head of $T$. Further, $S_i$ cannot appear in
the socle of $T$, because $T\in\mathscr R_F$ and $S_i\in\mathscr I_F$.
Therefore $\soc T=S_0$. This condition and $\beta$ completely
determine $T$, by \cite{Crawley-Boevey00}, Lemma~2 (2). With the
notations of \cite{BaumannKamnitzer12}, section~3 (see also
Example~\ref{ex:CompBKAIRT}), we have $T\cong N(\beta-\omega_0)$.
(To apply \cite{BaumannKamnitzer12}, Theorem~3.1, we note the
existence of $w\in W_0$ such that $\beta=w\alpha_0$, which
implies $\beta-\omega_0=-ws_0\omega_0$.) Equation~(3.1)
in~\cite{BaumannKamnitzer12} (or the proof of Lemma~2 in
\cite{Crawley-Boevey00}) then says that
$$\dim\hd_jT=\max(0,(\beta,\alpha_j)-\langle\omega_0,\alpha_j\rangle),$$
and we have seen that the left-hand side is zero for
$j\in I_0\setminus\{i\}$. Let us write $\beta=\iota(-\alpha)$,
with $\alpha\in\Phi^s_+$. Then $(\alpha,\alpha_j)\geq0$ for each
$j\in I_0\setminus\{i\}$. Now $\langle\varpi_i,\beta\rangle=0$,
so the support of $\alpha$ (a subset of $I_0$) avoids the node $i$,
and therefore $\alpha\in\Phi^s_J$ for a certain connected component
$J$ of $I_0\setminus\{i\}$. We then conclude that
$\alpha=\widetilde\alpha_J$, and therefore that $T=R_J$.

We now claim that the simple objects linked to $R_J$ are the $S_j$
with $j\in J$. By \cite{Bourbaki68}, chapitre~6, \S1,
n\textordmasculine\;6, Proposition~19, there is a sequence
$\beta_1$, \dots, $\beta_n$ of elements in $\{\alpha_j\mid
j\in J\}$ such that $\beta_1+\cdots+\beta_k$ is a root for each $k$
and $\beta_1+\cdots+\beta_n=\widetilde\alpha_J$. For
$k\in\{1,\ldots,n\}$, let $T_k$ be the simple $\Lambda$-module
with dimension-vector $\beta_k$. Let $N_{n+1}=R_J$, and for
$1<k\leq n$, let $N_k$ be the $\Lambda$-module with socle $S_0$
and dimension-vector
$$\delta-(\beta_1+\cdots+\beta_{k-1})=
\delta-\widetilde\alpha_J+(\beta_n+\cdots+\beta_k).$$
The existence and uniqueness of $N_k$ follows from Lemma~2 (2) in
\cite{Crawley-Boevey00}. Inspecting the proof of this result, we
see that $N_k$ is the middle term of a non-trivial extension of
$N_{k+1}$ by $T_k$, and thus $T_k$ is certainly linked to at least
one of the simple components of $N_{k+1}$. This conclusion also
holds for $k=1$, since
$\Ext^1_\Lambda(N_2,T_1)\neq0$ by Crawley-Boevey's formula
\eqref{eq:CrawleyBoeveyForm}. Thus, all the $T_k$ with
$1\leq k\leq n$ are linked to $R_J$. Since each $S_j$ with
$j\in J$ shows up among these modules $T_k$, we conclude that
all the $S_j$ with $j\in J$ are linked to $R_J$.

Now take two different connected components $J$ and $K$ of
$I_0\setminus\{i\}$. By Schur's lemma,
$$\Hom_\Lambda(S_j,S_k)=\Hom_\Lambda(S_j,R_K)=\Hom_\Lambda(R_J,S_k)=
\Hom_\Lambda(R_J,R_K)=0$$
for any $j\in J$ and $k\in K$. An easy calculation based on
Crawley-Boevey's formula~\eqref{eq:CrawleyBoeveyForm} then shows
that the $\Ext^1$ between $S_j$ or $R_J$ and $S_k$ or $R_K$ is
zero. So $J$ and $K$ give rise to different linkage classes.

To each connected component of $I_0\setminus\{i\}$ corresponds
thus a linkage class, formed by $R_J$ and the $S_j$ with $j\in J$.
The sum of the dimension-vectors of these objects is
$\delta-\widetilde\alpha_J+\sum_{j\in J}\alpha_j$, which is
smaller than or equal to $\delta$, with equality if and only if
$J$ is of type $A$.

At this point, we have established the desired property in the
case where $F\subseteq\overline{C_0^s}$. It remains to handle
the case of a general face $F$. Let $w\in W_0$ of minimal length
such that $w^{-1}F\subseteq\overline{C_0}$. Then $F$ is spanned by
a certain spherical chamber coweight $w\,\varpi_i$, where $i\in I_0$
and $w\in W_0$ is $I_0\setminus\{i\}$-reduced on the right. By
Theorem~\ref{th:SekiyaYamaura} (and Corollary~\ref{co:NwJred}),
we have equivalences of categories
$$\xymatrix@C=6.3em{\mathscr R_{\varpi_i}
\ar@<.6ex>[r]^(.55){I_w\otimes_\Lambda?}&
\ar@<.6ex>[l]^(.45){\Hom_\Lambda(I_w,?)}\mathscr R_F}.$$
We can then transfer to $\mathscr R_F$ the information
obtained above for $\mathscr R_{\varpi_i}$.

What is at stake is the fact that the sum of the
dimension-vectors of the simple objects in a linkage class
is at most $\delta$. As regards $\mathscr R_{\varpi_i}$, this
sum has the form $\delta-\beta_J$, where $J$ is a connected
component of $I_0\setminus\{i\}$ and
$\beta_J=\widetilde\alpha_J-\sum_{j\in J}\alpha_j$. Checking
the classification of root systems, we observe that $\beta_J$
is a root; using Corollary~\ref{co:NwJred}, we see that
$\beta_J\notin N_{w^{-1}}$. Therefore $w\beta_J$ is a positive
root and $w(\delta-\beta_J)=\delta-w\beta_J$ is less than or
equal to $\delta$, as desired.
\end{proof}

Recall the framework of section~\ref{ss:BasicDef}. The graph
$(I,E)$ can be endowed with several orientations $\Omega$
(we only consider acyclic orientations). The datum of $\Omega$
gives a quiver $Q$, whence an Euler form $\langle\,,\;\rangle_Q$
on $\mathbb ZI$, defined as
$$\langle\lambda,\mu\rangle_Q=\sum_{i\in I}\lambda_i\mu_i
-\sum_{a\in\Omega}\lambda_{s(a)}\mu_{t(a)}.$$
The symmetric bilinear form $(\,,\,):\mathbb ZI\times\mathbb
ZI\to\mathbb Z$ is then the symmetrization of $\langle\,,\;\rangle_Q$.

The imaginary root $\delta$ belongs to the kernel of $(\,,\,)$,
so $\langle\delta,?\rangle_Q$ induces a linear form on
$\mathfrak t^*$, in other words, an element
$\gamma_\Omega\in\mathfrak t$. For example, in type
$\widetilde A_1$, there are two orientations
$$\,\Omega':\,\xymatrix@C=2.5em{0\ar@/^/[r]^\alpha\ar@/_/[r]_\beta&1}
\quad\text{and}\quad\Omega'':\,\xymatrix@C=2.5em{0&1.
\ar@/_/[l]_{\overline\alpha}\ar@/^/[l]^{\overline\beta}}$$
The corresponding linear forms are the spherical chamber coweights
$\gamma_{\Omega'}=\gamma'$ and $\gamma_{\Omega''}=\gamma''$
of section~\ref{ss:TypeTildeA1}.

\begin{proposition}
\label{pr:AssChambCow}
The map $\Omega\mapsto\gamma_\Omega$ is an injection from the
set of all non-cyclic orientations of $(I,E)$ into $\Gamma$. In
type $\widetilde A$, this map is bijective.
\end{proposition}
\trivlist
\item[\hskip\labelsep{\itshape Sketch of proof.}]\upshape
We begin by studying the type $\widetilde A_n$. The vertices of
the graph $(I,E)$ are numbered consecutively from $0$ to $n$ and
we have $\delta=\alpha_0+\cdots+\alpha_n$. Let $\Omega$ be a
non-cyclic orientation. The number
$a_i=\gamma_\Omega(\pi(\alpha_i))=\langle\delta,\alpha_i\rangle_Q$
is equal to $1$, $0$ or $-1$ depending on the number of arrows
that terminate at $i$. When $i$ cyclically runs over $\{0,\ldots,n\}$,
$a_i$ alternatively takes the values $1$ and $-1$, with zeros
interspersed between these values. The sum of all the $a_i$ is
$\langle\delta,\delta\rangle_Q=0$, and the $a_i$ cannot be all
zero because $\Omega$ is not cyclic. There is thus a unique
sequence of values $b_i\in\{0,1\}$, for $i\in\{0,\ldots,n+1\}$,
such that $a_i=b_i-b_{i+1}$ and $b_0=b_{n+1}$. Now $\mathfrak t^*$
has a standard realization as a hyperplane of the vector space
with basis $\{\varepsilon_i\mid1\leq i\leq n+1\}$. In this context,
the $b_i$ are the coordinates of $\gamma_\Omega$ in the basis
$(\varepsilon_i^*)$ dual to $(\varepsilon_i)$. In this basis
$(\varepsilon_i^*)$, the spherical chamber coweights are the sums
$\varepsilon_{i_1}^*+\cdots+\varepsilon_{i_k}^*$ with
$1\leq k\leq n$ and $1\leq i_1<\cdots<i_k\leq n+1$. Certainly
$\gamma_\Omega$ matches this pattern, hence is a spherical chamber
coweight. We leave to the reader the routine verifications needed
to show the announced bijectivity.

Consider the following orientation in type $\widetilde D_n$.
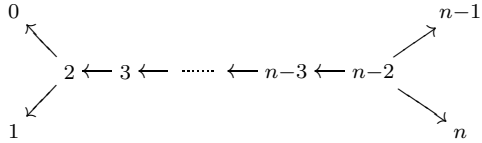
$$\xymatrix@=1em{\scriptstyle0&&&&&&&\scriptstyle n-1\\
&\scriptstyle2\ar[ul]\ar[dl]&\scriptstyle3\ar[l]&\ar[l]&
\ar@{.}[l]&\scriptstyle n-3\ar[l]&\scriptstyle n-2\ar[l]
\ar[ur]\ar[dr]&\\\scriptstyle1&&&&&&&\scriptstyle n}$$
A direct calculation shows that the associated coweight is
$(s_{n-1}s_n)(s_1\cdots s_{n-2})\varpi_{n-2}$. Since the graph
is a tree, any orientation $\Omega$ can be obtained from this one
by a sequence of reflections at sources. Noting that $\delta$ is
$W$-invariant and using \cite{BaumannKamnitzer12}, Lemma~7.2,
we deduce that the coweight $\gamma_\Omega$ is $W$-conjugate to
$\varpi_{n-2}$. We omit the proof of the injectivity of the map
$\Omega\mapsto\gamma_\Omega$, for it requires lengthy (but direct)
calculations in coordinates.

The types $\widetilde E_6$, $\widetilde E_7$ and $\widetilde E_8$
are dealt with similarly. One finds that the coweights $\gamma_\Omega$
are all $W$-conjugate to the fundamental coweight corresponding to
the branching point in the Dynkin diagram. For these exceptional
types, we used a computer to check the injectivity of the
map~$\Omega\mapsto\gamma_\Omega$.
\nobreak\noindent$\square$
\endtrivlist

Let us fix an orientation $\Omega$, whence a quiver $Q$. In
\cite{Ringel98a}, Ringel describes $\Lambda\mmod$ in terms of the
category $KQ\mmod$ of finite dimensional representations of $Q$.
More precisely, let $\tau$ denote the Auslander-Reiten translation
in $KQ\mmod$ and let $M$ be a $KQ$-module. Then the structures of
$\Lambda$-module on $M$ that extend the given structure of
$KQ$-module are in natural bijection with a certain subspace
$\mathcal N^{\tau^-}(M)$ of nilpotent elements in
$\mathcal O^{\tau^-}(M)=\Hom_{KQ}(\tau^{-1}M,M)$.

Recall that indecomposable $KQ$-modules are classified into
preprojective, preinjective and regular types. Every $KQ$-module $M$
can then be written as $M=I\oplus R\oplus P$, where $I$, $R$ and $P$
are the submodules of $M$ obtained by gathering all direct summands
in a Krull-Schmidt decomposition of $M$ which are respectively
preinjective, regular, and preprojective. (The subspaces $R$ and
$P$ depend on the choice of the Krull-Schmidt decomposition, but
$I$ and $I\oplus R$ do not; see~\cite{Crawley-Boevey92}, \S7,
Remark.)

\begin{proposition}
\label{pr:PRIQuiv}
Let $T$ be a $\Lambda$-module and decompose the restriction
$M=T|_Q$ as a sum $M=I\oplus R\oplus P$ as above. Then
$T_{\gamma_\Omega}^{\min}=I$ and
$T_{\gamma_\Omega}^{\max}=I\oplus R$ (as subspaces of $T$).
\end{proposition}
\begin{proof}
According to Ringel's construction, the datum of $T$ is equivalent
to the datum of $M$ and of $f\in\mathcal N^{\tau^-}(M)$.
By~\cite{Crawley-Boevey92}, \S7, Lemma~3, $f$ must map $\tau^{-1}I$
to $I$ and $\tau^{-1}(R\oplus I)$ to $R\oplus I\strut$, so $I$
and $R\oplus I$ are $\Lambda$-submodules of $T$.

Any nonzero quotient $KQ$-module of $I$ is preinjective (otherwise,
we would have a nonzero map from a preinjective to a preprojective
or a regular module), hence has a positive defect (\S7, Lemma~2
in~\cite{Crawley-Boevey92}). A fortiori, a nonzero quotient
$\Lambda$-module $X$ of $I$ satisfies
$\langle\gamma_\Omega,\dimvec X\rangle>0$. Therefore the
$\Lambda$-module $I$ belongs to
$\mathscr I_{\gamma_\Omega}$.

Similarly, a nonzero $KQ$-submodule of $P\oplus R$ cannot have
a preinjective direct summand, so has a nonpositive defect.
Thus a nonzero $\Lambda$-submodule $Y$ of $T/I$ satisfies
$\langle\gamma_\Omega,\dimvec Y\rangle\leq0$, and so
$T/I$ belongs to $\overline{\mathscr P}_{\gamma_\Omega}$.

We conclude that $I$ is the torsion submodule of $T$ with respect
to the torsion pair $(\mathscr I_{\gamma_\Omega},
\overline{\mathscr P}_{\gamma_\Omega})$, so
$T_{\gamma_\Omega}^{\min}=I$. The proof of the equality
$T_{\gamma_\Omega}^{\max}=I\oplus R$ is similar.
\end{proof}

\begin{other}{Remarks}
\label{rk:CompPrepQuiv}
\begin{enumerate}
\item
\label{it:CPQa}
This proposition explains our choice of the notation $\mathscr I_\theta$,
$\mathscr R_\theta$ and $\mathscr P_\theta$: when $\theta=\gamma_\Omega$,
the objects of these categories are the $\Lambda$-modules whose
restriction to $Q$ are preinjective, regular or preprojective,
respectively.
\item
\label{it:CPQb}
The abelian category $\mathcal R_Q$ of regular $KQ$-modules is
well understood (see~\cite{Crawley-Boevey92}, \S8). Indecomposable
objects are grouped into tubes, and there is no nonzero morphism or
extension between modules that belong to different tubes. Simple
objects in $\mathcal R_Q$ lie at the mouth of the tubes; two simple
objects are linked if and only if they belong to the same tube. The
sum of the dimension-vectors of the simple objects in a tube
$\mathcal T$ is equal to $\delta$. A tube is called homogeneous
if it has only one simple object; all but at most three tubes
are homogeneous.

This description fits well with Theorem~\ref{th:DescSimpRF}
and Proposition~\ref{pr:SumDimVecSimp}, with
$F=\mathbb R_{>0}\gamma_\Omega$. In fact, using Ringel's description
of $\Lambda$-modules, one easily shows that the arrows in
$\overline\Omega$ act by zero on any simple object in $\mathscr R_F$.
Therefore the map $T\mapsto T\bigl|_Q$ is a bijection from
$\Irr\mathscr R_F$ onto $\Irr\mathcal R_Q$. In this context,
the statements~\ref{it:DSRFa} and~\ref{it:DSRFb} in
Theorem~\ref{th:DescSimpRF} correspond to the cases where
$T\bigl|_Q$ belongs to an homogeneous tube or not.

Two simple objects in $\mathscr R_F$ are linked if their
restrictions to $Q$ are linked in $\Irr\mathcal R_Q$. Using
Proposition~\ref{pr:SumDimVecSimp}, we conclude that the bijection
$T\mapsto T\bigl|_Q$ maps linkage classes in $\Irr\mathscr R_F$
to linkage classes in $\Irr\mathcal R_Q$.
\item
\label{it:CPQc}
Let us keep $F=\mathbb R_{>0}\gamma_\Omega$, let us choose an
extending vertex $0$ in $I$, and let $i\in I_0$ be such that
$\gamma_\Omega$ is $W_0$-conjugated to $\varpi_i$. By
item~\ref{it:CPQb} above, we have $\sum_{S\in L}\dimvec S=\delta$
for each linkage class $L$ in $\mathscr R_F$. As we observed
during the course of the proof of Proposition~\ref{pr:SumDimVecSimp},
this implies that the connected components $J$ of $I_0\setminus\{i\}$
are of type~$A$. This is certainly compatible with the fact, noticed
in the proof of Proposition~\ref{pr:AssChambCow}, that $i$ is the
central node of $I_0$ when $I$ is of type
$\widetilde D$~or~$\widetilde E$.
\end{enumerate}
\end{other}

Let $\nu\in\mathbb NI$ be a dimension-vector. The nilpotent variety
$\Lambda(\nu)$ is a subvariety of the space of representations
of the double quiver $\overline Q$, which itself can be identified
with the cotangent of the space of representations $\Rep(KQ,\nu)$
of the quiver $Q$. It turns out that any irreducible component of
$\Lambda(\nu)$ is the closure of the conormal bundle of a
constructible subset $X\subseteq\Rep(KQ,\nu)$. The relevant subsets
$X$ were first described by Lusztig~\cite{Lusztig92} in the case of
a bipartite orientation $\Omega$, and by Ringel~\cite{Ringel98b} in
the general case of an acyclic orientation. We now explain how this
works.

Recall that an indecomposable $KQ$-module $N$ is regular if and
only if the Auslander-Reiten translation acts periodically on $N$:
there is a number $p>0$ such that $\tau^pN\cong N$. One says that
a finite-dimensional $KQ$-module $M$ is aperiodic if for any
indecomposable regular module $N$ in a non-homogeneous tube,
the sum $\bigoplus_{i=0}^{p-1}\tau^iN$ is not a direct summand of
$M$, where $p\geq2$ is the $\tau$-period of $N$. In addition,
recall that a homogeneous tube $\mathcal T$ contains exactly one
module in each dimension-vector $n\delta$; we denote this module
by $J(\mathcal T,n)$. Lastly, recall that the set of homogeneous
tubes is parameterized by the projective line $\mathbb P^1_K$,
minus at most three points.

Given $\nu\in\mathbb NI$, let $\mathscr S(\nu)$ be the set of all
pairs $(\sigma,\lambda)$, where $\sigma$ is an isomorphism
class of aperiodic modules and $\lambda$ is a partition, with the
further condition $\nu=\dimvec\sigma+|\lambda|\delta$
(see~\cite{Lusztig92}, \S4.13). For $(\sigma,\lambda)\in\mathscr
S(\nu)$, denote by $\ell$ the number of nonzero parts of $\lambda$,
write $\lambda=(\lambda_1,\ldots,\lambda_\ell)$, and let
$X(\sigma,\lambda)$ be the set of all points in
$\Rep(KQ,\nu)$ isomorphic to a module of the form
$M\oplus J(\mathcal T_1,\lambda_1)\oplus\cdots\oplus
J(\mathcal T_\ell,\lambda_\ell)$, where $\mathcal T_1$, \dots,
$\mathcal T_\ell$ are distinct homogeneous tubes and $M$ is an
aperiodic module in the isomorphism class $\sigma$. Let also
$\mathcal N(\sigma,\lambda)$ be the closure of the conormal bundle
of $X(\sigma,\lambda)$. Proposition~4.14 in~\cite{Lusztig92} claims
that the map $(\sigma,\lambda)\mapsto\mathcal N(\sigma,\lambda)$
is a bijection from $\mathscr S(\nu)$ onto $\mathfrak B(\nu)$.
Thanks to~\cite{Ringel98b}, Corollary~5.3,
$\mathcal N(\sigma,\lambda)$ can also be described as the closure
of $\{T\in\Lambda(\nu)\mid T\bigl|_Q\in X(\sigma,\lambda)\}$.

With all these tools in hand, one can prove that
$I(\gamma_\Omega,\lambda)=\mathcal N(0,\lambda)$
for each partition $\lambda$, where $0$ is the
isomorphism class of the trivial module. Thanks to
Proposition~\ref{pr:AssChambCow}, this construction
provides another proof (only valid in type $\widetilde A$)
for the results presented in section~\ref{ss:Cores}.

Pierre Baumann,
Institut de Recherche Math\'ematique Avanc\'ee,
Universit\'e de Strasbourg et CNRS,
7 rue Ren\'e Descartes,
67084 Strasbourg Cedex,
France.\\
\texttt{p.baumann@unistra.fr}
\medskip

Joel Kamnitzer,
Department of Mathematics,
University of Toronto,
Toronto, ON, M5S 2E4
Canada.\\
\texttt{jkamnitz@math.toronto.edu}
\medskip

Peter Tingley,
Department of Mathematics and Statistics,
Loyola University Chicago,
1032 W.\ Sheridan Road,
Chicago, IL 60660,
U.S.A.\\
\texttt{ptingley@luc.edu}

\begin{thebibliography}{99}
\bibitem{AmiotIyamaReitenTodorov10}
C.~Amiot, O.~Iyama, I.~Reiten, G.~Todorov, \textit{Preprojective algebras
and $c$-sortable words,} Proc.\ Lond.\ Math.\ Soc.\ (3) \textbf{104}
(2012), 513--539.
\bibitem{Anderson03}
J.~E.~Anderson, \textit{A polytope calculus for semisimple groups,}
Duke Math.\ J.\ \textbf{116} (2003), 567--588.
\bibitem{Assem90}
I.~Assem, \textit{Tilting theory --- an introduction,} Topics in
algebra, Part 1 (Warszawa, Poland, 1988), ed.\ by S.~Balcerzyk et al.,
Banach Center Publ.\ \textbf{26} (1990), 127--180.
\bibitem{BaumannDunlapKamnitzerTingley12}
P.~Baumann, T.~Dunlap, J.~Kamnitzer, P.~Tingley, \textit{Rank 2 affine
MV polytopes,} to appear in
\href{http://arxiv.org/abs/1202.6416}{Represent.\ Theory}.
\bibitem{BaumannKamnitzer12}
P.~Baumann, J.~Kamnitzer, \textit{Preprojective algebras and MV polytopes,}
Represent.\ Theory \textbf{16} (2012), 152--188.
\bibitem{Beck94}
J.~Beck, \textit{Convex bases of PBW type for quantum affine algebras,}
Comm.\ Math.\ Phys.~\textbf{165} (1994), 193--199.
\bibitem{BeckChariPressley99}
J.~Beck, V.~Chari, A.~Pressley, \textit{An algebraic characterization
of the affine canonical basis,} Duke Math.\ J.\ \textbf{99} (1999),
455--487.
\bibitem{BeckNakajima04}
J.~Beck, H.~Nakajima, \textit{Crystal bases and two-sided cells of
quantum affine algebras,} Duke Math.\ J.\ \textbf{123} (2004), 335--402.
\bibitem{BerensteinZelevinsky01}
A.~Berenstein, A.~Zelevinsky, \textit{Tensor product multiplicities,
canonical bases and totally positive varieties,} Invent.\ Math.\
\textbf{143} (2001), 77--128.
\bibitem{Bourbaki68}
N.~Bourbaki, \textit{Groupes et alg\`ebres de Lie, chapitres 4, 5 et 6,}
Hermann, 1968.
\bibitem{BravermanFinkelbergGaitsgory06}
A.~Braverman, M.~Finkelberg, D.~Gaitsgory, \textit{Uhlenbeck Spaces
via Affine Lie algebra,} The unity of mathematics, Progress in
Mathematics vol.~244, Birkh\"auser Boston, 2006, pp.~17--135.
\bibitem{BravermanGaitsgory01}
A.~Braverman, D.~Gaitsgory, \textit{Crystals via the affine
Grassmannian,} Duke Math.\ J.\ \textbf{107} (2001), 561--575.
\bibitem{Bridgeland09}
T.~Bridgeland, \textit{Spaces of stability conditions,}
Algebraic geometry (Seattle, WA, USA, 2005), ed.~by D.~Abramovich et al.,
Proceedings of Symposia in Pure Mathematics vol.~80, part~1, Amer.\
Math.\ Soc., 2009, pp.~1--21.
\bibitem{BuanIyamaReitenScott09}
A.~B.~Buan, O.~Iyama, I.~Reiten, J.~Scott, \textit{Cluster structures
for $2$-Calabi-Yau categories and unipotent groups,} Compos.\ Math.\
\textbf{145} (2009), 1035--1079.
\bibitem{CelliniPapi98}
P.~Cellini, P.~Papi, \textit{The structure of total reflection orders
in affine root systems,} J.\ Algebra \textbf{205} (1998), 207--226.
\bibitem{Crawley-Boevey92}
W.~Crawley-Boevey, \textit{Lectures on representations of quivers,}
lecture notes for a course given in Oxford in Spring 1992, available
at \href{http://www.amsta.leeds.ac.uk/~pmtwc/}
{\texttt{http://www.amsta.leeds.ac.uk/\raisebox{-4pt}{\string~}pmtwc/}}.
\bibitem{Crawley-Boevey00}
W.~Crawley-Boevey, \textit{On the exceptional fibres of Kleinian
singularities,} Amer.\ J.\ Math.\ \textbf{122} (2000), 1027--1037.
\bibitem{Crawley-BoeveySchroer02}
W.~Crawley-Boevey, J.~Schr\"oer, \textit{Irreducible components of varieties
of modules,} J.\ reine angew.\ Math.~\textbf{553} (2002), 201--220.
\bibitem{Dunlap10}
T.~Dunlap, \textit{Combinatorial representation theory of affine
$\mathfrak{sl}_2$ via polytope calculus,} PhD thesis,
Northwestern University, 2010.
\bibitem{FrenkelMalkinVybornov01}
I.~Frenkel, A.~Malkin, M.~Vybornov, \textit{Affine Lie algebras and
tame quivers,} Selecta Math.\ (N.S.) \textbf7 (2001), 1--56.
\bibitem{FrenkelSavage03}
I.~Frenkel, A.~Savage, \textit{Bases of representations of type $A$
affine Lie algebras via quiver varieties and statistical mechanics,}
Int.\ Math.\ Res.\ Not.\ \textbf{2003}, no.\ 28, 1521--1547.
\bibitem{GabrielRoiter92}
P.~Gabriel, A.~V.~Roiter, \textit{Representations of finite-dimensional
algebras,} with a chapter by B.~Keller, Encyclopaedia Math.\ Sci.\
vol.~73, Algebra~VIII, Springer-Verlag, 1992.
\bibitem{GaussentLittelmann05}
S.~Gaussent, P.~Littelmann, \textit{LS galleries, the path model, and
MV cycles,} Duke Math.\ J.\ \textbf{127} (2005), 35--88.
\bibitem{GelfandGoreskyMacPhersonSerganova87}
I.~M.~Gelfand, R~.M~.Goresky, R.~D.~MacPherson, V.~V.~Serganova,
\textit{Combinatorial geometries, convex polyhedra, and Schubert cells,}
Adv.\ Math.\ \textbf{63} (1987), 301--316.
\bibitem{GeissLeclercSchroer06}
C.~Gei\ss, B.~Leclerc, J.~Schr\"oer, \textit{Rigid modules over
preprojective algebras,} Invent.\ Math.\ \textbf{165} (2006), 589--632.
\bibitem{GeissLeclercSchroer07}
C.~Gei\ss, B.~Leclerc, J.~Schr\"oer, \textit{Semicanonical basis and
preprojecive algebras. II: a multiplication formula,} Compos.\ Math.\
\textbf{143} (2007), 1313--1334.
\bibitem{GeissLeclercSchroer11}
C.~Gei\ss, B.~Leclerc, J.~Schr\"oer, \textit{Kac-Moody groups and
cluster algebras,} Adv.\ Math.\ \textbf{228} (2011), 329--433.
\bibitem{Ito01}
K.~Ito, \textit{The classification of convex orders on affine root
systems,} Commun.\ Algebra \textbf{29} (2001), 5605--5630.
\bibitem{Ito05}
K.~Ito, \textit{Parametrizations of infinite biconvex sets in affine
root systems,} Hiroshima Math.\ J.\ \textbf{35} (2005), 425--451.
\bibitem{Ito10}
K.~Ito, \textit{A new description of convex bases of PBW type for
untwisted quantum affine algebras,} Hiroshima Math.\ J.\ \textbf{40}
(2010), 133--183.
\bibitem{IyamaReiten08}
O.~Iyama, I.~Reiten, \textit{Fomin-Zelevinsky mutation and tilting
modules over Calabi-Yau algebras,} Amer.\ J.\ Math.\ \textbf{130}
(2008), 1087--1149.
\bibitem{IyamaReiten10}
O.~Iyama, I.~Reiten, \textit{$2$-Auslander algebras associated with
reduced words in Coxeter groups,} Int.\ Math.\ Res.\ Not.\
\textbf{2011}, no.\ 8, 1782--1803.
\bibitem{Jiang13}
Y.~Jiang, \textit{Parametrizations of canonical bases and
irreducible components of nilpotent varieties,} to appear in
\href{http://arxiv.org/abs/1110.2937}{Int.\ Math.\ Res.\ Not.}
\bibitem{Kac90}
V.~G.~Kac, \textit{Infinite dimensional Lie algebras,} 3rd ed.,
Cambridge University Press, 1990.
\bibitem{Kamnitzer07}
J.~Kamnitzer, \textit{The crystal structure on the set of Mirkovi\'c-Vilonen
polytopes,} Adv.\ Math.~\textbf{215} (2007), 66--93.
\bibitem{Kamnitzer10}
J.~Kamnitzer, \textit{Mirkovi\'c-Vilonen cycles and polytopes,}
Ann.\ of Math.~\textbf{171} (2010), 245--294.
\bibitem{Kashiwara95}
M.~Kashiwara, \textit{On crystal bases,} Representations of groups
(Banff, AB, 1994), CMS Conf.\ Proc.\ vol.~16 (1995), pp.~155--197.
\bibitem{KashiwaraSaito97}
M.~Kashiwara, Y.~Saito, \textit{Geometric construction of crystal bases,}
Duke Math.\ J.\/ \textbf{89} (1997), 9--36.
\bibitem{Kimura07}
Y.~Kimura, \textit{Affine quivers and crystal bases,} Master thesis,
University of Kyoto (Japan), 2007.
\bibitem{King94}
A. King, \textit{Moduli of representations of finite dimensional
algebras,} Quart.\ J.\ Math.\ Oxford (2) \textbf{45} (1994), 515--530.
\bibitem{Lusztig90a}
G.~Lusztig, \textit{Canonical bases arising from quantized enveloping
algebras,} J.\ Amer.\ Math.\ Soc.~\textbf3 (1990), 447--498.
\bibitem{Lusztig90b}
G.~Lusztig, \textit{Canonical bases arising from quantized enveloping
algebras II,} Progr.\ Theoret.\ Phys.\ Suppl. \textbf{102} (1990),
175--201.
\bibitem{Lusztig91}
G.~Lusztig, \textit{Quivers, perverse sheaves, and quantized enveloping
algebras,} J.\ Am.\ Math.\ Soc.\ \textbf4 (1991), 365--421.
\bibitem{Lusztig92}
G.~Lusztig, \textit{Affine quivers and canonical bases,} Inst.\ Hautes
\'Etudes Sci.\ Publ.\ Math.\ \textbf{76} (1992), 111--163.
\bibitem{Lusztig93}
G.~Lusztig, \textit{Introduction to quantum groups,} Progress in
Mathematics vol.~110, Birkh\"auser Boston, 1993.
\bibitem{Lusztig96}
G.~Lusztig, \textit{Braid group action and canonical bases,}
Adv.\ Math.\ \textbf{122} (1996), 237--261.
\bibitem{MirkovicVilonen04}
I.~Mirkovi\'c, K.~Vilonen, \textit{Geometric Langlands duality
and representations of algebraic groups over commutative rings,}
Ann.\ of Math.\ \textbf{166} (2007), 95--143.
\bibitem{Mumford99}
D.~Mumford, \textit{The red book of varieties and schemes,}
Lecture Notes in Mathematics, vol.~1358, Springer-Verlag, 1999.
\bibitem{Muthiah11}
D.~Muthiah, \textit{Double MV cycles and the Naito-Sagaki-Saito crystal,}
preprint \href{http://arxiv.org/abs/1108.5404}{arXiv:1108.5404}.
\bibitem{MuthiahTingley13}
D.~Muthiah, P.~Tingley, \textit{Affine PBW bases and MV polytopes
in rank $2$,} \href{http://dx.doi.org/10.1007/s00029-012-0117-z}{to
appear in Selecta Math.}
\bibitem{NaitoSagakiSaito12}
S.~Naito, D.~Sagaki, Y.~Saito, \textit{Toward Berenstein-Zelevinsky data
in affine type A, parts~I and~II,} Representation theory of algebraic
groups and quantum groups '10 (Nagoya, Japan, 2010), ed.\ by S.~Ariki et
al., Contemp.\ Math.\ vol.~565, Amer.\ Math.\ Soc., 2012, pp.~143--216.
\bibitem{NaitoSagakiSaito13}
S.~Naito, D.~Sagaki, Y.~Saito, \textit{Toward Berenstein-Zelevinsky data
in affine type A, part~III,} Symmetries, integrable systems and
representations (Tokyo, Japan; Lyon, France, 2011), ed.\ by K.~Iohara
et al., Springer Proceedings in Mathematics and Statistics vol.~40,
Springer-Verlag, 2013, pp.~361--402.
\bibitem{Reineke03}
M.~Reineke, \textit{The Harder-Narasimhan system in quantum groups
and cohomology of quiver moduli,} Invent.\ Math.\ \textbf{152} (2003),
349--368.
\bibitem{Ringel98a}
C.~M.~Ringel, \textit{The preprojective algebra of a quiver,} Algebras
and modules II, Eighth international conference on representations of
algebras (Geiranger, Norway, 1996), ed.\ by I.~Reiten, CMS conf.\ Proc.\
vol.~24, Amer.\ Math.\ Soc., 1998, pp.~467--480.
\bibitem{Ringel98b}
C.~M.~Ringel, \textit{The preprojective algebra of a tame quiver: the
irreducible components of the module varieties,} Trends in the
representation theory of finite dimensional algebras (Seattle, WA, USA,
1997), ed.\ by E.~L.~Green et al., Contemp.\ Math.\ vol.~229, Amer.\
Math.\ Soc., 1998, pp.~293--306.
\bibitem{Rudakov97}
A.~Rudakov, \textit{Stability for an abelian category,} J.\ Algebra
\textbf{197} (1997), 231--245.
\bibitem{Saito94}
Y.~Saito, \textit{PBW basis of quantized universal enveloping algebras,}
Publ.\ Res.\ Inst.\ Math.\ Sci.\ \textbf{30} (1994), 209--232.
\bibitem{Shatz77}
S.~Shatz, \textit{The decomposition and specilization of algebraic
families of vector bundles,} Compos.\ Math.\ \textbf{35} (1977), 163--187.
\bibitem{SekiyaYamaura13}
Y.~Sekiya, K.~Yamaura, \textit{Tilting theoretical approach to moduli
spaces over preprojective algebras,} to appear in
\href{http://dx.doi.org/10.1007/s10468-012-9380-0}{Algebr.\ Represent.\
Theory}.
\end{thebibliography}
\end{document}